\RequirePackage{ifpdf}
\documentclass[phd,tocprelim]{cornell}
\usepackage{graphicx}
\usepackage{graphics}
\usepackage{subfigure}
\usepackage{epsfig}
\usepackage{subfigure}
\usepackage{caption}
\usepackage{palatino}

\usepackage{amsthm,amsmath,amsfonts}
\usepackage{color}
\usepackage[colorlinks,linkcolor=blue,citecolor=cyan,anchorcolor=blue]{hyperref}
\usepackage[numbers,sort]{natbib}
\usepackage{enumerate}
\usepackage{wrapfig}
\usepackage{epstopdf}
\newcommand{\R}{\mathbb{R}}
\newcommand{\D}{\mathbb{D}}

\newcommand{\Z}{\mathbb{Z}}
\newcommand{\E}{\mathbb{E}}
\newcommand{\EE}{\mathbb{E}}
\newcommand{\HH}{\mathcal{H}}
\newcommand{\lipone}{\text{\rm Lip(1)}}
\newcommand{\Prob}{\mathbb{P}}

\providecommand{\abs}[1]{\left\lvert#1\right\rvert}
\providecommand{\norm}[1]{\lVert#1\rVert}
\usepackage[normalem]{ulem}

\newtheorem{lemma}{Lemma}
\newtheorem{theorem}{Theorem}
\newtheorem{corollary}{Corollary}
\theoremstyle{definition}

\newtheorem{remark}{Remark}
\numberwithin{equation}{chapter}
\numberwithin{theorem}{chapter}
\numberwithin{lemma}{chapter}
\numberwithin{remark}{chapter}
\title {Stein's method for steady-state diffusion approximations}
\author {Anton Braverman}
\conferraldate {May}{2017}
\degreefield {Ph.D.}
\copyrightholder{Anton Braverman}
\copyrightyear{2017}

\begin{document}

\maketitle
\makecopyright

\begin{abstract}
Diffusion approximations have been a popular tool for performance analysis in queueing theory, with the main reason being tractability and computational efficiency. This dissertation is concerned with establishing theoretical guarantees on the performance of steady-state diffusion approximations of queueing systems. We develop a modular framework based on Stein's method that allows us to establish error bounds, or convergence rates, for the approximations. We apply this framework three queueing systems: the Erlang-C, Erlang-A, and $M/Ph/n+M$ systems. 

The former two systems are simpler and allow us to showcase the full potential of the framework. Namely, we prove that both
  Wasserstein and Kolmogorov distances between the stationary
  distribution of a normalized customer count process, and that of an
  appropriately defined diffusion process decrease at a rate of
  $1/\sqrt{R}$, where $R$ is the offered load. Futhermore, these error
  bounds are \emph{universal}, valid in any load condition from
  lightly loaded to heavily loaded. For the Erlang-C model, we also show that a diffusion approximation with state-dependent diffusion coefficient can achieve a rate of convergence of $1/R$, which is an order of magnitude faster when compared to approximations with constant diffusion coefficients. 
\end{abstract}

\begin{biosketch}
Anton received his Bachelors degree in Statistics and Mathematics from the University of Toronto in 2012. 
\end{biosketch}

%\begin{dedication}
%\end{dedication}

\begin{acknowledgements}
I am grateful to my parents for raising me to set high standards in life, and to my grandfather for being uncompromising in his attitude towards education. I am eternally grateful to Jim Dai for taking me on as an apprentice, and showing me how to stand on my own two feet. An even deeper sentiment goes out to my friends, who always provided me with an escape to normality whenever it was needed.
\\\\
\noindent Half of this research was fueled by my grandmother's delicious cooking. 
\end{acknowledgements}

\contentspage
%\tablelistpage
%\figurelistpage

\normalspacing \setcounter{page}{1} \pagenumbering{arabic}
\pagestyle{cornell} \addtolength{\parskip}{0.5\baselineskip}

\chapter{Introduction} \label{chap:intro}
Diffusion approximations have been a popular tool for performance analysis in queueing theory, with the main reason being tractability and computational efficiency. 
%As we will shortly see, these error bounds can be interpreted as convergence rates, and so we use the two terms interchangeably. 
As an example, in  \cite{DaiHe2013}, an algorithm was developed to compute the stationary distribution of the diffusion approximation of the $M/H_2/500+M$ system, which is a many-server queue with $500$ servers, Poisson arrivals, hyper-exponential service times and customer abandonment. 
The approximation is remarkably accurate; see,
for example, Figure 1 there. It was demonstrated there that computational efficiency, in terms of both time and memory, can be
achieved by diffusion approximations. For example, in the $M/H_2/500+M$
system, it took around 1
hour and peak memory usage of 5 GB to compute the stationary distribution of the customer count.  On the same computer,  it took less than 1 minute to compute the
stationary distribution of the corresponding diffusion approximation, and peak memory usage was less than 200 MB. This dissertation is concerned with establishing error bounds on steady-state diffusion approximations of queueing systems.

 The main technical driver of our results is a mathematical framework known as Stein's method. Stein's method is a powerful method used for studying approximations of probability distributions, and is best known for its ability to establish convergence rates. It has been widely used
in probability, statistics, and their wide range of applications such
as bioinformatics; see, for example, the survey papers \cite{Ross2011,
  Chat2014}, the recent book \cite{ChenGoldShao2011} and the
references within. Applications of Stein's method always involve some unknown distribution to be approximated, and an approximating distribution. For instance, the first appearance of the method  in \cite{Stei1972} involved the sum of identically distributed dependent random variables as the unknown, and the normal as approximating distribution. Other approximating distributions include the Poisson \cite{Chen1975}, binomial \cite{Ehm1991}, and multinomial \cite{Loh1992} distributions, just to name a few. To begin our discussion, we provide an example to illustrate the type of result that will be frequently encountered in this document.

\section{A Typical Result} \label{sec:INtypical}
Consider the Erlang-A and Erlang-C queuing systems. Both systems have $n$ homogeneous servers that serve customers in a first-come-first-serve manner. Customers arrive according to a Poisson process with rate $\lambda$, and customer service times are assumed to be i.i.d.\ having exponential distribution with mean $1/\mu$. In the Erlang-A system, each customer
has a patience time and when his waiting time in queue exceeds his
patience time, he abandons the queue without service; the patience
times are assumed to be i.i.d.\ having exponential distribution with mean
$1/\alpha$. We consider the birth-death process
\begin{align}
X=\{X(t), t\ge 0\}, \label{eq:CTMCunscaleddef}
\end{align} 
 where $X(t)$ is the number of customers in the system at time $t$.  In the Erlang-A system, $\alpha$ is assumed to be positive and
therefore the mean patience time is finite. This guarantees that the CTMC $X$ is positive recurrent. In the Erlang-C system, $\alpha=0$, and in order for the CTMC to be positive recurrent we need to assume that the offered load to the system, defined as $R = \lambda / \mu$, satisfies
 \begin{equation}
  \label{eq:erlang-cstable}
  R  < n.
\end{equation}
For both Erlang-A and Erlang-C systems, we use $X(\infty)$ to denote the random variable having the stationary distribution of $X$. 

Consider the case when $\alpha = 0$ and \eqref{eq:erlang-cstable} is satisfied. Set 
\begin{align*}
\tilde X(\infty) = (X(\infty) - R) /\sqrt{R},
\end{align*} 
and let $Y(\infty)$ denote a continuous random
variable on $\R$ having density
\begin{equation}
  \label{eq:stddenC}
\kappa \exp\Big(\frac{1}{\mu} {\int_0^xb(y)dy}\Big), \quad x \in \R,
\end{equation}
where $\kappa>0$ is a normalizing constant that makes the density integrate to one,
\begin{equation}
b(x) = \big((x+\zeta)^--\zeta^-\big)\mu \quad \text{ for $x \in \R$}, \quad \text{ and } \quad \zeta =\big(R -n\big)/\sqrt{R}. \label{eq:bandz}
\end{equation}
Although our choice of notation does not make this explicit, we highlight that the random variable $Y(\infty)$ depends on $\lambda, \mu$, and $n$, meaning that we are actually dealing with a \emph{family} of random variables $\{Y^{(\lambda, \mu, n)}(\infty)\}_{(\lambda, \mu, n)}$. The following theorem illustrates the type of result that can be obtained by Stein's method. 
\begin{theorem}
\label{thm:erlangCW}
Consider the Erlang-C system. For all $n \geq 1,\\ \lambda > 0$, and $\mu > 0$ satisfying $1 \leq R < n$,
\begin{equation}
  \label{eq:erlangCW}
d_W(\tilde X(\infty), Y(\infty)) :=  \sup \limits_{h(x) \in {\lipone}}  \big|\E h(\tilde X(\infty)) - \E h(Y(\infty))\big|\leq  \frac{190}{\sqrt{R}},
\end{equation}
where 
\begin{equation*}
\lipone=\{h: \R\to\R, \abs{h(x)-h(y)}\le
\abs{x-y},\ x,y \in \R\}.
\end{equation*}
\end{theorem}
Several points are worth mentioning. First, we note that Theorem~\ref{thm:erlangCW} is not a limit theorem. Steady-state approximations are usually justified by some kind of limit theorem. That is, one considers a sequence of queueing systems and proves that the corresponding sequence of steady-state distributions converges to some limiting distribution as traffic intensity approaches one, or as the number of servers goes to infinity. In contrast, our theorem holds for any finite parameter choices of $\lambda, n$, and $\mu $ satisfying \eqref{eq:erlang-cstable} and $R \geq 1$. Second, the error bound in \eqref{eq:erlangCW} is \emph{universal}, as it does not assume any relationship between $\lambda,n$, and $\mu$, other than the stability condition \eqref{eq:erlang-cstable} and the condition that $R \geq 1$.  Universal approximations were previously studied in \cite{GlynWard2003, GurvHuanMand2014}. One consequence of universality is that the error bound holds  when  parameters $\lambda,n$, and $\mu$  fall in  one of the 
following asymptotic regimes:
\begin{align*}
&n = \left \lceil R+ \beta R\right \rceil, \quad n = \left \lceil R + \beta \sqrt{R}\right \rceil,  \quad \text{ or }  \quad n = \left \lceil R + \beta \right \rceil,
 \end{align*}
 where $\beta > 0$ is fixed, while  $R \to \infty$.  
The first two parameter regimes above describe the quality-driven (QD), and quality-and-efficiency-driven (QED) regimes, respectively. The last regime is the nondegenerate-slowdown (NDS) regime, which was studied in \cite{Whit2003, Atar2012}. Third, as part of the universality of Theorem~\ref{thm:erlangCW}, we see that 
\begin{align}
\big|\E X(\infty) - \big( R + \sqrt{R}\E Y(\infty) \big) \big| \leq 190. \label{eq:const_error}
\end{align}
For a fixed $n$, let $\rho = R/n \uparrow 1$. One expects that $\E X(\infty)$ be on the order of $1/(1-\rho)$. Conventional heavy-traffic limit theorems often guarantee that the left hand side of \eqref{eq:const_error} is at most $o(1/\sqrt{1-\rho})$, whereas our error is bounded by a constant regardless of the load condition. This suggests that the diffusion approximation for the Erlang-C system is accurate not only as $R \to \infty$, but also in the heavy-traffic setting when $R \to n$. Table~\ref{tab1} contains some numerical results where we calculate the error on the left side of \eqref{eq:const_error}. The constant $190$ in \eqref{eq:const_error} is unlikely to be a sharp upper bound. In this thesis, we do not focus on optimizing such upper bounds, as Stein's method is not known for producing sharp constants.
\begin{table}[bt]
  \begin{center}
  \begin{tabular}{rcc | r cc }
\multicolumn{3}{c}{$n=5$} & \multicolumn{3}{c}{$n=500$} \\
\hline
$R$ & $\E X(\infty)$ & Error &
$R$ & $\E X(\infty)$ & Error \\
\hline
 3        &   3.35  &0.10  & 300 & 300.00 & $6\times 10^{-14}$  \\
 4        &  6.22  & 0.20  & 400 & 400.00 & $2\times 10^{-6}$  \\
 4.9        &  51.47  &0.28 & 490& 516.79 & 0.24   \\
 4.95        &  101.48  & 0.29 & 495& 569.15& 0.28  \\
 4.99        & 501.49  & 0.29 & 499 & 970.89& 0.32 \\
  \end{tabular}
  \caption{Comparing the error $\big|\E X(\infty) - \big( R + \sqrt{R}\E Y(\infty) \big) \big|$ for different system configurations.\label{tab1}}
  \end{center}
\end{table}
Theorem~\ref{thm:erlangCW} provides rates of convergence under the Wasserstein metric \cite{Ross2011}. The Wasserstein metric $d_W(\cdot, \cdot)$ is one of the most commonly studied metrics when Stein's method is concerned. This is because the the space $\lipone$ is relatively simple to work with, but is also rich enough so that convergence under the Wasserstein metric implies the convergence in
distribution \cite{GibbSu2002}. 

\section{Outline of the Stein Framework} \label{sec:INoutline}
Having seen the example in the previous section, let us briefly outline the key components of Stein's method. These are the Poisson equation, generator comparison, gradient bounds, moment bounds, and state-space collapse (SSC).  The generator comparison idea is also known as the generator approach, and is attributed to Barbour \cite{Barb1988, Barb1990} and G{\"o}tze \cite{Gotz1991}. Chapter~\ref{chap:erlangAC} is devoted to a detailed walkthrough of the first four of these components, while the SSC component is not required until Chapter~\ref{chap:phasetype}.

Consider two sequences of
stochastic processes $\{X^{(\ell)}\}_{\ell =1}^{\infty}$ and
$\{Y^{(\ell)}\}_{\ell=1}^{\infty}$ indexed by $\ell$, where $X^{(\ell)} = \{X^{(\ell)}(t)\in
\R^d, t \geq 0\}$ is a continuous-time Markov chain (CTMC) and $Y^{(\ell)} = \{Y^{(\ell)}(t) \in \R^d, t \geq 0\}$
is a diffusion process.
% We first assume the state spaces of $X^n$ and $Y^n$ have the same dimension.
%  There is nothing special about CTMC's and this
% framework holds more generally when $X^n$ is any Markov process taking
% values in a general state space. 
Suppose $X^{(\ell)}(\infty)$ and
$Y^{(\ell)}(\infty)$ are two random vectors having the stationary
distributions of $X^{(\ell)}$ and $Y^{(\ell)}$, respectively. % As $n$ grows, we
% expect $Y^n(\infty)$ and $X^n(\infty)$ to get closer together. 
Let $G_{X^{(\ell)}}$ and $G_{Y^{(\ell)}}$ be the generators of $X^{(\ell)}$ and $Y^{(\ell)}$,
respectively;  for a diffusion process, $G_{Y^{(\ell)}}$ is a second order
elliptic differential operator. For a function $h: \R^d \to \R$ in a "nice" (but
large enough) class, we wish to bound
\begin{displaymath}
\abs{\E h(X^{(\ell)}(\infty)) - \E h(Y^{(\ell)}(\infty))}.
\end{displaymath}
% \begin{displaymath}
% \sup \limits_{h} \abs{\E h(X^n(\infty)) - \E h^n(Y(\infty))} = O(1/n^\epsilon), \quad \text{ where $\epsilon > 0$}.
% \end{displaymath}
 % One generally expects a central limit theorem type convergence rate of $\epsilon = 1/2$. The reason for the slower convergence rate in (\ref{eq:introresult}) is due to the SSC component and will be discussed later on.;

 The first component is to set up the Poisson equation 
\begin{equation} \label{eq:intropoisson}
G_{Y^{(\ell)}} f_h(x) =\E h(Y^{(\ell)}(\infty)) -  h(x), \quad x \in \R^{d}.
\end{equation}
We then take the expectation of both sides above to see that
\begin{equation}
  \label{eq:PoissonRep}
\E h(Y^{(\ell)}(\infty)) - \E h(X^{(\ell)}(\infty))  = \E G_{Y^{(\ell)}} f_h(X^{(\ell)}(\infty)).  
\end{equation}
 When $d = 1$, the Poisson equation \eqref{eq:intropoisson} is an ordinary differential equation (ODE), and when $d > 1$, it is a partial
differential equation (PDE). To execute Stein's method, we require bounds on the derivatives of $f_h(x)$ (usually up to the third derivative). We refer to these as gradient bounds.

The next step is to rely on the following relationship between the generator and stationary distribution of a CTMC. One can check that a random vector $X^{(\ell)}(\infty) \in \R^d$ has the
stationary distribution of the CTMC $X^{(\ell)}$ if and only if
\begin{equation}
  \label{eq:bar}
\E G_{X^{(\ell)}} f (X^{(\ell)}(\infty)) = 0
\end{equation}
for all functions $f:\R^d\to \R$ that have compact support.  For a given $h(x)$, the
corresponding Poisson equation solution $f_h(x)$ does not have compact
support, but it is typically not hard to prove that
\eqref{eq:bar} continues to hold for $f_h(x)$.  Thus, it follows from \eqref{eq:PoissonRep} and \eqref{eq:bar} that
\begin{equation} \label{eq:introgencoup}
\E h(Y^{(\ell)}(\infty))- \E h(X^{(\ell)}(\infty)) = \E [G_{Y^{(\ell)}} f_h(X^{(\ell)}(\infty)) - G_{X^{(\ell)}} f_h (X^{(\ell)}(\infty))].
\end{equation}

The focus now falls on bounding the right side of  (\ref{eq:introgencoup}). To do so, we study
\begin{equation}
  \label{eq:generatorDiff}
  G_{X^{(\ell)}}f_h(x) -G_{Y^{(\ell)}}f_h(x)
\end{equation}
for each $x$ in the state space of $X^{(\ell)}$. By performing Taylor
expansion on $G_{X^{(\ell)}}f_h(x)$, we find that the difference in \eqref{eq:generatorDiff} involves the product of partial derivatives of $f_h(x)$ and terms related to the transition structure of $X^{(\ell)}$. The former are why we need gradient bounds, and the latter can typically be bounded by a polynomial of $x$.
% \begin{displaymath}
%   \frac{1}{\sqrt{n}} \text{gradient of } f_h(x)  \abs{x}^m.
% \end{displaymath}
Therefore, we also need bounds on various moments of
$\abs{X^{(\ell)}(\infty)}$,  which we refer to as moment bounds. The main challenge is that both gradient and moment bounds must be
\textit{uniform} in $\ell$. Once we have both gradient and moment bounds, the right hand side of \eqref{eq:generatorDiff} can be bounded. 

We point out that this procedure can also be carried out when $X^{(\ell)}$ itself is not a
CTMC, but a function of some higher dimensional CTMC $U^{(\ell)} = \{U^{(\ell)}(t)
\in \mathcal{U}, t \geq 0\}$, where the dimension of the state space
$\mathcal{U}$ is strictly greater than $d$. When this is the case, the CTMC $U^{(\ell)}$ must exhibit some form of SSC.  
This is the case in Chapter~\ref{chap:phasetype}, where we study the $M/Ph/n+M$ system.
This difference in
dimensions is partly responsible for  the computational speedup in diffusion
approximations; most complex stochastic processing systems
exhibit some form of SSC
\cite{Reim1984b,BellWill2005,FoscSalz1978,Harr1998,HarrLope1999,
  Whit1971,Will1998a,Bram1998a,DaiTezc2011,EryiSrik2012}.
 Let $G_U$ be the generator of $U^{(\ell)}$ and
$U^{(\ell)}(\infty)$ have its stationary distribution. Now, BAR \eqref{eq:bar} becomes 
$G_{U^{(\ell)}} F(U^{(\ell)}(\infty))=0$ for each `nice' $F:\mathcal{U}\to \R$. Furthermore,
(\ref{eq:introgencoup}) becomes
\begin{equation} \label{eq:introgencoupSSC}
\E h(X^{(\ell)}(\infty)) - \E h(Y^{(\ell)}(\infty)) = \E [G_{Y^{(\ell)}} f_h(X^{(\ell)}(\infty)) - G_{U^{(\ell)}} F_h(U^{(\ell)}(\infty))],
\end{equation}
where $F_h: \mathcal{U}\to \R$ is the lifting of $f_h:\R^d\to \R$ defined by letting $x \in \R^d$ be the projection of $u \in \mathcal{U}$ and then setting 
\begin{equation}
  \label{eq:lifting}
  F_h(u) = f_h(x).
\end{equation}
As before, we can perform Taylor expansion on $G_{U^{(\ell)}}F_h(u)$ to simplify
the difference $G_{U^{(\ell)}}F_h(u)-G_{Y^{(\ell)}}f_h(x)$.  To use this difference to bound
the right side of (\ref{eq:introgencoupSSC}), we need a steady-state
SSC result for $U^{(\ell)}(\infty)$, which tells us how to approximate
$U^{(\ell)}(\infty)$ from $X^{(\ell)}(\infty)$ and guarantees that this
approximation error is small. Typically, such an SSC result relies heavily on the structure of $U^{(\ell)}$. The SSC component will not appear until Chapter~\ref{chap:phasetype}.

\section{Related Literature}
This dissertation lies at the intersection of two mathematical communities: the queueing theory, and  Stein method communities. It is therefore appropriate to separate the literature review into two parts. We begin with the literature from queueing theory.

Diffusion approximations are a popular tool in queueing theory, and are usually ``justified'' by heavy traffic limit theorems. For example, a typical limit theorem would say that an appropriately scaled and centered version of the process $X$ in \eqref{eq:CTMCunscaleddef} converges to some limiting diffusion process as the system utilization $\rho$ tends to one.
%It is proved in \cite{DaiHeTezc2010} that for our $M/Ph/n+M$ systems,
%\begin{equation}
%  \label{eq:processconvere}
%  \tilde X^{(\lambda)}=\{\tilde X^{(\lambda)}(t), t\ge 0\} \Longrightarrow  Y=\{Y(t), t\ge
%0\}
%\end{equation}
%as $\lambda$ goes to infinity while satisfying (\ref{eq:square-root}) (we use the arrival rate $\lambda$ to index these systems instead of the abstract $\ell$ as before).
Proving such limit theorems has been an active area of research in
the last 50 years; see, for example,
\cite{Boro1964,Boro1965,IgleWhit1970,IgleWhit1970a, Harr1978,Reim1984}
for single-class queueing networks, \cite{Pete1991, Bram1998a,
  Will1998a} for multiclass queueing networks,
\cite{KangKellLeeWill2009, YaoYe2012} for bandwidth sharing networks,
\cite{HalfWhit1981,Reed2009,DaiHeTezc2010} for many-server queues. The
convergence used in these limit theorems is the convergence in
distribution on the path space $\D([0, \infty), \R^d)$, endowed with
Skorohod $J_1$-topology \cite{EthiKurt1986,Whit2002}. The
$J_1$-topology on $\D([0,\infty), \R^d)$ essentially means convergence
in $\D([0, T], \R^d)$ for each $T>0$. In particular, it says nothing
about the convergence at ``$\infty$''. Therefore, these limit theorems
do not justify steady-state convergence.

The jump from convergence on $\D([0, T], \R^d)$ to convergence of stationary distributions was first established in the
seminal paper \cite{GamaZeev2006}, where the authors prove an
interchange of limits for generalized Jackson networks of
single-server queues.  The results in \cite{GamaZeev2006} were
improved and extended by various authors for networks of
single-servers \cite{BudhLee2009, ZhanZwar2008, Kats2010}, for
bandwidth sharing networks~\cite{YaoYe2012}, and for many-server
systems \cite{Tezc2008, GamaStol2012, Gurv2014a}.  These ``interchange
of limits'' theorems are qualitative and thus do not provide rates of
convergence as in Theorem~\ref{thm:erlangCW}.

The first paper to have established convergence rates for steady-state diffusion approximations was \cite{GurvHuanMand2014}, which studied the Erlang-A system (many-server queue with customer abandonment) using an excursion based approach. Their approximation error bounds are universal. Although the authors in \cite{GurvHuanMand2014} did not study the Erlang-C system, their approach appears to be extendable to it as well. However, their method is not readily generalizable to the multi-dimensional setting. 

Following \cite{GurvHuanMand2014}, Gurvich \cite{Gurv2014} develops an approach to prove statements similar to Theorem~\ref{thm:erlangCW} for various queueing systems. Along the way, he independently rediscovers
many of the ideas central to Stein's method in the setting of
steady-state diffusion approximations. In particular, Gurvich's results also rely on the Poisson equation, generator comparison, gradient and moment bounds components discussed in Section~\ref{sec:INoutline}. Gurvich packages the necessary conditions to establish convergence rates into a single condition which requires the existence of a uniform Lyapunov function for the diffusion processes. In particular, this Lyapunov function provides the necessary moment and gradient bounds to establish convergence rates. However, his results are no longer immediately applicable when the SSC component is required, i.e.\ when dim$(\mathcal{U}) > d$. In contrast, Stein's method is a \textit{modular} framework. It allows one to separate a problem into its components, e.g.\ gradient bounds, moment bounds etc., and treat the difficulties of each component in isolation. 

%The immediate benefit we gain is the ability to apply this framework to cases when SSC occurs (dim$(\mathcal{U}) > d$). Moreover, although we also rely on Lyapunov functions to establish both moment and gradient bounds in our particular setting, our framework clearly illustrates that Lyapunov functions are merely tools one can use to establish these moment and gradient bounds; the bounds themselves are the actual drivers of our main results.
%  He relies on the existence of uniform Lyapunov functions for the diffusion processes. Putting the
%Lyapunov functions together with the probabilistic solution for
%(\ref{eq:intropoisson}) and a-priori Schauder estimates for elliptic
%PDEs (see \cite{GilbTrud1983}), he is able to obtain uniform gradient
%bounds for a large class of Poisson equations. Furthermore, he also
%obtains the necessary uniform moment bounds using these Lyapunov
%functions by showing that uniform moment bounds for the diffusion
%process imply the same moments are uniformly bounded for the
%CTMC. However, his result on uniform moment bounds no longer holds
%when dim$(\mathcal{U}) > d$ due to the need for SSC, which poses an
%additional technical challenge. We overcome this challenge for the
%$M/Ph/n+M$ system in Lemma~\ref{lemma:CTMCmoments}, in which moment
%bounds are established \emph{recursively}.

%The work in \cite{Gurv2014} is conceptually close to this paper. In that paper, Gurvich packages all the components required to prove his results into several conditions, with the main condition being the existence of uniform Lyapunov functions for the diffusion processes. 

We now discuss the relevant literature in the Stein method community. The first uses of Stein's method for stationary distributions of Markov processes traces back to \cite{Barb1988}, where it is pointed out that Stein's method can be applied anytime the approximating distribution is the stationary distribution of a Markov proccess.  That paper considers the multivariate Poisson, which is the stationary distribution of a certain multi-dimensional birth-death process. One of the major contributions of \cite{Barb1988} was to show how viewing the Poisson distribution as the stationary distribution of a Markov chain could be exploited to establish gradient bounds using coupling arguments; cf. the discussion around \eqref{eq:relval} of this document. A similar idea was subsequently used for the multivariate normal distribution through its connection to the multi-dimensional Ornstein--Uhlenbeck process in \cite{Barb1990, Gotz1991}.

Of the papers that use the connection between Stein's method and Markov processes, \cite{BrowXia2001} and the more recent \cite{KusuTudo2012} are the most relevant to this work. The former studies one-dimensional birth-death processes, with the focus being that many common distributions such as the Poisson, Binomial, Hypergeometric, Negative Binomial, etc., can be viewed as stationary distributions of a birth-death process. Although the Erlang-A and Erlang-C models are also birth-death processes, the focus in Chapter~\ref{chap:erlangAC} is on how well these models can be approximated by diffusions, e.g.\ qualitative features of the approximation like the universality in Theorem~\ref{thm:erlangCW}. Diffusion approximations go beyond approximations of birth-death processes, with the real interest lying in cases when a higher-dimensional Markov chain collapses to a one-dimensional diffusion, e.g.\ \cite{Tezc2008,Stol2004,DaiLin2008}, or when the diffusion approximation is multi-dimensional \cite{Harr1978,Reim1984,Pete1991, Bram1998a,  Will1998a}.

In \cite{KusuTudo2012}, the authors apply Stein's method to one-dimensional diffusions. The motivation is again that many common distributions like the gamma, uniform, beta, etc.,  happen to be stationary distributions of diffusions. Their chief result is to establish gradient bounds for a very large class of diffusion processes, requiring only the mild condition that the drift of the diffusion be a decreasing function. However, their result cannot be applied here, because it is impossible to say how their gradient bounds depend on the parameters of the diffusion. Detailed knowledge of this dependence is crucial, because we are dealing with a \emph{family} of approximating  distributions;  cf. \eqref{eq:stddenC} and the comments below (\ref{eq:bandz}).

Outside the diffusion approximation domain, Ying has recently
  successfully applied Stein's framework to establish error bounds
  for steady-state mean-field approximations \cite{ying2016,ying2016b}.
%Two additional recent lines of work deserve mention. In \cite{LeeuKnes2012}, the authors study how fast the Erlang-A model converges to  stationarity \st{using} by analyzing the spectral gap of the model. A spectral gap analysis for the Erlang-C model was done in \cite{GamaGold2013b}. Since one uses steady-state diffusion approximations, it is important to know how fast the original model reaches stationarity.
There is one additional recent line of work \cite{BlanGlyn2007,  JansLeeuZwar2008, JansLeeuZwar2008a,  JansLeeuZwar2011, LeeuZhanZwar2012} that deserves mention, where the theme is corrected diffusion approximations using asymptotic series expansions. In particular, \cite{JansLeeuZwar2011} considers the Erlang-C system and  \cite{LeeuZhanZwar2012} considers the Erlang-A system. In these papers, the authors derive series expansions for various steady-state quantities of interest like the probability of waiting $\Prob(X(\infty) \geq n)$. These types of series expansions are very powerful because they allow one to approximate steady-state quantities of interest within arbitrary precision. However, while accurate, these expansions vary for different performance metrics (e.g.\ waiting probability, expected queue length), and require non-trivial effort to be derived. They also depend on the choice of parameter regime, e.g.\ Halfin-Whitt. In contrast, the results provided by the Stein approach can be viewed as more robust because they capture multiple performance metrics and  multiple parameter regimes at the same time.

\section{Outline of Dissertation}
The rest of this document is structured as follows. Chapter~\ref{chap:erlangAC} serves as an introduction to Stein's method, where we outline the main steps of the procedure and carry them out on the Erlang-A and Erlang-C models. In Chapter~\ref{chap:dsquare} we work in the setting of the Erlang-C model. We prove that we can achieve a faster convergence rate by using a diffusion approximation with a state dependent diffusion coefficient. Finally, in Chapter~\ref{chap:phasetype}, we apply Stein's method to the $M/Ph/n+M$ queueing system, which is a significantly more complicated model than both the Erlang-A and Erlang-C systems. Each of the chapters requires its own moment and gradient bounds. We aggregate all moment bounds in Appendix~\ref{app:MOMENTS}, and all gradient bounds in Appendix~\ref{app:GRADIENTS}. 
%Instead of a centralized list of contributions, the contributions to the literature are listed on a chapter by chapter basis at the start of the each chapter. 

\section{Notation}
\label{sec:notation}
All random variables and stochastic processes are defined on a common
probability space $(\Omega, \mathcal{F}, \mathbb{P})$ unless otherwise
specified. For a sequence of random variables $\{X^n\}_{n=1}^{\infty}$, we write $X^n \Rightarrow X$ to denote convergence in distribution (also known as weak convergence) of $X^n$ to some random variable $X$. If $a > b$, we adopt the convention that $\sum \limits_{i=a}^b (\cdot) = 0$. For an
integer $d \geq 1$, $\R^d$ denotes the $d$-dimensional Euclidean space
and $\Z_+^d$ denotes the space of $d$-dimensional vectors whose
elements are non-negative integers. For $a,b \in \R$, we define $a
\vee b = \max\{a,b\}$ and $a \wedge b = \min\{a,b\}$. For $x \in \R$,
we define $x^+ = x \vee 0$ and $x^- = (-x)\vee 0$.  For $x \in \R^d$,
we use $x_i$ to denote its $i$th entry and $\abs{x}$ to denote its
Euclidean norm. For $x,y \in \R^d$, we write $x \leq y$ when $x_i \leq
y_i$ for all $i$ and when $x \leq y$ we define the vector interval
$[x,y] = \{z: x \leq z \leq y\}$. All vectors are assumed to be
column vectors. We let $x^T$ and $A^T$ denote the transpose of a
vector $x$ and matrix $A$, respectively. For a matrix $A$, we use
$A_{ij}$ to denote the entry in the $i$th row and $j$th column. We reserve $I$
for the identity matrix, $e$ for the vector of all ones and $e^{(i)}$
for the vector that has a one in the $i$th element and zeroes
elsewhere; the dimensions of these vectors will  be clear from the context.

\subsection{Probability Metrics}
For two random variables $U$ and $V$, define their Wasserstein distance, or Wasserstein metric, to be 
\begin{equation}
  \label{eq:dW}
  d_W(U, V) = \sup_{h(x) \in \lipone} \abs{\E[h(U)] -\E[h(V)]},
\end{equation}
where 
\begin{equation*}
\lipone=\{h: \R\to\R, \abs{h(x)-h(y)}\le
\abs{x-y}\}.
\end{equation*}
It is known, see for example \cite{Ross2011}, convergence under the 
Wasserstein metric implies convergence in distribution. We can replace $\lipone$ in (\ref{eq:dW}) by 
\begin{align}
{\cal H}_{K}=\{1_{(-\infty, a]}(x): a\in \R \}, \label{eq:classkolm}
\end{align}
and define 
\begin{align}
d_K(U, V) = \sup_{h(x) \in {\cal H}_{K}} \abs{\E[h(U)] -\E[h(V)]}. \label{eq:kolmogorov}
\end{align}
This quantity is known as the Kolmogorov distance, or Kolmogorov metric.

\chapter{Introduction to Stein's Method via the Erlang-A and Erlang-C Models} \label{chap:erlangAC}
 The goal of this Chapter is to introduce the reader to the main ideas behind Stein's method, and specifically in the context of steady-state diffusion approximations. We use the Erlang-A and Erlang-C systems as working examples to illustrate the technical aspects of the method. We begin this chapter with Section~\ref{sec:introchapCA}, where we recall some details about the Erlang-A and Erlang-C models. In Section~\ref{sec:CAresults}, we list the main results of this chapter. In Section~\ref{sec:CAroadmap}, we outline the key steps of the Stein framework: the Poisson equation, generator comparison, gradient bounds and moment bounds. In Section~\ref{sec:CAproofW} we prove Theorem~\ref{thm:erlangCW}, which is a result about the Wasserstein distance. In Section~\ref{sec:CAkolmogorov}, we discuss the Kolmogorov distance and the additional difficulties typically associated with it. Finally, we briefly discuss the approximation of higher moments in Section~\ref{sec:CAexten}. 
 
This chapter is based on \cite{BravDaiFeng2016}. The author would like to acknowledge Jiekun Feng, who contributed significantly to the contents of this chapter, and in particular to the results about the Erlang-A model.

\section{Chapter Introduction} \label{sec:introchapCA}
Section~\ref{sec:INtypical} already describes much of the focus of this chapter. We quickly recall some of the details about the Erlang-C and Erlang-A systems introduced there. Both systems have $n$ homogeneous servers that serve customers in a first-come-first-serve manner. Customers arrive according to a Poisson process with rate $\lambda$, and customer service times are assumed to be i.i.d.\ having exponential distribution with mean $1/\mu$. In the Erlang-A system, each customer
has a patience time and when his waiting time in queue exceeds his
patience time, he abandons the queue without service; the patience
times are assumed to be i.i.d.\ having exponential distribution with mean
$1/\alpha$. Recall that $X=\{X(t), t\ge 0\}$ is the customer count process. This process is positive recurrent when $\alpha > 0$, or if $\alpha = 0$ and the offered load $R = \lambda/\mu$ satisfies $R < n$. We use $X(\infty)$ to denote the random variable having the stationary distribution of $X$, and set $\tilde X(\infty) = (X(\infty) - R) /\sqrt{R}$. Theorem~\ref{thm:erlangCW} states that 
\begin{align}
d_W(\tilde X(\infty), Y(\infty)) =  \sup \limits_{h(x) \in {\lipone}}  \big|\E h(\tilde X(\infty)) - \E h(Y(\infty))\big|\leq  \frac{190}{\sqrt{R}}, \label{CW:dummy}
\end{align}
where $Y(\infty)$ is the random variable defined in \eqref{eq:stddenC}.

In addition to the discussion on universality below Theorem~\ref{thm:erlangCW} in Section~\ref{sec:INtypical}, there are two additional aspects that we will focus on in this chapter. From \eqref{CW:dummy}, we know that the first moment of $\tilde X(\infty)$ can be approximated universally by the first moment of $Y(\infty)$. It is natural to ask what can be said about the approximation of higher moments. We performed some numerical experiments in which we approximate the second and tenth moments of $\tilde X(\infty)$ in a system with $n = 500$. The results are displayed in Table~\ref{tab2}. One can see that the approximation errors grow as the offered load $R$ gets closer to $n$. We will see in Section~\ref{sec:CAexten} that this happens because the $(m-1)$th moment appears in the approximation error of the $m$th moment. A similar phenomenon was first observed for the $M/GI/1+GI$ model in Theorem 1 of \cite{GurvHuan2016}. 
\begin{table}[h]
  \begin{center}
  \resizebox{\columnwidth}{!}{
  \begin{tabular}{rcc | c c }
 $R$ & $\E (\tilde X(\infty))^2$ & $\big| \E (\tilde X(\infty))^2 - \E(Y(\infty))^2\big|$ & $\E(\tilde X(\infty))^{10}$ & $\big| \E (\tilde X(\infty))^{10} - \E(Y(\infty))^{10}\big|$ \\
\hline
 300        &  1                    & $4.55\times 10^{-15}$ & $9.77\times 10^2$ & 31.58\\
 400        &  1                    & $5.95\times 10^{-7}$ & $9.70\times 10^{2}$ & 24.44\\
 490        &  6.96                 & 0.11                 & $7.51\times 10^{9}$ & $7.01\times 10^{8}$\\
 495        &  31.56                & 0.27                & $9.10\times 10^{12}$ & $4.34\times 10^{11}$ \\
 499        &  $9.47\times 10^{2}$  & 1.59               & $1.07\times 10^{20}$ & $1.03\times 10^{18}$ \\
 499.9        &  $9.94\times 10^{4}$  & 16.50              & $1.13\times 10^{30}$ & $1.09\times 10^{27}$ \\
  \end{tabular}
  }
  \caption{Approximating the second and tenth moments of $\tilde X(\infty)$ with $n = 500$. The approximation error grows as $R$ approaches $n$ and suggests that the diffusion approximation of higher moments is not universal.\label{tab2}}
  \end{center}
\end{table} 

%9.765758518890070e+002
%9.698008827755839e+002
%7.512326134828012e+009
%9.100243423651367e+012
%1.073465781040412e+020
%
%error:
%31.575851889021806
%24.442750940677684
%7.007733407166281e+008
%4.337348528043447e+011
%1.034743104123404e+018

The rate of convergence in \eqref{CW:dummy} is for the  Wasserstein metric \cite{Ross2011}, which is usually the simplest metric to work with.  Another metric commonly studied in problems involving Stein's method is the Kolmogorov metric, which measures the distance between cumulative distribution functions of two random variables. The Kolmogorov distance between $\tilde X(\infty)$ and $Y(\infty)$ is
\begin{align*}
\sup_{h(x) \in \HH_K}  \big|\E h(\tilde X(\infty)) - \E h(Y(\infty))\big|, \quad \text{ where } \quad   {\cal H}_{K}=\{1_{(-\infty, a]}(x): a\in \R \}.
\end{align*}
Theorems~\ref{thm:erlangCK} and \ref{thm:erlangAK} of Section~\ref{sec:CAresults} involve the Kolmogorov metric. A general trend in Stein's method is that establishing convergence rates for the Kolmogorov metric often requires much more effort than establishing rates for the Wasserstein metric, and our problem is no exception. The extra difficulty always comes from the fact that the test functions belonging to the class $\HH_K$ are discontinuous, whereas the ones in $\lipone$ are Lipschitz-continuous. In Section~\ref{sec:CAkolmogorov}, we describe how to overcome this difficulty in our model setting. We now move on to state the main results of this chapter.

\section{Main results}
\label{sec:CAresults}
Recall the offered load $R = \lambda/\mu$. For notational convenience we define $\delta>0$ as
\begin{equation*}
    \delta = \frac{1}{ \sqrt{R}} = \sqrt{\frac{\mu}{\lambda}}.
\end{equation*}
Let $x(\infty)$ be the unique solution to the flow balance equation
\begin{align}
  \lambda = \big(x(\infty)\wedge n\big )\mu + \big(x(\infty)-n\big)^+ \alpha. \label{eq:xinf}
\end{align}
Here, $x(\infty)$ is interpreted as the equilibrium number of customers in the corresponding fluid model,
and is the point at which the arrival rate equals the departure rate. The latter is the sum of the service completion rate and the customer abandonment rate with $x(\infty)$ customers in the system. One can check that the flow balance equation has a unique solution $x(\infty)$ given by
 \begin{equation}
   \label{eq:equi}
  x(\infty) =
  \begin{cases}
    n + \frac{\lambda-n\mu }{\alpha} & \text{ if } R \geq n,\\
    R & \text{ if } R < n.
  \end{cases}
 \end{equation}
By noting that the number of busy servers $x(\infty) \wedge	n$ equals $n$ minus the number of idle servers $(x(\infty)-n)^-$, the equation in \eqref{eq:xinf} becomes
\begin{equation}
  \label{eq:lambdamudiff}
  \lambda - n \mu = \big(x(\infty)-n\big)^+\alpha - \big(x(\infty)-n\big)^-\mu.
\end{equation}
We note that $x(\infty)$ is well-defined even when $\alpha = 0$, because in that case we always assume that $R < n$.

We consider the CTMC
\begin{align} \label{eq:scaledctmc}
\tilde X = \{ \tilde X(t) := \delta(X(t) - x(\infty)),\ t \geq 0\},
\end{align}
and let the random variable $\tilde X(\infty)$ have its stationary distribution. Define 
\begin{equation}
\zeta =\delta\big(x(\infty) -n\big), \label{CA:zeta}
\end{equation}
and
\begin{equation}
  \label{eq:b-erlang-A}
  b(x) = \big((x+\zeta)^--\zeta^-\big)\mu -\big((x+\zeta)^+ -\zeta^+\big)\alpha \quad \text{ for } x\in \R,
\end{equation}
with convention that $\alpha$ is set to be zero in the Erlang-C
system. For intuition about the quantity $\zeta$, we note that in the Erlang-C system satisfying $R < n$, 
  \begin{align*}
    n = R -\zeta \sqrt{R}.
  \end{align*}
Thus,  $-\zeta=\abs{\zeta}>0$ is precisely the 
``safety coefficient'' in the square-root safety-staffing principle~\cite[equation (15)]{GansKoolMand2003}.
We point out that the event $\{\tilde X(t) = -\zeta \}$ corresponds to the event $\{X(t) = n\}$.

 Throughout this chapter, let $Y(\infty)$ denote a continuous random
variable on $\R$ having density
\begin{equation}
  \label{CA:stdden}
  \nu(x)= \kappa \exp\Big(\frac{1}{\mu} {\int_0^xb(y)dy}\Big),
\end{equation}
where $\kappa>0$ is a normalizing constant that makes the density integrate to one. Note that these definitions are consistent with \eqref{eq:stddenC} and \eqref{eq:bandz}.

%\begin{theorem}
%\label{thm:erlangCW}
%Consider the Erlang-C system ($\alpha = 0$). For all $\lambda, n$, and $\mu$ satisfying $n \geq 1$, and  $1 \leq R < n$,
%\begin{equation}
%  \label{eq:erlangCW}
%  d_W(\tilde X(\infty), Y(\infty)) \le 205\delta.
%\end{equation}
%\end{theorem}
%\begin{remark}
%One consequence of Theorem~\ref{thm:erlangCW} is that the error diffusion approximation for $\E X(\infty) - x(\infty)$, i.e.\ the unscaled first moment, is bounded by a constant. Observe with a fixed number of servers $n$, we know that $\E X(\infty) - x(\infty) \to \infty$ as $R \to n$. This suggests that the diffusion approximation for the Erlang-C system is accurate not only as $R \to \infty$, but also in the heavy-traffic setting when $R \to n$.
%\end{remark}

\begin{theorem}
\label{thm:erlangAW}
Consider the Erlang-A system ($\alpha > 0$). There exists an increasing function $C_W : \R_+ \to \R_+$ such that for all $n \geq 1, \lambda > 0, \mu>0$, and $\alpha>0$ satisfying $R \geq 1$,
\begin{equation}
  \label{eq:erlangAW}
  d_W(\tilde X(\infty), Y(\infty)) \le C_W(\alpha/\mu)  \delta.
\end{equation}
\end{theorem}
\begin{remark}
The proof of Theorems~\ref{thm:erlangCW} and \ref{thm:erlangAW} uses the same ideas. Therefore, for the sake of brevity, we only give an outline for the proof of Theorem~\ref{thm:erlangAW}  in Section~\ref{app:AWoutline}, without filling in all the details. It is for this reason that we do not write out the explicit form of $C_W(\alpha/\mu)$, although it can be obtained from the proof. The same is true for Theorem~\ref{thm:erlangAK} below.

\end{remark}
Given two random variables $U$ and $V$, \cite[Proposition 1.2]{Ross2011} implies that when $V$ has a density that is bounded by $C>0$, 
\begin{equation}
  \label{eq:dwdKbound}
  d_K(U, V) \le \sqrt{2C d_W(U, V)}.
\end{equation}
At best, \eqref{eq:dwdKbound} and Theorems~\ref{thm:erlangCW} and \ref{thm:erlangAW} imply a convergence rate of $\sqrt{\delta}$ for $d_K(\tilde X(\infty), Y(\infty))$. However, this bound is typically too crude, and the following two theorems show that convergence happens at rate $\delta$. Theorem~\ref{thm:erlangCK} is proved in Section~\ref{sec:kproof}. The proof of Theorem~\ref{thm:erlangAK} is outlined in Section~\ref{app:AKoutline}.

\begin{theorem}
\label{thm:erlangCK}
Consider the Erlang-C system ($\alpha = 0$). For all $n \geq 1,\\ \lambda > 0$, and $\mu>0$ satisfying $1 \leq R < n$,
\begin{equation}
  \label{eq:erlangCK}
  d_K(\tilde X(\infty), Y(\infty)) \le 156\delta.
\end{equation}
\end{theorem}

\begin{theorem}
\label{thm:erlangAK}
Consider the Erlang-A system ($\alpha > 0$). There exists an increasing function $C_K : \R_+ \to \R_+$ such that for all $n \geq 1, \lambda > 0, \mu>0$, and $\alpha>0$ satisfying $R \geq 1$,
\begin{equation}
  \label{eq:erlangAK}
  d_K(\tilde X(\infty), Y(\infty)) \le C_K(\alpha/\mu)\delta.
\end{equation}
\end{theorem}

 Theorems \ref{thm:erlangCW} and
\ref{thm:erlangCK} are new, but versions of Theorems~\ref{thm:erlangAW} and \ref{thm:erlangAK} were first proved in the
pioneering paper \cite{GurvHuanMand2014} using an excursion based approach. However, our notion of universality in those theorems is stronger than the one in \cite{GurvHuanMand2014}, because most of their results require $\mu$ and $\alpha$ to be fixed. The only exception is in Appendix C of that paper, where the authors consider the NDS regime with $\mu  = \mu(\lambda) = \beta \sqrt{\lambda}$ and $\lambda = n\mu + \beta_1 \mu$ for some $\beta> 0$ and $\beta_1 \in \R$.

%\blue{The reader may wonder why the constant in the Erlang-A theorems depends on $\alpha/\mu$, while the constant in the Erlang-C theorems does not depend on anything. Actually, numerical results in ?? suggest that the constant in the Erlang-A theorems can be made entirely independent of $\lambda,\mu, n$, or $\alpha$. However, despite our efforts, we were unable to prove this. The reason why $\alpha/\mu$ appears in the results is that the Erlang-C depends on only three parameters ($\lambda, \mu, n$), while the Erlang-A model also depends on $\alpha$. This means that more parameters show up in both gradient and moment bounds; compare for instance Lemma~\ref{lem:gradboundsCW} with \ref{lem:gradboundsAWunder}. }

We emphasize that both constants $C_W$ and $C_K$ are increasing in $\alpha/\mu$. That is, for an Erlang-A system with a higher abandonment rate with respect to its service rate, our error bound becomes larger. The reader may wonder why these constants depend on $\alpha/\mu$, while the constant in the Erlang-C theorems does not depend on anything. Despite our best efforts, we were unable to get rid of the dependency on $\alpha/\mu$. The reason is that the Erlang-C model depends on only three parameters ($\lambda, \mu, n$), while the Erlang-A model also depends on $\alpha$. 
As a result, both the gradient bounds and moment bounds have an extra factor $\alpha/\mu$  in the Erlang-A model. 
 For example, compare Lemma~\ref{lem:gradboundsCW} in Section~\ref{sec:CAwassergrad_bounds} with Lemma~\ref{lem:gradboundsAWunder} in Section~\ref{sec:CAapplemmas}.

%momerrs =
%
%  Columns 1 through 5
%
%   0.000000004842363   0.000439964103489   0.000000000000501   0.028631238692906   0.053289847064835
%
%  Columns 6 through 10
%
%   0.078811516176546   0.093337153172163   0.098813999447969   0.100638252330851   0.101224425709551
%
%>> kerrs
%
%kerrs =
%
%  Columns 1 through 5
%
%   0.003433475342745   0.010856652012603   0.034262472360977   0.076266227205540   0.097081571896822
%
%  Columns 6 through 10
%
%   0.054781386718694   0.020543376618708   0.020035288388012   0.021003855738943   0.021405912572790

\section{Outline of the Stein Framework}
\label{sec:CAroadmap}
In this section we introduce the main tools needed to prove
Theorems~\ref{thm:erlangCW}--\ref{thm:erlangAK}. However, the framework presented here is generic, and is not limited to the Erlang-A or Erlang-C systems. It can be applied  whenever one compares a Markov chain to a diffusion process; the content here will be referred to liberally in all chapters of this dissertation. The following is an informal outline of the rest of this section. 

We know that $\tilde X(\infty)$ follows the stationary distribution of the CTMC $\tilde X$, and that this CTMC has a generator $G_{\tilde X}$. To $Y(\infty)$, we will associate a diffusion process with generator $G_Y$. We will start by fixing a test function $h :\R \to \R$ and deriving the identity 
\begin{align}
\big| \E h(\tilde X(\infty)) - \E h(Y(\infty)) \big| = \big| \E G_{\tilde X} f_h(\tilde X(\infty)) - \E  G_{Y} f_h(\tilde X(\infty)) \big|, \label{eq:outline_identity}
\end{align}
where $f_h(x)$ is a solution to the Poisson equation
\begin{align*}
G_{Y} f_h(x) = \E h(Y(\infty)) - h(x), \quad x \in \R.
\end{align*}
We then focus on bounding the right hand side of \eqref{eq:outline_identity}, which is easier to handle than the left hand side. This is done by performing a Taylor expansion of $G_{\tilde X} f_h(x)$ in Section~\ref{sec:CAtaylor}. To bound the error term from the Taylor expansion, we require bounds on various moments of $\big| \tilde X(\infty) \big|$, as well as the derivatives of $f_h(x)$. We refer to the former as moment bounds, and the latter as gradient bounds. These are presented in Sections~\ref{sec:CAmomentbounds} and \ref{sec:CAwassergrad_bounds}, respectively.

\subsection{The Poisson Equation of a Diffusion Process}
\label{sec:diffPoisson}\
A one-dimensional diffusion process can be described by its generator 
\begin{equation}
  \label{CA:genericgenerator}
G f(x) =  \bar b(x)f'(x) +  \frac{1}{2} \bar a(x) f''(x) \text{ \quad for $x \in \R$}, \ f \in C^2(\R).
\end{equation}
The functions $\bar b(x)$ and $\bar a(x)$ are known as the drift, and diffusion coefficient, respectively. It is typically required that $\bar a(x) > 0$ for all $x \in \R$, and that both $\bar b(x)$ and $\bar a(x)$ satisfy some regularity condition, e.g.\ Lipschitz continuity.
  
 The random variable $Y(\infty)$ in
Theorems~\ref{thm:erlangCW}--\ref{thm:erlangAK} is well-defined and
its density is given in (\ref{CA:stdden}). It turns out that $Y(\infty)$ has the stationary distribution of a diffusion process $Y=\{Y(t), t\ge 0\}$. The process $Y$ is the one-dimensional piecewise
Ornstein--Uhlenbeck (OU) process, whose generator is given by
\begin{equation}
  \label{CA:GY}
G_Y f(x) =  b(x)f'(x) +  \mu f''(x) \text{ \quad for $x \in \R$}, \ f \in C^2(\R),
\end{equation}
where $b(x)$ is defined in (\ref{eq:b-erlang-A}).
Clearly,  $b(0)=0$, and  $b(x)$ is Lipschitz continuous. Indeed,
\begin{displaymath}
\abs{  b(x) -b(y)} \le (\alpha\vee \mu) \abs{x-y} \quad \text{ for } x, y\in \R.
\end{displaymath}
The generator in \eqref{CA:GY} has a \emph{constant} diffusion coefficient $\bar a(x) = 2\mu$. 

Since the diffusion process $Y$ depends on parameters $\lambda, n, \mu$, and
$\alpha$ in an arbitrary way, there is no appropriate way to talk about the limit of $Y(\infty)$ in terms of these parameters. Therefore, we call $Y$ a \emph{diffusion model}, as opposed to a
\emph{diffusion limit}. Having a diffusion model whose input
parameters are directly taken from the corresponding Markov chain
model is critical to achieve \emph{universal accuracy}. In other words, this diffusion model is accurate in any parameter regime, from underloaded, to critically loaded, and to overloaded. Diffusion models, not limits,  of queueing networks with a given set of parameters have been advanced in \cite{HarrWill1987,HarrNguy1993, DaiHe2013,GlynWard2003, Gurv2014, GurvHuanMand2014,GurvHuan2016}.

The main tool we use is known as the Poisson equation. It allows us to say that $Y(\infty)$ is a good estimate for $\tilde X(\infty)$ if the generator of $Y$ behaves similarly to the generator of $\tilde X$, where $\tilde X$ is defined in \eqref{eq:scaledctmc}. Let $\cal{H}$ be a class of functions $h : \R \to \R$, to be specified shortly.  For each function $h(x) \in \cal{H}$, consider the Poisson equation 
\begin{align} \label{CA:poisson}
G_Y f_h(x) =  b(x) f_h'(x) + \mu f_h''(x) = \E h(Y(\infty)) - h(x), \quad  x\in \R.
\end{align}
The solution to the Poisson equation is described by the following generic lemma. 
\begin{lemma} \label{lem:solution}
Let $\bar a:\R \to \R_+$ and $\bar b: \R \to \R$ be continuous functions, and assume that $\bar a(x) > 0$ for all $x \in \R$. Assume also that
\begin{align*}
\int_{-\infty}^{\infty}\frac{2}{\bar a(x)} \exp \Big({\int_{0}^{x} \frac{2 \bar b(u)}{\bar a(u)} du} \Big)dx < \infty,
\end{align*}
and let $V$ be a continuous random variable with density 
\begin{align*}
\frac{\frac{2}{\bar a(x)} \exp \Big({\int_{0}^{x} \frac{2 \bar b(u)}{\bar a(u)} du} \Big)}{\int_{-\infty}^{\infty}\frac{2}{\bar a(x)} \exp \Big({\int_{0}^{x} \frac{2 \bar b(u)}{\bar a(u)} du} \Big)dx}, \quad x \in \R. 
\end{align*}
Fix $h:\R \to \R$ satisfying $\E \abs{h(V)} < \infty$, and consider the Poisson equation 
\begin{align}
\frac{1}{2} \bar a(x) f_h''(x) + \bar b(x) f_h'(x) = \E h(V) - h(x), \quad x \in \R. \label{eq:genericpoisson}
\end{align}
There exists a solution $f_h(x)$ to this equation satisfying 
\begin{align}
f_h'(x) =&\ e^{-\int_{0}^{x} \frac{2\bar b(u)}{\bar a(u)}du}\int_{-\infty}^{x} \frac{2}{\bar a(y)} (\E h(V) - h(y)) e^{\int_{0}^{y} \frac{2\bar b(u)}{\bar a(u)}du} dy  \label{eq:fprimenegv1} \\
=&\ -e^{-\int_{0}^{x} \frac{2\bar b(u)}{\bar a(u)}du} \int_{x}^{\infty} \frac{2}{\bar a(y)} (\E h(V) - h(y)) e^{\int_{0}^{y} \frac{2\bar b(u)}{\bar a(u)}du} dy \label{eq:fprimeposv1}, \\
f_h''(x) =&\  - \frac{2\bar b(x)}{\bar a(x)} f_h'(x) +  \frac{2}{\bar a(x)}\big(\E h(V) - h(x) \big). \label{eq:fppv2}
\end{align}
\end{lemma}
\begin{proof}
The integrals in \eqref{eq:fprimenegv1} and \eqref{eq:fprimeposv1} are finite because $\E \abs{h(V)} < \infty$. One can verify directly that both forms of $f_h'(x)$ in \eqref{eq:fprimenegv1} and \eqref{eq:fprimeposv1} satisfy \eqref{eq:genericpoisson}. The form of $f_h''(x)$ follows from rearranging \eqref{eq:genericpoisson}.
\end{proof}

\begin{remark}
Provided $h(x)$, $\bar a(x)$, and $\bar b(x)/\bar a(x)$ are sufficiently differentiable, $f_h(x)$ can have more than two derivatives. For example, 
\begin{align}
f_h'''(x) =&\ -\Big(\frac{2\bar b(x)}{\bar a(x)}\Big)' f_h'(x) - \frac{2\bar b(x)}{\bar a(x)} f_h''(x) - \frac{2}{\bar a(x)}h'(x) - \frac{2\bar a'(x)}{\bar a^2(x)}\big( \E h(V) - h(x)\big).  \label{eq:fppp}
\end{align}
\end{remark}

%One may verify by differentiation that for all functions $h: \R \to \R$ satisfying $\E \big|h(Y(\infty))\big| < \infty$, the Poisson equation is solved by 
%\begin{align} \label{eq:poissonsolution}
%f_h(x) = a_1 + \int_{0}^{x} \bigg[ a_2 \frac{1}{\nu(u)} + \frac{1}{\nu(u)} \int_{-\infty}^{u} \frac{1}{\mu } \big(\E h(Y(\infty)) - h(y)\big) \nu(y) dy\bigg] du, 
%\end{align}
%where $a_1, a_2 \in \R$ are arbitrary constants, and $\nu(x)$ is as in \eqref{eq:stdden}. 

In this chapter, we take $\HH = \lipone$ when we deal
with the Wasserstein metric (Theorems~\ref{thm:erlangCW} and
\ref{thm:erlangAW}), and we choose $\HH = \HH_K$ (defined in \eqref{eq:classkolm}) when we deal with the
Kolmogorov metric (Theorems~\ref{thm:erlangCK} and
\ref{thm:erlangAK}). We claim that $\abs{\E h(Y(\infty))} <
\infty$. Indeed, when $\HH = \HH_K$, this clearly holds. 
When $\mathcal{H} = \lipone$, without loss of generality we take $h(0)= 0$ in \eqref{CA:poisson}, and use the Lipschitz
property of $h(x)$ to see that
\begin{align*}
\abs{\E h(Y(\infty))} \leq  \E \abs{Y(\infty)} < \infty, 
\end{align*}
where the finiteness of $\E \abs{Y(\infty)}$ will be proved in
\eqref{CW:fbound7}. 
%Furthermore, when $\HH = \HH_K$, we understand $f_h''(x)$ in \eqref{eq:poisson} to be the left derivative at the point $x = a$, because from the form of $f_h'(x)$ in \eqref{eq:poissonsolution} one can see that it is not differentiable at $x = a$.

 From (\ref{CA:poisson}), one has
\begin{align}
\big| \E h(\tilde X(\infty)) - \E h(Y(\infty)) \big| =&\ \big| \E G_Y f_h(\tilde X(\infty)) \big|.\label{eq:error1}
\end{align}
In (\ref{eq:error1}), $\tilde X(\infty)$ has the stationary
distribution of the CTMC $\tilde X$, not necessarily defined on the same
probability space of $Y(\infty)$.  Actually, $\tilde X(\infty)$ in
(\ref{eq:error1}) can be replaced by any other random variable,
although one does not expect the error on the right side to be small
if this random variable has no relationship with the diffusion process
$Y$.

\subsection{Comparing Generators}
\label{sec:compare}
To prove Theorems~\ref{thm:erlangCW}--\ref{thm:erlangAK}, we need to
bound the right side of (\ref{eq:error1}).  The CTMC $\tilde X$
defined in \eqref{eq:scaledctmc} also has a generator.  We bound the
right side of (\ref{eq:error1}) by showing that the diffusion
generator in (\ref{CA:GY}) is similar to the CTMC generator.

 For any $k \in \Z_+$, we define $x = x_k = \delta(k - x(\infty))$. Then for any function $f:\R\to\R$, the generator of $\tilde X$ is given by  
\begin{align}\label{eq:GX}
G_{\tilde X} f(x) = \lambda (f(x + \delta) - f(x)) + d(k) (f(x-\delta) - f(x)),
\end{align}
where
\begin{align} \label{eq:deathrate}
d(k) = \mu (k \wedge n) + \alpha (k-n)^+,
\end{align}
is the departure rate corresponding to the system having $k$ customers.  One may check that 
\begin{align}
b(x) = \delta( \lambda - d(k)). \label{eq:bk}
\end{align}
The relationship between $G_{\tilde X}$ and the stationary distribution of $\tilde X$ is illustrated by the following lemma.
\begin{lemma} \label{lem:gz}
Let $f(x): \R \to \R$ be a function such that $\abs{f(x)} \leq C(1+\abs{x})^3$ for some $C > 0$ (i.e.\ $f(x)$ is dominated by a cubic function), and assume that the CTMC $\tilde X$ is positive recurrent. Then
\begin{align*}
\EE \big[ G_{\tilde X} f(\tilde X(\infty)) \big] = 0.
\end{align*}
\end{lemma}
\begin{remark}
We will see in Lemma~\ref{lem:gradboundsCW} later this section, in Lemmas~\ref{lem:gradboundsCK} and \ref{lem:gradboundsAK} of Section~\ref{sec:CAkolmogorov}, and in  Lemma~\ref{lem:gradboundsAWunder} of Section~\ref{app:grad_A} that there is a family of solutions to the Poisson equation \eqref{CA:poisson} whose first derivatives grow at most linearly in both the Wasserstein and Kolmgorov settings, meaning that these solutions satisfy the conditions of Lemma~\ref{lem:gz}.
\end{remark}
The proof of Lemma~\ref{lem:gz} is provided in Section~\ref{app:misc}. Suppose for now that for any $h(x) \in \cal{H}$, the solution to the Poisson equation $f_h(x)$ satisfies the conditions of Lemma~\ref{lem:gz}. We can apply Lemma~\ref{lem:gz} to \eqref{eq:error1} to see that
\begin{align}
\big| \E h(\tilde X(\infty)) - \E h(Y(\infty)) \big| =&\ \big| \E G_Y f_h(\tilde X(\infty)) \big| \notag \\
=&\ \big| \E G_{\tilde X} f_h(\tilde X(\infty)) - \E G_Y f_h(\tilde X(\infty)) \big| \notag \\
\leq&\ \E \big| G_{\tilde X} f_h(\tilde X(\infty)) -  G_Y f_h(\tilde X(\infty)) \big|. \label{eq:gen_bound}
\end{align}
While the two random variables on the left side of
(\ref{eq:gen_bound}) are usually defined on different probability
spaces, the two random variables on the right side of
(\ref{eq:gen_bound}) are both functions of $\tilde X(\infty)$. Thus, we have 
achieved a coupling through Lemma \ref{lem:gz}. Setting up the Poisson equation is a generic first step one performs any time one wishes to apply Stein's method to a problem. The next step is to bound the equivalent of our $\big| \E G_Y f_h(\tilde X(\infty)) \big|$. This is usually done by using a coupling argument. However, this coupling is always problem specific, and is one of the greatest sources of difficulty one encounters when applying Stein's method. In our case, this generator coupling is natural because we deal with Markov processes $\tilde X$ and $Y$.

Since the generator completely characterizes the behavior of a Markov process, it is natural to expect that convergence of generators implies convergence of Markov processes. Indeed, the question of weak convergence was studied in detail, for instance in \cite{EthiKurt1986}, using the martingale problem of Stroock and Varadhan \cite{StroVara1979}. However, \eqref{eq:gen_bound} lets us go beyond weak convergence, both because different choices of $h(x)$ lead to different metrics of convergence, and also because the question of convergence rates can be answered. One interpretation of the Stein approach is to view $f_h(x)$ as a Lyapunov function that gives us information about $h(x)$. Instead of searching very hard for this Lyapunov function, the Poisson equation \eqref{CA:poisson} removes the guesswork. However, this comes at the cost of $f_h(x)$ being defined implicitly as the solution to a differential equation.

\subsection{Taylor Expansion}
\label{sec:CAtaylor}

To bound the right side of (\ref{eq:gen_bound}), 
we study the difference $G_{\tilde X} f_h(x) - G_Y f_h(x)$. For that we perform a Taylor expansion on $G_{\tilde X}f_h(x)$. To illustrate this, suppose
that $f_h''(x)$ exists for all $x \in \R$, and is absolutely
continuous. Then for any $k \in \Z_+$, and $x = x_k = \delta(k -
x(\infty))$, we recall that $b(x) = \delta(\lambda - d(k))$ in \eqref{eq:bk} to see
that
\begingroup
\allowdisplaybreaks
\begin{align*}
G_{\tilde X} f_h(x) =&\ \lambda (f_h(x + \delta) - f_h(x)) + d(k) (f_h(x-\delta) - f_h(x)) \notag \\
=&\ f_h'(x) \delta (\lambda - d(k)) + \frac{1}{2} \delta^2 f_h''(x)(\lambda + d(k)) \notag  \\
&+ \frac{1}{2} \lambda \delta^2 (f_h''(\xi) - f_h''(x)) + \frac{1}{2} d(k) \delta^2 (f_h''(\eta) - f_h''(x)) \notag \\
=&\ f_h'(x)  b(x) +  \frac{1}{2} \delta^2 (2\lambda - \frac{1}{\delta} b(x)) f_h''(x) \notag \\
&+ \frac{1}{2} \mu (f_h''(\xi) - f_h''(x)) + \frac{1}{2} ( \lambda -  \frac{1}{\delta}b(x))  \delta^2 (f_h''(\eta) - f_h''(x))\\
=&\ G_Y f_h(x) - \frac{1}{2} \delta f_h''(x)b(x) + \frac{1}{2} \mu (f_h''(\xi) - f_h''(x)) \\
&+ \frac{1}{2} ( \mu -  \delta b(x))   (f_h''(\eta) - f_h''(x)),
\end{align*}%
\endgroup
where $\xi \in [x, x+\delta]$ and $\eta \in [x-\delta,x]$. We invoke the absolute continuity of $f_h''(x)$ to get
\begin{align}
&\ \Big| \E h(\tilde X(\infty)) - \E h(Y(\infty)) \Big| \notag \\
 \leq&\ \frac{1}{2} \delta \E \Big[ \big|f_h''(\tilde X(\infty))b(\tilde X(\infty)) \big| \Big] + \frac{\mu}{2} \E \bigg[ \int_{\tilde X(\infty)}^{\tilde X(\infty) + \delta} \abs{f_h'''(y)}dy\bigg] \notag \\
&+ \frac{\mu}{2} \E \bigg[ \int_{\tilde X(\infty)-\delta}^{\tilde X(\infty)} \abs{f_h'''(y)}dy\bigg] + \frac{1}{2} \delta\E \bigg[  \big| b(\tilde X(\infty))\big| \int_{\tilde X(\infty)-\delta}^{\tilde X(\infty)} \abs{f_h'''(y)}dy\bigg]. \label{eq:first_bounds}
\end{align}
As one can see, to show that the right hand side of \eqref{eq:first_bounds} vanishes as $\delta \to 0$, we must be able to bound the derivatives of $f_h(x)$; we refer to these as gradient bounds. Furthermore, we will also need bounds on moments of $\big| \tilde X(\infty) \big|$; we refer to these as moment bounds. Both moment and gradient bounds will vary between the Erlang-A or Erlang-C setting, and the gradient bounds will be different for the Wasserstein, and Kolmogorov settings. Moment bounds will be discussed shortly, and gradient bounds 
in the Wasserstein setting
will be presented in Section~\ref{sec:CAwassergrad_bounds}.  We discuss the Kolmogorov setting separately in Section~\ref{sec:CAkolmogorov}. In that case we face an added difficulty because $f_h''(x)$ has a discontinuity, and we cannot use \eqref{eq:first_bounds} directly.
\subsection{Moment Bounds}
\label{sec:CAmomentbounds}
The following lemma presents the necessary moment bounds to bound \eqref{eq:first_bounds} in the Erlang-C model, and is proved in Appendix~\ref{app:momCproof}. These moment
  bounds are used for both the Wasserstein and Kolmogorov metrics.
\begin{lemma} \label{lem:moment_bounds_C}
Consider the Erlang-C model ($\alpha = 0$). For all $n \geq 1, \lambda > 0$, and $\mu>0$ satisfying $0 < R < n $,
\allowdisplaybreaks
\begin{align} 
&\E \Big[(\tilde X(\infty))^2 1(\tilde X(\infty) \leq -\zeta)\Big] \leq \frac{4}{3} + \frac{2\delta^2}{3}, \label{CW:xsquaredelta}\\
%&\E \Big[(\tilde X(\infty))^2 1(\tilde X(\infty) \leq -\zeta) \Big] \leq  3\zeta^2 + \abs{\zeta}(2 + \delta \abs{\zeta}), \label{eq:xsquarezeta} \\
&\E \Big[ \big|\tilde X(\infty) 1(\tilde X(\infty) \leq -\zeta)\big| \Big] \leq \sqrt{\frac{4}{3} + \frac{2\delta^2}{3}}, \label{CW:xminusdelta}\\
&\E \Big[ \big|\tilde X(\infty) 1(\tilde X(\infty) \leq -\zeta)\big| \Big] \leq 2\abs{\zeta} \label{CW:xminuszeta}\\
&\E \Big[\big|\tilde X(\infty)1(\tilde X(\infty) \geq -\zeta)\big| \Big] \leq \frac{1}{\abs{\zeta}} + \frac{\delta^2}{4\abs{\zeta}} + \frac{\delta}{2}, \label{CW:xplus}\\
&\Prob(\tilde X(\infty) \leq -\zeta) \leq (2+\delta)\abs{\zeta}. \label{CW:idle_prob}
\end{align}
\end{lemma} 
\noindent We see that \eqref{CW:xplus} immediately implies that when $\delta \leq 1$, 
\begin{align}
\abs{\zeta} \Prob( \tilde X(\infty) \geq -\zeta) \leq&\ \abs{\zeta} \wedge \E \Big[\big| \tilde X(\infty) 1(\tilde X(\infty) \geq -\zeta)\big| \Big]\notag \\
 \leq&\ \abs{\zeta} \wedge \Big(\frac{1}{\abs{\zeta}} + \frac{\delta^2}{4\abs{\zeta}} + \frac{\delta}{2} \Big) \notag \\
\leq&\ 7/4, \label{eq:xplusbound}
\end{align}
where to get the last inequality we considered separately the cases where $\abs{\zeta} \leq 1$ and $\abs{\zeta} \geq 1$. This bound will be used in the proofs of Theorems~\ref{thm:erlangCW} and \ref{thm:erlangCK}. 

One may wonder why the bounds are separated using the indicators $1\{\tilde X(\infty) \leq -\zeta\}$ and $1\{\tilde X(\infty) \geq -\zeta\}$. This is related to the drift $b(x)$ appearing in \eqref{eq:first_bounds}, and the fact that $b(x)$ takes different forms on the regions  $x\leq -\zeta$ and $x\geq -\zeta$. Furthermore, it may be unclear at this point why both \eqref{CW:xminusdelta} and \eqref{CW:xminuszeta} are needed, as the left hand side in both bounds is identical. The reason is that \eqref{CW:xminusdelta} is an $O(1)$ bound (we think of $\delta \leq 1$), whereas \eqref{CW:xminuszeta} is an $O(\abs{\zeta})$ bound. The latter is only useful when $\abs{\zeta}$ is small, but this is nevertheless an essential bound to achieve universal results. As we will see later, it negates $1/\abs{\zeta}$ terms that appear in \eqref{eq:first_bounds} from $f_h''(x)$ and $f_h'''(x)$.

For the Erlang-A model, we also require moment bounds similar to those stated in Lemma~\ref{lem:moment_bounds_C}. Both the proof, and subsequent usage, of the Erlang-A moment bounds are similar to the proof and subsequent usage of the Erlang-C moment bounds. We therefore delay their precise statement until Lemma~\ref{lem:moment_bounds_A_under} in Section~\ref{sec:CAapplemmas} to avoid distracting the reader with a bulky lemma.

%\begin{lemma}\label{lem:moment_bounds_A_over}
%Consider the Erlang-A model ($\alpha>0$). For all $\lambda, n$, and $\mu$ satisfying $n \geq 1$ and $R \geq n $,
%
%\end{lemma}

\subsection{Wasserstein Gradient Bounds} \label{sec:CAwassergrad_bounds}
Given a function $h(x)$, there are multiple solutions to the Poisson equation \eqref{CA:poisson}. Going forward, when we refer to a solution $f_h(x)$, we mean the solution in Lemma~\ref{lem:solution} with $\bar b(x) = b(x)$ and $\bar a(x) = 2\mu$. The following lemma presents Wasserstein gradient bounds for the Erlang-C model. It is proved in Section~\ref{app:wgradient}.
\begin{lemma} \label{lem:gradboundsCW}
Consider the Erlang-C model ($\alpha = 0$), and fix $ h(x) \in \lipone$. Then $f_h(x)$ is twice continuously differentiable, with an absolutely continuous second derivative. Furthermore, for all $n \geq 1, \lambda > 0$, and $\mu>0$ satisfying $0 < R < n $, 
\begin{align}
\abs{f_h'(x)} \leq&\
\begin{cases}
\frac{1}{\mu }(7.5 + 5/\abs{\zeta}), \quad x \leq -\zeta,\\
\frac{1}{\mu }\frac{1}{\abs{\zeta}}(x + 1 + 2/\abs{\zeta}), \quad x \geq -\zeta.
\end{cases}\label{eq:WCder1} \\
\abs{f_h''(x)} \leq&\ 
\begin{cases}
\frac{34}{\mu }( 1 + 1/\abs{\zeta}), \quad x \leq -\zeta,\\
\frac{1}{\mu \abs{\zeta}}, \quad x \geq -\zeta,
\end{cases} \label{eq:WCder2}
\end{align}
and for those $x \in \R$ where $f_h'''(x)$ exists, 
\begin{align}
\abs{f_h'''(x)} \leq&\ 
\begin{cases}
\frac{1}{\mu}(17 + 10/\abs{\zeta}) , \quad x \leq -\zeta,\\
2/\mu, \quad x \geq -\zeta.
\end{cases} \label{eq:WCder3}
\end{align}
\end{lemma}

\begin{remark}
This lemma validates the Taylor expansion used to obtain \eqref{eq:first_bounds}  because  $f_h''(x)$ is absolutely continuous. Furthermore, $f_h(x)$  satisfies the conditions of Lemma~\ref{lem:gz}, because $f_h'(x)$ grows at most linearly.
\end{remark}
\noindent Gradient bounds, also known as Stein factors, are central to any application of Stein's method. The problem of gradient bounds for diffusion approximations can be divided into two cases: the one-dimensional case, and the multi-dimensional case. In the former, the Poisson equation is an ordinary differential equation (ODE) corresponding to a one-dimensional diffusion process. In the latter, the Poisson equation is a partial differential equation (PDE) corresponding to a multi-dimensional diffusion process.

The one-dimensional case is simpler, because the explicit form of $f_h(x)$ is given to us by Lemma~\ref{lem:solution}. To bound $f_h'(x)$ and $f_h''(x)$ we can analyze \eqref{eq:fprimenegv1}--\eqref{eq:fppv2} directly, as we do in the proof of Lemma~\ref{lem:gradboundsCW}. In Appendix~\ref{app:gradboundsgeneric}, we see that this direct analysis can be used as a go-to method for one-dimensional diffusions. However, it fails in the multi-dimensional case, because closed form solutions for PDE's are not typically known. In this case, it helps to  exploit the fact that $f_h(x)$ satisfies
\begin{align}
f_h(x) = \int_{0}^{\infty} \Big(\E \big[h(Y(t))\ |\ Y(0) = x\big] - \E h(Y(\infty))\Big) dt, \label{eq:relval}
\end{align}
where $Y = \{Y(t), t\geq 0\}$ is a diffusion process with generator $G_{Y}$ \cite{PardVere2001}. To bound derivatives of $f_h(x)$ based on \eqref{eq:relval}, one may use coupling arguments to bound finite differences of the form $\frac{1}{s}(f_h(x+ s) - f_h(x))$. For examples of coupling arguments, see \cite{Barb1988, BarbBrow1992, BrowXia2001, BarbXia2006, GanXia2015, GorhMack2016}. A related paper to these types of gradient bounds is \cite{Stol2015}, where the author used a variant of \eqref{eq:relval} for the fluid model of a flexible-server queueing system as a Lyapunov function.  As an alternative to coupling, one may combine \eqref{eq:relval} with a-priori Schauder estimates from PDE theory, as was done in \cite{Gurv2014}. 

Just like we did with the moment bounds, we delay the Erlang-A gradient bounds to Lemma~\ref{lem:gradboundsAWunder} in Section~\ref{sec:CAapplemmas}. We are now ready to prove Theorem~\ref{thm:erlangCW}.

\section{Proof of Theorem~\ref{thm:erlangCW} (Erlang-C Wasserstein)}
\label{sec:CAproofW}
In this section we prove Theorem~\ref{thm:erlangCW}. Fixing $h(x) \in \lipone$, we see from Lemma~\ref{lem:gradboundsCW} that $f_h''(x)$ is absolutely continuous, implying that \eqref{eq:first_bounds} holds. We recall it here as
\begin{align}
&\ \Big| \E h(\tilde X(\infty)) - \E h(Y(\infty)) \Big| \notag \\
 \leq&\ \frac{1}{2} \delta \E \Big[ \big|f_h''(\tilde X(\infty))b(\tilde X(\infty)) \big| \Big] + \frac{\mu}{2} \E \bigg[ \int_{\tilde X(\infty)}^{\tilde X(\infty) + \delta} \abs{f_h'''(y)}dy\bigg] \notag \\
&+ \frac{\mu}{2} \E \bigg[ \int_{\tilde X(\infty)-\delta}^{\tilde X(\infty)} \abs{f_h'''(y)}dy\bigg] + \frac{1}{2} \delta\E \bigg[  \big| b(\tilde X(\infty))\big| \int_{\tilde X(\infty)-\delta}^{\tilde X(\infty)} \abs{f_h'''(y)}dy\bigg], \label{CW:second_bounds}
\end{align}
where $\delta = 1/\sqrt{R} = \sqrt{\mu/\lambda}.$
The proof of Theorem~\ref{thm:erlangCW} simply involves applying the moment bounds and gradient bounds to show that the error bound in \eqref{CW:second_bounds} is small.
\begin{proof}[Proof of Theorem~\ref{thm:erlangCW}]
Throughout the proof we assume that $R \geq 1$, or equivalently, $\delta \leq 1$. We bound each of the terms on the right side of \eqref{CW:second_bounds} individually. We recall here that the support of $\tilde X(\infty)$ is a $\delta$-spaced grid, and in particular this grid contains the point $-\zeta$. In the bounds that follow, we will often consider separately the cases where $\tilde X(\infty) \leq -\zeta - \delta$, and $\tilde X(\infty) \geq -\zeta$. We recall that 
\begin{align}
b(x) = \mu \big((x+\zeta)^--\zeta^-\big) = 
\begin{cases}
-\mu x, \quad x \leq -\zeta,\\
\mu \zeta, \quad x \geq -\zeta,
\end{cases} \label{eq:bexpand}
\end{align}
and apply the moment bounds \eqref{CW:xminusdelta}, \eqref{CW:xminuszeta}, and the gradient bound \eqref{eq:WCder2}, to see that
\begin{align*}
\E \Big[ \big|f_h''(\tilde X(\infty))b(\tilde X(\infty)) \big| \Big] \leq &\ 34(1 + 1/\abs{\zeta})\E \Big[\big|\tilde X(\infty) \big| 1(\tilde X(\infty) \leq -\zeta - \delta) \Big] \\
&+ \Prob(\tilde X(\infty) \geq -\zeta) \\
\leq &\ 34(1 + 1/\abs{\zeta})\bigg(2\abs{\zeta} \wedge \sqrt{\frac{4}{3} + \frac{2\delta^2}{3}}\bigg) + 1 \\
\leq &\ 34\Big(\sqrt{\frac{4}{3} + \frac{2\delta^2}{3}} + 2\Big) + 1 \\
\leq&\ 34\big(\sqrt{2} + 2\big) + 1 \leq 118.
\end{align*}
Next, we use \eqref{CW:idle_prob} and the gradient bound in \eqref{eq:WCder3} to get
\begin{align*}
&\ \frac{\mu }{2} \E \bigg[\int_{\tilde X(\infty)}^{\tilde X(\infty)+\delta} \abs{f_h'''(y)} dy \bigg] \\
\leq&\   \frac{\delta}{2}\Big((17 + 10/\abs{\zeta})\Prob(\tilde X(\infty) \leq -\zeta - \delta) + 2\Prob(\tilde X(\infty) \geq -\zeta)\Big) \\
\leq&\ \frac{\delta}{2}\Big(17 + \frac{10}{\abs{\zeta}}(3 \abs{\zeta}) \Big) \leq 24\delta.
\end{align*}
By a similar argument, we can show that
\begin{align*}
\frac{\mu }{2}\E \bigg[\int_{\tilde X(\infty)-\delta}^{\tilde X(\infty)} \abs{f_h'''(y)} dy \bigg] \leq 24\delta,
\end{align*}
with the only difference in the argument being that we consider the cases when $\tilde X(\infty) \leq -\zeta$ and $\tilde X(\infty) \geq -\zeta + \delta$, instead of $\tilde X(\infty) \leq -\zeta -\delta$ and $\tilde X(\infty) \geq -\zeta$. Lastly, we use the form of $b(x)$, the moment bounds \eqref{CW:xminusdelta}, \eqref{CW:xminuszeta}, and \eqref{eq:xplusbound}, and the gradient bound \eqref{eq:WCder3} to get 
\begin{align*}
&\ \frac{\delta}{2} \E \bigg[ \big| b(\tilde X(\infty)) \big| \int_{\tilde X(\infty)-\delta}^{\tilde X(\infty)} \abs{f_h'''(y)} dy \bigg] \\
\leq&\ \frac{\delta^2}{2}\Big( (17 + 10/\abs{\zeta})\E \Big[ \big| \tilde X(\infty)\big| 1(\tilde X(\infty) \leq -\zeta) \Big] + 2\abs{\zeta} \Prob(\tilde X(\infty) \geq -\zeta + \delta)\Big)\\
\leq&\ \frac{\delta^2}{2}\Big( (17 + 10/\abs{\zeta})\Big(2\abs{\zeta} \wedge \sqrt{\frac{4}{3} + \frac{2\delta^2}{3}}\Big) + 14/4 \Big)\\
\leq&\ \frac{\delta^2}{2}\bigg(17 \sqrt{2} + 20 + 14/4 \bigg) \leq 24\delta^2.
\end{align*}
Hence, from \eqref{eq:first_bounds} we conclude that for all $R \geq 1$, and $h(x) \in \lipone$,
\begin{align}
&\ \Big| \E h(\tilde X(\infty)) - \E h(Y(\infty)) \Big| \leq  \delta (118 + 24 + 24 + 24 \delta) \leq 190\delta, \label{eq:intermproofWC}
\end{align}
which proves Theorem~\ref{thm:erlangCW}.
\end{proof}
\section{Proof Outline for Theorem~\ref{thm:erlangAW} (Erlang-A Wasserstein)}
We begin by stating some necessary moment and gradient bounds, and then  outline the proof of Theorems~\ref{thm:erlangAW}.

\subsection{Erlang-A Moment and Gradient Bounds} \label{sec:CAapplemmas}

 The following lemma states the necessary moment bounds for the Erlang-A model. The underloaded and overloaded cases have to be handled separately. Since the drift $b(x)$ is different between the Erlang-A and Erlang-C models, the quantities bounded in the following lemma will resemble those in Lemma~\ref{lem:moment_bounds_C}, but will not be identical. Its proof is outlined in Appendix~\ref{app:mom_A}.
\begin{lemma}\label{lem:moment_bounds_A_under}
Consider the Erlang-A model ($\alpha>0$). Fix $n \geq 1, \lambda > 0, \\ \mu > 0$, and $\alpha > 0$. If $0 < R \leq n $ (an underloaded system), then
\allowdisplaybreaks
\begin{align}
&\E\Big[ \big(\tilde X(\infty)\big)^2 1(\tilde X(\infty)\leq -\zeta)\Big] \leq \frac{1}{3}\Big(\frac{\alpha}{\mu }\delta^2 + \delta^2 + 4 \Big), \label{eq:mwuK1}\\
&\E\Big[ \big|\tilde X(\infty)1(\tilde X(\infty)\leq -\zeta)\big|\Big] \leq \sqrt{\frac{1}{3}\Big(\frac{\alpha}{\mu }\delta^2 + \delta^2 + 4 \Big)}, \label{eq:mwu1}\\
&\E\Big[ \big|\tilde X(\infty)1(\tilde X(\infty)\leq -\zeta)\big|\Big] \leq 2\abs{\zeta} + \frac{\alpha}{\mu } \sqrt{\frac{1}{3}\Big(\frac{\mu }{\alpha}\delta^2 + \frac{\mu }{\alpha}4 + \delta^2  \Big)},  \label{eq:mwu2}\\
&\E \Big[ \big|\tilde X(\infty)1(\tilde X(\infty)\geq -\zeta)\big|\Big]\notag \\ 
\leq&\ \bigg(1 + \frac{\delta^2}{4} +  \frac{\delta}{2} \sqrt{\frac{1}{3}\Big(\frac{\alpha}{\mu }\delta^2 + \delta^2 + 4 \Big)}\bigg) \Big(\frac{\mu }{\mu \wedge \alpha} \wedge \frac{1}{\abs{\zeta}} \Big), \label{eq:mwu3}\\
&\E \Big[ (\tilde X(\infty)+\zeta)^2 1(\tilde X(\infty)\geq -\zeta)\Big]\leq \frac{1}{3}\Big(\frac{\mu }{\alpha}\delta^2 + \frac{\mu }{\alpha}4 + \delta^2  \Big), \label{eq:mwuK2}\\
&\E \Big[ (\tilde X(\infty)+\zeta)1(\tilde X(\infty)\geq -\zeta)\Big]\leq \sqrt{\frac{1}{3}\Big(\frac{\mu }{\alpha}\delta^2 + \frac{\mu }{\alpha}4 + \delta^2  \Big)}, \label{eq:mwu4}\\
&\E \Big[ (\tilde X(\infty)+\zeta)1(\tilde X(\infty)\geq -\zeta)\Big]\leq \frac{1}{\abs{\zeta}} \Big( \frac{\delta^2}{4}\frac{\alpha}{\mu } + \frac{\delta^2}{4} + 1 \Big), 
\label{eq:mwu5}\\
&\Prob(\tilde X(\infty)\leq -\zeta) \leq (2+\delta)\bigg(\abs{\zeta} + \frac{\alpha}{\mu } \sqrt{\frac{1}{3}\Big(\frac{\mu }{\alpha}\delta^2 + \frac{\mu }{\alpha}4 + \delta^2  \Big)}\bigg). \label{eq:mwu6}
\end{align}
and if $n \leq R$ (an overloaded system), then
\begin{align}
&\E\Big[ \big|\tilde X(\infty)1(\tilde X(\infty)\leq -\zeta)\big|\Big]\leq \sqrt{\frac{1}{\alpha \wedge \mu } \Big( \alpha \frac{\delta^2}{4} + \mu \Big) },
\label{eq:mwo7}\\
&\E\Big[ \big|\tilde X(\infty)1(\tilde X(\infty)\leq -\zeta)\big|\Big]\leq  \frac{1}{\abs{\zeta}} \Big(\frac{\delta^2}{4}+\frac{\mu}{\alpha}\Big),
\label{eq:mwo8}\\
&\E \Big[(\tilde X(\infty))^2 1(\tilde X(\infty) \geq -\zeta)\Big] \leq  \frac{1}{3}\Big(\delta^2 + 4\frac{\mu }{\alpha} \Big), \label{eq:mwo2}\\
&\E\Big[ \big|\tilde X(\infty)1(\tilde X(\infty)\geq -\zeta)\big|\Big] \leq  \sqrt{\frac{1}{3}\Big(\delta^2 + 4\frac{\mu }{\alpha} \Big)}, \label{eq:mwo1}\\
&\E \Big[\big| (\tilde X(\infty)+\zeta)1(\tilde X(\infty)\leq -\zeta)\big|\Big]\leq \frac{1}{\abs{\zeta}}\Big(\frac{\delta^2}{4}+1\Big),
\label{eq:mwo3}\\
&\E \Big[ (\tilde X(\infty)+\zeta)^21(\tilde X(\infty)\leq -\zeta)\Big]\leq  \frac{\delta^2}{4}\frac{\alpha}{\mu }+1, \label{eq:mwoK1}\\
&\E \Big[\big| (\tilde X(\infty)+\zeta)1(\tilde X(\infty)\leq -\zeta)\big|\Big]\leq  \sqrt{ \frac{\delta^2}{4}\frac{\alpha}{\mu }+1}, \label{eq:mwo4}\\
&\E \Big[\big| (\tilde X(\infty)+\zeta)1(\tilde X(\infty)\leq -\zeta)\big|\Big]\leq  \frac{\alpha}{\mu} \sqrt{\frac{1}{3}\Big(\delta^2 + 4\frac{\mu }{\alpha} \Big)},
\label{eq:mwo5}\\
&\Prob(\tilde X(\infty) \leq -\zeta)\leq (3+\delta)\frac{16}{\sqrt 2}\Big(\frac{\delta^2}{4}+1\Big)  \bigg(\Big(\frac{1}{\zeta}\vee \frac{\alpha}{\mu}\Big)\wedge \sqrt{\frac{\alpha}{\mu}}\bigg).\label{eq:mwo10}
\end{align}
\end{lemma}

The following Wasserstein gradient bounds are proved in Appendix~\ref{app:grad_A}.

\begin{lemma} \label{lem:gradboundsAWunder} 
Consider the Erlang-A model ($\alpha > 0$), and fix $h(x) \in \lipone$. Then $f_h(x)$ given in Lemma~\ref{lem:solution} is twice continuously differentiable, with an absolutely continuous second derivative. Furthermore, there exists a constant $C > 0$ independent of $\lambda, n, \mu$, and $\alpha$ such that for all $n \geq 1, \lambda > 0, \mu>0$, and $\alpha > 0$ satisfying $0 < R \leq n$ (an underloaded system), 
\begin{align}
\abs{f_h'(x)}\leq&\ 
\begin{cases}
C \left(\sqrt{\frac{\mu}{\alpha}}\wedge \frac{1}{\abs{\zeta}}+1\right)\frac{1}{\mu},\quad x\leq -\zeta,\\
C \left(\frac{\mu}{\alpha}+\sqrt{\frac{\mu}{\alpha}}\wedge \frac{1}{\abs{\zeta}}+1\right)\frac{1}{\mu},\quad x\geq -\zeta,
\end{cases} \label{eq:gwu1} \\
\abs{f_h''(x)} \leq&\
\begin{cases}
C \left(\sqrt{\frac{\mu}{\alpha}}\wedge \frac{1}{\abs{\zeta}}+1\right)\frac{1}{\mu}, \quad x\leq 0, \\
C \left[\left(\frac{\alpha}{\mu}+\sqrt{\frac{\alpha}{\mu}}+1\right)\left(\sqrt{\frac{\mu}{\alpha}}\wedge \frac{1}{\abs{\zeta}}\right)+1\right]\frac{1}{\mu},\quad x\in [0,-\zeta],\\
C\left(\frac{\alpha}{\mu}+\sqrt{\frac{\alpha}{\mu}}+1\right)\left(\sqrt{\frac{\mu}{\alpha}}\wedge \frac{1}{\abs{\zeta}}\right)\frac{1}{\mu},\quad x\geq -\zeta,
\end{cases} \label{eq:gwu2} \\
\end{align}
and for those $x \in \R$ where $f_h'''(x)$ exists,
\begin{align}
\abs{f_h'''(x)} \leq&\ 
\begin{cases}
C\left(\sqrt{\frac{\mu}{\alpha}}\wedge \frac{1}{\abs{\zeta}}+1\right)\frac{1}{\mu},
\quad x\leq 0,\\
C\left(\sqrt{\frac{\mu}{\alpha}}\wedge \frac{1}{\abs{\zeta}}+\frac{\alpha}{\mu}+\sqrt{\frac{\alpha}{\mu}}+1\right)\frac{1}{\mu}, \quad x\in [0,-\zeta], \\
C \left(\frac{\alpha}{\mu}+\sqrt{\frac{\alpha}{\mu}}+1\right)\frac{1}{\mu},\quad x\geq -\zeta,
\end{cases} \label{eq:gwu3}
\end{align}
and for all $n \geq 1, \lambda > 0, \mu>0$, and $\alpha > 0$ satisfying $n \leq R$ (an overloaded system),
\begin{align}
\abs{f_h'(x)}\leq&\
\begin{cases}
C \Big(\frac{1}{\mu}+ \frac{1}{\sqrt \alpha}\frac{1}{\sqrt\mu} + \frac{\zeta}{\mu} \wedge \frac{1}{\alpha}\Big), \quad x\leq-\zeta,\\
C \Big(\frac{1}{\mu}+\frac{1}{\sqrt\alpha}\frac{1}{\sqrt\mu}+\frac{1}{\alpha}\Big), \quad x\geq -\zeta,
\end{cases} \label{eq:gwo1} \\
\abs{f_h''(x)}\leq &\
\begin{cases}
C\Big(\frac{1}{\mu}+ \frac{1}{\sqrt \alpha}\frac{1}{\sqrt\mu} + \frac{\zeta}{\mu} \wedge \frac{1}{\alpha}\Big),\quad x\leq-\zeta,\\
 C\Big(\frac{\alpha}{\mu}+\sqrt{\frac{\alpha}{\mu}}+1\Big)\frac{1}{\mu}|x| +
C\Big(\frac{1}{\mu}+\frac{1}{\sqrt \alpha}\frac{1}{\sqrt\mu}\Big),\quad x\geq -\zeta,
\end{cases} \label{eq:gwo2}\\
\end{align}
and for those $x \in \R$ where $f_h'''(x)$ exists,
\begin{align}
\abs{f_h'''(x)}\leq&\ \frac{C}{\mu}\Big(1+\sqrt{\frac{\mu}{\alpha}}+\zeta \wedge \frac{\mu}{\alpha}\Big),\quad x\leq -\zeta,\label{eq:gwo3}\\
\abs{f_h'''(x)}\leq&\ \frac{C}{\mu}\Big(\frac{\alpha}{\mu}+\sqrt{\frac{\alpha}{\mu}}+1\Big)
\Big(1+\frac{\alpha}{\mu}x^2\Big)
+\frac{C}{\mu}\Big(\frac{\alpha}{\mu}+\sqrt{\frac{\alpha}{\mu}}\Big)\abs{x},\quad x\geq -\zeta,\label{eq:gwo41}\\
\abs{f_h'''(x)} \leq&\
\frac{C}{\mu}\Big(\frac{\alpha}{\mu}+\sqrt{\frac{\alpha}{\mu}}+1\Big)
+ \frac{C}{\mu}\Big(\frac{\alpha}{\mu}+\sqrt{\frac{\alpha}{\mu}}+1\Big)^2 \abs{x} ,\quad x\geq -\zeta. \label{eq:gwo42}
\end{align}

\end{lemma}

\subsection{Proof Outline} \label{app:AWoutline}
Proving Theorem~\ref{thm:erlangAW} consists of bounding the four error terms in \eqref{eq:first_bounds}. Since the procedure is very similar to the proof of Theorem~\ref{thm:erlangCW}, we will only outline which gradient and moment bounds need to be used to bound each error term.

We start with the underloaded case, when $R \leq n$. To bound the first term in \eqref{eq:first_bounds}, we use moment bounds \eqref{eq:mwu1}, \eqref{eq:mwu2}, and \eqref{eq:mwu4}, together with the gradient bounds in \eqref{eq:gwu2}. For the second and third terms, we use moment bound \eqref{eq:mwu6} and the gradient bounds in \eqref{eq:gwu3}. For the fourth term, we use moment bounds \eqref{eq:mwu1}--\eqref{eq:mwu4}, and the gradient bounds in \eqref{eq:gwu3}.

We now prove the overloaded case, when $R \geq n$. To bound the first term in \eqref{eq:first_bounds}, we use moment bounds \eqref{eq:mwo7}--\eqref{eq:mwo5}, together with the gradient bounds in \eqref{eq:gwo2}. For the second and third terms, we use moment bounds \eqref{eq:mwo2},\eqref{eq:mwo1}, and \eqref{eq:mwo10}, together with the gradient bounds in \eqref{eq:gwo3} and \eqref{eq:gwo41}. For the fourth term, we use moment bounds \eqref{eq:mwo7}--\eqref{eq:mwo5}, and gradient bounds in \eqref{eq:gwo3} and \eqref{eq:gwo42}.

\section{The Kolmogorov Metric}
\label{sec:CAkolmogorov}
In this section we prove Theorem~\ref{thm:erlangCK}, which is stated in the Kolmogorov setting. The biggest difference between the  Wasserstein and Kolmogorov settings is that in the latter, the test functions $h(x)$ used in the Poisson equation \eqref{CA:poisson} are discontinuous. For this reason, new gradient bounds need to be derived separately for the Kolmogorov setting; we present these new gradient bounds in Section~\ref{sec:kgrad}. Furthermore, the solution to the Poisson equation no longer has a continuous second derivative, meaning that the Taylor expansion we used to derive the upper bound in \eqref{eq:first_bounds} is invalid. We discuss an alternative to \eqref{eq:first_bounds} in Section~\ref{sec:ktaylor}. This alternative bound contains a new error term that cannot be handled by the gradient bounds, nor the moment bounds. This term appears because the solution to the Poisson equation has a discontinuous second derivative, and to bound it we present Lemma~\ref{lem:kolmfixC}. We then prove Theorem~\ref{thm:erlangCK}  in  Section~\ref{sec:kproof}, and outline the proof for Theorem~\ref{thm:erlangAK} in Section~\ref{app:AKoutline}.

\subsection{Kolmogorov Gradient Bounds}
\label{sec:kgrad}

Recall that in the Kolmogorov setting, we take the class of test functions for the Poisson equation \eqref{CA:poisson} to be $\HH_K$ defined in \eqref{eq:classkolm}. For the statement of the following two lemmas, we fix $a \in \R$ and set $h(x) = 1_{(-\infty, a]}(x)$. Recall that the Poisson equation has multiple solutions, but going forward we always work with the one from Lemma~\ref{lem:solution}. Furthermore, we use $f_a(x)$ instead of $f_h(x)$ to denote the solution to the Poisson equation.
\begin{lemma} \label{lem:gradboundsCK}
Consider the Erlang-C model ($\alpha = 0$). Then $f_a(x)$ is continuously differentiable, with an absolutely continuous derivative. Furthermore, for all $n \geq 1, \lambda > 0$, and $\mu > 0$ satisfying $0 < R  < n$, 
\begin{align}
\abs{f_a'(x)} \leq 
\begin{cases}
4/\mu , \quad x \leq -\zeta, \\
\frac{1}{\mu \abs{\zeta}}, \quad x \geq -\zeta,
\end{cases} \label{eq:KCder1}
\end{align}
and for all $x \in \R$,
\begin{align}
\abs{f_a''(x)} \leq 2/\mu, \label{eq:KCder2}
\end{align}
where $f_a''(x)$ is understood to be the left derivative at the point $x = a$.
\end{lemma}

\begin{lemma}\label{lem:gradboundsAK}
Consider the Erlang-A model ($\alpha > 0$). Then $f_a(x)$ is continuously differentiable, with an absolutely continuous derivative.  Fix $n \geq 1, \lambda > 0, \mu > 0$, and $\alpha > 0$. If $0 < R  \leq n$ (an underloaded system), then
\begin{align}
\abs{f_a'(x)} \leq 
\begin{cases}
\frac{1}{\mu }\sqrt{2\pi}e^{1/2} , \quad x \leq -\zeta, \\
\frac{1}{\mu } \Big(\sqrt{\frac{\pi}{2} \frac{\mu }{\alpha}} \wedge\frac{1}{\abs{\zeta}}\Big), \quad x \geq -\zeta,
\end{cases} \label{eq:ACuder1}
\end{align}
and if $n \leq R$ (an overloaded system), then
\begin{align}
\abs{f_a'(x)} \leq 
\begin{cases}
\frac{1}{\mu }\sqrt{\frac{\pi}{2}} , \quad x \leq -\zeta, \\
\frac{1}{\mu } \sqrt{\frac{\pi}{2}}\Big(1 + \sqrt{ \frac{\mu }{\alpha}} \Big), \quad x \geq -\zeta.
\end{cases} \label{eq:ACoder1}
\end{align}
Moreover,  for all  $\lambda>0, n \geq 1, \mu > 0$, and $\alpha > 0$, and all $x \in \R$,
\begin{align}
\abs{f_a''(x)} \leq 3/\mu, \label{eq:ACder2}
\end{align}
where $f_a''(x)$ is understood to be the left derivative at the point $x = a$.
\end{lemma}
\noindent Lemmas~\ref{lem:gradboundsCK} and \ref{lem:gradboundsAK} are proved in Appendix~\ref{app:kgradient}. Unlike the Wasserstein setting, these lemmas do not guarantee that $f_a''(x)$ is absolutely continuous. Indeed, for any $a \in \R$, 
substituting $h(x) = 1_{(-\infty, a]}(x)$ into~\eqref{eq:poisson} gives us
\begin{align*}
\mu f_a''(x) =  \Prob(Y(\infty) \leq a) - 1_{(-\infty, a]}(x) - b(x) f_a'(x).
\end{align*}
Since $b(x) f_a'(x)$ is a continuous function, the above equation implies that $f_a''(x)$ is discontinuous at the point $x = a$. Thus, we can no longer use the error bound in \eqref{eq:first_bounds}, and require a different expansion of $G_{\tilde X} f_a(x)$. 
\subsection{Alternative Taylor Expansion}
\label{sec:ktaylor}
To get an error bound similar to \eqref{eq:first_bounds}, we first define 
\begin{align}
\epsilon_1(x) =&\ \int_x^{x+\delta} (x+\delta -y)(f_a''(y)-f_a''(x-))dy, \label{CA:eps1def} \\
\epsilon_2(x) =&\  \int_{x-\delta}^{x} (y-(x-\delta))(f_a''(y)-f_a''(x-))dy. \label{CA:eps2def}
\end{align}
Now observe that
\begin{align}
f_a(x+\delta) -f_a(x) =&\  f_a'(x) \delta  + \int_x^{x+\delta} (x+\delta -y)f_a''(y)dy \notag  \\
=&\  f_a'(x) \delta  +  \frac{1}{2}\delta^2 f_a''(x-)\notag  \\
&+  \int_x
^{x+\delta} (x+\delta -y)(f_a''(y)-f_a''(x-))dy  \notag \\
 =&\  f_a'(x) \delta  + \frac{1}{2}\delta^2 f_a''(x-) +  \epsilon_1(x), \label{CA:ktaylor1}
\end{align}
and
\begin{align*}
  (f_a(x-\delta) -f_a(x))  =&\ - f_a'(x) \delta  +
 \int_{x-\delta}^x (y-(x-\delta))f_a''(y)dy  \\
=&\ -f_a'(x) \delta  +  \frac{1}{2}\delta^2 f_a''(x-) \\
&+  \int_{x-\delta}^{x} (y-(x-\delta))(f_a''(y)-f_a''(x-))dy  \\
=&\ -f_a'(x) \delta  +  \frac{1}{2}\delta^2 f_a''(x-) + \epsilon_2(x).
\end{align*}
For $k \in \Z_+$ and  $x=x_k=\delta(k-x(\infty))$, we recall the forms of $G_Y f_a(x)$ and $G_{\tilde X} f_a(x)$ from \eqref{CA:GY} and \eqref{eq:GX} to see that
\begin{align*}
  G_{\tilde X}f_a(x) =&\ \lambda  \delta f_a'(x) + \lambda \frac{1}{2}\delta^2 f_a''(x-) + \lambda \epsilon_1(x)\\
  & - d(k) \delta f_a'(x) + d(k) \frac{1}{2} \delta^2 f_a''(x-) + d(k)
  \epsilon_2(x)\\
=&\ b(x) f_a'(x) + \lambda \frac{1}{2}\delta^2 f_a''(x-) + \lambda \epsilon_1(x)\\
&+ (\lambda -\frac{1}{\delta}b(x)) \frac{1}{2} \delta^2 f_a''(x-) + (\lambda-\frac{1}{\delta}b(x) )
  \epsilon_2(x)\\
=&\ G_Y f(x) - b(x)\frac{1}{2} \delta f_a''(x-) + \lambda(\epsilon_1(x)+\epsilon_2(x)) - \frac{1}{\delta}b(x)\epsilon_2(x),
\end{align*}
where in the second equality we used the fact that $b(x) = \delta( \lambda - d(k))$, and in the last equality we use that $\delta^2 \lambda = \mu$. Combining this with \eqref{eq:gen_bound}, we have an error bound similar to \eqref{eq:first_bounds}:
\begin{align}
&\ \Big| \Prob(\tilde X(\infty) \leq a) - \Prob(Y(\infty) \leq a) \Big| \notag \\
 \leq&\ \frac{1}{2} \delta \E \Big[ \big|f_a''(\tilde X(\infty)-)b(\tilde X(\infty)) \big| \Big] + \lambda \E \Big[ \big| \epsilon_1(\tilde X(\infty)) \big| \Big] \notag \\
& + \lambda \E \Big[ \big|\epsilon_2(\tilde X(\infty))\big| \Big] + \frac{1}{\delta} \E \Big[ \big| b(\tilde X(\infty))\epsilon_2(\tilde X(\infty))\big| \Big], \label{eq:third_bounds}
\end{align}
where $\epsilon_1(x)$ and $\epsilon_2(x)$ are as in \eqref{CA:eps1def} and \eqref{CA:eps2def}, respectively. To bound the error terms in \eqref{eq:third_bounds} that are associated with $\epsilon_1(x)$ and $\epsilon_2(x)$, we need to analyze the difference $f_a''(y) - f_a''(x-)$ for $\abs{x-y} \leq \delta$. Since $f_a(x)$ is a solution to the Poisson equation (\ref{CA:poisson}), we see that for any $x, y \in \R$ with $y \neq a$, 
\begin{align*}
f_a''(y)-f_a''(x-) = \frac{1}{\mu} \big[ 1_{(-\infty, a]}(x) - 1_{(-\infty, a]}(y) + b(x)f_a'(x) - b(y)f_a'(y) \big].
\end{align*}
Therefore, for any $y \in [x, x + \delta]$  with $y \neq a$,
\begin{align}
&\ \abs{ f_a''(y)-f_a''(x-)} \notag \\
\le &\  \frac{1}{\mu} \big[ 1_{(a-\delta, a]}(x) + \abs{b(x)}\abs{f_a'(x) -f_a'(y)}
+ \abs{b(x)-b(y))}\abs{f_a'(y)} \big] \notag \\
 & \le  \frac{1}{\mu} \big[ 1_{(a-\delta, a]}(x) + \delta \abs{b(x)}\norm{f''} + \abs{b(x)-b(y))}\abs{f_a'(y)} \big], \label{eq:eps1b2}
\end{align}
and likewise, for any $y \in [x-\delta, x]$ with $y \neq a$, 
\begin{align}
&\ \abs{ f_a''(y)-f_a''(x-)} \notag \\
\le &\ \frac{1}{\mu} \big[ 1_{(a, a + \delta]}(x) + \abs{b(x)}\abs{f_a'(x) -f_a'(y)}
+ \abs{b(x)-b(y))}\abs{f_a'(y)} \big] \notag \\
 & \le  \frac{1}{\mu} \big[ 1_{(a, a + \delta]}(x) + \delta \abs{b(x)}\norm{f''} + \abs{b(x)-b(y))}\abs{f_a'(y)} \big].\label{eq:eps2b2}
\end{align}
The inequalities above contain the indicators $1_{(a-\delta, a]}(x)$ and $1_{(a, a + \delta]}(x)$. When we consider the upper bound in \eqref{eq:third_bounds}, these indicators will manifest themselves as probabilities $\Prob(a - \delta < \tilde X(\infty) \leq a)$ and $\Prob(a  < \tilde X(\infty) \leq a+\delta)$. To this end we present the following lemma, which will be used in the proof of Theorem~\ref{thm:erlangCK}.
\begin{lemma}\label{lem:kolmfixC}
Consider the Erlang-C model ($\alpha = 0$). Let $W$ be an arbitrary random variable with cumulative distribution function $F_W:\R \to [0,1]$. Let $\omega(F_W)$ be the modulus of continuity of $F_W$, defined as 
\begin{align*}
\omega(F_W) = \sup_{\substack{x, y \in \R \\ x \neq y}} \frac{\abs{F_W(x)-F_W(y)}}{\abs{x-y}}.
\end{align*}
Recall that $d_K(\tilde X(\infty), W)$ is the Kolmogorov distance between $X(\infty)$ and $W$. Then for any $a \in \R$, $n \geq 1$, and $0 < R< n$,  
\begin{align*}
\Prob( a - \delta < \tilde X(\infty) \leq a + \delta) \leq  \omega(F_W)2\delta  + d_K(\tilde X(\infty), W) + 9\delta^2 + 8\delta^4.
\end{align*}
\end{lemma}
\noindent This lemma is proved in Section~\ref{app:kolmfixC}. We will apply Lemma~\ref{lem:kolmfixC} with $W = Y(\infty)$ in the proof of Theorem~\ref{thm:erlangCK} that follows. The following lemma guarantees that the modulus of continuity of the cumulative distribution function of $Y(\infty)$ is bounded by a constant independent of $\lambda, n$, and $\mu$. Its proof is provided in Section~\ref{app:densboundC}.
\begin{lemma} \label{lem:densboundC}
Consider the Erlang-C model ($\alpha = 0$), and let $\nu(x)$ be the density of $Y(\infty)$, defined in \eqref{eq:stddenC}. Then for for all $n \geq 1, \lambda > 0$, and $\mu>0$ satisfying $0 < R<n $,
\begin{align*}
\abs{\nu(x)} \leq \sqrt{\frac{2}{\pi}} , \quad x \in \R.
\end{align*}
\end{lemma}
\noindent Lemmas~\ref{lem:kolmfixC} and \ref{lem:densboundC} are stated for the Erlang-C model, but one can easily repeat the arguments in the proofs of those lemmas to prove analogues for the Erlang-A model. Therefore, we state the following lemmas without proof.
\begin{lemma}\label{lem:kolmfixA}
Consider the Erlang-A model ($\alpha > 0$). Let $W$ be an arbitrary random variable with cumulative distribution function $F_W:\R \to [0,1]$. Let $\omega(F_W)$ be the modulus of continuity of $F_W$. Then for any $a \in \R$, $\alpha > 0$, $n \geq 1$, and $R > 0$,  
\begin{align*}
&\ \Prob( a - \delta < \tilde X(\infty) \leq a + \delta) \\
\leq&\  \omega(F_W)2\delta  + d_K(\tilde X(\infty), W) + 9 \Big( \frac{\alpha}{\mu } \vee 1 \Big)\delta^2 + 8\Big( \frac{\alpha}{\mu } \vee 1 \Big)^2\delta^4.
\end{align*}
\end{lemma}
\begin{lemma} \label{lem:densboundA}
Consider the Erlang-A model ($\alpha > 0$), and let $\nu(x)$ be the density of $Y(\infty)$. Fix $n \geq 1, \lambda > 0, \mu>0$, and $\alpha > 0$. If $0 < R \leq n $, then
\begin{align*}
\abs{\nu(x)} \leq \sqrt{\frac{2}{\pi}} , \quad x \in \R,
\end{align*}
and  if $n  \leq R$, then
\begin{align*}
\abs{\nu(x)} \leq \sqrt{\frac{2}{\pi}} \sqrt{\frac{\alpha}{\mu } }, \quad x \in \R.
\end{align*}
\end{lemma}

\subsection{Proof of Theorem~\ref{thm:erlangCK} (Erlang-C Kolmogorov)}
\label{sec:kproof}
\begin{proof}[Proof of Theorem~\ref{thm:erlangCK}] 
Throughout the proof we assume that $R \geq 1$, or equivalently, $\delta \leq 1$. For $h(x) = 1_{(-\infty, a]}(x)$, we let $f_a(x)$ be a solution the Poisson equation \eqref{CA:poisson} with parameter $a_2 = 0$. In this proof we will show that for all $a \in \R$, 
\begin{align}
&\ \abs{\Prob(\tilde X(\infty) \leq a) - \Prob(Y(\infty) \leq a)} \leq \frac{1}{2}\Prob(a - \delta < \tilde X(\infty) \leq a + \delta) + 59\delta, \label{eq:intermproofKC}
\end{align}
The upper bound in \eqref{eq:intermproofKC} is similar to \eqref{eq:intermproofWC}, however \eqref{eq:intermproofKC} has the extra term
\begin{align}
\frac{1}{2}\Prob(a - \delta < \tilde X(\infty) \leq a + \delta). \label{eq:extraterm}
\end{align}
The reason this term appears in the Kolmogorov setting but not in the Wasserstein setting is because $f_a''(x)$ is discontinuous in the Kolmogorov case, as opposed to the Wasserstein case where $f_h''(x)$ is continuous. Applying Lemmas~\ref{lem:kolmfixC} and \ref{lem:densboundC} to the right hand side of \eqref{eq:intermproofKC}, and taking the supremum over all $a \in \R$ on both sides, we see that 
\begin{align*}
&\ d_K(\tilde X(\infty), Y(\infty)) \leq  \frac{1}{2}d_K(\tilde X(\infty), Y(\infty)) + 2\sqrt{\frac{2}{\pi}} \delta  + 9\delta^2 + 8\delta^4 +  59\delta,
\end{align*}
or 
\begin{align*}
d_K(\tilde X(\infty), Y(\infty)) \leq 156\delta.
\end{align*}
We want to add that Lemma~\ref{lem:kolmfixC} makes heavy use of the birth-death structure of the Erlang-C model, and that it is not obvious how to handle \eqref{eq:extraterm} more generally.

To prove Theorem~\ref{thm:erlangCK} it remains to verify \eqref{eq:intermproofKC}, which we now do. The argument we will use is similar to the argument used to prove \eqref{eq:intermproofWC} in Theorem~\ref{thm:erlangCW}. We will bound each of the terms in \eqref{eq:third_bounds}, which we recall here as 
\begin{align*}
&\ \Big| \Prob(\tilde X(\infty) \leq a) - \Prob(Y(\infty) \leq a) \Big| \notag \\
 \leq&\ \frac{1}{2} \delta \E \Big[ \big|f_a''(\tilde X(\infty)-)b(\tilde X(\infty)) \big| \Big] + \lambda \E \Big[ \big| \epsilon_1(\tilde X(\infty))\big| \Big] \notag \\
& + \lambda \E \Big[ \big| \epsilon_2(\tilde X(\infty))\big| \Big] + \frac{1}{\delta} \E \Big[ \big| b(\tilde X(\infty))\epsilon_2(\tilde X(\infty))\big| \Big].
\end{align*}
We also recall the form of $b(x)$ from \eqref{eq:bexpand}. We use the moment bounds \eqref{CW:xminusdelta} and \eqref{eq:xplusbound}, and the gradient bound \eqref{eq:KCder2} to see that
\begin{align}
&\ \E \Big[ \big|f_a''(\tilde X(\infty)-)b(\tilde X(\infty)) \big|\Big]  \notag \\
\leq&\ \frac{2}{\mu}\E \Big[\big|b(\tilde X(\infty)) \big| \Big] \notag \\
=&\ 2  \E \Big[\big|\tilde X(\infty) 1(\tilde X(\infty) \leq  -\zeta - \delta)\big| \Big]+ 2 \abs{\zeta} \Prob (\tilde X(\infty) \geq -\zeta) \notag  \\
\leq&\ 2\sqrt{\frac{4}{3} + \frac{2\delta^2}{3}} + 2\Big(\abs{\zeta}\wedge \E \Big[ \big|\tilde X(\infty)\big| 1(\tilde X(\infty) \geq -\zeta) \Big]\Big) \notag \\
\leq&\ 2\sqrt{2} + \frac{14}{4} \leq 7. \label{eq:term1kolmc}
\end{align}
Next, we use \eqref{eq:eps1b2}, \eqref{eq:term1kolmc}, and the gradient bound \eqref{eq:KCder1} to get
\begingroup
\allowdisplaybreaks
\begin{align*}
&\ \lambda \E \Big[ \big|\epsilon_1(\tilde X(\infty))\big| \Big] \\
=&\ \lambda \E \bigg[\int_{\tilde X(\infty)}^{\tilde X(\infty)+\delta} (\tilde X(\infty)+\delta -y) \big|f_a''(y)-f_a''(\tilde X(\infty)-)\big|dy\bigg] \\
\leq&\ \frac{\lambda}{\mu} \E \bigg[ 1_{(a-\delta, a]}(\tilde X(\infty))\int_{\tilde X(\infty)}^{\tilde X(\infty)+\delta} (\tilde X(\infty)+\delta -y)dy\bigg] \\
&+ \frac{\lambda}{\mu}  \delta^3 \E \Big[ \big| b(\tilde X(\infty)) \big| \Big]\norm{f_a''} + \frac{\lambda}{\mu} \delta \E \bigg[\int_{\tilde X(\infty)}^{\tilde X(\infty)+\delta} \big|b(\tilde X(\infty))-b(y))\big| \big|f_a'(y)\big|dy \bigg] \\
\leq&\ \frac{1}{2} \Prob( a - \delta < \tilde X(\infty) \leq a) + 7 \delta +  4\delta \\
=&\ \frac{1}{2} \Prob( a - \delta < \tilde X(\infty) \leq a) + 11\delta,
\end{align*}%
\endgroup
where in the last inequality we used the fact that for $y \in [\tilde X(\infty), \tilde X(\infty) + \delta]$, 
\begin{align*}
b(\tilde X(\infty)) - b(y) = \mu \delta 1(\tilde X(\infty) \leq -\zeta - \delta).
\end{align*}
By a similar argument, one can check that
\begin{align*}
\lambda \E \Big[ \big|\epsilon_2(\tilde X(\infty))\big| \Big]\leq&\ \frac{1}{2} \Prob( a  < \tilde X(\infty) \leq a+\delta) +  11\delta,
\end{align*}
with the only difference in the argument being that we consider the cases when $\tilde X(\infty) \leq -\zeta$ and $\tilde X(\infty) \geq -\zeta + \delta$, instead of $\tilde X(\infty) \leq -\zeta -\delta$ and $\tilde X(\infty) \geq -\zeta$.
%\begin{align*}
%\lambda \E \abs{\epsilon_2(\tilde X(\infty))} =&\ \lambda \E \bigg|\int_{\tilde X(\infty)-\delta}^{\tilde X(\infty)} (y-(\tilde X(\infty)-\delta))(f_a''(y)-f_a''(\tilde X(\infty)))dy\bigg| \\
%\leq&\ \frac{\lambda}{\mu} \E 1_{(a, a+\delta]}(\tilde X(\infty))\int_{\tilde X(\infty)-\delta}^{\tilde X(\infty)} (y-(\tilde X(\infty)-\delta))dy \\
%&+ \frac{\lambda}{\mu} \E \delta^3 \abs{b(\tilde X(\infty))}\norm{f_a''} + \frac{\lambda}{\mu} \delta \E \int_{\tilde X(\infty)-\delta}^{\tilde X(\infty)} \abs{b(\tilde X(\infty))-b(y))}\abs{f_a'(y)}dy \\
%\leq&\ \frac{1}{2} \Prob( a < \tilde X(\infty) \leq a+\delta) + \delta \frac{3}{\mu}\E \big|b(\tilde X(\infty)) \big|+   \frac{\lambda}{\mu}\delta C\E\bigg[ 1(\tilde X(\infty) \leq -\zeta)\int_{\tilde X(\infty)-\delta}^{\tilde X(\infty)} \delta dy\bigg] \\
%\leq&\ \frac{1}{2} \Prob( a < \tilde X(\infty) \leq a+\delta) + \delta C_? +  C\delta.
%\end{align*}
Lastly, we use the first inequality in \eqref{eq:eps2b2} to see that
\begin{align*}
&\ \frac{1}{\delta} \E \Big[ \big| b(\tilde X(\infty))\epsilon_2(\tilde X(\infty))\big| \Big]\\
\leq&\ \frac{1}{\mu} \E \bigg[ \big| b(\tilde X(\infty))\big| \int_{\tilde X(\infty)-\delta}^{\tilde X(\infty)}  \Big[ 1_{(a,a+\delta]}(\tilde X(\infty)) \\
& \hspace{5cm}+\big|b(\tilde X(\infty))\big|\Big(\big|f_a'(\tilde X(\infty))\big| + \big|f_a'(y)\big| \Big) \\
& \hspace{5cm} + \big|b(\tilde X(\infty))-b(y))\big|\big|f_a'(y)\big| \Big]dy \bigg]\\
\leq&\  \delta \frac{1}{\mu} \E \Big[ \big|b(\tilde X(\infty))\big| \Big] + \delta\frac{1}{\mu} \E \Big[ \big|b^2(\tilde X(\infty))f_a'(\tilde X(\infty))\big| \Big] \\
&+  \frac{1}{\mu}\E \bigg[ \big| b^2(\tilde X(\infty))\big| \int_{\tilde X(\infty)-\delta}^{\tilde X(\infty)} \abs{f_a'(y)}dy\bigg] +4\delta^2 \E \Big[  \big| \tilde X(\infty)1(\tilde X(\infty) \leq -\zeta)\big| \Big]\\
\leq&\  \frac{7}{2} \delta  + \delta\frac{1}{\mu} \E \Big[ \big|b^2(\tilde X(\infty))f_a'(\tilde X(\infty))\big| \Big] \\
&+  \frac{1}{\mu}\E \bigg[ \big| b^2(\tilde X(\infty))\big| \int_{\tilde X(\infty)-\delta}^{\tilde X(\infty)} \abs{f_a'(y)}dy\bigg] +4 \sqrt{2}\delta^2,
\end{align*}
where in the last inequality we used \eqref{eq:term1kolmc} and the moment bound \eqref{CW:xminusdelta}. Now by \eqref{CW:xsquaredelta} and \eqref{eq:xplusbound},
\begin{align*}
&\ \delta\frac{1}{\mu} \E \Big[ \big| b^2(\tilde X(\infty))f_a'(\tilde X(\infty))\big| \Big] \\
\leq&\ 4\delta  \E \big[ \tilde X^2(\infty)1( \tilde X(\infty) \leq -\zeta ) \big] + \delta \abs{\zeta} \Prob(\tilde X(\infty) \geq -\zeta + \delta)\\
\leq&\ 8 \delta + \delta \frac{7}{4} \leq 10\delta,
\end{align*}
and similarly,
\begin{align*}
\frac{1}{\mu}\E \bigg[ \big| b^2(\tilde X(\infty))\big| \int_{\tilde X(\infty)-\delta}^{\tilde X(\infty)} \abs{f_a'(y)}dy\bigg]\leq&\ 10 \delta.
\end{align*}
Therefore,
\begin{align*}
\frac{1}{\delta} \E \Big[ \big| b(\tilde X(\infty))\epsilon_2(\tilde X(\infty))\big| \Big] \leq \frac{7}{2}\delta  + 20\delta + 4\sqrt{2}\delta^2 \leq 30 \delta.
\end{align*}
This verifies \eqref{eq:intermproofKC} and concludes the proof of Theorem~\ref{thm:erlangCK}.
\end{proof}

\subsection{Outline for Theorem~\ref{thm:erlangAK} (Erlang-A Kolmogorov)}\label{app:AKoutline}

The proof of Theorem~\ref{thm:erlangAK} is nearly identical to the proof of Theorem~\ref{thm:erlangCK}. Therefore, we only outline the key steps and differences. The goal is to obtain a version of \eqref{eq:intermproofKC}, from which the theorem follows by applying Lemmas~\ref{lem:kolmfixA} and \ref{lem:densboundA}. To get a version of \eqref{eq:intermproofKC}, we bound each of the terms in \eqref{eq:third_bounds}, just like we did in the proof of Theorem~\ref{thm:erlangCK}. The proof varies between the underloaded and overloaded cases.

We begin with the underloaded case ($1 \leq R \leq n$). To bound the first term in \eqref{eq:third_bounds}, we use moment bounds \eqref{eq:mwu1}, \eqref{eq:mwu3}, and \eqref{eq:mwu4}, together with gradient bound \eqref{eq:ACder2}. For the second and third terms in \eqref{eq:third_bounds} we use the gradient bound in \eqref{eq:ACuder1}. For the fourth error term, we use gradient bound \eqref{eq:ACuder1}, and moment bounds \eqref{eq:mwuK1}, \eqref{eq:mwu3}, and
\begin{align*}
&\ \E \Big[ \big( b(\tilde X(\infty))\big)^2 1( \tilde X(\infty) \geq -\zeta) \Big] \\
=&\ \alpha^2\E \Big[ \big( \tilde X(\infty)+ \zeta\big)^2 1( \tilde X(\infty) \geq -\zeta) \Big]+ \mu^2 \zeta^2\Prob(\tilde X(\infty) \geq -\zeta) \\
& + 2\alpha \mu \abs{\zeta}\E \Big[ (\tilde X(\infty) + \zeta)  1( \tilde X(\infty) \geq -\zeta)\Big] \\
\leq&\ \alpha^2 \frac{1}{3}\Big(\frac{\mu }{\alpha}\delta^2  + \frac{\mu }{\alpha}4 + \delta^2\Big)+ \mu^2 \zeta^2\Prob(\tilde X(\infty) \geq -\zeta) \\
&+ 2\alpha \mu \Big( \frac{\delta^2}{4}\frac{\alpha}{\mu }+ \frac{\delta^2}{4} + 1 \Big),
\end{align*}
where the last inequality follows from moment bounds \eqref{eq:mwuK2} and \eqref{eq:mwu5}.

In the overloaded case ($n \leq R$), to bound the first term in \eqref{eq:third_bounds} we use moment bounds \eqref{eq:mwo7}, \eqref{eq:mwo1}, and \eqref{eq:mwo4} with gradient bound \eqref{eq:ACder2}. To bound the second and third terms in \eqref{eq:third_bounds} we use gradient bound \eqref{eq:ACoder1}. To bound the fourth term in \eqref{eq:third_bounds}, we use gradient bound \eqref{eq:ACder2}, with moment bounds \eqref{eq:mwo2} and
\begin{align*}
&\ \E \Big[ \big( b(\tilde X(\infty))\big)^2 1( \tilde X(\infty) \leq -\zeta) \Big] \\
=&\ \mu^2\E \Big[ \big( \tilde X(\infty)+ \zeta\big)^2 1( \tilde X(\infty) \leq -\zeta) \Big]+ \alpha^2 \zeta^2\Prob(\tilde X(\infty) \leq -\zeta) \\
& + 2\alpha \mu \zeta\E \Big[\big| (\tilde X(\infty) + \zeta)  1( \tilde X(\infty) \leq -\zeta)\big|\Big] \\
\leq&\ \mu^2 \Big(\frac{\delta^2}{4}\frac{\alpha}{\mu }+1\Big) + \alpha^2 \Big(\frac{\delta^2}{4}+\frac{\mu}{\alpha}\Big) + 2\alpha \mu \Big(\frac{\delta^2}{4}+1\Big),
\end{align*}
where the last inequality follows from moment bounds \eqref{eq:mwo8}, \eqref{eq:mwo3}, and \eqref{eq:mwoK1}.

\section{Extension: Erlang-C Higher Moments}
\label{sec:CAexten}
In this section we consider the approximation of higher moments for the Erlang-C model. We begin with the following result.

\begin{theorem}\label{thm:poly}
Consider the Erlang-C system ($\alpha = 0$), and fix an integer $m > 0$. There exists a constant $C = C(m)$, such that for all $n \geq 1, \lambda > 0$, and $\mu>0$ satisfying $1 \leq R < n$,
\begin{equation}
  \label{eq:highmomCW}
\big|\E (\tilde X(\infty))^m - \E (Y(\infty))^m\big|\leq (1+1/\abs{\zeta}^{m-1})C(m)\delta,
\end{equation} 
where $\zeta$ is defined in \eqref{eq:bandz}.
%where
%\begin{displaymath}
%  {\cal P}_m=\{h\in C_1(\R): \abs{h(x)}\le \abs{x}^m, \quad \abs{h'(x)} \leq \abs{x}^{m-1} \text{ for all } x\in \R\}.
%\end{displaymath}
\end{theorem}
The proof of this theorem follows the standard Stein framework in Section~\ref{sec:CAroadmap}, but we do not provide it in this document. The most interesting aspect of \eqref{eq:highmomCW} is the appearance of $1/\abs{\zeta}^{m-1}$ in the bound on the right hand side, which of course only matters when $\abs{\zeta}$ is small. To check whether the bound is sharp, we performed some numerical experiments illustrated in Table~\ref{tab3}. The results suggest that the approximation error does indeed grow like $1/\abs{\zeta}^{m-1}$.

A better way to understand the growth parameter $1/\abs{\zeta}^{m-1}$ is through its relationship with $\E (\tilde X(\infty))^{m-1}$. We claim that  $\E (\tilde X(\infty))^{m-1} \approx 1/\abs{\zeta}^{m-1}$  for small values of $\abs{\zeta}$. The following lemma, which is proved in Section~\ref{app:order_mag},  is needed.
\begin{lemma} \label{lem:order_mag}
For any integer $m \geq 1$, and all $n \geq 1, \lambda > 0$, and $\mu>0$ satisfying $R < n$,
\begin{align}
\lim_{\zeta \uparrow 0} \abs{\zeta}^{m}\E (Y(\infty))^{m} = m!.
\end{align}
\end{lemma}
Multiplying both sides of \eqref{eq:highmomCW} by $\abs{\zeta}^{m}$ and applying Lemma~\ref{lem:order_mag}, we see that for all $n \geq 1, \lambda > 0$, and $\mu>0$ satisfying $1 \leq R < n$,
\begin{align*}
\lim_{\zeta \uparrow 0} \abs{\zeta}^{m}\E (\tilde X(\infty))^{m} = m!.
\end{align*}
In other words,  we can rewrite \eqref{eq:highmomCW} as 
\begin{align*}
&\ \big|\E (\tilde X(\infty))^m - \E (Y(\infty))^m\big|\\
\leq&\ \Big(1+\frac{1}{\abs{\zeta}^{m-1} \big|\E (\tilde X(\infty))^{m-1}\big|}\big|\E (\tilde X(\infty))^{m-1}\big|\Big)C(m)\delta \\
\leq&\ \Big(1+\big|\E (\tilde X(\infty))^{m-1}\big|\Big)\tilde C(m)\delta,
\end{align*}
where $\tilde C(m)$ is a redefined version of $C(m)$. That the approximation error in Table~\ref{tab3} increases is then attributed to the fact that $\E \tilde X(\infty)$ increases as $\zeta \uparrow 0$. As we mentioned before, the appearance of the $(m-1)$th moment in the approximation error of the $m$th moment was also observed recently in \cite{GurvHuan2016} for the virtual waiting time in the $M/GI/1+GI$ model, potentially suggesting a general trend.
\begin{table}[h!]
  \begin{center}
  \resizebox{\columnwidth}{!}{
  \begin{tabular}{rccc | ccc }
$R$ &$\abs{\zeta}$ & $\E (\tilde X(\infty))^2$ & Error & $\abs{\zeta}\times $Error & $\abs{\zeta}^{0.5}\times $Error & $\abs{\zeta}^{1.5}\times $Error \\
\hline
 499  & $4.48 \times 10^{-2}$  &  $9.47\times 10^{2}$  & 1.59 & $7.10\times 10^{-2}$ & 0.34 & $1.50\times 10^{-2}$ \\
 499.9  & $4.50 \times 10^{-3}$   &  $9.94\times 10^{4}$  & 16.50 & $7.38\times 10^{-2}$ & 1.10 & $4.94\times 10^{-3}$\\
 499.95 &  $2.20 \times 10^{-3}$ &  $3.99\times 10^{5}$  & 33.08 & $7.40\times 10^{-2}$ & 1.56& $3.50\times 10^{-3}$ \\
 499.99 & $4.47 \times 10^{-4}$  &  $9.99\times 10^{6}$  & 165.67 & $7.41\times 10^{-2}$ & 3.50 & $1.57\times 10^{-3}$\\
  \end{tabular}
  }
  \end{center}
  \caption{The error term above equals $\big| \E (\tilde X(\infty))^2 - \E (Y(\infty))^2 \big|$ and grows as $R \to n$. The error term still grows when multiplied by $\abs{\zeta}^{0.5}$, and the error term shrinks to zero when multiplied by $\abs{\zeta}^{1.5}$. However, when multiplied by $\abs{\zeta}$, the error term appears to converge to some limiting value, suggesting that the error does indeed grow at a rate of $1/\abs{\zeta}$. We observed consistent behavior for higher moments of $\tilde X(\infty)$ as well. \label{tab3}}
\end{table}

\section{Chapter Appendix}

\subsection{Miscellaneous Lemmas}  \label{app:misc}
In this section we prove Lemmas~\ref{lem:gz}, \ref{lem:kolmfixC},  \ref{lem:densboundC}, and \ref{lem:order_mag}. 

\subsubsection{Proof of Lemma~\ref{lem:gz}}
\label{app:gz}
\begin{proof}[Proof of Lemma~\ref{lem:gz}]
Let  $f(x): \R \to \R$ satisfy $\abs{f(x)} \leq C(1+\abs{x})^3$. A sufficient condition to ensure that
\begin{align*}
\EE  \big[ G_{\tilde X} f(\tilde X(\infty)) \big] = 0
\end{align*}
is given by \cite[Proposition 1.1]{Hend1997} (alternatively, see \cite[Proposition 3]{GlynZeev2008}). Namely, we require that 
\begin{align}
\EE \Big[\big| G_{\tilde X} (\tilde X(\infty), \tilde X(\infty)) f(\tilde X(\infty))\big| \Big] < \infty, \label{eq:gzcond}
\end{align}
where $G_{\tilde X} (x,x)$ is the diagonal entry of the generator matrix $G_{\tilde X}$ corresponding to state $x$. 

In the Erlang-C model, the transition rates of $\tilde X$ are bounded by $\lambda + n\mu$. Since $\abs{f(x)} \leq C(1+\abs{x})^3$, it suffices to show that $\EE (\tilde X(\infty))^3 < \infty$, or that $\EE (X(\infty))^3 < \infty$, where $X(\infty)$ has the stationary distribution of the CTMC $X$. Consider the function $V(k) = k^4$, where $k \in \Z_+$. Let $G_{X}$ be the generator of $X$, which is a simple birth death process with constant birth rate $\lambda$ and departure rate $\mu (k \wedge n)$ in state $k \in \Z_+$. Then for $k \geq n$, 
\begin{align}
G_{X}V(k) =&\ \lambda ( (k + 1)^4 - k^4) + n\mu ((k-1)^4 - k^4)  \notag \\
=&\ \lambda (4k^3 + 6k^2 + 4k + 1) + n\mu  (-4k^3 + 6k^2 - 4k + 1) \notag \\
=&\ -4k^3 ( n\mu - \lambda) +6k^2 ( n\mu + \lambda) -4k( n \mu - \lambda ) + (\lambda + n\mu). \label{eq:gz1}
\end{align}
It is not hard to see that there exists some $k_0 \in \Z_+$, and a constant $c > 0$ (that depends on $\lambda, n$, and $\mu$), such that for all $k \geq k_0$,
\begin{align}
-4k^3 ( n\mu - \lambda) +6k^2 ( n\mu + \lambda) -4k( n \mu - \lambda ) \leq -ck^3.\label{eq:gz3}
\end{align} 
We combine \eqref{eq:gz1}--\eqref{eq:gz3} to conclude that there exists some constant $d > 0$ (that depends on $\lambda, n$, and $\mu$) satisfying 
\begin{align*}
G_{X} V(k)  \leq -c k^3 + d 1(k < (k_0 \vee n)),
\end{align*}
and invoking \cite[Theorem 4.3]{MeynTwee1993b}, we see that $\EE (X(\infty))^3 < \infty$.

The case of the Erlang-A model is not very different. When $\alpha > 0$, the transition rates of the CTMC depend linearly on its state. Hence, to satisfy \eqref{eq:gzcond} we need to show that $\E (X(\infty))^4 < \infty$. This is readily proven by repeating the procedure above with the Lyapunov function $V(k) = k^5$, and we omit the details.
\end{proof}

\subsubsection{Proof of Lemma~\ref{lem:kolmfixC}}
\label{app:kolmfixC}
\begin{proof}[Proof of Lemma~\ref{lem:kolmfixC}]
We let $F_W(w)$ and $F_{\tilde X}(x)$ be the distribution functions of $W$ and $\tilde X(\infty)$, respectively. For any $a \in \R$, let  $\tilde a = \delta(a - x(\infty))$. We want to show that
\begin{align}
\Prob( \tilde a - \delta <  \tilde X(\infty) \leq \tilde a + \delta) =&\ F_{\tilde X}(\tilde a + \delta) - F_{\tilde X}(\tilde a - \delta) \notag \\
\leq&\ 2\delta \omega(F_W)  + d_K(\tilde X(\infty), W) + 9\delta^2 + 8\delta^4. \label{eq:fix_result}
\end{align}
Let $\{\pi_k\}_{k=0}^{\infty}$ be the distribution of $X(\infty)$, and  
\begin{align*}
k^* = \inf \{k \geq 0 : \pi_k \geq \nu_j, \text{ for all $j \neq k$}\}.
\end{align*}
Then for any $\tilde a \in \R$, 
\begin{align*}
F_{\tilde X}(\tilde a + \delta) - F_{\tilde X}(\tilde a - \delta) \leq 2\pi_{k^*},
\end{align*}
because $\tilde X(\infty)$ takes at most two values in the interval $(\tilde a - \delta, \tilde a + \delta]$. Observe that by the flow balance equations, we know that for any $k \in \Z_+$, 
\begin{align*}
 \pi_k=  \frac{d(k+1)}{\lambda}\pi_{k+1},
\end{align*}
where $d(k)$ is defined in \eqref{eq:deathrate}. Since $k^*$ is the maximizer of $\{\pi_k\}$, we know that 
\begin{align*}
d(k^*) \leq \lambda \leq d(k^*+1) \leq \lambda + \mu,
\end{align*}
where in the last inequality we have used the fact that the increase in departure rate between state $k^*$ and $k^*+1$ is at most $\mu$. Likewise, $d(k^*+i) \leq \lambda + i \mu$ for $i = 2,3$. Hence,
\begin{align*}
\pi_{k^*}=&\  \frac{d(k^*+1)}{\lambda}\pi_{k^*+1} \leq \Big(1 + \frac{\mu}{\lambda}\Big)\pi_{k^*+1} \leq  \pi_{k^*+1} + \delta^2,\\
\pi_{k^*}=&\  \frac{d(k^*+1)}{\lambda}\frac{d(k^*+2)}{\lambda} \pi_{k^*+2} \\
\leq&\ (1 + \delta^2)(1 + 2\delta^2)\pi_{k^*+2} \leq \pi_{k^*+2} + 3\delta^2  + 2\delta^4,\\
\pi_{k^*+1} =&\ \frac{d(k^*+2)}{\lambda}\frac{d(k^*+3)}{\lambda} \pi_{k^*+3} \\
\leq&\ (1 + 2\delta^2)(1 + 3\delta^2)\pi_{k^*+3} \leq \pi_{k^*+3} + 5\delta^2 + 6\delta^4,
\end{align*}
which implies that for any $\tilde a \in \R$, 
\begin{align*}
F_{\tilde X}(\tilde a + \delta) - F_{\tilde X}(\tilde a - \delta) \leq 2\pi_{k^*} \leq&\ \pi_{k^*} + \pi_{k^* + 1} + \delta^2\\
 =&\ F_{\tilde X}(\tilde k^* + \delta) - F_{\tilde X}(\tilde k^* - \delta) + \delta^2.
\end{align*}

There are now 4 cases to consider, with the first three being simple to handle. Recall that $\omega(F_W)$ is the modulus of continuity of $F_W(w)$. 
\begin{enumerate}
\item  \label{case:1} If $F_W(\tilde k^* - \delta) \leq F_{\tilde X}(\tilde k^* - \delta)$ and $F_W(\tilde k^* + \delta) \geq F_{\tilde X}(\tilde k^* + \delta)$, then 
\begin{align}
F_{\tilde X}(\tilde k^* + \delta) - F_{\tilde X}(\tilde k^* - \delta) \leq F_{W}(\tilde k^* + \delta) - F_{W}(\tilde k^* - \delta) \leq 2\delta \omega(F_W). \label{eq:case1}
\end{align}
\item  \label{case:2} If $F_W(\tilde k^* - \delta) \leq F_{\tilde X}(\tilde k^* - \delta)$ but $F_W(\tilde k^* + \delta) < F_{\tilde X}(\tilde k^* + \delta)$, then 
\begin{align}
&\ F_{\tilde X}(\tilde k^* + \delta) - F_{\tilde X}(\tilde k^* - \delta) \notag \\
 \leq&\  F_{\tilde X}(\tilde k^* + \delta) - F_{W}(\tilde k^* + \delta) +  F_{W}(\tilde k^* + \delta) - F_{W}(\tilde k^* - \delta) \notag \\
 \leq&\  2\delta \omega(F_W)  + d_K(\tilde X(\infty), W). \label{eq:case2}
\end{align}
\item \label{case:3} Similarly, if $F_W(\tilde k^* - \delta) > F_{\tilde X}(\tilde k^* - \delta)$ and $F_W(\tilde k^* + \delta) \geq F_{\tilde X}(\tilde k^* + \delta)$, then 
\begin{align}
&\ F_{\tilde X}(\tilde k^* + \delta) - F_{\tilde X}(\tilde k^* - \delta) \notag  \\
\leq&\  F_{W}(\tilde k^* + \delta) - F_{W}(\tilde k^* - \delta) + F_{W}(\tilde k^* - \delta) - F_{\tilde X}(\tilde k^* - \delta)  \notag \\
 \leq&\  2\delta \omega(F_W)  + d_K(\tilde X(\infty), W). \label{eq:case3}
\end{align}
\item \label{case:4}  Suppose $F_W(\tilde k^* - \delta) > F_{\tilde X}(\tilde k^* - \delta)$ and $F_W(\tilde k^* + \delta) < F_{\tilde X}(\tilde k^* + \delta)$, then we need to use a different approach. We know that
\begin{align*}
F_{\tilde X}(\tilde k^* + \delta) - F_{\tilde X}(\tilde k^* - \delta) =&\ \pi_{k^*} + \pi_{k^*+1} \\
\leq&\ \pi_{k^*+2} + \pi_{k^* + 3} + 8\delta^2 + 8\delta^4 \\
=&\ F_{\tilde X}(\tilde k^* + 3\delta) - F_{\tilde X}(\tilde k^* + \delta)+ 8\delta^2 + 8\delta^4.
\end{align*}
Since  $F_{W}(\tilde k^* + \delta) \leq F_{\tilde X}(\tilde k^* + \delta)$, we are either in  case~\ref{case:1}  or \ref{case:2} for the difference  $F_{\tilde X}(\tilde k^* + 3\delta) - F_{\tilde X}(\tilde k^* + \delta)$, and hence we have
\begin{align*}
F_{\tilde X}(\tilde k^* + 3\delta) - F_{\tilde X}(\tilde k^* + \delta) \leq 2\delta \omega(F_W)  + d_K(\tilde X(\infty), W).
\end{align*}
\end{enumerate}
This proves \eqref{eq:fix_result}, concluding the proof of this lemma.

\end{proof}

\subsubsection{Proof of Lemma~\ref{lem:densboundC} }
\label{app:densboundC}
\begin{proof}[Proof of Lemma~\ref{lem:densboundC}]
In the Erlang-C model,
\begin{align}
\nu(x) = 
\begin{cases}
a_{-} e^{-\frac{1}{2}x^2}, \quad x \leq - \zeta,\\
a_{+} e^{-\abs{\zeta} x}, \quad x \geq -\zeta.
\end{cases} \label{eq:densEC}
\end{align}
To bound this density, we need to bound $a_-$ and $a_+$. We know that $\nu(x)$ must integrate to one, which implies that 
\begin{align*}
a_- \int_{-\infty}^{-\zeta} e^{-\frac{1}{2}y^2} dy + a_+ \int_{-\zeta}^{\infty} e^{-\abs{\zeta} y} dy = 1
\end{align*}
Furthermore, since $\nu(x)$ is continuous at $x = -\zeta$, 
\begin{align*}
a_- e^{-\frac{1}{2}\zeta^2} = a_{+} e^{-\zeta^2}.
\end{align*}
Combining these two facts, we see that 
\begin{align}
 a_- = \frac{1}{\int_{-\infty}^{-\zeta} e^{-\frac{1}{2}y^2} dy + e^{\frac{1}{2}\zeta^2} \int_{-\zeta}^{\infty} e^{-\abs{\zeta} y} dy} \leq \frac{1}{\int_{-\infty}^{0} e^{-\frac{1}{2}y^2} dy} = \sqrt{\frac{2}{\pi}}, \label{eq:aminus}
\end{align}
and 
\begin{align}
a_+ = \frac{1}{e^{-\frac{1}{2}\zeta^2}\int_{-\infty}^{-\zeta} e^{-\frac{1}{2}y^2} dy +  \int_{-\zeta}^{\infty} e^{-\abs{\zeta} y} dy} \leq \frac{1}{e^{-\frac{1}{2}\zeta^2}\int_{-\infty}^{0} e^{-\frac{1}{2}y^2} dy} = e^{\frac{1}{2}\zeta^2}\sqrt{\frac{2}{\pi}}. \label{eq:aplus}
\end{align}
Therefore, for $x \leq -\zeta$, 
\begin{align*}
\abs{\nu(x)} \leq a_- \leq  \sqrt{\frac{2}{\pi}},
\end{align*}
and for $x \geq -\zeta$, we recall that $\zeta < 0$ to see that
\begin{align*}
\abs{\nu(x)} \leq a_+ e^{-\abs{\zeta} x} \leq \sqrt{\frac{2}{\pi}}e^{\frac{1}{2}\zeta^2}e^{-\abs{\zeta} x} \leq \sqrt{\frac{2}{\pi}}.
\end{align*}
\end{proof}

\subsubsection{Proof of Lemma~\ref{lem:order_mag} } \label{app:order_mag}
\begin{proof}[Proof of Lemma~\ref{lem:order_mag} ]
The density of $Y(\infty)$ is given in \eqref{eq:densEC}, and so
\begin{align*}
\E (Y(\infty))^m = a_- \int_{-\infty}^{-\zeta} y^{m}e^{-\frac{1}{2}y^2} dy + a_+ \int_{-\zeta}^{\infty} y^{m}e^{-\abs{\zeta} y} dy,
\end{align*}
where $a_-$ and $a_+$ are as in \eqref{eq:aminus} and \eqref{eq:aplus}. In particular, 
\begin{align*}
 a_- = \frac{1}{\int_{-\infty}^{-\zeta} e^{-\frac{1}{2}y^2} dy + e^{\frac{1}{2}\zeta^2} \int_{-\zeta}^{\infty} e^{-\abs{\zeta} y} dy}  = \frac{1}{\int_{-\infty}^{-\zeta} e^{-\frac{1}{2}y^2} dy + \frac{1}{\abs{\zeta}} e^{-\frac{1}{2}\zeta^2}},
\end{align*}
which implies that 
\begin{align*}
\lim_{\zeta \uparrow 0} \abs{\zeta}^{m} a_- \int_{-\infty}^{-\zeta} y^{m}e^{-\frac{1}{2}y^2} dy = 0.
\end{align*}
Furthermore, 
\begin{align*}
a_+ = \frac{1}{e^{-\frac{1}{2}\zeta^2}\int_{-\infty}^{-\zeta} e^{-\frac{1}{2}y^2} dy +  \int_{-\zeta}^{\infty} e^{-\abs{\zeta} y} dy} = \frac{1}{e^{-\frac{1}{2}\zeta^2}\int_{-\infty}^{-\zeta} e^{-\frac{1}{2}y^2} dy + \frac{1}{\abs{\zeta}} e^{-\zeta^2}},
\end{align*}
and using integration by parts,
\begin{align*}
\int_{-\zeta}^{\infty} y^{m}e^{-\abs{\zeta} y} dy =&\ e^{-\zeta^2}\sum_{j=0}^{m} \frac{m!}{(m-j)!} \frac{1}{\abs{\zeta}^{j+1}}\abs{\zeta}^{m-j} \\
=&\ e^{-\zeta^2}\sum_{j=0}^{m-1} \frac{m!}{(m-j)!} \frac{1}{\abs{\zeta}^{j+1}}\abs{\zeta}^{m-j} + 
e^{-\zeta^2} \frac{m!}{\abs{\zeta}^{m+1}}.
\end{align*}
Hence, 
\begin{align*}
\lim_{\zeta \uparrow 0} \abs{\zeta}^{m} a_+ \int_{-\zeta}^{\infty} y^{m}e^{-\abs{\zeta} y} dy = 
m!.
\end{align*}
%\begin{align*}
%\int_{-\zeta}^{\infty} y^{m}e^{-\abs{\zeta} y} dy =&\  \frac{\abs{\zeta}^m}{\abs{\zeta}} e^{-\zeta^2} + \frac{m}{\abs{\zeta}} \int_{-\zeta}^{\infty} y^{m-1}e^{-\abs{\zeta} y} dy \\
%=&\ \frac{1}{\abs{\zeta}}\abs{\zeta}^m e^{-\zeta^2} + \frac{m}{\abs{\zeta}}\frac{1}{\abs{\zeta}}\abs{\zeta}^{m-1}  e^{-\zeta^2} + \frac{m}{\abs{\zeta}} \frac{m-1}{\abs{\zeta}} \int_{-\zeta}^{\infty} y^{m-2}e^{-\abs{\zeta} y} dy  \\ 
%&\ \ldots \\
%=&\ 
%\end{align*}

\end{proof}

\chapter{State Dependent Diffusion Coefficient: Faster Convergence Rates} \label{chap:dsquare}

Choosing a diffusion approximation  involves selecting a drift $\bar b(x)$ and a diffusion coefficient $\bar a(x)$.  When choosing a diffusion approximation of a Markov chain, one would think that best course of action would be to choose $\bar b(x)$ and $\bar a(x)$ based on the infinitesimal drift and variance of the Markov chain, respectively. While the drift of the diffusion $\bar b(x)$ is usually matched exactly to the infinitesimal drift of the Markov chain, the diffusion coefficient $\bar a(x)$ is often taken to be a constant, even when the infinitesimal variance of the Markov chain is state dependent; see \cite{HalfWhit1981, Atar2012, Ward2012, GurvHuanMand2014} just to name a few. However, not everyone uses a constant $\bar a(x)$. State-dependent diffusion coefficients are used for example in strong approximation theorems in \cite{MandMassReim1998}; see \cite[Remark 2.2]{GurvHuanMand2014} for further discussion.  In \cite[p. 116]{GlynWard2003},  the authors compare two diffusion approximations, one  with constant and one with state-dependent $\bar a(x)$. Numerically, they find that the latter does perform a little better, but overall they are unenthusiastic about promoting its use. The main reason being that a state-dependent diffusion coefficient makes the transient behavior of the diffusion process more difficult to compute, and their observed accuracy gains are not sufficient to justify this extra difficulty.

The purpose of this chapter is to strongly promote the use of state-dependent diffusion coefficients $\bar a(x)$ that more accurately capture the infinitesimal variance of the Markov chain. Working in the setting of the Erlang-C model, we prove in Theorem~\ref{thm:w2} that the error from an approximation with a state-dependent diffusion coefficient goes to zero an order of magnitude faster than the error from an approximation with a constant diffusion coefficient. We will also see that a state-dependent diffusion coefficient does not increase the difficulty of computing the stationary distribution of the diffusion. 

Going forward, the reader is assumed to be familiar with the content of Chapter~\ref{chap:erlangAC}. In particular, we assume familiarity with the Stein framework from Section~\ref{sec:CAroadmap}. We begin the chapter with Section~\ref{sec:DSmain}, where we present Theorem~\ref{thm:w2} and some numerical results that go along with it. In Section~\ref{sec:DSroadmap}, we present the ingredients needed to prove Theorem~\ref{thm:w2} and carry out the proof in Section~\ref{sec:DSproofW}. Section~\ref{sec:DSappendix} is a short  appendix for the chapter.

\section{Main Result}
\label{sec:DSmain}
We adopt the notation of Chapter~\ref{chap:erlangAC}, which we recall briefly below. The Erlang-C system has $n$ servers, arrival rate $\lambda$, and service rate $\mu$. The quantity $R = \lambda/\mu$ is known as the offered load, and we set $\delta = 1/\sqrt{R}$ for convenience. The customer count process is $X = \{X(t), t \geq 0\}$ and the scaled and centered process is $\tilde X = \{\delta(X(t) - R), t \geq 0\}$. When $R < n$, these processes are positive recurrent, and  $X(\infty)$ and $\tilde X(\infty)$ are the random variables having the respective stationary distributions. The process $\tilde X$ has generator 
\begin{align}
G_{\tilde X} f(x) = \lambda (f(x + \delta) - f(x)) + d(k) (f(x-\delta) - f(x)), \label{DS:GX}
\end{align}
where $k \in \Z_+$, $x = x_k = \delta(k - x(\infty))$, and
\begin{align*}
d(k) = \mu (k \wedge n),
\end{align*}
is the departure rate corresponding to the system having $k$ customers.  We also recall $\zeta = \delta(R - n)$, which was defined in \eqref{CA:zeta}. The approximation to $\tilde X(\infty)$ was $Y(\infty)$, a continuous random variable with density $\nu(x)$ given in \eqref{CA:stdden}. The random variable $Y(\infty)$ corresponds to a diffusion process with drift
\begin{align}
b(x) = 
\begin{cases}
-\mu x, \quad x \leq -\zeta,\\
\mu \zeta, \quad x \geq -\zeta,
\end{cases} \label{DS:bdef}
\end{align}
and diffusion coefficient $2\mu$.

In this chapter, we propose a different diffusion approximation. Namely, let $Y_S(\infty)$ be the continuous random variable with density 
\begin{equation}
  \label{eq:stdden}
  \nu_S(x)= \frac{\kappa}{a(x)} \exp\Big({\int_0^x \frac{2b(y)}{a(y)}dy}\Big), \quad x \in \R,
\end{equation}
where $\kappa > 0$ is a normalization constant, and 
\begin{align}
&a(x)  = 
\begin{cases}
\mu , \quad x \leq -1/\delta, \\
\mu (2 + \delta x), \quad x \in [-1/\delta, -\zeta], \\
\mu (2 + \delta \abs{\zeta}), \quad x \geq -\zeta,
\end{cases}.  \label{DS:adef}
\end{align}

One may check that for $k \in \Z_+$ and $x = \delta(k-R)$, 
\begin{align}
b(x) = \delta( \lambda - d(k)), \quad \text{ and } \quad a(x) = \delta(\lambda + d(k) 1(k > 0)). \label{DS:abk}
\end{align}
The random variable $Y_S(\infty)$ has the stationary distribution of a diffusion process on the real line with drift $b(x)$ and \emph{state dependent} diffusion coefficient $a(x)$. In contrast, in Chapter~\ref{chap:erlangAC} we used a constant diffusion coefficient of $2\mu$. The following is the main result of this chapter.
\begin{theorem}
\label{thm:w2}
There exists a constant $C > 0$ (independent of $\lambda, n$, and $\mu$), such that for all $n \geq 1, \lambda > 0$, and $\mu > 0$ satisfying $1 \leq R < n $,
\begin{equation}
  \label{eq:newmain}
  d_{W_2}(\tilde X(\infty), Y_S(\infty)) := \sup_{h(x) \in W_2} \big| \EE h(\tilde X(\infty)) - \EE h(Y_S(\infty)) \big| \le \frac{C}{R},
\end{equation}
where 
\begin{align}
W_2 = \big\{h: \R \to \R\ \big|\  h(x), h'(x) \in \lipone \big\}. \label{eq:spacew2}
\end{align}
\end{theorem}
\noindent Theorem~\ref{thm:w2} should be compared with Theorem~\ref{thm:erlangCW} of Chapter~\ref{chap:erlangAC}. The former has a convergence rate of $1/R$ versus the $1/\sqrt{R}$ rate of the latter. The class of functions $W_2$ in \eqref{eq:newmain} is not significantly smaller than $\lipone$, meaning that the two statements are comparable. We will see in Section~\ref{sec:DSappendix} that $W_2$ is a rich enough class of functions to imply convergence in distribution.

 Theorem~\ref{thm:w2} can also be compared to the results in \cite{GurvHuanMand2014, Gurv2014} and Chapter~\ref{chap:phasetype} (which is based in \cite{BravDai2017}), all of which study convergence rates for steady-state diffusion approximations of various models. A rate of $1/R$ is an order of magnitude better than the rates in any of the previously mentioned papers, whose rates are equivalent to $1/\sqrt{R}$ in our model.

\subsection{Numerical Study}
\label{sec:numeric}
Before moving on to the proof of Theorem~\ref{thm:w2}, we present some numerical results to complement the theorem. The results in this section show that $Y_S(\infty)$ consistently outperforms $Y(\infty)$. In Table~\ref{tabbenefit} we see that for large or heavily loaded systems, i.e. when $R$ is either large or close to $n$, the approximation $Y(\infty)$ performs reasonably well, and the accuracy gained from using $Y_S(\infty)$ is not as impressive. However, the accuracy gain of $Y_S(\infty)$ is much more significant for smaller systems with lighter loads. In Table~\ref{tabrates} we see that the errors of $Y(\infty)$ and $Y_S(\infty)$ indeed decrease at a rate of $1/\sqrt{R}$ and $1/R$, respectively. Furthermore, the table suggests that the approximation error of the second moment also decreases at a rate of $1/R$, even though \eqref{eq:newmain} does not guarantee this. Numerically, we observed a rate of $1/R$ for higher moments as well. This is not surprising, as there is nothing preventing us from repeating the analysis in this chapter for higher moments. 
\begin{table}[h]
  \begin{center}
  \resizebox{\columnwidth}{!}{
   \begin{tabular}{rc|cc|cc }
\multicolumn{6}{c}{$n=5$} \\
$R$ & $\EE \tilde X(\infty)$ & $\big| \EE Y(\infty) - \EE \tilde X(\infty)\big|$  &Relative Error & $\big| \EE Y_S(\infty) - \EE \tilde X(\infty)\big|$& Relative Error\\
%& & &$(Y_0(\infty))$&\multicolumn{1}{c}{} & $(Y(\infty))$\\
\hline
 3        &   0.20  &$5.87\times 10^{-2}$  & 28.69\% & $9.34\times 10^{-3}$ & 4.57\% \\
 4        &  1.11  & $9.91\times 10^{-2}$  & 8.95\% & $1.12\times 10^{-2}$ & 1.08\% \\
 4.9        &  21.04  &$1.28\times 10^{-1}$ & 0.61\% & $1.29\times 10^{-2}$ & 0.06\% \\
 4.95        &  43.39  & $1.29\times 10^{-1}$ & 0.30\% & $1.29\times 10^{-2}$ & 0.03\% \\
 4.99        & 222.26  & $1.30\times 10^{-1}$& 0.06\% & $1.29\times 10^{-2}$ & 0.006\% \\
  \end{tabular}}
  \\~\\
  \resizebox{\columnwidth}{!}{
  \begin{tabular}{rc|cc|cc }
\multicolumn{6}{c}{$n=100$} \\
$R$ & $\EE \tilde X(\infty)$ & $\big| \EE Y(\infty) - \EE \tilde X(\infty)\big|$  &Relative Error & $\big| \EE Y_S(\infty) - \EE \tilde X(\infty)\big|$& Relative Error\\
\hline
 60 & $2.97\times 10^{-7}$ & $2.73\times 10^{-7}$ & 91.83\% & $5.11\times 10^{-8}$ &17.24\% \\
 80 & $8.79\times 10^{-3}$  & $2.25\times 10^{-3}$ & 25.60\%& $1.03\times 10^{-4}$ & 1.17\% \\
 98 & $3.84$ & $2.85\times 10^{-2}$ & 0.74\% & $7.00\times 10^{-4}$ & 0.02\% \\
99 & $8.78$ & $3.04\times 10^{-2}$ & 0.35\% & $7.26\times 10^{-4}$ & 0.008\% \\
99.8 & $48.74$ & $3.19\times 10^{-2}$ & 0.07\% & $7.46\times 10^{-4}$ & 0.002\% \\
  \end{tabular}
  }
  \end{center}
  \caption{The new approximation $Y_S(\infty)$ consistently outperforms $Y(\infty)$. \label{tabbenefit}}
\end{table}

\begin{table}[h!]
\begin{center}\setlength\tabcolsep{3pt}
\begin{tabular}{rc | c | c | c }
 $n$ & $R$ &$\EE \tilde X(\infty)$ & $\big| \EE \tilde X(\infty) - \EE Y(\infty)\big|$ & $\big| \EE \tilde X(\infty) - \EE Y_S(\infty)\big|$ \\
\hline
5    & 4       & 1.11 & $9.9 \times 10^{-2}$ & $1.2 \times 10^{-2}$ \\
50   & 46.59   & 1.04 & $3.2 \times 10^{-2}$ & $1.2 \times 10^{-3}$ \\
500  & 488.94  & 1.02 & $1.0 \times 10^{-2}$ & $1.2 \times 10^{-4}$  \\
5000 &  4965   & 1.01 & $3.3 \times 10^{-3}$ & $1.2 \times 10^{-5}$  \\
\end{tabular}
\end{center}

\begin{center}\setlength\tabcolsep{3pt}
\begin{tabular}{rc | c | c | c }
 $n$ & $R$ &$\EE (\tilde X(\infty))^2$ & $\big| \EE (\tilde X(\infty))^2 - \EE (Y(\infty))^2\big|$ & $\big| \EE (\tilde X(\infty))^2 - \EE (Y_S(\infty))^2\big|$ \\
\hline
5    & 4       & 6.54 & 1.00  & $6 \times 10^{-2}$ \\
50   & 46.59   & 5.84 &  0.30 & $5.7 \times 10^{-3}$ \\
500  & 488.94  & 5.63 & 0.092 &  $5.6 \times 10^{-4}$  \\
5000 &  4965   & 5.57 & 0.029 &  $5.5 \times 10^{-5}$  \\
\end{tabular}
\end{center}

\caption{As the offered load increases by a factor of $10$, the approximation error of $Y(\infty)$, derived with a \emph{constant} diffusion coefficient, shrinks at a rate of $\sqrt{10}$, whereas the approximation error of $Y_S(\infty)$,
derived with a \emph{state-dependent} diffusion coefficient, shrinks at a rate of $10$. Similar experiments for moments higher than the second yield consistent results. \label{tabrates}}
\end{table} 

Furthermore, although Theorem~\ref{thm:w2} is only stated in the context of the $W_2$ metric, we show that $Y_S(\infty)$ is a superior approximation to $Y(\infty)$ when it comes to estimating the both the probability mass function (PMF), and cumulative distribution function (CDF). Let $\{\pi_k\}_{k=0}^{\infty}$ be the distribution of $X(\infty)$. For $k \in \Z_+$ define 
\begin{align*}
\pi^{Y}_k =& \Prob\Big(Y(\infty) \in \big[\delta(k - R)-\delta/2, \delta (k - R)+\delta/2\big]\Big),\\
\pi^{Y_S}_k =& \Prob\Big(Y_S(\infty) \in \big[\delta(k - R)-\delta/2, \delta (k - R)+\delta/2\big]\Big).
\end{align*}
Results for the PMF are displayed in Figure~\ref{fig1} and Table~\ref{tabpmf}, and results for the CDF are in Table~\ref{tabkolm}. We observe numerically that the Kolmogorov distance converges to zero at a rate of $1/\sqrt{R}$ as opposed to $1/R$. However, $Y_S(\infty)$ still performs better. 
 \begin{figure}[h]
\centerline{\includegraphics[width=90mm,keepaspectratio]{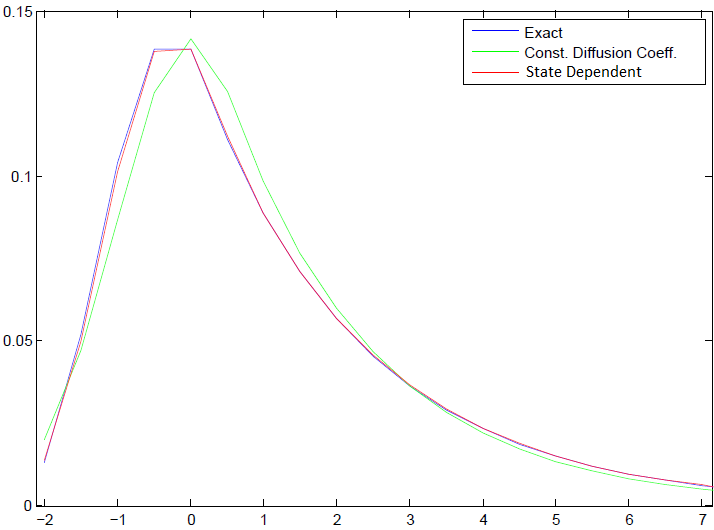}}
 \caption{The plot above corresponds to a small system with $n = 5$ and $R = 4$. The blue, green and red lines are $\pi_k, \pi^{Y}_k$, and $\pi^{Y_S}_k$, respectively. \label{fig1}}
  \end{figure}
\begin{table}[h]
  \begin{center}
  \resizebox{\columnwidth}{!}{
   \begin{tabular}{rcc|ccc }
\multicolumn{3}{c}{$n=5$}& \multicolumn{3}{c}{$n=100$} \\
$R$ & $\sup_{k \in \Z_+} \big|\pi_k - \pi^{Y}_k \big|$ & $\sup_{k \in \Z_+} \big|\pi_k - \pi^{Y_S}_k \big|$   &R & $\sup_{k \in \Z_+} \big|\pi_k - \pi^{Y}_k \big|$ & $\sup_{k \in \Z_+} \big|\pi_k - \pi^{Y_S}_k \big|$  \\
%& & &$(Y_0(\infty))$&\multicolumn{1}{c}{} & $(Y(\infty))$\\
\hline
 3        &   $2.72\times 10^{-2}$  &$5.84\times 10^{-3}$  & 60 & $1.59\times 10^{-3}$ & $2.95\times 10^{-5}$ \\
 4        &  $1.72\times 10^{-2}$  & $2.67\times 10^{-3}$  & 80  & $1.16\times 10^{-3}$ & $1.92\times 10^{-5}$ \\
 4.9        & $2.51\times 10^{-3}$  &$3.54\times 10^{-4}$  & 98  & $3.59\times 10^{-4}$ & $9.81\times 10^{-6}$ \\
 4.95        &  $1.28\times 10^{-3}$  & $1.78\times 10^{-4}$ & 99  & $2.07\times 10^{-4}$ & $5.80\times 10^{-6}$ \\
 4.99        & $2.61\times 10^{-4}$  & $3.62\times 10^{-5}$& 99.98 & $4.71\times 10^{-5}$ & $1.34\times 10^{-6}$ 
  \end{tabular}}
  \begin{tabular}{rc | c | c  }
 $n$ & $R$ &$\sup_{k \in \Z_+} \big|\pi_k - \pi^{Y}_k \big|$ & $\sup_{k \in \Z_+} \big|\pi_k - \pi^{Y_S}_k \big|$\\
\hline
5    & 4       &  $1.72 \times 10^{-2}$ & $2.67 \times 10^{-3}$ \\
50   & 46.59   &  $1.41 \times 10^{-3}$ & $4.78 \times 10^{-5}$ \\
500  & 488.94  &  $1.38 \times 10^{-4}$ & $1.27 \times 10^{-6}$  \\
5000 &  4965   &  $1.37 \times 10^{-5}$ & $3.81 \times 10^{-8}$ 
\end{tabular}
\end{center}
  \caption{Approximating the probability mass function of $\tilde X(\infty)$. \label{tabpmf}}
\end{table}

\begin{table}[h]
  \begin{center}
  \resizebox{\columnwidth}{!}{
   \begin{tabular}{rcc|ccc }
\multicolumn{3}{c}{$n=5$}& \multicolumn{3}{c}{$n=100$} \\
$R$ & $d_K(\tilde X(\infty), Y(\infty))$ & $d_K(\tilde X(\infty), Y_S(\infty))$   &R & $d_K(\tilde X(\infty), Y(\infty))$ & $d_K(\tilde X(\infty), Y_S(\infty))$ \\
%& & &$(Y_0(\infty))$&\multicolumn{1}{c}{} & $(Y(\infty))$\\
\hline
 3        &   $1.32\times 10^{-1}$  &$9.27\times 10^{-2}$  & 60 & $3.43\times 10^{-2}$ & $2.58\times 10^{-2}$ \\
 4        &  $8.76\times 10^{-2}$  & $6.41\times 10^{-2}$  & 80  & $2.93\times 10^{-2}$ & $2.23\times 10^{-2}$ \\
 4.9        & $1.32\times 10^{-2}$  &$9.48\times 10^{-3}$  & 98  & $1.03\times 10^{-2}$ & $8.10\times 10^{-3}$ \\
 4.95        &  $6.84\times 10^{-3}$  & $4.84\times 10^{-3}$ & 99  & $5.86\times 10^{-3}$ & $4.53\times 10^{-3}$ \\
 4.99        & $1.41\times 10^{-3}$  & $9.84\times 10^{-4}$& 99.98 & $1.31\times 10^{-3}$ & $9.93\times 10^{-4}$ 
  \end{tabular}}

\begin{tabular}{rc | c | c  }
 $n$ & $R$ &$d_K(\tilde X(\infty), Y(\infty))$ & $d_K(\tilde X(\infty), Y_S(\infty))$\\
\hline
5    & 4       &  $8.76 \times 10^{-2}$ & $6.41 \times 10^{-2}$ \\
50   & 46.59   &  $2.60 \times 10^{-2}$ & $2.11 \times 10^{-2}$ \\
500  & 488.94  &  $7.98 \times 10^{-3}$ & $6.48 \times 10^{-3}$  \\
5000 &  4965   &  $2.50 \times 10^{-3}$ & $2.03 \times 10^{-3}$ 
\end{tabular}
  \end{center}
  \caption{Approximating the cumulative distribution function of $\tilde X(\infty)$. In the second table, as $R$ increases by a factor of $10$, both $d_K(\tilde X(\infty), Y(\infty))$ and $d_K(\tilde X(\infty), Y_S(\infty))$ decrease by a factor of $\sqrt{10}$, and not $10$. \label{tabkolm}}
\end{table}
\clearpage

\section{Proof Components}
\label{sec:DSroadmap}
The proof of Theorem~\ref{thm:w2} uses the Stein framework developed in Section~\ref{sec:CAroadmap}. We assume familiarity with that section, and now state the main ingredients needed to prove Theorem~\ref{thm:w2}. As we mentioned,  the random variable $Y_S(\infty)$ is associated to a diffusion process with generator
\begin{equation}
  \label{eq:GY}
G_{Y_S} f(x) =  b(x)f'(x) +  \frac{1}{2}a(x) f''(x) \text{ \quad for $x \in \R$}, \ f \in C^2(\R),
\end{equation}
where $b(x)$ and $a(x)$ are defined in \eqref{DS:bdef} and \eqref{DS:adef}, respectively. Fix $h(x) \in W_2$ with $h(0) = 0$, and consider the Poisson equation 
\begin{align} \label{DS:poisson}
G_{Y_S} f_h(x)  = \E h(Y(\infty)) - h(x), \quad  x\in \R.
\end{align}
We use the Lipschitz property of $h(x)$ to see that
\begin{align*}
\abs{\E h(Y_S(\infty))} \leq  \E \big|Y_S(\infty)\big| < \infty, 
\end{align*}
where the finiteness of $\EE \big|Y_S(\infty)\big|$ will be proved in \eqref{DS:fbound7}. Just as was done in \eqref{eq:gen_bound}, we can take expected values on both sides of (\ref{DS:poisson}) with respect to $\tilde X(\infty)$ and apply Lemma~\ref{lem:gz} to get
\begin{align}
\big| \EE h(\tilde X(\infty)) - \EE h(Y_S(\infty)) \big| =&\ \big| \EE G_{Y_S} f_h(\tilde X(\infty)) \big| \notag \\
=&\ \big| \EE G_{\tilde X} f_h(\tilde X(\infty)) - \EE G_{Y_S} f_h(\tilde X(\infty)) \big| \notag \\
\leq&\ \EE \big| G_{\tilde X} f_h(\tilde X(\infty)) -  G_{Y_S} f_h(\tilde X(\infty)) \big|. \label{DS:poissonright}
\end{align}
We will shortly see in Lemma~\ref{lem:gb} that there is indeed a  solution $f_h(x)$ to the Poisson equation \eqref{DS:poisson} with a bounded second derivative. This means that it can be bounded by a quadratic polynomial, and hence satisfy the conditions of Lemma~\ref{lem:gz}. The following section presents the necessary moment and gradient bounds.

\subsection{Moment Bounds and Gradient Bounds}
\label{sec:momentbounds}
Recall that $\zeta < 0$. We begin with several moment bounds.
\begin{lemma}
\label{lem:lastbounds}
For all $n \geq 1, \lambda > 0$, and $\mu > 0$ satisfying $1 \leq R < n $,
\begin{align}
\Big(1+\frac{1}{\abs{\zeta}}\Big)\EE \Big[\big|\tilde X(\infty) 1(\tilde X(\infty) \leq -\zeta) \big| \Big] \leq&\ \sqrt{2} + 2, \label{eq:lb1}\\
\Big(1+\frac{1}{\abs{\zeta}}\Big)\EE \Big[(\tilde X(\infty))^2  1(\tilde X(\infty) \leq -\zeta) \Big] \leq&\ 9,\label{eq:lb2}\\
\abs{\zeta}\Prob(\tilde X(\infty) \geq -\zeta) \leq&\ 2,\label{eq:lb3}\\
\zeta^2\Prob(\tilde X(\infty) \geq -\zeta) \leq&\ 20, \label{eq:lb4}
\end{align}  
and if $0 < R < n$, then
\begin{align}
&\Prob(\tilde X(\infty) \leq -\zeta) \leq (2+\delta)\abs{\zeta}. \label{eq:idle_prob_v2}
\end{align}
\end{lemma}

\begin{lemma}
\label{lem:pibounds}
Let $\{\pi_k\}_{k=0}^{\infty}$ be the distribution of $X(\infty)$. For all $n \geq 1, \lambda > 0$, and $\mu > 0$ satisfying $0 < R < n $,
\begin{align}
\pi_0 \leq&\ 4(2+\delta)\delta^2\abs{\zeta}, \quad \text{ when } \abs{\zeta} \leq 1, \label{DS:pi0}
\end{align}
and 
\begin{align}
\pi_n \leq \delta \abs{\zeta}. \label{DS:pin}
\end{align}
\end{lemma}
\noindent Lemmas~\ref{lem:lastbounds} and \ref{lem:pibounds} are proved in Section~\ref{app:DSmoment}. Next we present the gradient bounds, which are proved in Section~\ref{app:DSgradbounds}.
\begin{lemma} \label{lem:gb}
Fix $h(x) \in W_2$ with $h(0) = 0$ and consider the Poisson equation \eqref{DS:poisson}. There exists a solution $f_h(x)$ such that $f_h''(x)$ is absolutely continuous, $f_h'''(x)$ exists and is continuous everywhere except the points $x = -1/\delta$ and $x=-\zeta$, and $\lim_{u \uparrow x} f_h'''(u)$ and $\lim_{u \downarrow x} f_h'''(u)$ both exist at those two points. Moreover, there exists a constant $C > 0$ independent of $\lambda, n$, and $\mu$, such that for all $n \geq 1, \lambda > 0$, and $\mu > 0$ satisfying $1 \leq R < n $,
\begin{align}
\abs{f_h'(x)} 
\leq&\
\begin{cases}
\frac{C}{\mu }\Big(1 +  \frac{1}{\abs{\zeta}}\Big), \quad x \leq -\zeta,\\
\frac{C}{\mu \abs{\zeta}}\Big(x + 1 + \frac{1}{\abs{\zeta}}\Big), \quad x \geq -\zeta,
\end{cases}\label{DS:WCder1} \\
\abs{f_h''(x)} \leq&\ 
\begin{cases}
\frac{C}{\mu }\Big(1 +  \frac{1}{\abs{\zeta}}\Big), \quad x \leq -\zeta,\\
\frac{C}{\mu \abs{\zeta}}, \quad x \geq -\zeta,
\end{cases} \label{DS:WCder2}
\end{align}
and 
\begin{align}
 \abs{f_h'''(x)} \leq&\ 
\begin{cases}
\frac{C}{\mu }\Big(1 + \frac{1}{\abs{\zeta}}\Big), \quad x \leq -\zeta,\\
\frac{C}{\mu }, \quad x > -\zeta,
\end{cases} \label{DS:WCder3}
\end{align}
where $f_h'''(x)$ is interpreted as the left derivative at the points $x = -1/\delta$ and $x = -\zeta$.
\end{lemma}
\noindent The gradient bounds in Lemma~\ref{lem:gb} involve only the first three derivatives of $f_h(x)$, and are not sufficient for us. We require the following bounds on the fourth derivative (when it exists). These are proved in Appendix~\ref{app:eterm}.
\begin{lemma}
\label{lem:eterm}
Fix $h(x) \in W_2$ with $h(0)=0$, and let $f_h(x)$ be the solution to the Poisson equation \eqref{eq:poisson} from Lemma~\ref{lem:gb}. Consider only those $x \in \R$ such that $x = \delta (k - R)$ for some $k \in \Z_+$. Then there exists a constant $C>0$ independent of $\lambda, n$, and $\mu$, such that for all $n \geq 1, \lambda > 0$, and $\mu > 0$ satisfying $1 \leq R < n$,
\begin{align}
\abs{f_h'''(x-) - f_h'''(y)} \leq&\ \frac{C\delta}{\mu } \bigg[ 1(x \leq -\zeta)(1 + \abs{x})\Big(1 + \frac{1}{\abs{\zeta}}\Big)\notag \\
& \hspace{1cm} + 1(x \geq -\zeta+\delta) (1+ \abs{\zeta})\bigg], \quad  y \in (x-\delta, x) \label{eq:eboundleft}
\end{align}
and 
\begin{align}
&\ \abs{f_h'''(x-) - f_h'''(y)}  \notag \\
\leq&\ \frac{C\delta}{\mu } \bigg[ 1(x \leq -\zeta-\delta)(1 + \abs{x})\Big(1 + \frac{1}{\abs{\zeta}}\Big) + 1(x \geq -\zeta) (1+ \abs{\zeta}) \notag \\
& \hspace{1cm}+ \frac{1}{\delta}\Big(1 + \frac{1}{\abs{\zeta}}\Big)1(x\in \{-1/\delta, -\zeta\}) \bigg], \quad y \in (x, x + \delta). \label{eq:eboundright}
\end{align}
\end{lemma}
\begin{remark}
The upper bound in \eqref{eq:eboundright} has an extra term compared to the bound in \eqref{eq:eboundleft}. This terms is the result of the discontinuity of $f_h'''(x)$ at $x = -1/\delta$ and $x = -\zeta$.
\end{remark}

\subsection{Taylor Expansion}
\label{sec:DStaylor}
In this section we perform a Taylor expansion on $G_{\tilde X} f_h(x)$ to get a handle on the difference $G_{\tilde X} f_h(x) - G_{Y_S} f_h(x)$. The Taylor expansion here is similar to the one in Section~\ref{sec:ktaylor}, except that now we expand to four terms, whereas the expansion in Section~\ref{sec:ktaylor} was done only up to three terms. Lemma~\ref{lem:gb} guarantees that $f_h''(x)$ is absolutely continuous, and that $f_h'''(x)$ is continuous everywhere except the points $x = -1/\delta$, and $x = -\zeta$. We write $f_h'''(x-)$ to denote $\lim_{u \uparrow x}f_h'''(u)$. We first define 
\begin{align}
\tilde \epsilon_1(x) =&\ \frac{1}{2} \int_x^{x+\delta} (x+\delta -y)^2(f_h'''(y)-f_h'''(x-))dy, \quad x \in \R, \label{DS:eps1def} \\
\tilde \epsilon_2(x) =&\  -\frac{1}{2} \int_{x-\delta}^{x} (y-(x-\delta))^2(f_h'''(y)-f_h'''(x-))dy, \quad x \in \R. \label{DS:eps2def}
\end{align}
Now observe that
\begin{align*}
f_h(x+\delta) -f_h(x) =&\  f_h'(x) \delta  +  \frac{1}{2}\delta^2 f_h''(x) +  \int_x
^{x+\delta} (x+\delta -y)(f_h''(y)-f_h''(x))dy  \\
=&\  f_h'(x) \delta  +  \frac{1}{2}\delta^2 f_h''(x) + \frac{1}{6}\delta^3 f_h'''(x-) \\
&+ \frac{1}{2} \int_x
^{x+\delta}  (x+\delta -y)^2(f_h'''(y)-f_h'''(x-))dy  \\
=&\  f_h'(x) \delta  + \frac{1}{2}\delta^2 f_h''(x) + \frac{1}{6}\delta^3 f_h'''(x-)+  \tilde \epsilon_1(x),
\end{align*}
where the first equality is the same as in \eqref{CA:ktaylor1}. 
Similarly, one can check that
\begin{align*}
 (f_h(x-\delta) -f_h(x))% & = - f_h'(x) \delta  +
% \int_{x-\delta}^x (y-(x-\delta))f_h''(y)dy  \\
%&= -f_h'(x) \delta  +  \frac{1}{2}\delta^2 f_h''(x) +  \int_{x-\delta}^{x} (y-(x-\delta))(f_h''(y)-f_h''(x-))dy  \\
%&= -f_h'(x) \delta  +  \frac{1}{2}\delta^2 f_h''(x) +\frac{1}{6}\delta^3 f_h'''(x-) + \frac{1}{2}\int_{x-\delta}^{x} (y-(x-\delta))^2(f_h'''(y)-f_h'''(x-))dy  \\
&= -f_h'(x) \delta  +  \frac{1}{2}\delta^2 f_h''(x)-\frac{1}{6}\delta^3 f_h'''(x-)  + \tilde \epsilon_2(x).
\end{align*}
Recall from \eqref{DS:abk}  that for any $k \in \Z_+$, and $x = x_k = \delta(k -
R)$,  $b(x) = \delta(\lambda - d(k))$ and $a(x) = \delta(\lambda + d(k) 1(k > 0))$. Therefore, 
\begin{align}
  G_{\tilde X}f_h(x) =&\ \lambda  \delta f_h'(x) + \lambda \frac{1}{2}\delta^2 f_h''(x) + \frac{1}{6}\lambda \delta^3 f_h'''(x-) + \lambda \tilde \epsilon_1(x) \notag \\
  & - d(k) \delta f_h'(x) + d(k) \frac{1}{2} \delta^2 f_h''(x) -d(k) \frac{1}{6}\delta^3 f_h'''(x-) + d(k)
  \tilde \epsilon_2(x) \notag \\
=&\ b(x) f_h'(x) + (\lambda +d(k)1(x \geq -1/\delta)) \frac{1}{2} \delta^2 f_h''(x) \notag \\
& + \frac{1}{6}\delta^3(\lambda - d(k))f_h'''(x-) + \lambda \tilde \epsilon_1(x)+ (\lambda-\frac{1}{\delta}b(x) )
  \tilde \epsilon_2(x) \notag \\
=&\ G_Y f_h(x) + \frac{1}{6}\delta^2b(x)f_h'''(x-) + \lambda(\tilde \epsilon_1(x)+\tilde \epsilon_2(x)) - \frac{1}{\delta}b(x)\tilde \epsilon_2(x), \label{DS:taylor}
\end{align}
and
\begin{align}
\Big| \E h(\tilde X(\infty)) - \E h(Y_S(\infty)) \Big| \leq&\ \frac{1}{6} \delta^2 \E \Big[ \big|f_h'''(\tilde X(\infty)-)b(\tilde X(\infty)) \big| \Big] + \lambda \E\Big[ \big| \tilde \epsilon_1(\tilde X(\infty))\big|\Big] \notag \\
&+ \lambda \E \Big[ \big| \tilde \epsilon_2(\tilde X(\infty))\big|\Big] +  \frac{1}{\delta}  \E\Big[ \big|b(\tilde X(\infty)) \tilde \epsilon_2(\tilde X(\infty))\big|\Big]. \label{DS:first_bounds}
\end{align}
The Taylor expansion in \eqref{DS:taylor} reveals the reason this approximation is better than the one in Chapter~\ref{chap:erlangAC}. This approximation is able to capture the entire second order term in the Taylor expansion of $G_{\tilde X} f_h(x)$ (i.e.\ all the terms that correspond to $f_h'(x)$ and $f_h''(x)$). In contrast, the constant diffusion coefficient approximation in Chapter~\ref{chap:erlangAC} uses a $a(0) = 2\mu$ for the diffusion coefficient. Comparing \eqref{DS:first_bounds} to \eqref{eq:first_bounds}, we see that there is an extra error term of the form 
\begin{align*}
\frac{1}{2}\delta^2\EE \Big[ \big| f_h''(\tilde X(\infty)) \big( a(\tilde X(\infty)) - a(0) \big)   \big| \Big] = \frac{1}{2}\delta^2\EE \Big[ \big| f_h''(\tilde X(\infty)) b(\tilde X(\infty))  \big| \Big],
\end{align*}
which turns out to be on the order of $\delta$, not $\delta^2$. We are now ready to prove Theorem~\ref{thm:w2}.
\section{Proof of Theorem~\ref{thm:w2} (Faster convergence rates)}
\label{sec:DSproofW}
Fix $h(x) \in {W_2}$ with $h(0) = 0$, and let $f_h(x)$ be as in Lemma~\ref{lem:gb}. We will focus on bounding \eqref{DS:first_bounds}, which we recall here as
\begin{align}
\Big| \E h(\tilde X(\infty)) - \E h(Y_S(\infty)) \Big| \leq&\ \frac{1}{6} \delta^2 \E \Big[ \big|f_h'''(\tilde X(\infty)-)b(\tilde X(\infty)) \big| \Big] + \lambda \EE\Big[ \big| \tilde \epsilon_1(\tilde X(\infty))\big|\Big] \notag \\
&+ \lambda \E\Big[ \big| \tilde \epsilon_2(\tilde X(\infty))\big|\Big] +  \frac{1}{\delta}  \E\Big[ \big|b(\tilde X(\infty)) \tilde \epsilon_2(\tilde X(\infty))\big|\Big], \label{eq:second_bounds}
\end{align}
where
\begin{align*}
\tilde \epsilon_1(x) =&\ \frac{1}{2} \int_x^{x+\delta} (x+\delta -y)^2(f_h'''(y)-f_h'''(x-))dy, \\
\tilde \epsilon_2(x) =&\  -\frac{1}{2} \int_{x-\delta}^{x} (y-(x-\delta))^2(f_h'''(y)-f_h'''(x-))dy.
\end{align*}
\begin{proof}[Proof of Theorem~\ref{thm:w2}]
Throughout the proof we assume that $R \geq 1$, or equivalently, $\delta \leq 1$. We will use $C > 0$ to denote a generic constant that may change from line to line, but does not depend of $\lambda,n$, and $\mu$. Suppose we know that for some positive constants $c_1, \ldots, c_4 > 0$ independent of $\lambda, n$, and $\mu$, 
\begin{align}
&\ \Big| \E h(\tilde X(\infty)) - \E h(Y_S(\infty)) \Big| \leq  \delta^2 (c_1 + c_2  +c_3+ \delta c_4)+ C\delta(\pi_0+\pi_n), \label{inl:thm}
\end{align}
where $\{\pi_k\}_{k=0}^{\infty}$ is the distribution of $X(\infty)$. Then to prove the theorem we would only need to show that 
\begin{align*}
 \pi_0, \pi_n \leq C \delta.
\end{align*}
One way to prove this is to appeal to Theorem~\ref{thm:erlangCK}, which states that the Kolmogorov distance 
\begin{align*}
d_K(\tilde X(\infty), Y(\infty)) = \sup_{a \in \R} \big| \Prob( \tilde X(\infty) \leq a) - \Prob( Y(\infty) \leq a) \big| \leq 156 \delta
\end{align*}
for all $n \geq 1$ and $1 \leq R < n$, where $Y(\infty)$ is the random variable with density $\nu(x)$ defined in \eqref{CA:stdden}. We would then have that 
\begin{align}
\pi_n =&\ \Prob( -\zeta - \delta/2 \leq \tilde X(\infty) \leq  -\zeta + \delta/2) \notag \\
=&\ \Prob( -\zeta - \delta/2 \leq Y(\infty) \leq  -\zeta + \delta/2)  \notag \\
&+ \Prob( -\zeta - \delta/2 \leq \tilde X(\infty) \leq  -\zeta + \delta/2) - \Prob( -\zeta - \delta/2 \leq Y(\infty) \leq  -\zeta + \delta/2)  \notag  \\
\leq&\ \delta \norm{\nu} + 2d_K(\tilde X(\infty), Y(\infty)) \leq \delta C, \label{eq:pikolmbound}
\end{align}
where in the last inequality we apply Lemma~\ref{lem:densboundC}, which states that $\nu(x)$ is always bounded by $\sqrt{2/\pi}$. The same argument can be used to bound $\pi_0$. 
%There is another way to prove Theorem~\ref{thm:w2} from \eqref{inl:thm} without relying on any results in \cite{BravDaiFeng2016}. We provide an outline for this argument at the end of this section after proving Theorem~\ref{thm:w2}.

To conclude the theorem it remains to verify \eqref{inl:thm}, which we do by bounding each of the terms on the right side of \eqref{eq:second_bounds} individually. We recall here that the support of $\tilde X(\infty)$ is a $\delta$-spaced grid, and in particular this grid contains the points $-1/\delta$ and $-\zeta$. In the bounds that follow, we will often consider separately the cases where $\tilde X(\infty) \leq -\zeta$, and $\tilde X(\infty) \geq -\zeta+\delta$. We recall that 
\begin{align*}
b(x) = \mu \big((x+\zeta)^- + \zeta\big),
\end{align*}
and apply the gradient bound \eqref{DS:WCder3} together with \eqref{eq:lb1} and \eqref{eq:lb3} of Lemma~\ref{lem:lastbounds} to see that
%moment bounds \eqref{eq:xminusdelta}, \eqref{eq:xminuszeta}, and \eqref{eq:xplusbound} to see that
\begin{align*}
\EE \Big[ \big|f_h'''(\tilde X(\infty)-)b(\tilde X(\infty)) \big| \Big] \leq &\ C\Big(1 + \frac{1}{\abs{\zeta}}\Big)\EE \Big[\big|\tilde X(\infty) 1(\tilde X(\infty) \leq -\zeta) \big| \Big] \\
&+ C\abs{\zeta}\Prob(\tilde X(\infty) \geq -\zeta+\delta) \\
\leq&\ C(\sqrt{2} + 2) + 2C =: c_1.
\end{align*}
To bound the next term, we use \eqref{eq:eboundright} from Lemma~\ref{lem:eterm} to see that
\begin{align*}
\lambda \E\Big[ \big| \tilde \epsilon_1(\tilde X(\infty))\big|\Big] \leq&\ \frac{\mu }{2} \EE \bigg[\int_{\tilde X(\infty)}^{\tilde X(\infty)+\delta} \abs{f_h'''(\tilde X(\infty)-)-  f_h'''(y)} dy \bigg]\\
\leq&\ C\delta^2 \EE \bigg[ 1(\tilde X(\infty) \leq -\zeta -\delta)\Big(1 + \big|\tilde X(\infty))\big|\Big)\Big(1 + \frac{1}{\abs{\zeta}}\Big)  \\
&\hspace{1cm}+ 1(\tilde X(\infty) \geq -\zeta) (1+ \abs{\zeta})\\
&\hspace{1cm}+ \frac{1}{\delta}\Big(1 + \frac{1}{\abs{\zeta}}\Big)1(\tilde X(\infty)\in \{-1/\delta, -\zeta\})\bigg]\\
\leq&\ C\delta^2 \bigg[\EE \Big[ \big|\tilde X(\infty)1(\tilde X(\infty) \leq -\zeta -\delta) \big|\Big]\Big(1 + \frac{1}{\abs{\zeta}}\Big)  \\
&\hspace{1cm} + \Prob(\tilde X(\infty) \leq -\zeta -\delta)\Big(1 + \frac{1}{\abs{\zeta}}\Big) \\
&\hspace{1cm} + \Prob(\tilde X(\infty) \geq -\zeta) (1+ \abs{\zeta})+\frac{1}{\delta}\Big(1 + \frac{1}{\abs{\zeta}}\Big)(\pi_0 + \pi_n)\bigg],
\end{align*}
where in the last inequality we used the fact that $\Prob(\tilde X(\infty)=-1/\delta)$ and $\Prob(\tilde X(\infty)=-\zeta)$ equal $\pi_0$ and $\pi_n$, respectively. We first use \eqref{eq:lb1}, \eqref{eq:lb3}, and \eqref{eq:idle_prob_v2} to see that
%moment bounds \eqref{eq:xminusdelta}, \eqref{eq:xminuszeta}, \eqref{eq:idle_prob}, and \eqref{eq:xplusbound} to see that
\begin{align*}
&\lambda \EE\Big[ \big| \tilde \epsilon_1(\tilde X(\infty))\big|\Big]\\
 \leq&\ C\delta^2 \Big((\sqrt{2} + 2) + (1 + 3) +(1 + 2)\Big)+ C\delta\Big(1 + \frac{1}{\abs{\zeta}}\Big)(\pi_0+\pi_n) \\ 
\leq&\ C\delta^2 + C\delta(\pi_0+\pi_n) + \frac{C\delta}{\abs{\zeta}} (\pi_0+\pi_n)\\
=&\ C\delta^2 + C\delta(\pi_0+\pi_n) + \frac{C\delta}{\abs{\zeta}} \pi_0 1(\abs{\zeta} \geq 1) + \frac{C\delta}{\abs{\zeta}} \pi_0 1(\abs{\zeta} \leq 1) + \frac{C\delta}{\abs{\zeta}} \pi_n \\
\leq&\ C\delta^2 + C\delta(\pi_0+\pi_n) + \frac{C\delta}{\abs{\zeta}} \pi_0 1(\abs{\zeta} \leq 1) + \frac{C\delta}{\abs{\zeta}} \pi_n.
\end{align*}
Next, we apply the bounds on $\pi_0$ and $\pi_n$ from \eqref{DS:pi0} and \eqref{DS:pin} to conclude that
\begin{align*}
\lambda \EE\Big[ \big| \tilde \epsilon_1(\tilde X(\infty))\big|\Big] \leq&\  c_2\delta^2 + C\delta(\pi_0+\pi_n).
\end{align*}
We move on to bound the next term in \eqref{eq:second_bounds}. Using \eqref{eq:eboundleft} from Lemma~\ref{lem:eterm},
\begin{align*}
\lambda \EE\Big[ \big| \tilde \epsilon_2(\tilde X(\infty))\big|\Big] \leq&\ \frac{\mu }{2} \EE \bigg[\int_{\tilde X(\infty)-\delta}^{\tilde X(\infty)} \abs{f_h'''(\tilde X(\infty)-)-  f_h'''(y)} dy \bigg]\\
\leq&\ C\delta^2 \EE\Big[ 1(\tilde X(\infty) \leq -\zeta)\Big(1 + \big|\tilde X(\infty)\big|\Big)\Big(1+\frac{1}{\abs{\zeta}}\Big) \\
&\hspace{1.5cm}+ 1(\tilde X(\infty) \geq -\zeta+\delta) (1+ \abs{\zeta})\Big]\\
\leq&\ C\delta^2 \bigg[\Prob(\tilde X(\infty) \leq -\zeta)\Big(1+\frac{1}{\abs{\zeta}}\Big) \\
&\hspace{1.5cm} + \E\Big[\big|\tilde X(\infty)1(\tilde X(\infty) \leq -\zeta)\big| \Big]\Big(1+\frac{1}{\abs{\zeta}}\Big) \\
& \hspace{1.5cm}+ \Prob(\tilde X(\infty) \geq -\zeta+\delta) (1+ \abs{\zeta})\bigg].
\end{align*}
Now \eqref{eq:lb1}, \eqref{eq:lb3}, and \eqref{eq:idle_prob_v2} imply that
\begin{align*}
\lambda \EE\Big[ \big| \tilde \epsilon_2(\tilde X(\infty))\big|\Big] \leq&\ C\delta^2 \big(4 + (\sqrt{2}+2)+ (1+2)\big) =: c_3 \delta^2.
\end{align*}
For the last term in \eqref{eq:second_bounds}, we use the form of $b(x)$ together with \eqref{eq:eboundleft} from Lemma~\ref{lem:eterm} to see that
\begin{align*}
&\frac{1}{\delta} \EE\Big[ \big|b(\tilde X(\infty)) \tilde \epsilon_2(\tilde X(\infty))\big|\Big]\\
 \leq&\ \frac{\delta}{2} \EE \bigg[\big|b(\tilde X(\infty))\big|\int_{\tilde X(\infty)-\delta}^{\tilde X(\infty)} \abs{f_h'''(\tilde X(\infty)-)-  f_h'''(y)} dy \bigg]\\
\leq&\ C\delta^3 \bigg[ \EE \Big[\big|\tilde X(\infty)(1 + \big|\tilde X(\infty)\big|) 1(\tilde X(\infty) \leq -\zeta)\big|\Big] \Big(1+\frac{1}{\abs{\zeta}}\Big) \notag \\
& \hspace{1cm} + \Prob(\tilde X(\infty) \geq -\zeta+\delta)\abs{\zeta} (1+ \abs{\zeta})\bigg]\\
\leq&\ C\delta^3 \bigg[ \EE \Big[\big|\tilde X(\infty) 1(\tilde X(\infty) \leq -\zeta)\big|\Big]\Big(1+\frac{1}{\abs{\zeta}}\Big) \\
&\hspace{1cm}+ \EE \Big[\big(\tilde X(\infty)\big)^21(\tilde X(\infty) \leq -\zeta)\big|\Big] \Big(1+\frac{1}{\abs{\zeta}}\Big) \notag \\
& \hspace{1cm} + \Prob(\tilde X(\infty) \geq -\zeta+\delta)(\abs{\zeta}+\abs{\zeta}^2 )\bigg].
\end{align*}
We apply \eqref{eq:lb1}--\eqref{eq:lb4} from Lemma~\ref{lem:lastbounds} to conclude that
\begin{align*}
\frac{1}{\delta} \EE\Big[ \big|b(\tilde X(\infty)) \tilde \epsilon_2(\tilde X(\infty))\big|\Big] \leq&\ C\delta^3 \big( (\sqrt{2}+2) + 9 +2 + 20\big) =: c_4\delta^3.
\end{align*}
Therefore, we have shown that for all $R \geq 1$, and $h(x) \in {W_2}$ with $h(0)=0$, \eqref{inl:thm} holds, concluding the proof of Theorem~\ref{thm:w2}.
\end{proof}

\section{Chapter Appendix}

\subsection{The $W_2$ Metric}
\label{sec:DSappendix}
Let $W_2$ be the class of functions defined in \eqref{eq:spacew2}, i.e.\ the class of differentiable functions $h(x): \R \to \R$ such that both $h(x)$ and $h'(x)$ belong to $\lipone$. 
%\begin{itemize}
%\item $h'(x)$ exists Lebesgue almost everywhere, and $h'(x-)$ exists everywhere.
%\item $h''(x)$ exists Lebesgue almost everywhere.
%\item $\norm{h'} < 1$ and $\norm{h''} < 1$.
%\end{itemize}
For two random variables $U$ and $V$, define their $W_2$ distance to be 
\begin{equation}
  \label{eq:dW2}
  d_{W_2}(U, V) = \sup_{h\in {W_2}} \abs{\E[h(U)] -\EE[h(V)]},
\end{equation}
and recall the Kolmogorov distance $d_K(U, V)$ defined in \eqref{eq:kolmogorov}. In this section we prove the following relationship between the $W_2$ and Kolmogorov distances. This lemma is a modified version of \cite[Proposition 1.2]{Ross2011}. 
\begin{lemma}
\label{lem:pmetrics}
Let $U, V$ be two random variables, and assume that $V$ has a density bounded by some constant $C > 0$. If $d_{W_2}(U, V) < 4C$, then 
\begin{equation*}
  d_K(U, V) \le 5\Big(\frac{C}{2}\Big)^{2/3} d_{W_2}(U, V)^{1/3}.
\end{equation*}
\end{lemma}
\noindent We wish to combine Lemma~\ref{lem:pmetrics} with Theorem~\ref{thm:w2}, but to do so we need a bound on the density of $Y_S(\infty)$.
\begin{lemma} \label{lem:densbound}
Let $\nu_S(x) : \R \to \R$ be the density of $Y_S(\infty)$, whose form is given in \eqref{eq:stdden}. Then for all $n \geq 1, \lambda > 0$, and $\mu > 0$ satisfying $1  \geq R < n $, 
\begin{align*}
\nu_S(x) \leq 4, \quad x \in \R.
\end{align*}
\end{lemma}
\noindent Now combining Theorem~\ref{thm:w2} with Lemmas~\ref{lem:pmetrics} and \ref{lem:densbound} implies that $d_K\big(\tilde X(\infty), Y_S(\infty) \big)$ converges to zero at a rate of $1/R^{1/3}$. However, we believe this rate to be sub-optimal, and that $d_K\big(\tilde X(\infty), Y_S(\infty) \big)$ actually vanishes at a rate of $1/\sqrt{R}$. This is supported by numerical results in Section~\ref{sec:numeric}.

\begin{proof}[Proof of Lemma~\ref{lem:pmetrics} ]
Fix $a \in \R$ and let $h(x) = 1_{(-\infty, a]}(x)$. Now fix $\epsilon \in (0, 2)$ and define the smoothed version
\begin{align*}
h_{\epsilon}(x) = 
\begin{cases}
1, \quad &x \leq a, \\
-\frac{2}{\epsilon^2}(x-a)^2 + 1, \quad &x \in [a, a + \epsilon/2], \\
\frac{2}{\epsilon^2}\big[x - (a + \epsilon/2) \big]^2 - \frac{2}{\epsilon}(x-a) + \frac{3}{2}, \quad &x \in [a + \epsilon/2, a + \epsilon], \\
0, \quad &x \geq a + \epsilon.
\end{cases}
\end{align*}
Since we chose $\epsilon < 2$, it is not hard to see that 
\begin{align*}
\abs{h_{\epsilon}'(x)} \leq \frac{4}{\epsilon^2}, \quad \abs{h_{\epsilon}''(x)} \leq \frac{4}{\epsilon^2}, \quad x \in \R,
\end{align*}
where $h_{\epsilon}''(x)$ is interpreted as the left derivative of $h_{\epsilon}'(x)$ for $x \in \{a, a+\epsilon/2, a+\epsilon\}$. Therefore, $\frac{\epsilon^2}{4} h_{\epsilon}(x) \in W_2$. Then 
\begin{align*}
\EE h(U) - \EE h(V) =&\ \EE h(U) - \EE h_{\epsilon}(V) + \EE \big[ h_{\epsilon}(V) - h(V)\big]\\
\leq&\ \EE h_{\epsilon}(U) - \EE h_{\epsilon}(V) + C \int_{a}^{a+\epsilon} h_{\epsilon}(x) dx \\
=&\ \EE h_{\epsilon}(U) - \EE h_{\epsilon}(V) + C\epsilon/2 \\
\leq&\ \frac{4}{\epsilon^2} d_{W_2}(U,V) + C\epsilon /2,
\end{align*}
Choose $\epsilon = \Big(\frac{2d_{W_2}(U,V)}{C}\Big)^{1/3}$, which lies in $(0,2)$ by our assumption that $d_{W_2}(U,V) < 4C$. Then
\begin{align*}
\EE h(U) - \EE h(V) \leq 5\Big(\frac{C}{2}\Big)^{2/3} d_{W_2}(U, V)^{1/3}.
\end{align*}
Using the function $\tilde h_{\epsilon}(x) := h_{\epsilon}(x+\epsilon)$, a similar argument can be repeated to show that 
\begin{align*}
 \EE h(V) - \EE h(U) =&\ \EE h(V) - \EE \tilde h_{\epsilon}(V) + \EE \tilde h_{\epsilon}(V) - \EE h(U) \\ 
\leq&\ 5\Big(\frac{C}{2}\Big)^{2/3} d_{W_2}(U, V)^{1/3},
\end{align*}
concluding the proof.
\end{proof}

\begin{proof}[Proof of Lemma~\ref{lem:densbound}]
One can check (see also \eqref{DS:pdef} in Section~\ref{app:DSgradbounds})  that \eqref{eq:stdden} implies 
\begin{align*}
\nu_S(x) =
\begin{cases}
\frac{a_1}{\mu }e^{-x^2}, \quad x \leq -1/\delta,\\
\frac{a_2}{\mu(2 + \delta x)}e^{\frac{2}{\delta^2} [2\log(2 + \delta x) - \delta x]}, \quad x \in [-1/\delta, -\zeta], \\
\frac{a_3}{\mu(2 + \delta \abs{\zeta})}e^{\frac{-2\abs{\zeta} x }{2+\delta \abs{\zeta}}}, \quad x \geq -\zeta,
\end{cases}
\end{align*}
where the constants $a_1, a_2, a_3$ make the $\nu_S(x)$ continuous and integrate to one. To prove that $\nu_S(x)$ is bounded, we need to bound these three constants. We know that 
\begin{align}
&\ a_1 \int_{-\infty}^{-1/\delta} \frac{1}{\mu }e^{-y^2} dy + a_2 \int_{-1/\delta}^{-\zeta} \frac{1}{\mu(2 + \delta y)}e^{\frac{2}{\delta^2} [2\log(2 + \delta y) - \delta y]} dy  \notag \\
&+ a_3 \int_{-\zeta}^{\infty} \frac{1}{\mu(2 + \delta \abs{\zeta})}e^{\frac{-2\abs{\zeta} y }{2+\delta \abs{\zeta}}}dy = 1. \label{eq:densintegral}
\end{align}
We first bound $\nu_S(x)$ when $x \leq -1/\delta$. Since $a_1$ and $a_2$ are chosen to make $\nu_S(x)$ continuous at $x = -1/\delta$, we know that $a_1 e^{-1/\delta^2} = a_2 e^{2/\delta^2}$, or $a_2 = a_1 e^{-3/\delta^2}$. Substituting this into \eqref{eq:densintegral}, we see that

\begin{align*}
a_1 \leq \frac{1}{e^{-3/\delta^2}\int_{-1/\delta}^{-\zeta} \frac{1}{\mu(2 + \delta y)}e^{\frac{2}{\delta^2} [2\log(2 + \delta y) - \delta y]} dy} \leq \frac{2\mu }{e^{-3/\delta^2}\int_{-1/\delta}^{0} e^{\frac{2}{\delta^2} [2\log(2 + \delta y) - \delta y]} dy}.
\end{align*}
The derivative of $e^{\frac{2}{\delta^2} [2\log(2 + \delta y) - \delta y]} $ is positive on the interval $[-1/\delta, 0]$. Therefore, on the interval $[-1/\delta, 0]$, this function achieves its minimum at $y = -1/\delta$, implying that $e^{\frac{2}{\delta^2} [2\log(2 + \delta y) - \delta y]}\geq e^{2/\delta^2}$ for $y \in [-1/\delta, 0]$, and 
\begin{align*}
a_1 \leq \frac{2\mu }{e^{-3/\delta^2}\int_{-1/\delta}^{0} e^{\frac{2}{\delta^2} [2\log(2 + \delta y) - \delta y]} dy} \leq \frac{2\mu }{e^{-3/\delta^2}\int_{-1/\delta}^{0} e^{2/\delta^2} dy} = 2 \mu \delta e^{1/\delta^2}.
\end{align*}
Hence, for $x \leq -1/\delta$, 
\begin{align*}
\nu_S(x) \leq 2 \mu \delta e^{1/\delta^2} \frac{1}{\mu }e^{-x^2} \leq 2\delta  \leq 2,
\end{align*}
where in the last inequality we used the fact that $R \geq 1$, or $\delta \leq 1$. We now bound $\nu_S(x)$ when $x \in [-1/\delta, -\zeta]$. By \eqref{eq:densintegral}, 
\begin{align*}
a_2 \leq \frac{1}{\int_{-1/\delta}^{-\zeta} \frac{1}{\mu(2 + \delta y)}e^{\frac{2}{\delta^2} [2\log(2 + \delta y) - \delta y]} dy} \leq \frac{2\mu }{\int_{-1/\delta}^{0} e^{\frac{2}{\delta^2} [2\log(2 + \delta y) - \delta y]} dy}.
\end{align*}
Using the Taylor expansion 
\begin{align*}
2\log(2 + \delta y) = 2\log(2) + \frac{2}{2}\delta y - \frac{2}{(2 + \xi(\delta y))^2} \frac{(\delta y)^2}{2},
\end{align*}
where $\xi(\delta y) \in [\delta y, 0]$, we see that
\begin{align*}
e^{\frac{2}{\delta^2} [2\log(2 + \delta y) - \delta y]} =&\ e^{\frac{4}{\delta^2} \log(2)} e^{\frac{2}{\delta^2} \Big[ \frac{-\delta^2 y^2}{(2 + \xi(\delta y))^2} \Big] } \geq e^{\frac{4}{\delta^2} \log(2)}e^{-2y^2}, \quad y \in [-1/\delta, 0].
\end{align*}
Therefore, 
\begin{align*}
a_2 \leq \frac{2\mu }{\int_{-1/\delta}^{0} e^{\frac{2}{\delta^2} [2\log(2 + \delta y) - \delta y]} dy} \leq \frac{2\mu }{e^{\frac{4}{\delta^2} \log(2)} \int_{-1}^{0} e^{-2y^2}dy} = \frac{2\mu  e^{-\frac{4}{\delta^2} \log(2)}}{\int_{-1}^{0} e^{-2y^2}dy},
\end{align*}
where in the second inequality we used the fact that $\delta \leq 1$. We conclude that for $x \in [-1/\delta, -\zeta]$, 
\begin{align*}
\nu_S(x) = \frac{a_2}{\mu(2 + \delta x)}e^{\frac{2}{\delta^2} [2\log(2 + \delta x) - \delta x]} \leq \frac{2 e^{-\frac{4}{\delta^2} \log(2)}e^{\frac{2}{\delta^2} [2\log(2 + \delta x) - \delta x]}}{\int_{-1}^{0} e^{-2y^2}dy} \leq \frac{2}{\int_{-1}^{0} e^{-2y^2}dy} \leq 4,
\end{align*}
where in the second last inequality we used the fact that on the interval $[-1/\delta, -\zeta]$, the function $e^{\frac{2}{\delta^2} [2\log(2 + \delta x) - \delta x]}$ achieves its maximum at $x = 0$. This fact can be checked by differentiating the function.

Lastly, we bound $\nu_S(x)$ when $x \geq -\zeta$. By \eqref{eq:densintegral}, 
\begin{align*}
a_3 \leq \frac{1}{\int_{-\zeta}^{\infty} \frac{1}{\mu(2 + \delta \abs{\zeta})}e^{\frac{-2\abs{\zeta} y }{2+\delta \abs{\zeta}}}dy} = 2\mu \abs{\zeta}e^{\frac{2\zeta^2}{2+\delta \abs{\zeta}}},
\end{align*}
which means that for $x \geq -\zeta$,
\begin{align}
\nu_S(x) = \frac{a_3}{\mu(2 + \delta \abs{\zeta})}e^{\frac{-2\abs{\zeta} x }{2+\delta \abs{\zeta}}} \leq \frac{2\abs{\zeta}}{2 + \delta \abs{\zeta}} \leq \abs{\zeta}, \label{eq:nuboundzeta}
\end{align}
which a useful bound only when $\abs{\zeta}$ is small, say $\abs{\zeta} \leq 1$. Now suppose $\abs{\zeta} \geq 1$. Since $\nu_S(x)$ is continuous at $x = -\zeta$, we have
\begin{align*}
a_2 = a_3 e^{\frac{-2\zeta^2 }{2+\delta \abs{\zeta}}}e^{-\frac{2}{\delta^2} [2\log(2 + \delta \abs{\zeta}) - \delta \abs{\zeta}]}.
\end{align*}
We insert this into \eqref{eq:densintegral} to see that for $x \geq -\zeta$, 
\begin{align*}
a_3 \leq&\ \frac{e^{\frac{2\zeta^2 }{2+\delta \abs{\zeta}}}}{\int_{-1/\delta}^{-\zeta} \frac{1}{\mu(2 + \delta y)}e^{\frac{2}{\delta^2} [2\log(2 + \delta y) - \delta y]}e^{-\frac{2}{\delta^2} [2\log(2 + \delta \abs{\zeta}) - \delta \abs{\zeta}]} dy} \\
\leq&\ \frac{e^{\frac{2\zeta^2 }{2+\delta \abs{\zeta}}}}{\int_{0}^{-\zeta} \frac{1}{\mu(2 + \delta y)} dy} 
\leq \frac{e^{\frac{2\zeta^2 }{2+\delta \abs{\zeta}}}}{\int_{0}^{-\zeta} \frac{1}{\mu(2 + \delta \abs{\zeta})} dy}
= \mu \frac{2 + \delta \abs{\zeta}}{\abs{\zeta}}e^{\frac{2\zeta^2 }{2+\delta \abs{\zeta}}},
\end{align*}
where in the second inequality we used that 
\begin{align*}
e^{\frac{2}{\delta^2} [2\log(2 + \delta y) - \delta y]}e^{-\frac{2}{\delta^2} [2\log(2 + \delta \abs{\zeta}) - \delta \abs{\zeta}]} \geq 1 , \quad y \in [0,-\zeta],
\end{align*}
which is true because the derivative of the function $e^{\frac{2}{\delta^2} [2\log(2 + \delta y) - \delta y]}$ is negative on the interval $[0, -\zeta]$. Therefore, for $x \geq -\zeta$, 
\begin{align*}
\nu_S(x) = \frac{a_3}{\mu(2 + \delta \abs{\zeta})}e^{\frac{-2\abs{\zeta} x }{2+\delta \abs{\zeta}}} \leq \frac{1}{\abs{\zeta}} e^{\frac{2\zeta^2 }{2+\delta \abs{\zeta}}} e^{\frac{-2\abs{\zeta} x }{2+\delta \abs{\zeta}}} \leq \frac{1}{\abs{\zeta}}.
\end{align*}
Together with \eqref{eq:nuboundzeta}, this implies that $\nu_S(x) \leq 1$ for $x \geq -\zeta$. This concludes the proof of this lemma.
\end{proof}

\chapter{Moderate Deviations in the Erlang-C Model} \label{chap:md}
This chapter focuses again on the Erlang-C model. We adopt the notation from previous chapters, and refer the reader to Section~\ref{sec:DSmain} for a quick recap of the model and notation. In Theorem~\ref{thm:erlangCK} of Chapter~\ref{chap:erlangAC}, we proved a bound on the Kolmogorov distance between the steady-state customer count $\tilde X(\infty)$ and the diffusion approximation $Y(\infty)$. Namely, we showed that
\begin{align*}
d_K(\tilde X(\infty), Y(\infty)) = \sup_{z \in \R} \big| \Prob(\tilde X(\infty) \leq z) - \Prob(Y(\infty) \leq z) \big| \le \frac{156}{\sqrt{R}},
\end{align*}
where $R$ is the offered load to the system. The Kolmogorov distance represents the absolute error between the cumulative distribution functions (CDF). However, when $\Prob(\tilde X(\infty) \leq z)$ is small, the absolute error is a poor indicator of performance, and the relative error becomes more important. It turns out that Stein's method can also be used to prove error bounds on the relative error, and the goal of this chapter is to do this for the Erlang-C model. Our main result is Theorem~\ref{thm:MDmain} contained in Section~\ref{sec:MDmain}, which shows that there exists a constant $C>0$ independent of $\lambda, n$, and $\mu$, such that for $z = \frac{1}{\sqrt{R}}(k-R), k > n, k \in \Z_+$,
\begin{align}
&\bigg|\frac{\Prob(\tilde X(\infty) \geq z)}{\Prob(\tilde Y_S(\infty)\geq z)} - 1\bigg| \notag \\
\leq&\ \frac{(1-\rho)}{\rho}+ \frac{2(1-\rho)^2}{\rho} + \frac{1}{R} + Ce^{\zeta^2}\Big(\frac{1}{R}+ \frac{1}{\sqrt{R}} \frac{1-\rho}{\rho} + \frac{(1-\rho)^2}{\rho^2}\Big) \notag \\
&+ Ce^{2\zeta^2}\zeta^2\Big(\frac{1}{R}+\frac{1}{\sqrt{R}}\frac{1-\rho}{\rho} + \frac{(1-\rho)^2}{\rho^2}\Big) \min\Big\{(z\vee 1), R\Big(\frac{1}{\abs{\zeta}}+\frac{1}{\sqrt{R}}\Big)^3\Big\}, \label{MD:introresult}
\end{align}
where $n$ is the number of servers in the system,  $R$ is the offered load, $\rho = R/n$ is the system utilization, $\zeta = (R-n)/\sqrt{R}$, and  $Y_S(\infty)$ is the diffusion approximation defined in \eqref{eq:stdden} of Chapter~\ref{chap:dsquare}. 
In particular, the bound in \eqref{MD:introresult} says that in the quality-and-efficiency-driven (QED) regime where $n = \left \lceil R + \beta \sqrt{R}\right \rceil$ for some $\beta > 0$, 
\begin{align*}
&\bigg|\frac{\Prob(\tilde X(\infty) \geq z)}{\Prob(\tilde Y_S(\infty)\geq z)} - 1\bigg| \leq \frac{C(\beta)}{\sqrt{R}} + C(\beta) \min\Big\{\frac{1}{R}(z\vee 1), 1\Big\}.
\end{align*}

Stein's method has been used to prove bounds on the relative error of the CDF approximation in \cite{ChenFangShao2013a, ChenFangShao2013b, ShaoZhou2016, ChenShaoWuXu2016, ChangShaoZhou2016}. These results are referred to as moderate deviations results, which date back to Cram\'er  \cite{Cram1938}, who derived expansions for tail probabilities of sums of independent random variables in terms of the normal distribution. The  following is a typical moderate deviations result  \cite[Chapter 8, equation (2.41)]{Petr1975}. If $X_1, \ldots, X_m$ are i.i.d.\ random variables with $\E X_i = 0$, $\text{Var}(X_i) = 1$, and $\E e^{t_0 \abs{X_i}} < \infty$ for some $t_0 > 0$, then 
\begin{align*}
\frac{\Prob(X_1 + \ldots + X_m > z)}{1 - \Phi(z)} = 1+ O(1)\frac{1+z^3}{\sqrt{m}}, \quad 0 \leq z \leq a_0 m^{1/6},
\end{align*}
where $\Phi(z)$ is the CDF of the standard normal, $O(1)$ is bounded a constant, and both the bound on $O(1)$ and $a_0$ are independent of $m$. The name ``moderate'' deviations comes from the restriction $z \in [0, a_0 m^{1/6}]$, which makes the bound valid as long as $\Prob(X_1 + \ldots + X_m > z)$ is not too small. This type of range restriction on $z$ is always present in moderate deviations results. In contrast, \eqref{MD:introresult} does not have an upper bound on the value that $z$ can take.

The rest of this chapter is structured as follows. We state and prove our main results in Section~\ref{sec:MDmain}, and prove some auxiliary lemmas in Section~\ref{sec:MDaux}. The author would like to thank Xiao Fang, who provided him with a preliminary version of the moderate deviations result for the Erlang-C system. 
\section{Main Result} \label{sec:MDmain}
In this section we state and prove the main result of this chapter. We assume familiarity with the Stein framework introduced in Section~\ref{sec:CAroadmap}. We also refer the reader to Section~\ref{sec:DSmain} for a quick summary of notation. In addition to the notation used there, we let $\rho = \lambda/n\mu$ be the utilization in the Erlang-C system.
\begin{theorem} \label{thm:MDmain}
Recall that $\zeta = \delta(R-n)$.  There exists a constant $C > 0$ such that for any $k \in \Z_+$, $k > n$, $z = \delta (k-R)$, $\lambda > 0, \mu > 0$, and $n \geq 1$ satisfying $\rho \geq 1/2$,
\begin{align}
&\bigg|\frac{\Prob(\tilde X(\infty) \geq z)}{\Prob(Y_S(\infty) \geq z)} - 1\bigg| \notag \\
\leq&\ \frac{(1-\rho)}{\rho}+ \frac{2(1-\rho)^2}{\rho} + \delta^2 + Ce^{\zeta^2}\Big(\delta^2+\delta \frac{1-\rho}{\rho} + \frac{(1-\rho)^2}{\rho^2}\Big) \notag \\
&+ Ce^{2\zeta^2}\zeta^2\Big(\delta^2+\delta \frac{1-\rho}{\rho} + \frac{(1-\rho)^2}{\rho^2}\Big) \min\Big\{(z\vee 1), \frac{1}{\delta^2} \Big(\frac{1}{\abs{\zeta}}+\delta\Big)^3\Big\}.
\end{align}
\end{theorem}
\noindent  To supplement the theorem, we present some numerical results below. 
Recall that in addition to $Y_S(\infty)$, the diffusion approximation with state-dependent diffusion coefficient, we also have $Y(\infty)$, the approximation with constant diffusion coefficient; cf. \eqref{CA:stdden}. From the results in Chapter~\ref{chap:dsquare}, it is natural to anticipate that $Y_S(\infty)$ is a better approximation, and this is correct. Figure~\ref{fig:MDrelerror} displays the relative error of approximating $\Prob(\tilde X(\infty) \geq z)$ when $n = 100$ and $\rho = 0.9$. We see a \emph{qualitative} difference in the approximation quality of $Y(\infty)$ and $Y_S(\infty)$. The relative error of the former increases linearly in $z$, whereas the error of the latter is bounded no matter how large $z$ becomes. These results are consistent for other choices of $n$ and $\rho$. 

In contrast to the universal approximation results we saw in the previous chapters, the upper bound in Theorem~\ref{thm:MDmain} only decreases as $\rho \uparrow 1$. However, we believe that universality still holds, and that the current statement of Theorem~\ref{thm:MDmain} is simply a shortcoming of the author's proof. To support this, we present Table~\ref{tab:MDerror}, which shows the relative error when $n$ increases while $\rho$ is fixed at $0.6$. As we had hoped, the relative error of the approximation decreases as $n$ grows, which suggests that the current statement of Theorem~\ref{thm:MDmain} can be improved upon. 

\begin{figure}[h!]\centering
    \hspace{-1.5cm} \includegraphics[clip, trim=.8cm 5.5cm 0.5cm 4cm, width=90mm,keepaspectratio]{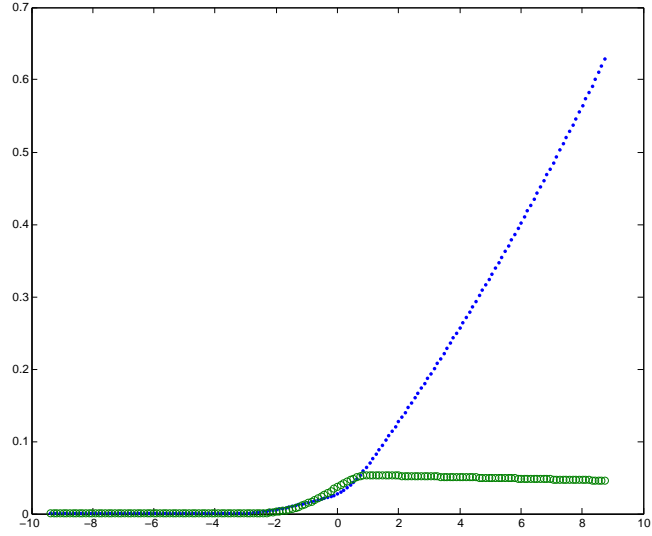}
   \caption{$n =100, \rho = 0.9$. The plot above shows the relative error of approximating $\Prob(\tilde X(\infty) \geq z)$. The x-axis displays the value of $z$; when $z = 8$, $\Prob(X(\infty) \geq z) \approx 10^{-4}$. The blue dots correspond to $\Big|\frac{\Prob(\tilde X(\infty) \geq z)}{\Prob(Y(\infty) \geq z)} - 1\Big|$, the error of the constant diffusion coefficient approximation. The green circles correspond to $\Big|\frac{\Prob(X(\infty) \geq z)}{\Prob(Y_S(\infty)\geq z)} - 1\Big|$, the error from the state-dependent coefficient approximation. \label{fig:MDrelerror}}
\end{figure}

\begin{table}[h]
  \begin{center}
  \resizebox{\columnwidth}{!}{
   \begin{tabular}{r|c| c |c }
$n$ & $\Prob(\tilde X(\infty) \geq 2.4)$ & $\big| \frac{\Prob(\tilde X(\infty) \geq 2.4)}{\Prob(Y_S(\infty) \geq 2.4)} - 1 \big|$ & $\big| \frac{\Prob(\tilde X(\infty) \geq 2.4)}{\Prob(Y(\infty) \geq 2.4)} - 1 \big|$\\
%& & &$(Y_0(\infty))$&\multicolumn{1}{c}{} & $(Y(\infty))$\\
\hline
100 & 0.0146 & 0.1347 & 0.3090 \\
200 & 0.0117  & 0.1027 & 0.2473 \\
400 &  0.0104  & 0.0767 & 0.1890 \\
800 & 0.0102 &  0.0561 & 0.1386 \\
1600 &  0.0094 & 0.0409 &  0.1029 \\
% 15        &   0.900  &$2.07\times 10^{-2}$  & 2.29\% & $2.31\times 10^{-3}$ & 0.26\% \\
% 30        &  0.907  & $1.40\times 10^{-2}$  & 1.54\% & $1.16\times 10^{-3}$ & 0.13\% \\
% 60        &  0.912  &$9.55\times 10^{-3}$ & 1.05\% & $5.68\times 10^{-4}$ & 0.06\% \\
% 125        &  0.915  & $6.43\times 10^{-3}$ & 0.70\% & $2.49\times 10^{-4}$ & 0.03\% \\
% 250        & 0.917  & $4.45\times 10^{-3}$& 0.49\% & $9.71\times 10^{-5}$ & 0.01\% \\
% 500        & 0.918  & $3.09\times 10^{-3}$& 0.34\% & $1.95\times 10^{-5}$ & 0.002\% \\
% 1000        & 0.919  & $2.15\times 10^{-3}$& 0.23\% & $2.03\times 10^{-5}$ & 0.002\% \\
  \end{tabular}}
  \end{center}
  \caption{ Approximating $\Prob(\tilde X(\infty) \geq z)$ with $\rho = 0.6$. The value $z = 2.4$ was chosen because $\Prob(\tilde X(\infty) \geq 2.4) \approx 0.01$. Even though Theorem~\ref{thm:MDmain} does not guarantee sharp error bounds for $\rho = 0.6$, the relative error of the diffusion approximation still shrinks. \label{tab:MDerror}}
\end{table} 
\clearpage

\subsection{Proof of the Main Result}

The rest of this section is dedicated to proving Theorem~\ref{thm:MDmain}. To reduce notational clutter, going forward we let
\begin{align}
W = \tilde X(\infty), \quad \text{ and } \quad Y_S = Y_S(\infty). \label{MD:notation}
\end{align}
The proof of Theorem~\ref{thm:MDmain} follows the standard Stein framework. We recall the generator $G_{Y_S}$ defined in \eqref{eq:GY}, as 
\begin{align*}
G_{Y_S} f(x) = b(x) f'(x) + \frac{1}{2}a(x) f''(x), 
\end{align*}
where $a(x)$ and $b(x)$ are as in \eqref{DS:bdef} and \eqref{DS:adef}, respectively.
%
%\begin{align*}
%b(x) =&\ \delta(\lambda - d(k)), \\
%a(x) =&\ \delta^2(\lambda + d(k)1(k\geq 0)) = \delta^2 \lambda + \big(\delta^2 \lambda- \delta b(x)\big) 1(x \geq -1/\delta),
%\end{align*}
%and for $x = \delta(k-R)$,
%\begin{align*}
%d(k) = \mu(k \wedge n).
%\end{align*}
Fix $z \in \R$ and suppose $f_z(w)$ satisfies the Poisson equation
\begin{align*}
b(w) f_z'(w) + \frac{1}{2}a(w) f_z''(w) =  \Prob(Y_S \geq z) - 1(w \geq z)
\end{align*}
Let $G_{\tilde X}$ be the generator of the CTMC associated to $W = \tilde X(\infty)$, whose form can be found in \eqref{DS:GX}. Using the Taylor expansion performed in Section~\ref{sec:ktaylor}, one can check that for $w = \delta(k-R)$, 
\begin{align*}
&G_{\tilde X} f_z(w) - G_{\tilde Y} f_z(w) \\
=&\ \lambda \int_{w}^{w+\delta} (w+\delta -y ) f_z''(y) dy + d(k) \int_{w-\delta}^{w} (y-(w-\delta))f_z''(y) dy - \frac{1}{2}a(w) f_z''(w) \\
=&\ \int_{0}^{\delta} f_z''(w+y) \lambda(\delta -y) dy + \int_{-\delta}^{0} f_z''(w+y)(y+\delta) (\lambda - b(w)/\delta) dy- \frac{1}{2}a(w) f_z''(w),
\end{align*}
where $d(k) = \mu(k \wedge n)$, and $f_z''(w)$ is understood to be the left derivative at the points $w = z$ and $w = -\delta R$. Lemma~\ref{lem:gz} tells us that $\E G_W f_z(w) = 0$, and we conclude that
\begin{align}
&\Prob(Y_S\geq z) - \Prob(W \geq z) \notag \\
=&\ \E G_{\tilde Y} f_z(W) - \E G_{\tilde X} f_z(W) \notag \\
=&\ \E \Big[\frac{1}{2}a(W) f_z''(W) \Big] \notag \\
&- \E \bigg[\int_{0}^{\delta} f_z''(W+y) \lambda(\delta -y) dy + \int_{-\delta}^{0}f_z''(W+y)(\lambda - b(W)/\delta)(y+\delta) dy \bigg]. \label{MD:exp1}
\end{align}
The proof of Theorem~\ref{thm:MDmain} revolves around bounding the right hand side above. Define
\begin{align*}
K_W(y) = 
\begin{cases}
(\lambda - b(W)/\delta) (y+\delta) \geq 0, \quad y \in [-\delta, 0], \\
\lambda(\delta-y) \geq 0, \quad y \in [0,\delta].
\end{cases}
\end{align*}
It can be checked that
\begin{align}
\int_{-\delta}^{0} K_W(y) dy =&\ \frac{1}{2} \delta^2 \lambda -\frac{1}{2}\delta b(W),\\
\int_{0}^{\delta} K_W(y) dy =&\ \frac{1}{2}\delta^2 \lambda,\\
\int_{-\delta}^{\delta} K_W(y) dy =&\ \frac{1}{2}a(W) \label{MD:k3}
\end{align}
Together with \eqref{MD:k3}, the expansion in \eqref{MD:exp1} then implies that
\begin{align}
&\Prob(Y_S\geq z) - \Prob(W \geq z) \notag \\
 =&\ -\E \bigg[\int_{-\delta}^{\delta} (f_z''(W+y)-f_z''(W)) K_W(y) dy  \bigg] \notag \\
=&\ \E \bigg[\int_{-\delta}^{\delta} \Big(\frac{2b(W+y)}{a(W+y)}f_z'(W+y) - \frac{2b(W)}{a(W)}f_z'(W) \Big) K_W(y) dy  \bigg]  \notag \\
&+\E \bigg[\int_{-\delta}^{\delta} \Big( \frac{2}{a(W)} 1(W\geq z) - \frac{2}{a(W+y)}1(W+y\geq z) \Big) K_W(y) dy  \bigg] \notag  \\
&+\Prob(Y_S \geq z)\E \bigg[\int_{-\delta}^{\delta} \Big(  \frac{2}{a(W+y)} -\frac{2}{a(W)} \Big) K_W(y) dy  \bigg], \label{MD:diff}
\end{align}
where we used $f_z''(w) = -\frac{2b(w)}{a(w)}f_z'(w) + \frac{2}{a(w)}\big(\Prob(Y_S \geq z) - 1(w\geq z) \big)$  in the last equation. The following lemma is assumed for now, and will be proved at the end of Section~\ref{sec:MDaux}.
\begin{lemma}\label{lem:termbyterm}
There exists a constant $C>0$, independent of $\lambda, \mu$ or $n$, such that
\begin{align}
& \bigg|\E\bigg[\int_{-\delta}^{\delta} \Big(\frac{2b(W+y)}{a(W+y)}f_z'(W+y) - \frac{2b(W)}{a(W)}f_z'(W) \Big) K_W(y) dy\bigg]\bigg|  \notag \\
\leq&\ Ce^{\zeta^2}\Big(\delta^2+\delta \frac{1-\rho}{\rho} + \frac{(1-\rho)^2}{\rho^2}\Big)  + \frac{2(1-\rho)^2}{1+\rho} \frac{\Prob(W\geq z)}{\Prob(Y_S \geq z)} \notag \\
&+ Ce^{2\zeta^2}\zeta^2\Big(\delta^2+\delta \frac{1-\rho}{\rho} + \frac{(1-\rho)^2}{\rho^2}\Big) \min\Big\{(z\vee 1), \frac{1}{\delta^2} \Big(\frac{1}{\abs{\zeta}}+\delta\Big)^3\Big\}.\label{eq:lem3}
\end{align} 
\end{lemma}
\noindent We now prove Theorem~\ref{thm:MDmain}.
\begin{proof}[Proof of Theorem~\ref{thm:MDmain} ]
Throughout the proof we will let $C>0$ be a positive constant that may change from line to line, but will always be independent of $\lambda, n$, and $\mu$. We begin by bounding the second and third terms on the right hand side of \eqref{MD:diff}. Since we assumed that $z = \delta(k-R)$ and $k > n$, this implies that $z \geq -\zeta + \delta$. Observe that
\begin{align*}
&\int_{-\delta}^{\delta} \Big( \frac{2}{a(W)} 1(W\geq z) - \frac{2}{a(W+y)}1(W+y\geq z) \Big) K_W(y) dy  \\ 
=&\ 1(W = z) - 1(W = z)\int_{0}^{\delta} \frac{2}{a(W+y)}1(W+y\geq z) K_W(y) dy\\
=&\ 1(W = z) - 1(W = z) \frac{2}{a(-\zeta)} \frac{1}{2}\lambda \delta^2 \\
=&\ 1(W = z)\frac{1+\delta\abs{\zeta}}{2+\delta\abs{\zeta}} \\
=&\ 1(W = z)\frac{1}{1+\rho},
\end{align*} 
where in the second equality we used the fact that $a(z+y) = a(-\zeta)$ for $y \in [0,\delta]$, and in the last equality we used the fact that $\delta \abs{\zeta} = \delta^2(n-R) = 1/\rho - 1$. The flow-balance equations of the  Erlang-C model imply that
\begin{align}
\Prob(W = z) = (1-\rho) \Prob(W \geq z). \label{MD:tail}
\end{align} 
 Therefore,
\begin{align}
&\ \E \bigg[\int_{-\delta}^{\delta} \Big( \frac{2}{a(W)} 1(W\geq z) - \frac{2}{a(W+y)}1(W+y\geq z) \Big) K_W(y) dy  \bigg] \notag \\
=&\  \Prob(W = z)\frac{1}{1+\rho} = \Prob(W \geq z)\frac{1-\rho}{1+\rho}. \label{eq:lem1}
\end{align} 
To bound the third term in \eqref{MD:diff}, observe that
\begin{align}
\bigg|\int_{-\delta}^{\delta} \Big(  \frac{2}{a(W+y)} -\frac{2}{a(W)} \Big) K_W(y) dy  \bigg| =&\ \bigg|2\int_{-\delta}^{\delta} \Big(  \frac{a(W)-a(W+y)}{a(W+y)a(W)}  \Big) K_W(y) dy \bigg| \notag  \\
\leq&\ 2\frac{\delta^2 \mu }{\mu a(W)} \int_{-\delta}^{\delta}  K_W(y) dy =\delta^2, \label{eq:lem2}
\end{align}
where in  inequality we used the fact that $K_W(y) \geq 0$, $1/a(w) \leq 1/\mu$, and  $a'(w) \leq \mu\delta$ for all $w \in \R$.
Applying the bounds in \eqref{eq:lem3}, \eqref{eq:lem1}, and \eqref{eq:lem2} to \eqref{MD:diff}, we arrive at
\begin{align*}
&\bigg|\frac{\Prob(W \geq z)}{\Prob(Y_S\geq z)} - 1\bigg|  \\
\leq&\ Ce^{\zeta^2}\Big(\delta^2+\delta \frac{1-\rho}{\rho} + \frac{(1-\rho)^2}{\rho^2}\Big)  + \frac{2(1-\rho)^2}{1+\rho} \frac{\Prob(W\geq z)}{\Prob(Y_S \geq z)} \notag \\
&+ Ce^{2\zeta^2}\zeta^2\Big(\delta^2+\delta \frac{1-\rho}{\rho} + \frac{(1-\rho)^2}{\rho^2}\Big) \min\Big\{(z\vee 1), \frac{1}{\delta^2} \Big(\frac{1}{\abs{\zeta}}+\delta\Big)^3\Big\} \\
&+ \delta^2 + \frac{\Prob(W \geq z)}{\Prob(Y_S \geq z)} \frac{1-\rho}{1+\rho}\\
=&\ \delta^2 + Ce^{\zeta^2}\Big(\delta^2+\delta \frac{1-\rho}{\rho} + \frac{(1-\rho)^2}{\rho^2}\Big) \notag \\
&+ Ce^{2\zeta^2}\zeta^2\Big(\delta^2+\delta \frac{1-\rho}{\rho} + \frac{(1-\rho)^2}{\rho^2}\Big) \min\Big\{(z\vee 1), \frac{1}{\delta^2} \Big(\frac{1}{\abs{\zeta}}+\delta\Big)^3\Big\} \\
&+ \frac{\Prob(W \geq z)}{\Prob(Y_S \geq z)} \frac{1-\rho + 2(1-\rho)^2}{1+\rho}.
\end{align*}
It remains to bound $\frac{\Prob(W \geq z)}{\Prob(Y_S \geq z)}$. For convenience, let us define 
\begin{align*}
\psi(z) =&\ \delta^2 + Ce^{\zeta^2}\Big(\delta^2+\delta \frac{1-\rho}{\rho} + \frac{(1-\rho)^2}{\rho^2}\Big) \notag \\
&+ Ce^{2\zeta^2}\zeta^2\Big(\delta^2+\delta \frac{1-\rho}{\rho} + \frac{(1-\rho)^2}{\rho^2}\Big) \min\Big\{(z\vee 1), \frac{1}{\delta^2} \Big(\frac{1}{\abs{\zeta}}+\delta\Big)^3\Big\}.
\end{align*}
%\begin{align*}
%\bigg|\frac{\Prob(W \geq z)}{\Prob(Y_S\geq z)} - 1\bigg|  \leq&\  C e^{3\zeta^2/2}\Big( \min\big\{z(1-\rho)^2,1/\abs{\zeta}\big\} + \delta^2 + \delta(1-\rho)\Big)  \\
%&+ \frac{\Prob(W \geq z)}{\Prob(Y_S \geq z)} \frac{1-\rho + 2(1-\rho)^2}{1+\rho} +\delta^2
%\end{align*}
 Rearranging the inequality above, we see that 
\begin{align*}
 1 +  \psi(z) \geq&\ \frac{\Prob(W \geq z)}{\Prob(Y_S \geq z)} \Big( 1 -  \frac{1-\rho + 2(1-\rho)^2}{1+\rho} \Big) \\
=&\ \frac{\Prob(W \geq z)}{\Prob(Y_S \geq z)}\Big( \frac{\rho}{1+\rho} +  \frac{\rho - 2(1-\rho)^2}{1+\rho} \Big) \\
\geq&\ \frac{\Prob(W \geq z)}{\Prob(Y_S \geq z)}\frac{\rho}{1+\rho},
\end{align*}
where in the last inequality we used the fact that $\rho - 2(1-\rho)^2 \geq 0$ for $\rho \in [1/2, 1]$. Therefore, 
\begin{align*}
\frac{\Prob(W \geq z)}{\Prob(Y_S \geq z)} \leq&\ \frac{1+\rho}{\rho} (1+\psi(z)),
\end{align*}
and we conclude that 
\begin{align*}
\bigg|\frac{\Prob(W \geq z)}{\Prob(Y_S\geq z)} - 1\bigg| \leq&\ \psi(z) + \frac{1+\rho}{\rho} (1+\psi(z))\frac{1-\rho + 2(1-\rho)^2}{1+\rho}\\
=&\ \frac{1-\rho + 2(1-\rho)^2}{\rho} +  C\psi(z).
\end{align*}
\end{proof}
\section{Auxiliary Proofs} \label{sec:MDaux}
 Having proved Theorem~\ref{thm:MDmain}, we now describe how to prove \eqref{eq:lem3}. Attempting to bound the left hand side of \eqref{eq:lem3} in its present form will not yield anything useful. This following lemma manipulates the left hand side into something more manageable using a combination of Taylor's theorem and the Poisson equation. 
\begin{lemma}\label{lem:error_expansion}
Assume $z \geq -\zeta + \delta$, and let $r(w) = 2b(w)/a(w)$. Then 
\begin{align}
&\int_{-\delta}^{\delta} \Big(\frac{2b(W+y)}{a(W+y)}f_z'(W+y) - \frac{2b(W)}{a(W)}f_z'(W) \Big) K_W(y) dy  \notag \\ 
=&\ \frac{1}{6}\delta^2 b(W)\frac{2b(W)}{a(W)}f_z''(W) + \Big(\frac{2b(W)}{a(W)}\Big)^2 \int_{-\delta}^{\delta}K_W(y)\int_{0}^{y}\int_{0}^{s} f_z''(W+u) dudsdy \notag \\
&+  \frac{2b(W)}{a(W)}f_z'(W)  \int_{-\delta}^{\delta}K_W(y)\int_{0}^{y}\int_{0}^{s} r'(W+u)duds dy \notag \\
&- 1(W=z) \frac{2b(W)}{a(W)}\frac{2}{a(W)}\int_{-\delta}^{0}yK_W(y) dy \notag \\
&+ \Prob(Y_S \geq z)\frac{2b(W)}{a(W)}\int_{-\delta}^{\delta}K_W(y)\int_{0}^{y}\Big(\frac{2}{a(W+s)} - \frac{2}{a(W)} \Big)dsdy \notag \\
&+ 1(W = -1/\delta)f_z'(W)\int_{0}^{\delta} K_W(y)\int_{0}^{y}r'(W+s)ds  dy \notag \\
&+ 1(W = -\zeta) f_z'(W)\int_{-\delta}^{0} K_W(y)\int_{0}^{y}r'(W+s)ds  dy \notag \\
&+ 1(W \in [-1/\delta + \delta, -\zeta - \delta]) f_z'(W) \int_{-\delta}^{\delta} K_W(y)\int_{0}^{y}\int_{0}^{s}r''(W+u)duds  dy \notag \\
&+ 1(W \in [-1/\delta + \delta, -\zeta - \delta]) f_z'(W) r'(W) \frac{1}{6}\delta^2b(W) \label{MD:verbose}
\end{align}
\end{lemma}
\noindent Examining the right hand side of \eqref{MD:verbose}, we see that we will again need moment and gradient bounds to bound its expected value. One of the moment bounds we will need is
\begin{align}
\E \Big[(\tilde X(\infty))^2 1(\tilde X(\infty) \leq -\zeta)\Big] \leq \frac{4}{3} + \frac{2\delta^2}{3}. \label{MD:xsquaredelta}
\end{align}
This was proved in \eqref{CW:xsquaredelta} of Chapter~\ref{chap:erlangAC}. The following lemma presents the necessary gradient bounds. It is proved in Section~\ref{app:MDgb}.
\begin{lemma} \label{lem:MDgradient_bounds}
There exists a constant $C > 0$ such that for any $\lambda > 0, \mu > 0$, and $n \geq 1$,
\begin{align}
\abs{f_z'(w)} \leq&\ \frac{1}{\mu }e^{\zeta^2}(3+\abs{\zeta}), \quad x \leq -\zeta, \label{MD:gb1} \\ 
f_z'(w) =&\ \frac{\Prob(Y_S \leq z)}{\mu \abs{\zeta}}, \quad w \geq -\zeta, \label{MD:gb2} \\
\frac{1}{\Prob(Y_S \geq z)} \abs{f_z''(w)}  \leq&\ \frac{C}{\mu}e^{\zeta^2}(1 + \abs{\zeta} + \zeta^2), \quad w \leq -\zeta, \label{MD:gb3}\\
\frac{1}{\Prob(Y_S \geq z)}\abs{f_z''(w)}  \leq&\ \frac{C}{\mu}e^{w\frac{2\abs{b(-\zeta)}}{a(-\zeta)}}e^{\zeta^2} (1+\abs{\zeta} + \zeta^2), \quad w \in [-\zeta, z], \label{MD:gb4}\\
f_z''(w)  =&\  0 ,\ \quad w \geq z. \label{MD:gb5}
\end{align}
\end{lemma}
Recall that $2\abs{b(-\zeta)}/a(-\zeta) = 2\abs{\zeta}/(2+\delta \abs{\zeta})$. The appearance of $e^{w\frac{2\abs{b(-\zeta)}}{a(-\zeta)}}$ in \eqref{MD:gb4} means that we require bounds on the moment generating function of $W$. The following lemma contains what we need, and is proved in Section~\ref{sec:MDmgfbound}.
\begin{lemma} \label{lem:mgfbound}
There exists a constant $C > 0$ such that for any $\lambda > 0, \mu > 0$, and $n \geq 1$ satisfying $\rho \geq 0.1$, and any $\gamma > \frac{2+\delta\abs{\zeta}}{2\abs{\zeta}}$, 
\begin{align}
&\E \Big(e^{\big(\frac{2\abs{\zeta}}{2+\delta\abs{\zeta}} - \frac{1}{\gamma}\big)W}1 (W \geq -\zeta )\Big) \leq \gamma C e^{\frac{2\zeta^2}{2+\delta\abs{\zeta}}}, \label{MD:mgfz} \\
&\E \Big(e^{\frac{2\abs{\zeta}}{2+\delta\abs{\zeta}}W}1 (W \geq -\zeta )\Big) \leq \frac{1}{\delta^2} \Big(\frac{1}{\abs{\zeta}}+\delta\Big)^3 C e^{\frac{2\zeta^2}{2+\delta\abs{\zeta}}}. \label{MD:mgf}
\end{align}
\end{lemma}
We are now ready to prove Lemma~\ref{lem:termbyterm}.
\begin{proof}[Proof of Lemma~\ref{lem:termbyterm}]
We prove this lemma by taking expected values on both sides of \eqref{MD:verbose}, and bounding the terms on the right hand side one at a time. Namely, we will bound $1/\Prob(Y_S \geq z)$ times 
\begin{align}
&\ \frac{1}{6}\delta^2 \bigg|\E \bigg[b(W)\frac{2b(W)}{a(W)}f_z''(W)\bigg] \bigg| \notag \\
&+ \bigg|\E \bigg[\Big(\frac{2b(W)}{a(W)}\Big)^2 \int_{-\delta}^{\delta}K_W(y)\int_{0}^{y}\int_{0}^{s} f_z''(W+u) dudsdy \bigg]\bigg|\notag \\
&+ \bigg|\E \bigg[ \frac{2b(W)}{a(W)}f_z'(W)  \int_{-\delta}^{\delta}K_W(y)\int_{0}^{y}\int_{0}^{s} r'(W+u)duds dy\bigg]\bigg| \notag \\
&+ \bigg|\E \bigg[1(W=z) \frac{2b(W)}{a(W)}\frac{2}{a(W)}\int_{-\delta}^{0}yK_W(y) dy \bigg] \bigg| \notag \\
&+ \Prob(Y_S \geq z)\bigg| \E \bigg[\frac{2b(W)}{a(W)}\int_{-\delta}^{\delta}K_W(y)\int_{0}^{y}\Big(\frac{2}{a(W+s)} - \frac{2}{a(W)} \Big)dsdy \bigg]\bigg| \notag \\
&+ \bigg| \E \bigg[1(W = -1/\delta)f_z'(W)\int_{0}^{\delta} K_W(y)\int_{0}^{y}r'(W+s)ds  dy \bigg]\bigg| \notag \\
&+ \bigg| \E \bigg[1(W = -\zeta) f_z'(W)\int_{-\delta}^{0} K_W(y)\int_{0}^{y}r'(W+s)ds  dy \bigg]\bigg|\notag \\
&+ \bigg| \E \bigg[1(W \in [-1/\delta + \delta, -\zeta - \delta]) f_z'(W) \int_{-\delta}^{\delta} K_W(y)\int_{0}^{y}\int_{0}^{s}r''(W+u)duds  dy \bigg] \bigg|\notag \\
&+ \bigg|\E \bigg[ 1(W \in [-1/\delta + \delta, -\zeta - \delta]) f_z'(W) r'(W) \frac{1}{6}\delta^2b(W)\bigg]\bigg|, \label{MD:forlemma}
\end{align}
one line at a time. We begin with the first line in \eqref{MD:forlemma}:
\begin{align*}
& \frac{1}{\Prob(Y_S \geq z)}\bigg|\E \bigg[\frac{1}{6}\delta^2 b(W)\frac{2b(W)}{a(W)}f_z''(W) \bigg]\bigg| \\
\leq&\ \delta^2\frac{C}{\mu}e^{\zeta^2}(1 + \abs{\zeta} + \zeta^2)  \E \bigg[\frac{2b^2(W)}{a(W)} 1(W \leq -\zeta)\bigg] \\
&+ \delta^2 \frac{C}{\mu}e^{\zeta^2} (1+\abs{\zeta} + \zeta^2)\frac{2b^2(-\zeta)}{a(-\zeta)}\E \bigg[e^{W \frac{2\abs{\zeta}}{2+\delta\abs{\zeta}}} 1(W \in [-\zeta, z])\bigg] \\
\leq&\ C\delta^2 e^{\zeta^2}(1 + \abs{\zeta} + \zeta^2) \E \bigg[W^2 1(W \leq -\zeta)\bigg] \\
&+ \delta^2 \frac{C}{\mu}e^{\zeta^2} (1+\abs{\zeta} + \zeta^2)\frac{2b^2(-\zeta)}{a(-\zeta)}\E \bigg[e^{W \frac{2\abs{\zeta}}{2+\delta\abs{\zeta}}} 1(W \in [-\zeta, z])\bigg] \\
\leq&\ C\delta^2 e^{\zeta^2}(1 + \abs{\zeta} + \zeta^2)+ \delta^2 \frac{C}{\mu}e^{\zeta^2} (1+\abs{\zeta} + \zeta^2)\frac{2b^2(-\zeta)}{a(-\zeta)}\E \bigg[e^{W \frac{2\abs{\zeta}}{2+\delta\abs{\zeta}}} 1(W \in [-\zeta, z])\bigg]\\
\leq&\ C\delta^2 e^{\zeta^2}(1 + \abs{\zeta} + \zeta^2)+ \delta^2 C e^{\zeta^2} \zeta^2(1+\abs{\zeta} + \zeta^2)\E \bigg[e^{W \frac{2\abs{\zeta}}{2+\delta\abs{\zeta}}} 1(W \in [-\zeta, z])\bigg],
\end{align*}
where we used \eqref{MD:xsquaredelta} in the third inequality. If $z \leq \frac{2+\delta \abs{\zeta}}{2\abs{\zeta}}$, then 
\begin{align*}
\E \bigg[e^{W \frac{2\abs{\zeta}}{2+\delta\abs{\zeta}}} 1(W \in [-\zeta, z])\bigg] \leq 3.
\end{align*} 
If $z > \frac{2+\delta \abs{\zeta}}{2\abs{\zeta}}$, we use \eqref{MD:mgfz} with $z = \gamma$ there to see that
\begin{align*}
\E \bigg[e^{W \frac{2\abs{\zeta}}{2+\delta\abs{\zeta}}} 1(W \in [-\zeta, z])\bigg] =&\ \E \bigg[e^{W \big(\frac{2\abs{\zeta}}{2+\delta\abs{\zeta}}- 1/z\big)} e^{W/z}1(W \in [-\zeta, z])\bigg] \\
\leq&\ z C e^{\frac{2\zeta^2}{2+\delta\abs{\zeta}}} \leq z C e^{\zeta^2}.
\end{align*}
Using \eqref{MD:mgf},
\begin{align*}
\E \bigg[e^{W \frac{2\abs{\zeta}}{2+\delta\abs{\zeta}}} 1(W \in [-\zeta, z])\bigg] \leq&\  \frac{1}{\delta^2} \Big(\frac{1}{\abs{\zeta}}+\delta\Big)^3 C e^{\frac{2\zeta^2}{2+\delta\abs{\zeta}}} \leq  \frac{1}{\delta^2} \Big(\frac{1}{\abs{\zeta}}+\delta\Big)^3 C e^{\zeta^2}.
\end{align*}
Hence, 
\begin{align*}
&\frac{1}{\Prob(Y_S \geq z)}\bigg|\E \bigg[\frac{1}{6}\delta^2 b(W)\frac{2b(W)}{a(W)}f_z''(W) \bigg]\bigg| \\
\leq&\ C\delta^2 e^{\zeta^2}(1 + \abs{\zeta} + \zeta^2) + C\delta^2 e^{2\zeta^2}\zeta^2(1 + \abs{\zeta} + \zeta^2) \min\Big\{(z \vee 1), \frac{1}{\delta^2} \Big(\frac{1}{\abs{\zeta}}+\delta\Big)^3\Big\}.
\end{align*}
Moving on to the second line of \eqref{MD:forlemma}:
\begin{align*}
&\frac{1}{\Prob(Y_S \geq z)}\bigg|\E \bigg[\Big(\frac{2b(W)}{a(W)}\Big)^2 \int_{-\delta}^{\delta}K_W(y)\int_{0}^{y}\int_{0}^{s} f_z''(W+u) dudsdy \bigg]\bigg| \\
\leq&\ \Big(\frac{2b(-\zeta)}{a(-\zeta)}\Big)^2 \E\bigg[ \int_{-\delta}^{\delta} K_W(y)\int_{0}^{y} \int_{0}^{s} 1( W+u\geq -\zeta) f_z''(W+u) du dsdy\bigg] \\
&+  \E\bigg[ \Big(\frac{2b(W)}{a(W)}\Big)^2\int_{-\delta}^{\delta} K_W(y)\int_{0}^{y} \int_{0}^{s}  1( W+u\leq -\zeta)f_z''(W+u) du dsdy\bigg] \\
\leq&\ \Big(\frac{2b(-\zeta)}{a(-\zeta)}\Big)^2 \E\bigg[ \int_{-\delta}^{\delta} K_W(y)\int_{0}^{y} \int_{0}^{s}  1( W+u \in [-\zeta, z]) \\
& \hspace{6cm} \times  \frac{C}{\mu}e^{\zeta^2}(1 + \abs{\zeta} + \zeta^2)e^{(W+u) \frac{2\abs{\zeta}}{2+\delta\abs{\zeta}}} du dsdy\bigg] \\
&+  \E\bigg[ \Big(\frac{2b(W)}{a(W)}\Big)^2 \int_{-\delta}^{\delta} K_W(y)\int_{0}^{y} \int_{0}^{s}  1( W\leq -\zeta )\frac{C}{\mu}e^{\zeta^2}(1 + \abs{\zeta} + \zeta^2) du dsdy\bigg],
\end{align*}
where in the second inequality we used the gradient bounds from Lemma~\ref{lem:MDgradient_bounds}. To bound the first term, note that 
\begin{align*}
&\Big(\frac{2b(-\zeta)}{a(-\zeta)}\Big)^2 \E\bigg[ \int_{-\delta}^{\delta} K_W(y)\int_{0}^{y} \int_{0}^{s} 1( W+u \in [-\zeta, z]) \\
&\hspace{6cm} \times  \frac{C}{\mu}e^{\zeta^2}(1 + \abs{\zeta} + \zeta^2)e^{(W+u) \frac{2\abs{\zeta}}{2+\delta\abs{\zeta}}} du dsdy\bigg] \\
\leq&\ \Big(\frac{2b(-\zeta)}{a(-\zeta)}\Big)^2 \E\bigg[ 1( W \in[-\zeta, z])\int_{-\delta}^{\delta} C\delta^2 K_W(y) \frac{1}{\mu}e^{\zeta^2}(1 + \abs{\zeta} + \zeta^2)e^{(W+\delta) \frac{2\abs{\zeta}}{2+\delta\abs{\zeta}}}dy\bigg]\\
\leq&\ \Big(\frac{2b(-\zeta)}{a(-\zeta)}\Big)^2  \frac{C\delta^2}{\mu}e^{\zeta^2}(1 + \abs{\zeta} + \zeta^2)\E\bigg[ 1( W \in[-\zeta, z])e^{W \frac{2\abs{\zeta}}{2+\delta\abs{\zeta}}}\int_{-\delta}^{\delta}  K_W(y)dy\bigg]\\
=&\ \Big(\frac{2b(-\zeta)}{a(-\zeta)}\Big)^2  \frac{C\delta^2}{\mu}e^{\zeta^2}(1 + \abs{\zeta} + \zeta^2)\E\bigg[ 1( W \in[-\zeta, z])e^{W \frac{2\abs{\zeta}}{2+\delta\abs{\zeta}}} \frac{1}{2}	a(W)\bigg]\\
=&\ \frac{1}{2}	a(-\zeta)\Big(\frac{2b(-\zeta)}{a(-\zeta)}\Big)^2  \frac{C\delta^2}{\mu}e^{\zeta^2}(1 + \abs{\zeta} + \zeta^2)\E\bigg[ 1( W \in[-\zeta, z])e^{W \frac{2\abs{\zeta}}{2+\delta\abs{\zeta}}} \bigg]\\
\leq&\ C\delta^2 e^{2\zeta^2}(1 + \abs{\zeta} + \zeta^2) \frac{2b^2(-\zeta)}{\mu a(-\zeta)} \min\Big\{z, \frac{1}{\delta^2} \Big(\frac{1}{\abs{\zeta}}+\delta\Big)^3\Big\}\\
\leq&\ C\delta^2 e^{2\zeta^2}(1 + \abs{\zeta} + \zeta^2) \zeta^2 \min\Big\{z, \frac{1}{\delta^2} \Big(\frac{1}{\abs{\zeta}}+\delta\Big)^3\Big\}.
\end{align*}
For the second term, 
\begin{align*}
& \E\bigg[ \Big(\frac{2b(W)}{a(W)}\Big)^21( W\leq -\zeta )\int_{-\delta}^{\delta} K_W(y)\int_{0}^{y} \int_{0}^{s}  \frac{C}{\mu}e^{\zeta^2}(1 + \abs{\zeta} + \zeta^2) du dsdy\bigg] \\ 
\leq&\ C\delta^2 e^{\zeta^2}(1 + \abs{\zeta} + \zeta^2) \E\bigg[ \Big(\frac{2b(W)}{a(W)}\Big)^21( W\leq -\zeta )\frac{1}{\mu}\int_{-\delta}^{\delta} K_W(y)dy\bigg] \\ 
=&\ C\delta^2 e^{\zeta^2}(1 + \abs{\zeta} + \zeta^2) \E\bigg[ \Big(\frac{2b(W)}{a(W)}\Big)^21( W\leq -\zeta )\frac{a(W)}{2\mu}\bigg] \\ 
=&\ C\delta^2e^{\zeta^2}(1 + \abs{\zeta} + \zeta^2)\E \bigg[\frac{2b^2(W)}{\mu a(W)} 1(W \leq -\zeta)\bigg] \\ 
\leq&\ C\delta^2 e^{\zeta^2}(1 + \abs{\zeta} + \zeta^2).
\end{align*}
Hence, 
\begin{align*}
&\frac{1}{\Prob(Y_S \geq z)}\bigg|\E \bigg[\Big(\frac{2b(W)}{a(W)}\Big)^2 \int_{-\delta}^{\delta}K_W(y)\int_{0}^{y}\int_{0}^{s} f_z''(W+u) dudsdy \bigg]\bigg| \\ 
\leq&\  C\delta^2 e^{\zeta^2}(1 + \abs{\zeta} + \zeta^2) + C\delta^2 e^{2\zeta^2}\zeta^2(1 + \abs{\zeta} + \zeta^2) \min\Big\{(z \vee 1), \frac{1}{\delta^2} \Big(\frac{1}{\abs{\zeta}}+\delta\Big)^3\Big\}.
\end{align*}
We now bound the third line in \eqref{MD:forlemma}:
\begin{align*}
&\frac{1}{\Prob(Y_S \geq z)}\bigg|\E \bigg[\frac{2b(W)}{a(W)}f_z'(W)  \int_{-\delta}^{\delta}K_W(y)\int_{0}^{y}\int_{0}^{s} g'(W+u)duds dy\bigg]\bigg| \\
\leq&\ \frac{1}{\Prob(Y_S \geq z)}\bigg|\E \bigg[\frac{2b(W)}{a(W)}f_z'(W)  \int_{-\delta}^{\delta}K_W(y)\int_{0}^{y}\int_{0}^{s} 4\times 1(W+u \leq -\zeta) duds dy\bigg]\bigg| \\
\leq&\ \frac{1}{\Prob(Y_S \geq z)}\bigg|\E \bigg[\frac{2b(W)}{a(W)}f_z'(W) 1(W\leq -\zeta)C\delta^2 \int_{-\delta}^{\delta}K_W(y)dy\bigg]\bigg| \\
\leq&\ C\delta^2 \E \bigg[\abs{b(W)f_z'(W)} 1(W\leq -\zeta)\bigg] \\
\leq&\ C\delta^2 \E \bigg[ \mu \abs{W} \frac{1}{\mu } e^{\zeta^2}(1+\abs{\zeta}) 1(W\leq -\zeta)\bigg] \\
\leq&\ C\delta^2  e^{\zeta^2}(1+\abs{\zeta}),
\end{align*}
where in the second last inequality we used \eqref{MD:gb1}, and in the last inequality we used \eqref{MD:xsquaredelta}. We now bound the fourth line in \eqref{MD:forlemma}:
\begin{align*}
&\frac{1}{\Prob(Y_S \geq z)}\bigg|\E \bigg[1(W=z) \frac{2b(W)}{a(W)}\frac{2}{a(W)}\int_{-\delta}^{0}yK_W(y) dy\bigg]\bigg| \\ 
\leq&\ \frac{\Prob(W=z)}{\Prob(Y_S \geq z)}\bigg| \frac{2b(z)}{a(z)}\frac{2}{a(z)}\delta\int_{-\delta}^{\delta}K_W(y) dy\bigg| \\
=&\ \frac{\Prob(W=z)}{\Prob(Y_S \geq z)}\delta\bigg|\frac{2b(z)}{a(z)}\bigg| \\
=&\ \delta(1-\rho) \frac{\Prob(W\geq z)}{\Prob(Y_S \geq z)}\frac{2|b(-\zeta)|}{a(-\zeta)} \\
=&\ \delta(1-\rho)\frac{2\abs{\zeta}}{2+\delta \abs{\zeta}} \frac{\Prob(W\geq z)}{\Prob(Y_S \geq z)},
\end{align*}
where in the second last equality we used \eqref{MD:tail}.  We now bound the fifth line in \eqref{MD:forlemma}:
\begin{align*}
&\frac{1}{\Prob(Y_S \geq z)}\bigg|\E \bigg[\Prob(Y_S \geq z)\frac{2b(W)}{a(W)}\int_{-\delta}^{\delta}K_W(y)\int_{0}^{y}\Big(\frac{2}{a(W+s)} - \frac{2}{a(W)} \Big)dsdy\bigg]\bigg| \\ 
=&\ \E \bigg[\frac{2b(W)}{a(W)}\int_{-\delta}^{\delta}K_W(y)\int_{0}^{y}2\Big|  \frac{a(W)-a(W+s)}{a(W+s)a(W)}  \Big|dsdy\bigg]\\
\leq&\ \E \bigg[\frac{2b(W)}{a(W)}\int_{-\delta}^{\delta}K_W(y)\int_{0}^{y}  \frac{2\mu \delta \abs{s}}{a(W+s)a(W)} dsdy\bigg]\\
\leq&\ \E \bigg[\frac{2b(W)}{a(W)}\int_{-\delta}^{\delta}K_W(y)\int_{0}^{y}  \frac{2\mu \delta^2}{\mu a(W)} dsdy\bigg]\\
=&\ \delta^3\E \bigg[\frac{2b(W)}{a(W)}\frac{2}{ a(W)}\int_{-\delta}^{\delta}K_W(y) dy\bigg]\\
=&\ \delta^3\E \bigg[\frac{2b(W)}{a(W)}\bigg]\\
\leq&\ \delta^3C(1+\abs{\zeta}),
\end{align*}
where in the first inequality we used the fact that $a'(w) \leq \mu\delta$ for all $w \in \R$, and in the last inequality we used \eqref{MD:xsquaredelta}. 
 We now bound the sixth line in \eqref{MD:forlemma}:
\begin{align*}
&\frac{1}{\Prob(Y_S \geq z)}\bigg|\E \bigg[1(W = -1/\delta)f_z'(W)\int_{0}^{\delta} K_W(y)\int_{0}^{y}g'(W+s)ds  dy\bigg]\bigg| \\ 
\leq&\ \Prob(W = -1/\delta)\frac{3}{\mu} \int_{0}^{\delta} K_W(y)\int_{0}^{y}|g'(-1/\delta+s)|ds  dy\\ 
\leq&\ \Prob(W = -1/\delta)\frac{3}{\mu} \int_{0}^{\delta} 4\delta K_W(y)dy\\
=&\ \Prob(W = -1/\delta)\frac{3}{\mu} 4\delta\frac{\lambda \delta^2 }{2}\\
\leq&\ C\delta\Prob(W = -1/\delta) \\
\leq&\ C\delta^2,
\end{align*}
where we obtained the first inequality from \eqref{MD:gb1}. The term in the seventh line is bounded similarly:
\begin{align*}
&\frac{1}{\Prob(Y_S \geq z)}\bigg|\E \bigg[1(W = -\zeta) f_z'(W)\int_{-\delta}^{0} K_W(y)\int_{0}^{y}g'(W+s)ds  dy\bigg]\bigg| \\ 
\leq&\ \Prob(W = -\zeta)\frac{e^{\zeta^2}(3+\abs{\zeta})}{\mu} \int_{-\delta}^{0} 4\delta K_W(y)dy\\
=&\ \Prob(W = -\zeta)\frac{e^{\zeta^2}(3+\abs{\zeta})}{\mu} 4\delta\frac{1}{2}(\delta^2 \lambda - \delta b(-\zeta))\\
\leq&\ C\delta\Prob(W = -\zeta)e^{\zeta^2}(1+\abs{\zeta}) \\
\leq&\ C\delta^2e^{\zeta^2}(1+\abs{\zeta}).
\end{align*}
We now bound the eighth line in \eqref{MD:forlemma}:
\begin{align*}
&\frac{1}{\Prob(Y_S \geq z)}\bigg|\E \bigg[1(W \in [-1/\delta + \delta, -\zeta - \delta])  f_z'(W) \\
&\hspace{5cm} \times \int_{-\delta}^{\delta} K_W(y)\int_{0}^{y}\int_{0}^{s}g''(W+u)duds  dy \bigg]\bigg| \\ 
\leq&\ \frac{e^{\zeta^2}(3+\abs{\zeta})}{\mu} \E \bigg[1(W \in [-1/\delta + \delta, -\zeta - \delta])  \int_{-\delta}^{\delta} 8\delta^2 K_W(y)dy \bigg]\\
=&\ \frac{e^{\zeta^2}(3+\abs{\zeta})}{\mu}4\delta^2 \E \bigg[1(W \in [-1/\delta + \delta, -\zeta - \delta])  a(W) \bigg]\\
\leq&\ \frac{e^{\zeta^2}(3+\abs{\zeta})}{\mu}4\delta^2 \E \bigg[1(W \leq -\zeta ) \mu (2+\delta\abs{W}) \bigg]\\
\leq&\ C\delta^2e^{\zeta^2}(1+\abs{\zeta}).
\end{align*}
Finally, we bound the ninth line in \eqref{MD:forlemma}:
\begin{align*}
&\frac{1}{\Prob(Y_S \geq z)}\bigg|\E \bigg[1(W \in [-1/\delta + \delta, -\zeta - \delta])  f_z'(W)g'(W) \frac{1}{6}\delta^2b(W)\Big)\bigg]\bigg| \\ 
\leq&\ C\frac{e^{\zeta^2}(3+\abs{\zeta})}{\mu} \E \big[1(W \leq -\zeta)  \delta^2 \abs{b(W)}\big]\\
\leq&\ C\delta^2e^{\zeta^2}(1+\abs{\zeta}).
\end{align*}
Combining these nine bounds together, we arrive at the final bound of 
\begin{align*}
&C\delta^2e^{\zeta^2}(1+\abs{\zeta} + \zeta^2)  +\delta(1-\rho)\frac{2\abs{\zeta}}{2+\delta \abs{\zeta}} \frac{\Prob(W\geq z)}{\Prob(Y_S \geq z)} \frac{\Prob(W\geq z)}{\Prob(Y_S \geq z)}\\
&+ C\delta^2 e^{2\zeta^2}\zeta^2(1 + \abs{\zeta} + \zeta^2) \min\Big\{(z \vee 1), \frac{1}{\delta^2} \Big(\frac{1}{\abs{\zeta}}+\delta\Big)^3\Big\}.
\end{align*}
Combining the above with the fact that $\delta\abs{\zeta} = \delta^2(n-R) = 1/\rho - 1$ concludes the proof.
\end{proof}

\subsection{Proof of Lemma~\ref{lem:error_expansion} (Error term)}
Recall that $r(w) = \frac{2b(w)}{a(w)}$. Using the forms of $a(w)$ and $b(w)$ in \eqref{DS:adef} and \eqref{DS:bdef}, it is not hard to check that 
\begin{align*}
r'(w) = \begin{cases}
2, \quad w \leq -1/\delta, \\
\frac{-4}{(2+\delta w)^2}, \quad w \in (-1/\delta, -\zeta], \\
0, \quad w > -\zeta,
\end{cases}
\end{align*}
where $r'(w)$ is understood to be the left derivative at the points $w = -1/\delta$ and $w = -\zeta$. Assume for now that for all $y \in (-\delta, \delta)$, 
\begin{align}
&\frac{2b(W+y)}{a(W+y)}f_z'(W+y) - \frac{2b(W)}{a(W)}f_z'(W) \notag \\
=&\ y \frac{2b(W)}{a(W)}f_z''(W) +  \frac{2b(W)}{a(W)}\int_{0}^{y}\int_{0}^{s} \Big(\frac{2b(W)}{a(W)} f_z''(W+u) + r'(W+u)f_z'(W) \Big)duds  \notag  \\
&+ \frac{2b(W)}{a(W)}\int_{0}^{y}\Big(\frac{2}{a(W+s)} 1(W+s \geq z) - \frac{2}{a(W)} 1(W \geq z) \Big)ds \notag \\
&+ \Prob(Y_S \geq z)\frac{2b(W)}{a(W)}\int_{0}^{y}\Big(\frac{2}{a(W+s)} - \frac{2}{a(W)} \Big)ds + f_z'(W) \int_{0}^{y}r'(W+s)ds. \label{eq:interm}
\end{align}
We postpone verifying \eqref{eq:interm} to the end of this proof. Since $z \geq -\zeta + \delta$ and $a(w) = a(-\zeta)$ for $w \geq -\zeta$, we see that
\begin{align}
&\frac{2b(W)}{a(W)}\int_{0}^{y}\Big(\frac{2}{a(W+s)} 1(W+s \geq z) - \frac{2}{a(W)} 1(W \geq z) \Big)ds \notag \\
=&\ 1(W=z) \frac{2b(W)}{a(W)}\frac{2}{a(W)}\int_{0}^{y} \big(1(s \geq 0) - 1 \big)ds\\
=&\ 1(W=z) \frac{2b(W)}{a(W)}\frac{2}{a(W)}\big(-y1(y \geq 0)\big). \label{eq:interm2}
\end{align}
Combining \eqref{eq:interm}--\eqref{eq:interm2} with the fact that $\int_{-\delta}^{\delta}yK_W(y)dy = \frac{1}{6}\delta^2 b(W)$, we arrive at
\begin{align*}
&\int_{-\delta}^{\delta} \Big(\frac{2b(W+y)}{a(W+y)}f_z'(W+y) - \frac{2b(W)}{a(W)}f_z'(W) \Big) K_W(y) dy  \notag \\ 
=&\ \frac{1}{6}\delta^2 b(W)\frac{2b(W)}{a(W)}f_z''(W) + \Big(\frac{2b(W)}{a(W)}\Big)^2 \int_{-\delta}^{\delta}K_W(y)\int_{0}^{y}\int_{0}^{s} f_z''(W+u) dudsdy \\
&+  \frac{2b(W)}{a(W)}f_z'(W)  \int_{-\delta}^{\delta}K_W(y)\int_{0}^{y}\int_{0}^{s} r'(W+u)duds dy \\
&- 1(W=z) \frac{2b(W)}{a(W)}\frac{2}{a(W)}\int_{-\delta}^{0}yK_W(y) dy\\
&+ \Prob(Y_S \geq z)\frac{2b(W)}{a(W)}\int_{-\delta}^{\delta}K_W(y)\int_{0}^{y}\Big(\frac{2}{a(W+s)} - \frac{2}{a(W)} \Big)dsdy \\
&+ f_z'(W)\int_{-\delta}^{\delta}K_W(y) \int_{0}^{y}r'(W+s)dsdy
\end{align*}
We are almost done, but the last term on the right hand side above requires some additional manipulations. Since $r'(w) = 0$ for $w \geq -\zeta$ and $K_W(y)=0$ for $W = -1/\delta$ and $y \in [-\delta, 0]$,
\begin{align*}
&\int_{-\delta}^{\delta} K_W(y)\int_{0}^{y}r'(W+s)ds  dy \\
 =&\ 1(W = -1/\delta)\int_{0}^{\delta} K_W(y)\int_{0}^{y}r'(W+s)ds  dy\\
&+ 1(W = -\zeta) \int_{-\delta}^{0} K_W(y)\int_{0}^{y}r'(W+s)ds  dy \\
&+ 1(W \in [-1/\delta + \delta, -\zeta - \delta])  \int_{-\delta}^{\delta} K_W(y)\int_{0}^{y}r'(W+s)ds  dy,
\end{align*}
and for $W \in [-1/\delta + \delta, -\zeta - \delta]$, 
\begin{align*}
&\int_{-\delta}^{\delta} K_W(y)\int_{0}^{y}r'(W+s)ds  dy \\
=&\  \int_{-\delta}^{\delta} K_W(y)\int_{0}^{y} \big(r'(W+s) - r'(W)\big)ds  dy +r'(W) \int_{-\delta}^{\delta} yK_W(y)dy \\
=&\  \int_{-\delta}^{\delta} K_W(y)\int_{0}^{y} \int_{0}^{s}r''(W+u)duds  dy + r'(W) \frac{1}{6}\delta^2b(W).
\end{align*}

To conclude the proof, we verify \eqref{eq:interm}. Since
\begin{align*}
\frac{d}{dx} \big(r(x) f_z'(x)\big) = r(x) f_z''(x) +  r'(x) f_z'(x), 
\end{align*}
it follows from the Fundamental Theorem of Calculus that 
\begin{align}
&\frac{2b(W+y)}{a(W+y)}f_z'(W+y) - \frac{2b(W)}{a(W)}f_z'(W)  \notag \\
=&\ \int_{0}^{y} \Big(\frac{2b(W)}{a(W)} f_z''(W+s) + r'(W+s)f_z'(W)\Big)ds. \label{eq:byparts}
\end{align}
Now 
\begin{align*}
\int_{0}^{y}  f_z''(W+s)ds =&\  yf_z''(W) + \int_{0}^{y}  \big(f_z''(W+s) - f_z''(W) \big)ds \\
=&\ y f_z''(W) +  \int_{0}^{y}\Big(\frac{2b(W+s)}{a(W+s)}f_z'(W+s) - \frac{2b(W)}{a(W)}f_z'(W) \Big)ds   \\
&+ \int_{0}^{y}\Big(\frac{2}{a(W+s)} 1(W+s \geq z) - \frac{2}{a(W)} 1(W \geq z) \Big)ds\\
&+ \Prob(Y_S \geq z)\int_{0}^{y}\Big(\frac{2}{a(W+s)} - \frac{2}{a(W)} \Big)ds,
\end{align*}
and applying \eqref{eq:byparts} once again, we see that this equals 
\begin{align*}
&y f_z''(W) +  \int_{0}^{y}\int_{0}^{s} \Big(\frac{2b(W)}{a(W)} f_z''(W+u) + r'(W+u)f_z'(W) \Big)duds   \\
&+ \int_{0}^{y}\Big(\frac{2}{a(W+s)} 1(W+s \geq z) - \frac{2}{a(W)} 1(W \geq z) \Big)ds\\
&+ \Prob(Y_S \geq z)\int_{0}^{y}\Big(\frac{2}{a(W+s)} - \frac{2}{a(W)} \Big)ds,
\end{align*}
thus proving \eqref{eq:interm}.

\subsection{Moment Generating Function Bound} \label{sec:MDmgfbound}
\begin{proof}[Proof of Lemma~\ref{lem:mgfbound}]
Throughout the proof we will let $C>0$ be a positive constant that may change from line to line, but will always be independent of $\lambda, n$, and $\mu$. Recall that  $\zeta = \delta(R - n)$ and that the random variable $W$ lives on the lattice $\delta(\Z_+ - R)$. Fix $r > 0$ and $M \in \{\delta(k-R) :\ k \geq n\}$. Consider the test function $f(w) = e^{r \phi(w)}$,
where 
\begin{align*}
\phi(w) = 
\begin{cases}
-\zeta, \quad &w \leq -\zeta, \\
w, \quad &w \in [-\zeta, M], \\
M, \quad &w \geq M.
\end{cases}
\end{align*}
For $w = \delta(k-R)$, we have
\begin{align*}
G_{W} f(w) =&\ \lambda(f(w+\delta) - f(w))1(w \in [-\zeta, M-\delta]) \\
&+ \mu(k \wedge n) (f(w-\delta) - f(w))1(w \in [-\zeta+\delta, M])\\
=&\ \lambda f(w) (e^{\delta r}-1)1(w \in [-\zeta, M-\delta]) \\
&+ n\mu f(w) (e^{-\delta r}-1)1(w \in [-\zeta+\delta, M]) \\
=&\ \lambda f(w) (e^{\delta r}-1)1(w \in [-\zeta, M]) - \lambda f(M) (e^{\delta r}-1)1(w = M) \\
&+ n\mu f(w) (e^{-\delta r}-1)1(w \in [-\zeta, M]) - \mu n f(-\zeta) (e^{-\delta r}-1)1(w=-\zeta),
\end{align*}
Since $\E G_{W} f(W) = 0$, we take the expectation in the equation above to see that
\begin{align}
& - \big(\lambda (e^{\delta r}-1) + n\mu (e^{-\delta r}-1)\big) \E \big(f(W) 1(W \in [-\zeta, M])\big) \notag \\
=&\  - \lambda f(M) (e^{\delta r}-1)\Prob(W = M) +  n \mu f(-\zeta) (1 - e^{-\delta r})\Prob(W=-\zeta). \label{eq:mgfbar}
\end{align}
First, note that the right hand side is bounded by
\begin{align}
n \mu f(-\zeta) (1 - e^{-\delta r})\Prob(W=-\zeta) =&\ \lambda f(-\zeta) (1 - e^{-\delta r})\Prob(W=-\zeta-\delta) \notag \\
\leq&\ \lambda f(-\zeta) \delta r \Prob(W=-\zeta-\delta)\notag \\
\leq&\ \lambda f(-\zeta) \delta r C\delta \notag \\
=&\ rf(-\zeta) C\mu , \label{eq:rhs}
\end{align}
where the first equality follows from the flow-balance equations of the CTMC corresponding to $W$, and the last inequality follows from the same logic used to prove \eqref{eq:pikolmbound} of Section~\ref{sec:DSproofW}. Now let $\gamma > \frac{a(-\zeta)}{2\abs{b(-\zeta)}}$ and set $r = \frac{2\abs{b(-\zeta)}}{a(-\zeta)} - \frac{1}{\gamma}$. Assume we can prove that
\begin{align}
-\big(\lambda (e^{\delta r}-1) + n\mu (e^{-\delta r}-1)\big) \geq 
\mu \Big(\frac{r}{\gamma}\frac{1+\rho}{2\rho} + \frac{r^4 \delta^2}{120}\Big), \quad \rho \geq 0.1. \label{eq:bounds}
\end{align}
Then using \eqref{eq:mgfbar} and \eqref{eq:rhs} we get
\begin{align*}
&\E \big(f(W) 1(W \in [-\zeta, M])\big) \leq  \gamma\frac{2\rho}{1+\rho} f(-\zeta) C \leq \gamma C e^{\frac{2\abs{b(-\zeta)}}{a(-\zeta)} \abs{\zeta}}, \quad r = \frac{2\abs{b(-\zeta)}}{a(-\zeta)} - \frac{1}{\gamma}, \\
&\E \big(f(W) 1(W \in [-\zeta, M])\big) \leq  \frac{1}{r^3\delta^2} C e^{\frac{2\abs{b(-\zeta)}}{a(-\zeta)} \abs{\zeta}}, \quad r = \frac{2\abs{b(-\zeta)}}{a(-\zeta)},
\end{align*}
and taking $M \to \infty$ then establishes the claim in the lemma. 

We now verify \eqref{eq:bounds}. Using the Taylor expansions
\begin{align*}
e^{\delta r}-1 =&\ r\delta + \frac{1}{2}(r\delta)^2 + \frac{1}{6}(r\delta)^3 + \frac{1}{24}(r\delta)^4 +  \frac{1}{120}(r\delta)^5 e^{\xi(\delta r)} \\
e^{-\delta r} - 1 =&\ -r\delta + \frac{1}{2}(r\delta)^2 - \frac{1}{6}(r\delta)^3 + \frac{1}{24}(r\delta)^4 - \frac{1}{120}(r\delta)^5 e^{\eta(-\delta r)},
\end{align*}
where $\xi(\delta r) \in [0,\delta r]$ and $\eta (-\delta r) \in [-\delta r, 0]$ (the fifth order expansion is necessary), we rewrite the left side of \eqref{eq:mgfbar} as 
\begin{align*}
& -\Big((\lambda - n\mu ) \big(r \delta +\frac{1}{6}(r\delta)^3 \big) +( \lambda + n\mu)\big(\frac{1}{2}(r\delta)^2 + \frac{1}{24}(r\delta)^4 \big)\Big)\E \big(f(W) 1(W \in [-\zeta, M])\big)\\ 
&- \frac{1}{120}(r\delta)^5 \big( \lambda e^{\xi(\delta r)} - n\mu e^{\eta(-\delta r)}   \big)\E \big(f(W) 1(W \in [-\zeta, M])\big).
\end{align*}
Recalling that $\delta(\lambda - n\mu) = \mu \zeta$, $\lambda \delta^2 = \mu$, and $n\delta ^2 = 1/\rho$, the quantity above becomes 
\begin{align}
& \mu \Big(\abs{\zeta}\big(r +\frac{1}{6}r^3\delta^2 \big) -( 1+ \frac{1}{\rho})\big(\frac{1}{2}r^2 + \frac{1}{24}r^4\delta^2 \big)\Big)\E \big(f(W) 1(W \in [-\zeta, M])\big)\notag  \\ 
&+ \frac{\mu}{120}r^5\delta^3 \big( \frac{1}{\rho} e^{\eta(-\delta r)}  - e^{\xi(\delta r)} \big)\E \big(f(W) 1(W \in [-\zeta, M])\big) \notag \\ 
\geq&\ \mu \Big(\abs{\zeta}\big(r +\frac{1}{6}r^3\delta^2 \big) +( 1+ \frac{1}{\rho})\big(\frac{1}{2}r^2 + \frac{1}{24}r^4\delta^2 \big)- \frac{1}{120}r^5\delta^3  e^{\delta r} \Big)\notag \\
&\hspace{7cm} \times \E \big(f(W) 1(W \in [-\zeta, M])\big) \notag \\
=&\ \mu \Big(r\big(\abs{\zeta} - (1+\frac{1}{\rho}) \frac{1}{2}r\big) +  \frac{1}{6}r^3\delta^2\big(\abs{\zeta} - (1+\frac{1}{\rho}) \frac{1}{4}r\big)- \frac{1}{120}r^5\delta^3  e^{\delta r} \Big)\notag \\
&\hspace{7cm} \times \E \big(f(W) 1(W \in [-\zeta, M])\big) \label{eq:lhs}
\end{align}
Now if $r = \frac{2\abs{b(-\zeta)}}{a(-\zeta)} - \frac{1}{\gamma}$ for some $\gamma > \frac{a(-\zeta)}{2\abs{b(-\zeta)}}$, then 
\begin{align*}
 \abs{\zeta} - (1+\frac{1}{\rho}) \frac{1}{2}r =&\ \abs{\zeta} - (1+\frac{1}{\rho}) \frac{1}{2}\big(\frac{2\abs{\zeta}}{2+\delta \abs{\zeta}} - \frac{1}{\gamma} \big) \\
 =&\ \abs{\zeta}\frac{2+\delta\abs{\zeta} - 1-\frac{1}{\rho}}{2+\delta \abs{\zeta}} +   (1+\frac{1}{\rho}) \frac{1}{2} \frac{1}{\gamma} \\
 =&\ \abs{\zeta}\frac{2+(\frac{1}{\rho}-1) - 1-\frac{1}{\rho}}{2+\delta \abs{\zeta}} +   (1+\frac{1}{\rho}) \frac{1}{2} \frac{1}{\gamma} \\
 =&\ \frac{1}{\gamma}\frac{1+\rho}{2\rho},
\end{align*}
where in the third equality we used the fact that $\delta \abs{\zeta} = \delta^2(n-R) = \frac{1}{\rho}-1$. The right hand side of \eqref{eq:lhs}
then equals 
\begin{align*}
&\mu \Big(\frac{r}{\gamma}\frac{1+\rho}{2\rho}+  \frac{1}{6}r^3\delta^2\big(\frac{1}{2}\abs{\zeta} + \frac{r}{z}\frac{1+\rho}{4\rho} \big) - \frac{1}{120} r^5\delta^3 e^{\delta r} \Big)\E \big(f(W) 1(W \in [-\zeta, M])\big) \\
\geq&\ \mu \Big(\frac{r}{\gamma}\frac{1+\rho}{2\rho}+  \frac{1}{12}r^3\delta^2\abs{\zeta}  - \frac{1}{120} r^5\delta^3 e^{\delta r} \Big)\E \big(f(W) 1(W \in [-\zeta, M])\big) \\
\geq&\ \mu \Big(\frac{r}{\gamma}\frac{1+\rho}{2\rho}+  \frac{1}{12}r^4\delta^2  - \frac{1}{120} r^5\delta^3 e^{\delta r} \Big)\E \big(f(W) 1(W \in [-\zeta, M])\big),
\end{align*}
where in the last inequality we used the fact that $\abs{\zeta} \geq \frac{2\abs{\zeta}}{2 + \delta \abs{\zeta}} \geq r$. Now 
\begin{align*}
r\delta \leq \frac{2\delta\abs{\zeta}}{2+\delta \abs{\zeta}} = \frac{2(1-\rho)}{\rho(2+\delta \abs{\zeta})} = \frac{2(1-\rho)}{2\rho + (1-\rho)} = \frac{2(1-\rho)}{1 + \rho},
\end{align*}
and so it can be checked that 
\begin{align*}
 \frac{1}{12}r^4\delta^2  - \frac{1}{120} r^5\delta^3 e^{\delta r} = \frac{r^4 \delta^2}{12} \big(1 - \frac{1}{10} r\delta e^{\delta r}\big) \geq&\ \frac{r^4 \delta^2}{12} \big(1 - \frac{1}{10} \frac{2(1-\rho)}{1 + \rho}e^{\frac{2(1-\rho)}{1 + \rho}}\big)\\
  \geq&\ \frac{r^4 \delta^2}{12} \frac{1}{10}
\end{align*}
whenever $\rho \geq 0.1$.

\end{proof}

\chapter{Steady-State Diffusion Approximation of the $M/P{h}/{n}+M$ Model} \label{chap:phasetype}
This chapter is based on \cite{BravDai2017}. We ignore any notation defined in previous chapters, and start fresh with notation (although much of the notation will be similar to the previous chapters). In this chapter, we apply the Stein framework introduced in Chapter~\ref{chap:erlangAC} to the $M/Ph/n+M$ system, which serves as a building
block to model large-scale service systems such as customer contact
centers \cite{GansKoolMand2003,AksiArmoMehr2007} and hospital
operations \cite{Armoetal2011,Shietal2015}. In such a system, there
are $n$ identical servers, the arrival process is Poisson (the symbol
$M$) with rate $\lambda$, the service times are i.i.d.\ having a
phase-type distribution (the symbol $Ph$) with $d$ phases and mean $1/\mu$, the
patience times of customers are i.i.d.\ having an exponential
distribution (the symbol $+M$) with mean $1/\alpha<\infty$. When the
waiting time of a customer in queue exceeds her patience time, the
customer abandons the system without service; once the service of a
customer is started, the customer does not abandon.

Let $X_i(t)$ be the number of customers in phase $i$ at time $t$ for
$i=1, \ldots, d$, where $d$ is the number of phases in the service
time distribution. Let $X(t)$ be the corresponding vector. Then the
system size process $X=\{X(t), t\ge 0\}$ has a
unique stationary distribution for any arrival rate $\lambda$ and any
server number $n$ due to customer abandonment; although $X$ is not a Markov chain, it is a function of a Markov chain with a unique stationary distribution,  see Section~\ref{sec:CTMCrep} for details.  In Theorem \ref{thm:main} of this chapter, we
prove that
%  Unlike the
% previous results, our framework provides us with convergence rates:
\begin{equation} \label{eq:introresult}
\sup \limits_{h \in \mathcal{H}}  \abs{\E \big[h(\tilde X^{(\lambda)}(\infty))\big] - \E \big[h(Y(\infty))\big]} \leq \frac{C}{\sqrt{\lambda}} \quad \text{for any } \lambda > 0 \text{ and } n \ge 1 
\end{equation}
satisfying
\begin{equation}
  \label{eq:square-root}
  n \mu = \lambda + \beta \sqrt{\lambda},
\end{equation}
where $\beta\in \R$ is some constant and $\mathcal{H}$ is some class
of functions $h:\R^d\to \R$. This is known as the Halfin-Whitt, or quality- and efficiency-driven (QED) regime \cite{HalfWhit1981}. In (\ref{eq:introresult}), $\tilde
X^{(\lambda)}(\infty)$ is a random vector having the stationary distribution of
a properly scaled version of $X = X^{(\lambda)}$ that depends on the arrival rate $\lambda$, number of servers $n$, the service time distribution, and the abandonment rate $\alpha$, and $Y(\infty)$ is a random vector
having the stationary distribution of a $d$-dimensional piecewise Ornstein-Uhlenbeck
(OU) process $Y=\{Y(t), t\ge 0\}$. The stationary distribution of $X^{(\lambda)}$ exists even when $\beta$ is negative because $\alpha$ is assumed to be positive. The constant $C$ depends on the
service time distribution, abandonment rate $\alpha$, the  constant $\beta$ in (\ref{eq:square-root}), and the choice of ${\cal
  H}$, but $C$ is independent of the arrival rate $\lambda$ and the number of servers $n$.  Unlike the results in Chapters~\ref{chap:erlangAC} and \ref{chap:dsquare}, which were universal and did not rely on any particular parameter regime, we do require the QED regime to prove the result in \eqref{eq:introresult}. The reason for this is the additional difficulty in establishing gradient and moments bounds due to the multi-dimensional nature of $\tilde X^{(\lambda)}(\infty)$ and the approximation $Y(\infty)$.  
  
  Two different classes $\mathcal{H}$ will used in our Theorem
\ref{thm:main}. First, we take $\cal{H}$ to be the class of polynomials up to a certain
order. In this case, (\ref{eq:introresult}) provides rates of
convergence for steady-state moments. Second, ${\cal H}$ is taken to be ${\cal W}^{(d)}$, the class of all $1$-Lipschitz functions
\begin{equation}
  \label{eq:Wd}
     {\cal W}^{(d)} = \{h: \R^d\to\R: \abs{h(x)-h(y)}\le \abs{x-y}\}.
\end{equation}
In this case, (\ref{eq:introresult}) provides rates of convergence for
stationary distributions under the Wasserstein metric;
convergence under Wasserstein metric implies the convergence in
distribution \cite{GibbSu2002}.

As previously mentioned in Section~\ref{chap:intro}, the authors of \cite{DaiHe2013} develop an algorithm to compute the
distribution of $Y(\infty)$.  The algorithm is more computationally efficient, in terms of both time and memory, than computing the distribution of $\tilde X^{(\lambda)}(\infty)$. For example, in an $M/H_2/500+M$
system studied in \cite{DaiHe2013}, where the system has $500$ servers
and a hyper-exponential service time distribution, it took around 1
hour and peak memory usage of 5 GB to compute the distribution of
$X^{(\lambda)}$.  On the same computer, it took less
than 1 minute to compute the
distribution of $Y(\infty)$,  and peak memory usage was less than 200 MB. Theorem \ref{thm:main} quantifies the steady-state diffusion approximations developed in \cite{DaiHe2013}.

 In
\cite{DaiDiekGao2014}, the authors prove the convergence of
distribution $\tilde X^{(\lambda)}(\infty)$ to that of $Y(\infty)$ by proving an
interchange of limits.  The proof technique follows that of the
seminal paper \cite{GamaZeev2006}, where the authors prove an
interchange of limits for generalized Jackson networks of
single-server queues.  The results in \cite{GamaZeev2006} were
improved and extended by various authors for networks of
single-servers \cite{BudhLee2009, ZhanZwar2008, Kats2010}, for
bandwidth sharing networks~\cite{YaoYe2012}, and for many-server
systems \cite{Tezc2008, GamaStol2012, Gurv2014a}.  These ``interchange
limits theorems'' are qualitative and thus do not provide rates of
convergence as in (\ref{eq:introresult}).

 Our use of Stein's method in this chapter has two important features that were not present in the previous chapters. Unlike the Erlang-A and Erlang-C models, which are relatively simple one-dimensional birth death processes, the $M/Ph/n+M$ model is a multi-dimensional Markov chain, and the corresponding diffusion approxmiation is also multi-dimensional. This means that our usual approach for deriving gradient bounds does not hold anymore, and we rely on ideas from \cite{Gurv2014} to solve this problem. The second feature of this chapter is state-space collapse (SSC). We will see that the Markov chain representing the $M/Ph/n+M$ system lives in a higher dimensional space than the diffusion approximation. Therefore, certain SSC error bounds need to be established in order for us to carry out Stein's method. 

In Chapter~\ref{chap:dsquare} we discussed the benefits of using a diffusion approximation with a state-dependent diffusion coefficient. The approximation $Y(\infty)$ in  \eqref{eq:introresult} is based on a diffusion process with a constant diffusion coefficient. Nothing is proved about the approximation with state-dependent diffusion coefficient, because the multi-dimensional nature of the $M/Ph/n+M$ model makes this task much more difficult. However, this does not prevent us from evaluating the approximation numerically, which we do in Section~\ref{sec:PHstatedep}. Our observations depend on the type of service-time distribution we use. Namely, we observe a difference between the cases when the first service phase is deterministic or random. In the former case, no SSC is required, and the state-dependent coefficient approximation performs better. Namely, we observe the phenomenon of faster convergence rates of $1/\lambda$,  analogous to what was proved in Chapter~\ref{chap:dsquare}. In the latter case, SSC is required, and we do not have faster convergence rates. This is because the SSC error is of order $1/\sqrt{\lambda}$ and does not vanish with the use of a state-dependent diffusion coefficient.
% However, this does not hold for all service time distributions. In particular, we try two different service time distributions belonging to the family of phase-type distributions: the 2-phase Coxian and the 2-phase hyper-exponential. In both cases, the diffusion approximation is 2-dimensional. With the Coxian distribution, the system can be represented by a 2-dimensional CTMC and no SSC is required. However, with the hyper-exponential distribution, the CTMC has more than two dimensions, and SSC is required. 

The rest of the chapter is structured as
follows. We begin with Section~\ref{sec:model}, where we formally define the $M/Ph/n+M$
system as well as the diffusion process whose steady-state
distribution will approximate the system. Section~\ref{sec:mainresult}
states our main results. Section~\ref{sec:CTMCrep} describes the continuous-time Markov chain (CTMC)
representation of the $M/Ph/n+M$ system. Section~\ref{sec:mainproof} sets up the Poisson equation, gradient bounds, and Taylor expansion of the CTMC generator. Section~\ref{sec:state-space-collapse} deals with SSC. Moment bounds and the proof of our main
result can be found in Section~\ref{sec:thmpf}. Section~\ref{sec:PHstatedep} contains numerical results evaluating the performance of an approximation with state-dependent diffusion coefficient.

\section{Models} \label{sec:model} 

In this section, we give additional description of the 
$M/Ph/n+M$ system and the corresponding diffusion model.

\subsection{The $M/Ph/n+M$ System}
The basic description of the $M/Ph/n+M$ queueing system was given
  in the first paragraph of the introduction. Here, we describe the
  dynamics of the system.  Upon arrival to the system with idle
servers, a customer begins service immediately. Otherwise, if all
servers are busy, the customer enters an infinite capacity queue to
wait for service. When a server completes serving a customer,
the server becomes idle if the queue is empty, or takes a customer
from the queue under the first-come-first-served service policy if
it is nonempty. Recall that the $Ph$ indicates that customer service times are
i.i.d.\ following a phase-type distribution. We shall provide a
definition of a phase-type distribution shortly below. The phase-type distribution
can approximate any positive-valued distribution \cite[Theorem
III.4.2]{Asmu2003}.

%Recall that $\lambda$ denotes the arrival rate of the system.  We use
%$1/\alpha$ to denote the mean patience time.  In our study, we take
%the service time distribution and $\alpha$ fixed, but allow the
%arrival rate $\lambda$ and  the number of servers
%$n$ to grow without bound.
%Throughout this paper,
% we assume that $n$ follows
%the square-root-safety staffing rule in \eqref{eq:square-root}.
%%  Namely,
%% \begin{equation} \label{eq:halfin-whitt}
%%  n \mu = \lambda + \beta \sqrt{\lambda},
%% \end{equation}
%% where $\beta \in \R$ is some constant and $1/\mu$ is again the mean service time.
%%  Without loss of generaltiy, in the rest of this paper, we assume
%% \begin{equation}
%%   \label{eq:muequalone}
%%   \mu=1.
%% \end{equation}
%%  arrival, service and abandonment rates, respectively, of the system. The
%% service rate is defined to be $\mu = 1/m$, where $m$ is the
%% mean service time.
%%  we
%% assume that $n$ follows $\lambda$ through relationship
%% (\ref{eq:halfin-whitt}) with the constant $\beta$ being fixed.  
%In
%the pioneering paper of \cite{HalfWhit1981}, the authors studied these
%systems as $\lambda\to\infty$ and $n$ grows to infinity following
%\eqref{eq:square-root}. This parameter regime is now known as the
%Halfin-Whitt regime. In this regime, the system has high server
%utilization and at the same time has small customer waiting time and
%abandonment fraction.  Therefore, this regime is also known as the
%quality- and efficiency-driven (QED) regime, a term coined by
%\cite{GansKoolMand2003}.

\subsubsection*{Phase-type Service Time  Distribution}\label{sec:phase-type-distr}
A phase-type distribution is assumed to have $d \geq 1$ phases. Each phase-type distribution is determined by the tuple $(p,\nu,P)$, where $p \in \R^d$ is a vector of non-negative entries whose sum is equal to one, $\nu \in \R^d$ is a vector of positive entries and $P$ is a $d \times d$ sub-stochastic matrix. We assume that $P$ is transient, i.e. 
\begin{equation}
(I-P)^{-1} \text{ \quad exists,} \label{eq:transience}
\end{equation}
and without loss of generality, we also assume that the diagonal entries of $P$ are zero ($P_{ii}=0$).

A random variable is said to have a phase-type distribution with parameters $(p,\nu,P)$ if it is equal to the absorption time of the following CTMC. The state space of the CTMC is $\{1, ... ,d+1\}$, with $d+1$ being the absorbing state. The CTMC starts off in one of the states in $\{1,...,d\}$ according to distribution $p$. For $i = 1, ... ,d $, the time spent in state $i$ is exponentially distributed with mean $1/\nu_i$. Upon leaving state $i$, the CTMC transitions to state $j= 1, ... ,d $ with probability $P_{ij}$, or gets absorbed into state $d+1$ with probability $1- \sum_{j=1}^d P_{ij}$.

%\subsubsection*{Service-time Distribution}

The CTMC above is a useful way to describe the service times in the
$M/Ph/n+M$ system. Upon arrival to the system, a customer is assigned
her first service phase according to distribution $p$. If the customer is forced to wait in queue because all servers are busy, she is still assigned a first service phase, but this phase of service will not start until a server takes on this customer for service. Once a customer with initial phase $i$ enters service, her service time is the time until absorption to state $d+1$ by the CTMC. We
assume without loss of generality that for each service phase $i$,
either
\begin{equation} \label{eq:noredundancy}
p_i > 0 \text{ or } P_{ji}> 0 \text{ for some $j$}.
\end{equation}
This simply means that there are no redundant phases. 

We now define some useful quantities for future use. Define 
\begin{equation}
R = (I-P^T)\text{diag}(\nu) \quad  \text{and} \quad \gamma = \mu R^{-1}p, \label{eq:defph}
\end{equation}
where the matrix $\text{diag}(\nu)$ is the $d\times d$ diagonal matrix 
with diagonal entries given by the components of $\nu$. One may verify that $\sum \limits_{i=1}^d \gamma_i = 1$. One can interpret $\gamma_i$ to be the fraction of phase $i$ service load on the $n$ servers.

For concreteness, we provide two examples of phase-type distributions when $d=2$. The first example is the two-phase hyper-exponential distribution, denoted by $H_2$. The corresponding tuple of parameters is $(p,\nu,P)$, where
\begin{displaymath}
p = (p_1, p_2)^T, \quad \nu = (\nu_1,\nu_2)^T, \text{ \quad and \quad } P = 0.
\end{displaymath}
Therefore, with probability $p_i$, the service time follows an exponential distribution with mean $1/\nu_i$.

The second example is the Erlang-$2$ distribution, denoted by $E_2$. The corresponding tuple of parameters is $(p,\nu,P)$, where
\begin{displaymath}
p = (1,0)^T, \quad \nu = (\theta, \theta)^T, \text{ \quad and \quad } P = \begin{pmatrix}
0 & 1\\ 
0 & 0
\end{pmatrix}.
\end{displaymath}
An $E_2$ random variable is a sum of two i.i.d.\ exponential random variables, each having mean $1/\theta$.

\subsection{System Size Process and Diffusion Model}
\label{sec:diffusion}

Before we state the main results, we introduce the process we wish to approximate, as well as the approximating diffusion process -- the piecewise OU process.
Recall that  $X = \{X(t) \in \R^d, t \geq 0\}$ is the system size process, where 
\begin{displaymath}
X(t) = (X_1(t), ... , X_d(t))^T,
\end{displaymath}
and $X_i(t)$ is the number of customers of phase $i$ in the system (queue + service) at time $t$. We emphasize that $X$ is not a CTMC, but it is a deterministic function of a higher-dimensional CTMC, which will be described in Section \ref{sec:CTMCrep}. 

The process $X$ depends on $\lambda, n, \alpha, p, P$, and $\nu$. However, in this chapter we keep $\alpha, p, P$, and $\nu$ fixed, and allow $\lambda$ and $n$ to vary according to \eqref{eq:square-root}. For the remainder of the chapter we write $X^{(\lambda)}$ to emphasize the dependence of $X$ on $\lambda$; the dependence of $X^{(\lambda)}$ on $n$ is implicit through \eqref{eq:square-root}.

Recall the definition of $\gamma$ in (\ref{eq:defph}) and define the
scaled random variable
\begin{equation}
  \label{eq:scaledctmcsd}
\tilde X^{(\lambda)}(\infty) = \delta (X^{(\lambda)}(\infty) - \gamma n),
\end{equation}
where, for convenience, we let 
\begin{equation}
  \label{eq:delta}
\delta = 1/\sqrt{\lambda}.  
\end{equation}
%Instead of $\tilde X(\infty)$, it would be more appropriate to write $\tilde X^{(n, \beta,\alpha, p, \nu, P)}(\infty)$ to emphasize the dependence on the primitive parameters. Since we keep $\beta,\alpha$ and $(p,\nu, P)$ fixed throughout the paper, and  $\lambda$ and $n$ are connected via \eqref{eq:square-root}
To approximate $\tilde X^{(\lambda)}(\infty)$, we introduce the piecewise OU process $Y =
\{Y(t), t\geq 0\}$. This is a $d$-dimensional diffusion process satisfying
\begin{equation}
Y(t) = Y(0) - p \beta t - R \int_0^t{\big(Y(s) -p(e^{T}Y(s))^+\big) ds} - \alpha p \int_0^t {(e^{T}Y(s))^+ds} + \sqrt{\Sigma} B(t). \label{eq:defou}
\end{equation}
Above, $B(t)$ is the $d$-dimensional standard Brownian motion and $\sqrt{\Sigma}$ is any $d\times d$ matrix satisfying
\begin{equation} \label{eq:OUdiffusioncoeff}
\sqrt{\Sigma}\sqrt{\Sigma}^T= \Sigma = \text{diag}(p) + \sum_{k=1}^d \gamma_k \nu_k H^k + (I-P^T)\text{diag}(\nu) \text{diag}(\gamma)(I-P),
\end{equation}
where the matrix $H^k$ is defined as 
\begin{displaymath}
H^k_{ii} = P_{ki}(1-P_{ki}), \quad H^k_{ij} = -P_{ki}P_{kj} \quad \text{ for $j\neq i$}.
\end{displaymath} 
Comparing the form of $\Sigma$ above to (2.24) of \cite{DaiHe2013} confirms that it is positive definite. Thus $\sqrt{\Sigma}$ exists. Observe that $Y$ depends only on $\beta, \alpha, p, P$, and $\nu$, all of which are held constant throughout this chapter.

%\begin{lemma} \label{lemma:posdef}
%Let $\Sigma$ be defined as in (\ref{eq:defouvar}). Then, $\Sigma$ is positive definite.
%\end{lemma}
%The lemma will be proved in Section~\ref{sec:verify-posit-defin}

The diffusion process in (\ref{eq:defou}) has been studied by \cite{DiekGao2013}. They prove that $Y$ is positive recurrent by finding an appropriate Lyapunov function. In particular, this means that $Y$ admits a stationary distribution.

\section{Main Results} \label{sec:mainresult}

We now state our main results.

\begin{theorem} \label{thm:main}
For every integer $m > 0$, there exists a constant $C_m = C_m(\beta,\alpha,p,\nu,P)>0$ such that for all locally Lipschitz functions $h : \R^d \to \R$ satisfying 
\begin{displaymath}
\abs{h(x)} \leq \abs{x}^{2m} \quad \text{  for  } x\in \R^d,
\end{displaymath}
we have
\begin{displaymath} \label{eq:thm1}
\abs{ \E h(\tilde X^{(\lambda)}(\infty)) - \E h(Y(\infty))} \leq \frac{C_m}{\sqrt{\lambda}}
\quad \text{for all } \lambda >0
\end{displaymath} 
satisfying \eqref{eq:square-root}, which we recall below as
\begin{displaymath}
 n \mu = \lambda + \beta \sqrt{\lambda}.
\end{displaymath}
\end{theorem}
Theorem~\ref{thm:main} will be proved in  Section~\ref{sec:thmpf}. As a consequence of the theorem, we immediately have the following corollary.
\begin{corollary} \label{corol:main}
There exists a constant $C_1 = C_1(\beta,\alpha,p,\nu,P)>0$ such that

\begin{displaymath}
  \sup \limits_{h \in \mathcal{W}^{(d)}}  \abs{\E h(\tilde X^{(\lambda)}(\infty)) - \E h(Y(\infty))} \leq \frac{C_1}{\sqrt{\lambda}} \quad \text{for all } \lambda >0 
\end{displaymath}
satisfying \eqref{eq:square-root}, where $W^{(d)}$ is defined in \eqref{eq:Wd}.
In particular, 
\begin{displaymath}
\tilde X^{(\lambda)}(\infty) \Rightarrow Y(\infty) \text{ \quad as  \quad } \lambda \rightarrow \infty. 
\end{displaymath}
\end{corollary}
\begin{proof}
Suppose $h \in \mathcal{W}^{(d)}$. Without loss of generality, we may assume that $h(0) = 0$, otherwise we may simply consider $h(x) - h(0)$. By definition of $\mathcal{W}^{(d)}$, 
\begin{displaymath}
\abs{h(x)} \leq \abs{x} \quad \text{ for } x\in \R^d
\end{displaymath}
and the result follows \text{from Theorem~\ref{eq:thm1} with $m=1$}.
\end{proof}

\begin{remark}\label{rem:lambdarange}
  For any fixed $\beta\in \R$, there are only finitely many
  combinations of $\lambda \in (0, 4)$ and integer $n\ge 1$ satisfying
  \eqref{eq:square-root}. Therefore, it suffices to prove Theorem
  \ref{thm:main} by restricting $\lambda\ge 4$, a convenience for technical purposes.
\end{remark}

\section{Markov Representation} \label{sec:CTMCrep}
The $M/Ph/n+M$ system can be represented as a CTMC
\begin{displaymath}
U^{(\lambda)} = \{ U^{(\lambda)}(t), t \geq 0\}
\end{displaymath} taking values in $\mathcal{U}$, the set of finite sequences $\{u_1, ... , u_k\}$ . The sequence $u=\{u_1, ... , u_k\}$ encodes the service phase of each customer and their order of arrival to the system. For example, the sequence $\{5, 1, 4\}$ corresponds to $3$ customers in the system, with the service phases of the first, second and third customers (in the order of their arrival to the system) being $5$, $1$ and $4$, respectively. We use $\abs{u}$ to denote the length of the sequence $u$. The irreducibility of the CTMC $U^{(\lambda)}$ is guaranteed by
(\ref{eq:transience}) and (\ref{eq:noredundancy}).

 We remark here that $U^{(\lambda)}$ is not the simplest Markovian representation of the $M/Ph/n+M$ system. Another way to represent this system would be to consider a $d+1$ dimensional CTMC that keeps track of the total number of customers in the system, as well as the total number of customers in each phase that are currently in service; this $d+1$ dimensional CTMC is used in \cite{DaiHeTezc2010}. In this chapter we use the infinite dimensional CTMC $U^{(\lambda)}$ because the system size process $X^{(\lambda)}$ cannot be recovered sample path wise from the $d+1$ dimensional CTMC, it can only be recovered from $U^{(\lambda)}$. Also, the CTMC $U^{(\lambda)}$ will play an important role in our SSC argument in Section~\ref{sec:state-space-collapse}.

In addition to the system size process $X^{(\lambda)}$, we define the queue size process $Q^{(\lambda)} = \{Q^{(\lambda)}(t) \in \Z^d_+, t \geq 0\}$, where 
\begin{displaymath}
Q^{(\lambda)}(t) = (Q^{(\lambda)}_1(t), ... , Q^{(\lambda)}_d(t))^T,
\end{displaymath}
and $Q^{(\lambda)}_i(t)$ is the number of customers of phase $i$ in the queue at time $t$.
Then $X^{(\lambda)}_i(t) - Q^{(\lambda)}_i(t) \geq 0$ is the number phase $i$ customers in service at time $t$.

To recover $X^{(\lambda)}(t)$ and $Q^{(\lambda)}(t)$ from $U^{(\lambda)}(t)$, we define the projection functions $\Pi_X : \mathcal{U} \to \R^d$  and $\Pi_Q : \mathcal{U} \to \R^d$. For each $u\in {\cal U}$ and each phase $i\in \{1, \ldots, d\}$,
\begin{displaymath}
\left(\Pi_X(u)\right)_{i} = \sum_{k=1}^{\abs{u}} 1_{\{ u_k = i\}}
\quad \text{ and }
\quad 
\left(\Pi_Q(u)\right)_{i} = \sum_{k = n+1}^{\abs{u}} 1_{\{ u_k = i\}}.
\end{displaymath}
It is clear that on each sample path
\begin{equation} \label{eq:projectionCTMC}
X^{(\lambda)}(t) = \Pi_X(U^{(\lambda)}(t)) \text{\quad and \quad} Q^{(\lambda)}(t) = \Pi_Q(U^{(\lambda)}(t)) \quad \text{ for } t\ge 0.
\end{equation} 
Because there is customer abandonment the Markov chain $U^{(\lambda)}$ can be
proved to be positive recurrent with a unique stationary
distribution \cite{DaiDiekGao2014}. We use $U^{(\lambda)}(\infty)$ to denote the random element that has
the stationary distribution. It follows that $X^{(\lambda)}(\infty)=\Pi_X(U^{(\lambda)}(\infty))$ has the stationary distribution of $X^{(\lambda)}$, and  $\tilde X^{(\lambda)}(\infty)$ in \eqref{eq:scaledctmcsd} is given by
\begin{equation}
  \label{eq:Xinfty}
  \tilde X^{(\lambda)}(\infty) = \delta (\Pi_X(U^{(\lambda)}(\infty))-\gamma n).
\end{equation}

For $u \in \mathcal{U}$, we define 
\begin{equation}\label{eq:projlittle}
x = \delta(\Pi_X(u) - \gamma n), \quad  q = \Pi_Q(u) \text{ \quad and \quad} z = \Pi_X(u) - q.
\end{equation}
When the CTMC is in state $u$, we interpret $(\Pi_X(u))_i$, $q_i$, and $z_i$ as the number of the
  phase $i$ customers in system, in queue, and in service,
  respectively. It follows that $z\ge 0$. 
  
Let $G_{U^{(\lambda)}}$ be the generator of the CTMC $U^{(\lambda)}$. To describe it, we introduce the lifting operator $A$. For any function $f: \R^d \to \R$, we define $Af: \mathcal{U} \to \R$ by 
\begin{equation} \label{eq:lifter}
Af(u) = f(\delta (\Pi_X(u)-\gamma n)) = f(x).
\end{equation}
Hence, for any function $f: \R^d \to \R$, the generator acts on the lifted version $Af$ as follows:
\begin{eqnarray}
G_{U^{(\lambda)}} Af(u) &=& \sum \limits_{i=1}^d \lambda p_i( f(x + \delta e^{(i)}) - f(x)) + \sum \limits_{i=1}^d \alpha q_i (f(x - \delta e^{(i)}) - f(x)) \nonumber\\
&& + \sum \limits_{i=1}^d \nu_i z_i \Big{[} \sum \limits_{j=1}^d P_{ij}f(x+\delta e^{(j)}-\delta e^{(i)}) \nonumber\\
&& {} + (1-\sum \limits_{j=1}^d P_{ij})f(x-\delta e^{(i)}) -f(x)  \Big{]}.
\label{eq:tildeG}
\end{eqnarray}
Observe that $G_{U^{(\lambda)}} Af(u)$ does not depend on the entire sequence $u$;
it depends on $x$, $q$, and the function $f$ only.

\section{Applying Stein's  Method} \label{sec:mainproof} 
In this section, we prepare the ingredients needed to prove Theorem~\ref{thm:main} using the Stein framework introduced in Section~\ref{sec:CAroadmap}. We prove Theorem~\ref{thm:main} in Section \ref{sec:thmpf}.
\subsection{Poisson Equation}
%The main idea behind Stein's method is that instead of bounding 
%\begin{equation}
%  \label{eq:diff}
%\E h(\tilde X^{(\lambda)}(\infty)) - \E h(Y(\infty)),  
%\end{equation}
%one solves 
Consider the Poisson equation 
\begin{equation} \label{eq:poisson}
G_Y f_h(x) = \E h(Y(\infty)) - h(x),
\end{equation}
where the generator $G_Y$ of the diffusion process $Y$, applied to a function $f(x) \in C^2(\R^d)$, is given by 
\begin{eqnarray}
&G_Y f(x) \notag \\
&= \sum \limits_{i=1}^d \partial_i f(x) \Big{[} p_i \beta - \nu_i(x_i - p_i(e^Tx)^+) - \alpha p_i (e^Tx)^+ + \sum \limits_{j=1}^d P_{ji}\nu_j(x_j-p_j(e^Tx)^+)\Big{]} \notag \\
&+ \frac{1}{2}\sum \limits_{i,j=1}^d \Sigma _{ij} \partial_{ij} f(x)
\quad \text{ for } x\in \R^d. \label{eq:DMgen}
\end{eqnarray}
Taking expected values in \eqref{eq:poisson} with respect to $\tilde X^{(\lambda)}(\infty)$, we focus on bounding the left hand side 
\begin{equation}\label{eq:poissonLHS}
\E G_Y f_h(\tilde X^{(\lambda)}(\infty)).
\end{equation}

The following lemma, based on the results of \cite{Gurv2014},  guarantees the existence of a solution to (\ref{eq:poisson}) and provides gradient bounds for it. The proof of this lemma is given in Section~\ref{app:PHgradbounds}.

\begin{lemma} \label{lemma:poisson}
For any locally Lipschitz function $h: \R^d \to \R$ satisfying $\abs{h(x)} \leq \abs{x}^{2m}$, equation (\ref{eq:poisson})
has a solution $f_h(x)$. Moreover, there exists a constant $C(m,1)>0$ (depending only on $(\beta,\alpha,p,\nu,P)$) such that  for $x\in \R^d$
\begin{eqnarray}
\abs{f_h(x)} &\leq & C(m,1)(1+\abs{x}^2)^m, \label{eq:gradbound1}\\
\abs{\partial_i f_h(x)} &\leq & C(m,1)(1+\abs{x}^2)^m(1+\abs{x}), \label{eq:gradbound2} \\
\abs{\partial_{ij} f_h(x)} &\leq & C(m,1)(1+\abs{x}^2)^m(1+\abs{x})^2, \label{eq:gradbound3} \\
\sup \limits_{y\in \R^d:\abs{y-x} < 1} \frac{\abs{\partial_{ij} f_h(y)-\partial_{ij}f_h(x)}}{\abs{y-x}} &\leq &C(m,1) (1+\abs{x}^2)^m(1+\abs{x})^3. \label{eq:gradbound4}
\end{eqnarray}
\end{lemma}

\subsection{Comparing Generators}
% For a general diffusion process $Z = \{Z(t) \in \R^k, t\geq 0\}$ with stationary distribution $Z(\infty)$ and infinitesimal generator $G_Z$, we have the following \textit{basic adjoint relationship} (BAR):
% \begin{lemma}
% Let $W$ be a random vector in $\R^k$, then
% \begin{displaymath}
% W \stackrel{d}{=} Z(\infty) \iff \E f(G_Z(W)) = 0 \text{ \quad for all \quad } f \in C_b^2(\R^k).
% \end{displaymath} 
% \end{lemma}
% This BAR is the principle behind bounding (\ref{eq:poissonLHS}). One hopes that if $\tilde X(\infty)$ is close to $Y(\infty)$, then (\ref{eq:poissonLHS}) will be close to $0$. The following lemma is a similar statement about our CTMC $(\tilde X,\tilde Q)$.
%Let $W^{(\lambda)}$ denote the random variable $G_Y f_h(\tilde X^{(\lambda)}(\infty))$ in (\ref{eq:poissonLHS}). To
%prove $\abs{\E W^{(\lambda)}}$ small, a common approach in using the Stein's method is to
%find a coupling $\tilde W^{(\lambda)}$ for $W^{(\lambda)}$ so that 
%\begin{eqnarray}
%&&\text{$\abs{\E\tilde W^{(\lambda)}}$ is
%small, and }  \hspace{.6\textwidth}\label{item:1}   \\
%&&  \text{$\E\abs{W^{(\lambda)}-\tilde W^{(\lambda)}}$ is small.} \label{item:2}
%\end{eqnarray}
%Constructing an
%effective coupling is an art that is problem specific. See \cite{Ross2011}
%for a recent survey that includes examples of various couplings.
%
%Recall that $G_{U^{(\lambda)}} Af_h(U^{(\lambda)}(\infty))$ to construct
%the coupling, where $A$ is the lifting operator defined in (\ref{eq:lifter}). 
The following is an analogue of Lemma~\ref{lem:gz}.
\begin{lemma} \label{lemma:CTMCbar}
Let $h: \R^d \to \R$ satisfy $\abs{h(x)} \leq \abs{x}^{2m}$. The function $f_h(x)$ given by (\ref{eq:poisson}) satisfies
\begin{equation} \label{eq:CTMCbar}
\E G_{U^{(\lambda)}} Af_h(U^{(\lambda)}(\infty)) = 0.
\end{equation}
\end{lemma}
\noindent To prove the lemma, we need finite moments of the steady-state system size.

\begin{lemma}\label{lem:finteexpmoment}
(a) Let  $L(u) =\exp(e^T \Pi_X(u))$ for $u \in \mathcal{U}$. Then 
\begin{equation} \label{eq:CTMCmgfexist}
\E L(U^{(\lambda)}(\infty))<\infty.
\end{equation}
(b) all moments of $e^TX^{(\lambda)}(\infty)$ are finite.
\end{lemma}
\begin{proof}
One may verify that
\begin{displaymath}
G_{U^{(\lambda)}} L(u) \leq \lambda(\exp(1)-1)L(u) - \alpha (e^T \Pi_X(u) -n)^+(1- \exp(-1))L(u).
\end{displaymath}
It follows that there exist a positive constant $C= C(\lambda, n, \alpha)$ such that,  whenever $e^T \Pi_X(u)$ is large enough,
\begin{equation} \label{eq:CTMCfosterlyap}
G_{U^{(\lambda)}} L(u) \leq -CL(u) + 1.
\end{equation}
Part (a) follows from \cite[Theorem 4.2]{MeynTwee1993b}. Part (b)
follows from \eqref{eq:CTMCmgfexist} and the equality $e^T
\Pi_X(U^{(\lambda)}(\infty)) = e^T X^{(\lambda)}(\infty)$.
\end{proof}
\noindent The function $L(u)$ is said to be a Lyapunov function.
Inequality (\ref{eq:CTMCfosterlyap}) is known as a Foster-Lyapunov
condition and guarantees that the CTMC is positive recurrent; see, for example, 
\cite{MeynTwee1993b}.

\begin{proof}[Proof of Lemma~\ref{lemma:CTMCbar}]
A sufficient condition for (\ref{eq:CTMCbar}) to hold is given by \cite[Proposition 1.1]{Hend1997} (alternatively, see \cite[Proposition 3]{GlynZeev2008}), namely
\begin{equation}
  \label{eq:glynnzeevi}
  \E\Big[ \abs{G_{U^{(\lambda)}}(U^{(\lambda)}(\infty),U^{(\lambda)}(\infty))} \abs{Af_h(U^{(\lambda)}(\infty))}\Big]<\infty.
\end{equation}
Above, $G_{U^{(\lambda)}}(u,u)$ is the $u$th diagonal entry of the generator matrix $G_{U^{(\lambda)}}$.
 In our case, the left side of (\ref{eq:glynnzeevi}) is equal to 
\begin{eqnarray*}
& =& \E\Big[ \abs{G_{U^{(\lambda)}}(U^{(\lambda)}(\infty),U^{(\lambda)}(\infty))} \abs{f_h(\tilde X^{(\lambda)}(\infty))}\Big]\\
&=&\E \abs{\lambda + \alpha(e^TX^{(\lambda)}(\infty) - n)^+ + \sum \limits_{i=1}^d \nu_i (X^{(\lambda)}_i(\infty) - Q^{(\lambda)}_i(\infty)} \abs{ f_h(\tilde X^{(\lambda)}(\infty))} \notag \\
&\leq & \E \abs{\lambda + (\alpha \vee \max_i \{\nu_i\}) e^TX^{(\lambda)}(\infty)} \abs{ f_h(\tilde X^{(\lambda)}(\infty))},
\end{eqnarray*}
where the first equality follows from \eqref{eq:Xinfty} and \eqref{eq:lifter}.
One may apply (\ref{eq:gradbound1}) and (\ref{eq:CTMCmgfexist}) to see that the quantity above is finite.
\end{proof}
\subsection{Taylor Expansion}
To prove 
that
\begin{equation*}
  \label{eq:compareGener}
\abs{\E h(\tilde X^{(\lambda)}(\infty))) - \E h(Y(\infty))}  =\abs{ \E G_{U^{(\lambda)}} Af_h(U^{(\lambda)}(\infty)) - G_Y f_h(\tilde X^{(\lambda)}(\infty))}
\end{equation*}
is small, we perform Taylor expansion on $G_{U^{(\lambda)}} Af_h(u)$, which is defined in (\ref{eq:tildeG}):
\begin{eqnarray}
&& G_{U^{(\lambda)}} Af_h(u) \nonumber \\
&=& \sum \limits_{i=1}^d \lambda p_i\big( \delta \partial_i f_h(x) + \frac{\delta^2}{2} \partial_{ii}f_h(\xi_{i}^+)\big) +\alpha q_i\big(-\delta \partial_i f_h(x)+  \frac{\delta^2}{2}\partial_{ii}f_h(\xi_{i}^-)\big)\nonumber\\
&& + \sum \limits_{i=1}^d \nu_i z_i(1-\sum \limits_{j=1}^d P_{ij})\big(-\delta \partial_i f_h(x)+ \frac{\delta^2}{2}\partial_{ii}f_h(\xi_{i}^-)\big)
\nonumber \\
&& +\sum \limits_{i=1}^d \sum \limits_{j=1}^d \nu_i z_i P_{ij}\Big(-\delta \partial_i f_h(x)  + \delta \partial_j f_h(x) + \frac{\delta^2}{2}\partial_{ii}f_h(\xi_{ij}) \nonumber \\
&& \hspace{3.5cm}+ \frac{\delta^2}{2}\partial_{jj}f_h(\xi_{ij})- \delta^2 \partial_{ij}f_h(\xi_{ij})\Big), \label{eq:GXtaylor}
\end{eqnarray}
where $\xi_{i}^+ \in [x, x+\delta e^{(i)}]$, $\xi_{i}^-\in
[x-\delta e^{(i)}, x] $ and $\xi_{ij}$ lies somewhere between $x$ and
$x-\delta e^{(i)}+\delta e^{(j)}$.  Using the gradient bounds in Lemma~\ref{lemma:poisson}, we have the following lemma, which will be proved 
in Section \ref{app:taylorexp}.

\begin{lemma}\label{lemma:diffbound} There exists a constant
  $C(m,2)>0$ (depending only on $(\beta,\alpha,p,\nu,P)$) such that for any $u \in \mathcal{U}$,
  \begin{eqnarray}
    \label{eq:diffbound}
&&G_{U^{(\lambda)}} Af_h(u) - G_Y f_h(x) \nonumber\\
&=&  \sum \limits_{i=1}^d \partial_i f_h(x)\Big{[}(\nu_i - \alpha-\sum \limits_{j=1}^d P_{ji}\nu_j)( \delta q_i - p_i(e^T x)^+)\Big{]} + E(u),
  \end{eqnarray}
  where $q$ and $x$ are as in (\ref{eq:projlittle}), $\delta$ as in \eqref{eq:delta}, and $E(u)$ is an error term that satisfies 

\begin{displaymath}
\abs{E(u)} \leq \delta\, C(m,2) (1+\abs{x}^2)^m (1+\abs{x})^4.
\end{displaymath}
\end{lemma}

\section{State Space Collapse}\label{sec:state-space-collapse}
%In this section, we  prove Lemma~\ref{lemma:SSC}.

One of the challenges we face comes from the fact that our CTMC $U^{(\lambda)}$ is infinite-dimensional, while the approximating diffusion process is only $d$-dimensional. 
Recall the process $(X^{(\lambda)},Q^{(\lambda)})$ defined in (\ref{eq:projectionCTMC}) and the lifting operator $A$ acting on functions $f:\R^d \to \R$, as defined in (\ref{eq:lifter}). When acting on the lifted functions $Af(U^{(\lambda)}(\infty))$, the CTMC generator $G_{U^{(\lambda)}}$ depends on both $\tilde X^{(\lambda)}(\infty)$ and $Q^{(\lambda)}(\infty)$, but its approximation $G_Y f(\tilde X^{(\lambda)}(\infty))$ only depends on $\tilde X^{(\lambda)}(\infty)$. This is captured in \eqref{eq:diffbound} by the term
\begin{displaymath}
\sum \limits_{i=1}^d \partial_i f_h(x)\Big{[}(\nu_i - \alpha-\sum \limits_{j=1}^d P_{ji}\nu_j)( \delta q_i - p_i(e^T x)^+)\Big{]}.
\end{displaymath}
To bound this term, observe that for any $1 \leq i \leq d$, 
\begin{eqnarray}
&&\Big(\nu_i - \alpha-\sum \limits_{j=1}^d P_{ji}\nu_j\Big)\partial_i f_h(x)\big( \delta q_i - p_i(e^T x)^+\big) \notag \\
 &=& \Big(\nu_i - \alpha-\sum \limits_{j=1}^d P_{ji}\nu_j\Big)\Big(\partial_i f_h(x)- \partial_i f_h\big(x - \delta q + p(e^T x)^+\big) \Big)\big( \delta q_i - p_i(e^T x)^+\big) \notag  \\
 &&+\ \Big(\nu_i - \alpha-\sum \limits_{j=1}^d P_{ji}\nu_j\Big)\partial_i f_h\big(x - \delta q + p(e^T x)^+\big)\big( \delta q_i - p_i(e^T x)^+\big) \notag \\
 &=& \Big(\nu_i - \alpha-\sum \limits_{j=1}^d P_{ji}\nu_j\Big)\sum_{k=1}^{d}\partial_{ik} f_h(\xi)( \delta q_k - p_k(e^T x)^+)\big( \delta q_i - p_i(e^T x)^+\big)  \notag \\
 &&+\ \Big(\nu_i - \alpha-\sum \limits_{j=1}^d P_{ji}\nu_j\Big)\partial_i f_h\big(\delta (z - \gamma n) + p(e^T x)^+\big)\big( \delta q_i - p_i(e^T x)^+\big), \label{eq:interm_ssc}
\end{eqnarray}
where $z$, defined in \eqref{eq:projlittle}, is a vector that represents the number of customers of each type in service, and $\xi$ is some point between $x$ and $x - \delta q + p(e^T x)^+$. In particular, there exists some constant $C$ that doesn't depend on $\lambda$ and $n$, such that
\begin{equation} \label{eq:xi_ineq}
\abs{\xi} \leq \abs{x}  + \delta \abs{q} + \abs{p}(e^T x)^+ \leq C \abs{x},
\end{equation}
because $\delta q_i \leq (e^T x)^+$ for each $1 \leq i \leq d$ (i.e.\ the number of phase $i$ customers in queue can never exceed the queue size).

In order to bound the expected value of \eqref{eq:interm_ssc}, we must prove a
relationship between $\tilde X^{(\lambda)}(\infty)$ and $Q^{(\lambda)}(\infty)$. Intuitively, the number of customers of phase $i$
waiting in the queue should be approximately equal to a fraction $p_i$
of the total queue size. The following two lemmas bound
the error caused by the SSC approximation. They are proved at the end of this section.

\begin{lemma}
\label{lemma:SSC}
Let $Z^{(\lambda)}(\infty) = X^{(\lambda)}(\infty) - Q^{(\lambda)}(\infty)$  be the vector representing the number of customers of each type in service in steady-state. Then conditioned on $(e^T \tilde X^{(\lambda)}(\infty))^+$, the random vectors $Q^{(\lambda)}(\infty)$ and $Z^{(\lambda)}(\infty)$ are independent.
%\begin{enumerate}
%\item  \label{item:indep}
%\item Conditioned on $(e^T \tilde X^{(\lambda)}(\infty))^+ = \ell > 0$, the random vector $Q^{(\lambda)}(\infty)$ has a multinomial distribution with $\ell$ trials and mean $\ell p$. \label{item:multinom}
%\end{enumerate}
Furthermore,
\begin{equation} \label{eq:multinom}
\E \Big [\delta Q^{(\lambda)}(\infty) - p(e^T \tilde X^{(\lambda)}(\infty))^+ \Big|\ (e^T \tilde X^{(\lambda)}(\infty))^+ \Big] = 0,
\end{equation}
and for any integer $m>0$, there exists $C(m, 3)>0$ (depending only on $(\beta,\alpha,p,\nu,P)$) such that for all $\lambda>0$ and $n\ge 1$ satisfying \eqref{eq:square-root},
\begin{equation}
  \label{eq:sscscaled}
  \E \Big [\abs{\delta Q^{(\lambda)}(\infty) - p(e^T \tilde X^{(\lambda)}(\infty))^+}^{2m}\Big ] \leq \delta^m\, C(m, 3)\E [(e^T \tilde X^{(\lambda)}(\infty))^+]^m,
\end{equation}
where $\delta=1/\sqrt{\lambda}$ as in \eqref{eq:delta}. 
\end{lemma}

\begin{lemma}
\label{lemma:SSC_bound}
For any integer $m>0$, there exists $C(m, 4)>0$ (depending only on $(\beta,\alpha,p,\nu,P)$) such that for any locally Lipschitz function $h: \R^d \to \R$ satisfying $\abs{h(x)} \leq \abs{x}^{2m}$, and all $\lambda>0$ and $n\ge 1$ satisfying \eqref{eq:square-root}
\begin{eqnarray}
&&\abs{\sum \limits_{i=1}^d  \E \bigg[\partial_i f_h(\tilde X^{(\lambda)}(\infty))\Big{[}(\nu_i - \alpha-\sum \limits_{j=1}^d P_{ji}\nu_j)( \delta Q^{(\lambda)}_i(\infty) - p_i(e^T \tilde X^{(\lambda)}(\infty))^+)\Big{]}\bigg]} \notag \\
&\leq & \delta C(m, 4)\E \Big[\big((e^T \tilde X^{(\lambda)}(\infty))^+\big)^2\Big]\sqrt{\E \Big[1 + \abs{\tilde X^{(\lambda)}(\infty)}^8\Big]}
\end{eqnarray}
where $f_h(x)$ is the solution to the Poisson equation \eqref{eq:poisson}.
\end{lemma}
\begin{proof}[Proof of Lemma~\ref{lemma:SSC}]
We begin by proving \eqref{eq:sscscaled}, for which it suffices to show that for all $\lambda>0$ and $n\ge 1$
    satisfying \eqref{eq:square-root}
\begin{displaymath}
\E \Big [\abs{Q^{(\lambda)}(\infty) - p(e^T X^{(\lambda)}(\infty) - n)^+}^{2m}\Big ] \leq C(m, 3)\E [(e^TX^{(\lambda)}(\infty) - n)^+]^m.
\end{displaymath}
  We first prove a version of \eqref{eq:sscscaled} for any finite time $t \geq 0$. Then,
  $(e^TX^{(\lambda)}(t) - n)^+$ is the total number of customers waiting in queue
  at time $t$. Assume that the system is empty at time $t = 0$, i.e.\ $X^{(\lambda)}(0) = 0$. Fix a phase $i$.  Upon arrival to the system, a
  customer is assigned to service phase $i$ with probability $p_i$. Consider the sequence $\{\xi_j:j=1, 2, \ldots\}$, where $\xi_j$ is one if the $j$th customer to enter the system was assigned to phase $i$, and zero otherwise. Then $\{\xi_j:j=1, 2, \ldots\}$ is a sequence of iid Bernoulli random
  variables with $\Prob(\xi_j=1)=p_i$. For $t > 0$, define $A(t)$ and $B(t)$ to be the total number of customers to have entered the system, and entered service by time $t$, respectively. Also let $\zeta_j(t)$ be the indicator of whether customer $j$ is still waiting in queue at time $t$. Then 
  \begin{eqnarray}
    \label{eq:rigL}
 &&    (e^TX^{(\lambda)}(t) - n)^+ = \sum_{j=B(t)+1}^{A(t)} \zeta_j(t), \\
 &&   Q^{(\lambda)}_i(t) = \sum_{j= B(t) + 1}^{A(t)} \xi_j \zeta_j(t).
    \label{eq:rigQ}
  \end{eqnarray}
Let $Z^{(\lambda)}(t) = X^{(\lambda)}(t) - Q^{(\lambda)}(t)$ be the vector keeping track of the customer types in service at time $t$ and let $B(\ell,p_i)$ be a binomial random variable with $\ell\in \Z_+$ trials and success probability $p_i$. Assuming $X^{(\lambda)}(0)=0$, by a sample path construction of the process $U^{(\lambda)}$ one can verify that for any time $t \geq 0$, the following three properties hold.
First, for any $z \in \Z_+^d$, $a,b \in \Z_+$ \text{with $a\ge 1$}, and $x_1, \ldots, x_a, y_1, \ldots, y_a \in \{0,1\}$,
\begin{eqnarray}
&&
\Prob\big(\xi_{b+1} = x_1, \ldots, \xi_{b+a}=x_{a}\ |\ A(t)= b+a, B(t)=b, Z^{(\lambda)}(t) = z, \notag\\
&& {} \hspace{2in} \zeta_{b+1}=y_1, \ldots, \zeta_{b+a}=y_{a} \big) \notag \\
& =&\ \Prob\big(\xi_{1} = x_1\big)\Prob\big(\xi_{2} = x_2\big)\ldots \Prob\big(\xi_{a}=x_{a} \big)\notag \\
& =& \ p_i^{\sum_{i=1}^a x_i}(1- p_i)^{a-\sum_{i=1}^a x_i}. \label{eq:rig1}
\end{eqnarray}
The right side of \eqref{eq:rig1} is independent of $b$, $z$, $y_1,
\ldots, y_a$.  It then follows from \eqref{eq:rigL}, \eqref{eq:rigQ} 
and \eqref{eq:rig1} that  for any integer $\ell\ge 1$, $q_i\in \Z_+$, and $z \in \Z_+^d$,
\begin{align}
&\Prob\big(Q^{(\lambda)}_i(t) = q_i\ |\ (e^TX^{(\lambda)}(t) - n)^+ = \ell, Z^{(\lambda)}(t) = z\big) \notag \\
=&\ \Prob\big(Q^{(\lambda)}_i(t) = q_i\ |\ (e^TX^{(\lambda)}(t) - n)^+ = \ell\big) \notag \\
=&  \Prob\big(B(\ell,p_i) =q_i \big). \label{eq:rig3}
\end{align}
% Third, for any $\ell, q_i \in \Z_+$,
% \begin{align}
% & \Prob\big(Q_i^{(\lambda)}(t) = q_i\ |\ (e^TX^{(\lambda)}(t) - n)^+ = \ell\big)= \Prob\big(B(\ell,p_i) =q_i \big). \label{eq:rig3}
% \end{align}
Since \eqref{eq:rig3} holds for all $t \geq 0$, it holds in stationarity as well.

We now say a few words about how to construct $U^{(\lambda)}$ and
argue \eqref{eq:rig1}--\eqref{eq:rig3}. One would start with four
primitive sequences: a sequence of inter-arrival times, potential
service times, patience times, and routing decisions. The sequence of
potential service times would hold all the service information about
each customer provided they were patient enough to get into
service. The routing sequence would represent the phase each customer
is assigned upon entering the system.

To see why \eqref{eq:rig1} is true, we first observe that at any time
$t> 0$, the random variable $A(t)$ depends only on the inter-arrival
time primitives; in particular, it is independent of the routing
  sequence $\{\xi_j, j\ge 1\}$. Second, any customer to arrive after
customer number $B(t)=b$ has no impact on any of the servers at
any point in time during $[0,t]$. In particular, the primitives
including $\{\xi_{b+j}, j\ge 1\}$ associated to those customers
are independent of $B(t)=b$ and $Z^{(\lambda)}(t)$. Lastly, the
decisions of those customers whether to abandon or not
by time $t$ depends only on their arrival times, patience times, and
the service history in the interval $[0,t]$. In particular,
  the sequence $\{\zeta_{b+j}(t), j\ge 1\}$ is independent of
  $\{\xi_{b+j}, j\ge 1\}$. This proves the 
the first equality in \eqref{eq:rig1}.

%  Since we assumed that the system is initially empty at $t = 0$, one can verify that
%  $(Q^{(\lambda)}_i(t),(e^TX^{(\lambda)}(t) - n)^+)$ has the same distribution as
%\begin{displaymath}
%\left( \sum_{j=1}^{(e^T X^{(\lambda)}(t) - n)^+} \xi_j, \quad (e^T X^{(\lambda)}(t) - n)^+ \right),
%\end{displaymath}
%where $\{\xi_j:j\ge 1\}$ is independent of $(e^T X^{(\lambda)}(t) - n)^+$.

% Applying \eqref{eq:rig1}, we see that $\xi_{B(t)+1}, \ldots, \xi_{B(t) + A(t)}$ are independent of the queue length at time $t$. Now if  $(e^TX^{(\lambda)}(t) - n)^+ = \ell > 0$, then exactly $\ell$ of the $\zeta_{B(t)+1}, \ldots, \zeta_{B(t) +A(t)}$ are equal to zero. Since $\xi_{B(t)+1}, \ldots, \xi_{B(t) + A(t)}$ are independent of  $\zeta_{B(t)+1}, \ldots, \zeta_{B(t) +A(t)}$, it doesn't matter which of the $\zeta$'s are zero, only how many are zero. Therefore, each element of $Q^{(\lambda)}(t)$ is the sum of $\ell$ Bernoulli random variables that are independent of everything else in the system.

We now move on to complete the proof of this lemma. We use \eqref{eq:rig3} to see that for any positive integer $N$, 
\begin{eqnarray}
&& \E\Big( [Q^{(\lambda)}_i(t) - p_i(e^TX^{(\lambda)}(t) - n)^+]^{2m}1_{\{ (e^T X^{(\lambda)}(t) - n)^+ \le N\}}\Big)\nonumber\\
% &=& \sum \limits_{\ell=1}^{N}\E \Big{[}\big(\sum_{j=1}^\ell \xi_j - p_i \ell\big)^{2m}\Big{|} (e^TX(t) - n)^+ = \ell\Big{]}\Prob((e^TX(t) - n) = \ell)\nonumber\\
&=& \sum \limits_{\ell=1}^{N}\E \Big{[}\big(B(\ell,p_i) - p_i \ell\big)^{2m}\Big{]}\Prob((e^TX^{(\lambda)}(t) - n) = \ell)\nonumber\\
&\leq&\sum \limits_{\ell=1}^{N}C(m,6) \ell^m\Prob((e^TX^{(\lambda)}(t) - n) = \ell)\nonumber\\
&=& C(m,6) \E \Big([(e^TX^{(\lambda)}(t) - n)^+]^m1_{\{ (e^T X^{(\lambda)}(t) - n)^+ \le N\}}\Big),\label{eq:sscproof}
\end{eqnarray}
where we have used the fact that there is a constant $C(m,6)>0$ such that
\begin{displaymath}
\E \Big{[}\big(B(\ell,p_i) - p_i \ell\big)^{2m}\Big{]} \le   
C(m,6) \ell^m \quad \text{ for all } \ell \ge 1;
\end{displaymath}
see, for example,  (4.10) of \cite{Knob2008}. Letting $t\to\infty$ in both sides of 
(\ref{eq:sscproof}), by the dominated convergence theorem, one has
\begin{eqnarray*}
&& \E\Big( [Q^{(\lambda)}_i(\infty) - p_i(e^TX^{(\lambda)}(\infty) - n)^+]^{2m}1_{\{ (e^T X^{(\lambda)}(\infty) - n)^+ \le N\}}\Big) \\
& &\le  C(m,6) \E \Big([(e^TX^{(\lambda)}(\infty) - n)^+]^m1_{\{ (e^T X^{(\lambda)}(\infty) - n)^+ \le N\}}\Big).
\end{eqnarray*}
Letting $N\to\infty$, by the monotone convergence theorem, one has 
\begin{eqnarray*}
\E (Q^{(\lambda)}_i(\infty) - p_i(e^TX^{(\lambda)}(\infty) - n)^+)^{2m} \leq C(m,6) \E \big[(e^TX^{(\lambda)}(\infty) - n)^+\big]^m.
\end{eqnarray*}
Then \eqref{eq:sscscaled} follows from this inequality for each $i$ and the fact that there is a constant $B_m>0$ such that  $\abs{x}^{2m}\le B_m \sum_{i=1}^d (x_i)^{2m}$ for all $x\in \R^d$. One can check that \eqref{eq:multinom} can be obtained by an argument very similar to the one used to prove \eqref{eq:sscscaled}.
\end{proof}

\begin{proof}[Proof of Lemma~\ref{lemma:SSC_bound}]
Recall that 
\begin{displaymath}
Z^{(\lambda)}(\infty) = X^{(\lambda)}(\infty) - Q^{(\lambda)}(\infty)
\end{displaymath} 
is the vector representing the number of customers of each type in service in steady-state. Then from \eqref{eq:interm_ssc} we have
\begin{eqnarray*}
&&\E \bigg[\partial_i f_h(\tilde X^{(\lambda)}(\infty))\big( \delta Q^{(\lambda)}_i(\infty) - p_i(e^T \tilde X^{(\lambda)}(\infty))^+\big) \bigg]\\
 &=& \sum_{k=1}^{d}\E \bigg[\partial_{ik} f_h(\xi)\big( \delta Q^{(\lambda)}_k(\infty) - p_k(e^T \tilde X^{(\lambda)}(\infty))^+\big)\big( \delta Q^{(\lambda)}_i(\infty) - p_i(e^T \tilde X^{(\lambda)}(\infty))^+\big)\bigg] \\
 &&+\ \E \bigg[\partial_i f_h\Big(\delta (Z^{(\lambda)}(\infty) - \gamma n) + p(e^T \tilde X^{(\lambda)}(\infty))^+\Big)\big( \delta Q^{(\lambda)}_i(\infty) - p_i(e^T \tilde X^{(\lambda)}(\infty))^+\big)\bigg].
\end{eqnarray*}
By Lemma~\ref{lemma:SSC}, the second expected value equals zero. For the first term, one can use the Cauchy-Schwarz inequality, together with the gradient bound \eqref{eq:gradbound3} and the SSC result \eqref{eq:sscscaled} to see that for all $1 \leq i,k \leq d$,
\begin{eqnarray*}
&&\E \bigg[\partial_{ik} f_h(\xi)\big( \delta Q^{(\lambda)}_k(\infty) - p_k(e^T \tilde X^{(\lambda)}(\infty))^+\big)\big( \delta Q^{(\lambda)}_i(\infty) - p_i(e^T \tilde X^{(\lambda)}(\infty))^+\big)\bigg] \\
&\leq & \bigg(\E \Big[\big(\partial_{ik} f_h(\xi)\big)^2\Big]\bigg)^{1/2} \bigg(\E \bigg[\Big( \delta Q^{(\lambda)}_k(\infty) - p_k(e^T \tilde X^{(\lambda)}(\infty))^+\Big)^4\bigg]\bigg)^{1/4} \\
&&\hspace{4cm} \times \bigg( \E \bigg[\Big( \delta Q^{(\lambda)}_i(\infty) - p_i(e^T \tilde X^{(\lambda)}(\infty))^+\Big)^4	\bigg]\bigg)^{1/4} \\
&\leq & \delta C(2, 3)\E \big[(e^T \tilde X^{(\lambda)}(\infty))^+\big]^2\sqrt{\E \Big[\big(\partial_{ik} f_h(\xi)\big)^2\Big]} \\
&\leq & \delta C(2, 3)\E \big[(e^T \tilde X^{(\lambda)}(\infty))^+\big]^2C(m,1)\sqrt{\E \Big[(1+\abs{\xi}^2)^2(1+\abs{\xi})^4\Big]}.
\end{eqnarray*}
We now combine everything together with the fact that $\xi$ satisfies \eqref{eq:xi_ineq} to conclude that there exists a constant $C(m,4)$ that does not depend on $\lambda$ or $n$, such that
\begin{eqnarray*}
&&\abs{\sum \limits_{i=1}^d \partial_i \E \bigg[f_h(\tilde X^{(\lambda)}(\infty))\Big{[}(\nu_i - \alpha-\sum \limits_{j=1}^d P_{ji}\nu_j)( \delta Q^{(\lambda)}_i(\infty) - p_i(e^T \tilde X^{(\lambda)}(\infty))^+)\Big{]}\bigg]} \notag \\
&\leq & \delta C(m, 4)\E \big[(e^T \tilde X^{(\lambda)}(\infty))^+\big]^2\sqrt{\E \Big[1 + \abs{\tilde X^{(\lambda)}(\infty)}^8\Big]},
\end{eqnarray*}
which concludes the proof of the lemma.
\end{proof}

\section{Proof of Theorem~\ref{thm:main}} \label{sec:thmpf}

To prove Theorem~\ref{thm:main}, we need an additional lemma on
uniform bounds for moments of scaled system size.  It will be proved
in Section~\ref{app:CTMCmomentsproof}.
\begin{lemma}
\label{lemma:CTMCmoments}
For any integer $m \geq 0$, there exists a constant  $C(m, 5)>0$ (depending only on $(\beta,\alpha,p,\nu,P)$) such that
\begin{equation} \label{eq:CTMCunifmombound}
\E \abs{\tilde X^{(\lambda)}(\infty)}^{m} \leq C(m, 5).
\end{equation}
\end{lemma}
\noindent We remark that in the special case when the service time distribution
is taken to be hyper-exponential, it is proved in \cite{GamaStol2012} that
\begin{displaymath}
\limsup \limits_{\lambda \rightarrow \infty} \E \exp\Big(\theta \abs{\tilde X^{(\lambda)} (\infty)}\Big) < \infty
\end{displaymath}
for $\theta$ in a neighborhood around zero. The proof relies on a result that allows one to compare the
system with an infinite-server system, whose stationary distribution is  known to be Poisson.

% Note that
% \begin{displaymath}
% \tilde Q_i(\infty) \leq (e^T \tilde X(\infty))^+ \leq C \abs{\tilde X(\infty)}.
% \end{displaymath}

\begin{proof}[Proof of Theorem~\ref{thm:main}]
  
It follows from Lemmas~\ref{lemma:diffbound} and \ref{lemma:SSC_bound} that

\begin{eqnarray}
&&\abs{\E h(\tilde X^{(\lambda)}(\infty)) - \E h(Y(\infty))} =
\abs{\E G_{U^{(\lambda)}} Af_h(U^{(\lambda)}(\infty)) - \E G_Y f_h(\tilde X^{(\lambda)}(\infty))}  \nonumber \\
&& \leq \abs{\sum \limits_{i=1}^d  \E \bigg[\partial_i f_h(\tilde X^{(\lambda)}(\infty))\Big{[}(\nu_i - \alpha-\sum \limits_{j=1}^d P_{ji}\nu_j)( \delta Q^{(\lambda)}_i(\infty) - p_i(e^T \tilde X^{(\lambda)}(\infty))^+)\Big{]}\bigg]} \nonumber \\
&& \qquad \qquad +  \delta C(m,2)  \E \Big{[}(1+\abs{\tilde X^{(\lambda)}(\infty)}^2)^m(1 + \abs{\tilde X^{(\lambda)}(\infty)})^4\Big{]} \nonumber \\
&&\qquad  \leq \delta C(m, 4)\E \Big[\big((e^T \tilde X^{(\lambda)}(\infty))^+\big)^2\Big]\sqrt{\E \Big[1 + \abs{\tilde X^{(\lambda)}(\infty)}^8\Big]}\nonumber \\
&& \qquad \qquad +  \delta C(m,2)  \E \Big{[}(1+\abs{\tilde X^{(\lambda)}(\infty)}^2)^m(1 + \abs{\tilde X^{(\lambda)}(\infty)})^4\Big{]}
.\label{eq:proofbound}
\end{eqnarray}
By Lemma~\ref{lemma:CTMCmoments}, there are constants $B_1(m), B_2(m)>0$ (depending only on $(\beta,\alpha,p,\nu,P)$) such that 
\begin{eqnarray*}
&& \E \Big[\big((e^T \tilde X^{(\lambda)}(\infty))^+\big)^2\Big]\sqrt{\E \Big[1 + \abs{\tilde X^{(\lambda)}(\infty)}^8\Big]} \le B_1(m), \\
&& \E \Big{[}(1+\abs{\tilde X^{(\lambda)}(\infty)}^2)^m(1 + \abs{\tilde X^{(\lambda)}(\infty)})^4\Big{]}\le B_2(m).
\end{eqnarray*}
Therefore, the right side of (\ref{eq:proofbound}) is less than or equal to 
\begin{eqnarray*}
\lefteqn{ \delta C(m,4) B_1(m) + \delta C(m,2) B_2(m) }\\
&\leq&  \Big(C(m,4) B_1(m) +  C(m,2) B_2(m) \Bigr) \frac{1}{\sqrt{\lambda}} \quad \text{ for }\lambda > 0.
\end{eqnarray*}
%where the first inequality follows from (\ref{eq:sscscaled}) and the last inequality is obtained by considering separately the cases when $\lambda \geq 1$ and $\lambda < 1$. By Remark~\ref{rem:lambdarange}, the latter case occurs finitely many times, meaning that the last inequality holds with redefined constants (depending on $\beta$ from (\ref{eq:square-root})).
 This concludes the proof of Theorem~\ref{thm:main}.
\end{proof}

\section{State Dependent Diffusion Coefficient} \label{sec:PHstatedep}
In Chapter~\ref{chap:dsquare}, we showed that using a state-dependent diffusion coefficient yields a much better approximation for the Erlang-C model. In this section we explore the use of a state-dependent diffusion coefficient for the $M/Ph/n+M$ model. We perform a numerical study, as the multi-dimensional nature of the $M/Ph/n+M$ model makes it difficult to prove any rigorous bounds.

To understand which diffusion approximation to use, we first group the terms on the right hand side of \eqref{eq:GXtaylor} by partial derivatives to see that 
\begin{align}
G_{U^{(\lambda)}} Af(u) \approx &\  \sum \limits_{i=1}^d \partial_i f(x)\delta \big(\lambda p_i - \alpha q_i - \nu_i z_i  \big) \notag \\
&+  \sum \limits_{i=1}^d \partial_{ii} f(x)\frac{1}{2}\delta^2 \big( \lambda p_{i}+ \alpha q_i  + \nu_i z_i + \sum_{j=1}^{d} P_{ji} \nu_j z_j \big) \notag \\
&- \sum_{i=1}^{d} \sum_{j=1}^{d} \partial_{ij} f(x) \delta^2 \nu_i z_i P_{ij}, \quad u \in \mathcal{U}, \label{PH:sdgenpart1}
\end{align}
where $z$, $q$  and $x$ are defined in \eqref{eq:projlittle}. We wish to replace $z$ and $q$ by functions of $x$. We know that 
\begin{align*}
x_i = \delta(z_i + q_i - \gamma_i n).
\end{align*}
 Lemma~\ref{lemma:SSC} tells us to use the approximation  
\begin{align}
\delta q_i \approx p_i(e^{T} x)^+ = p_{i} \big(\sum_{i}^{d} x_i\big)^+, \label{PH:qi}
\end{align} 
which suggests that 
\begin{align*}
z_i =  \frac{1}{\delta} (x_i - q_i) + \gamma_i n \approx  \frac{1}{\delta} \big(x_i - p_i(e^{T} x)^+\big) + \gamma_i n.
\end{align*}
The state space of the CTMC makes it so $z_i$ can never be negative, i.e.\ the number of customers in service is never negative. Therefore, 
\begin{align}
z_i = (z_i)^+  \approx  \Big( \frac{1}{\delta} \big(x_i - p_i(e^{T} x)^+\big) + \gamma_i n\Big)^+. \label{PH:zi}
\end{align}
We apply  \eqref{PH:qi} and \eqref{PH:zi} to \eqref{PH:sdgenpart1} to arrive at the diffusion approximation with generator 
\begin{align}
G f(x) =&\  \sum \limits_{i=1}^d \partial_i f(x) \Big(\delta\lambda p_i - \alpha p_i(e^{T} x)^+ - \nu_i \big(  x_i - p_i(e^{T} x)^+ + \delta\gamma_i n\big)^+  \Big)  \notag \\
&+ \sum \limits_{i=1}^d \partial_{ii} f(x)\frac{1}{2}\Big( \delta^2 \lambda p_{i}+ \delta\alpha p_i(e^{T} x)^+  + \delta\nu_i \big(  x_i - p_i(e^{T} x)^+ + \delta\gamma_i n\big)^+  \notag \\
& \hspace{3cm} + \delta\sum_{j=1}^{d} P_{ji} \nu_j\big(  x_j - p_j(e^{T} x)^+ + \delta\gamma_j n\big)^+ \Big) \notag \\
&- \sum_{i=1}^{d} \sum_{j=1}^{d} \partial_{ij} f(x) \delta \nu_i  P_{ij}\big(  x_i - p_i(e^{T} x)^+ + \delta\gamma_i n\big)^+. \label{PH:statedep}
\end{align}
Comparing the generator in \eqref{PH:statedep} to $G_Y$ in \eqref{eq:DMgen}, we see that the  coefficients of the second derivatives are state-dependent in the former, but constant in the latter. Although we are not guaranteed that the diffusion process with generator given by \eqref{PH:statedep} is positive recurrent, we assume it is, and use a modified version of the finite element algorithm in \cite{DaiHe2013} to compute its stationary distribution.

In the rest of this section, we will be interested in approximating the steady-state total customer count in the system. For convenience, we define
\begin{align}
T := X_1^{(\lambda)}(\infty) + X_2^{(\lambda)}(\infty) \quad \text{ and } \quad \tilde T := \tilde X_1^{(\lambda)}(\infty) + \tilde X_2^{(\lambda)}(\infty) = \delta(T-n), \label{PH:t}
\end{align}
where $\tilde X^{(\lambda)}(\infty)$ is defined in \eqref{eq:Xinfty}.
 We set 
\begin{align}
W := Y_1(\infty) + Y_2(\infty), \label{PH:w}
\end{align}
where $Y(\infty)$ has the steady-state distribution of the diffusion process with generator $G_Y$. The random variable $W$ is the constant diffusion coefficient approximation to $\tilde T$. Analogously to \eqref{PH:w}, we let $W_S$ be the approximation to $\tilde T$ based on the diffusion process with generator in \eqref{PH:statedep}. The code used in the following numerical study is publicly available at \url{https://github.com/anton0824/mphnplusm}.
%, we observe the following nuance. The state space of the CTMC makes it so $z_i$ can never be negative, i.e.\ the number of customers in service is never negative. The restriction that $z_i \geq 0$ (and $q_i \geq 0$) imposes a restriction that $x_i \geq -\delta \gamma_i n$. However, the diffusion model is defined on $\R^d$, and does not care about any restrictions on $x_i$. 
  
\subsection{$M/C_2/n+M$ Model -- No State Space Collapse}
We first focus on the special case of the $M/C_2/n+M$ model. The $C_2$ stands for a $2$-phase Coxian distribution. The corresponding tuple of parameters is $(p,\nu,P)$, where
\begin{displaymath}
p = (1,0)^T, \quad \nu = (\nu_1, \nu_2)^T, \text{ \quad and \quad } P = \begin{pmatrix}
0 & P_{12}\\ 
0 & 0
\end{pmatrix}.
\end{displaymath}
It can be checked that in this case, $1/\mu = 1/\nu_1 + P_{12}/\nu_1$, $\gamma_1 = \mu/\nu_1$, and $\gamma_2 = \mu P_{12}/\nu_2$. 
All customers start out in phase 1, and after completing that phase they move on to phase 2 with probability $P_{12}$, or leave system with probability $1-P_{12}$. Choosing the parameters $\nu_1, \nu_2, P_{12}$ is often done by first choosing the desired mean $1/\mu$ and 
squared coefficient of  variation $c_s^2$; the squared coefficient of variation of a random variable $Z$ equals $\text{Var}(Z)/(\E(Z))^2$. After choosing $1/\mu$ and $c_s^2$, we then set $\nu_1 = 2\mu, P_{12}=1/(2c_s^2), \nu_2 = P_{12}\nu_1$. In the following example, we choose $\mu = 1$ and $c_s^2 = 24$. 

The algorithm of \cite{DaiHe2013} that we use to compute the density of $W$ and $W_S$ require choosing a reference density, truncation rectangle, and a mesh resolution. To generate Table~\ref{tab:PHstatedep}, and Figures~\ref{fig:PHdens} and \ref{fig:PHrelerr}, we used a truncation rectangle of $[-10, 35] \times [-10, 35]$, and a lattice mesh in which all finite elements are $0.5 \times 0.5$ squares. The reference density used is similar to (3.21) and (3.23) of \cite{DaiHe2013}, but with one exception. With a $C_2$ service time distribution, any customer in the buffer must be a type-$1$ customer, and therefore type-$2$ customers never abandon the system. Therefore, using the notation of \cite{DaiHe2013}, we choose 
\begin{align*}
r_2(z) = \exp\bigg( - \frac{z}{\mu(c_a^2 + c_s^2)}  - \frac{\gamma_j^2 \beta^2}{1 + c_a^2}\bigg), \quad \text{ for } z \geq 0 . 
\end{align*}

Since all customers start out in phase $1$ of service, the $M/C_2/n+M$ model can be represented by a $2$-dimensional CTMC. Namely, $\{(X_1(t), X_2(t)), t \geq 0\}$ is a CTMC. This fact is important, because the diffusion approximation is also $2$-dimensional, and no SSC is required. This means that \eqref{PH:qi} and \eqref{PH:zi} are actually equalities, not just approximations, and that the diffusion generator completely captures the first and second derivative terms of the Taylor expansion in \eqref{eq:GXtaylor}. We observed in Chapter~\ref{chap:dsquare} that capturing the first and second derivative terms in the generator of the Erlang-C model gave us faster convergence rates. By similar logic, we expect the approximation in \eqref{PH:statedep} to have a faster convergence rate of $1/\lambda$ as opposed to $1/\sqrt{\lambda}$. Table~\ref{tab:PHstatedep} is consistent with this expectation, and shows that when approximating $\E|\tilde T|$, the errors from using $W$ and $W_S$ shrink at rates $1/\sqrt{\lambda}$ and $1/\lambda$, respectively. Similar results were observed for higher moments of $\tilde T$ as well. 

\begin{table}[h]
  \begin{center}
  \resizebox{\columnwidth}{!}{
   \begin{tabular}{rc|cc|cc }
$n$ & $\E |\tilde T|$ & $\big| \EE |W| - \EE |\tilde T|\big|$  &Relative Error & $\big| \EE |W_S| - \EE |\tilde T|\big|$& Relative Error\\
%& & &$(Y_0(\infty))$&\multicolumn{1}{c}{} & $(Y(\infty))$\\
\hline
 15        &   0.900  &$2.07\times 10^{-2}$  & 2.29\% & $2.31\times 10^{-3}$ & 0.26\% \\
 30        &  0.907  & $1.40\times 10^{-2}$  & 1.54\% & $1.16\times 10^{-3}$ & 0.13\% \\
 60        &  0.912  &$9.55\times 10^{-3}$ & 1.05\% & $5.68\times 10^{-4}$ & 0.06\% \\
 125        &  0.915  & $6.43\times 10^{-3}$ & 0.70\% & $2.49\times 10^{-4}$ & 0.03\% \\
 250        & 0.917  & $4.45\times 10^{-3}$& 0.49\% & $9.71\times 10^{-5}$ & 0.01\% \\
 500        & 0.918  & $3.09\times 10^{-3}$& 0.34\% & $1.95\times 10^{-5}$ & 0.002\% \\
 1000        & 0.919  & $2.15\times 10^{-3}$& 0.23\% & $2.03\times 10^{-5}$ & 0.002\% \\
  \end{tabular}}
  \end{center}
  \caption{Approximation error of $\E | \tilde T|$. In each row, $\mu = 1$ and $\lambda = n$. The approximation with the state-dependent diffusion coefficient outperforms the one with constant diffusion coefficient. We see that as $\lambda$ doubles, the error of $\E |W|$ decreases by a factor of $\sqrt{2}$, while the error of $\E |W_S|$ decreases by a factor of $2$. When $n=250,500$, and $1000$, the value of $\E|W_S|$ is so close to $\E | \tilde T|$ that the approximation error reported in the table is due to the numerical error in the finite element  algorithm used to compute $\E|W_S|$. \label{tab:PHstatedep}}
\end{table} 
Another criterion by which we evaluate the diffusion approximations is how well they approximate the probability mass function (pmf) of $T$, the unscaled total customer count. Figure~\ref{fig:PHdens} contains plots the pmf of $T$ together with the constant and state-dependent coefficient approximations. We see that the benefit of the latter approximation is more pronounced for the smaller-sized system. We refer the reader to Figure~\ref{fig:PHrelerr}, which plots the relative error of approximating $\Prob(T \geq k)$. We see from that figure that when approximating tail events, e.g.\ when $\Prob(T \geq k) \leq 0.05$,  the state-dependent coefficient approximation performs significantly better.

\begin{figure}[h]\centering
   \begin{minipage}{0.4\textwidth}
   \hspace{-1.5cm}
    \includegraphics[clip, trim=.5cm 6.5cm 0.5cm 7cm, width=75mm,keepaspectratio]{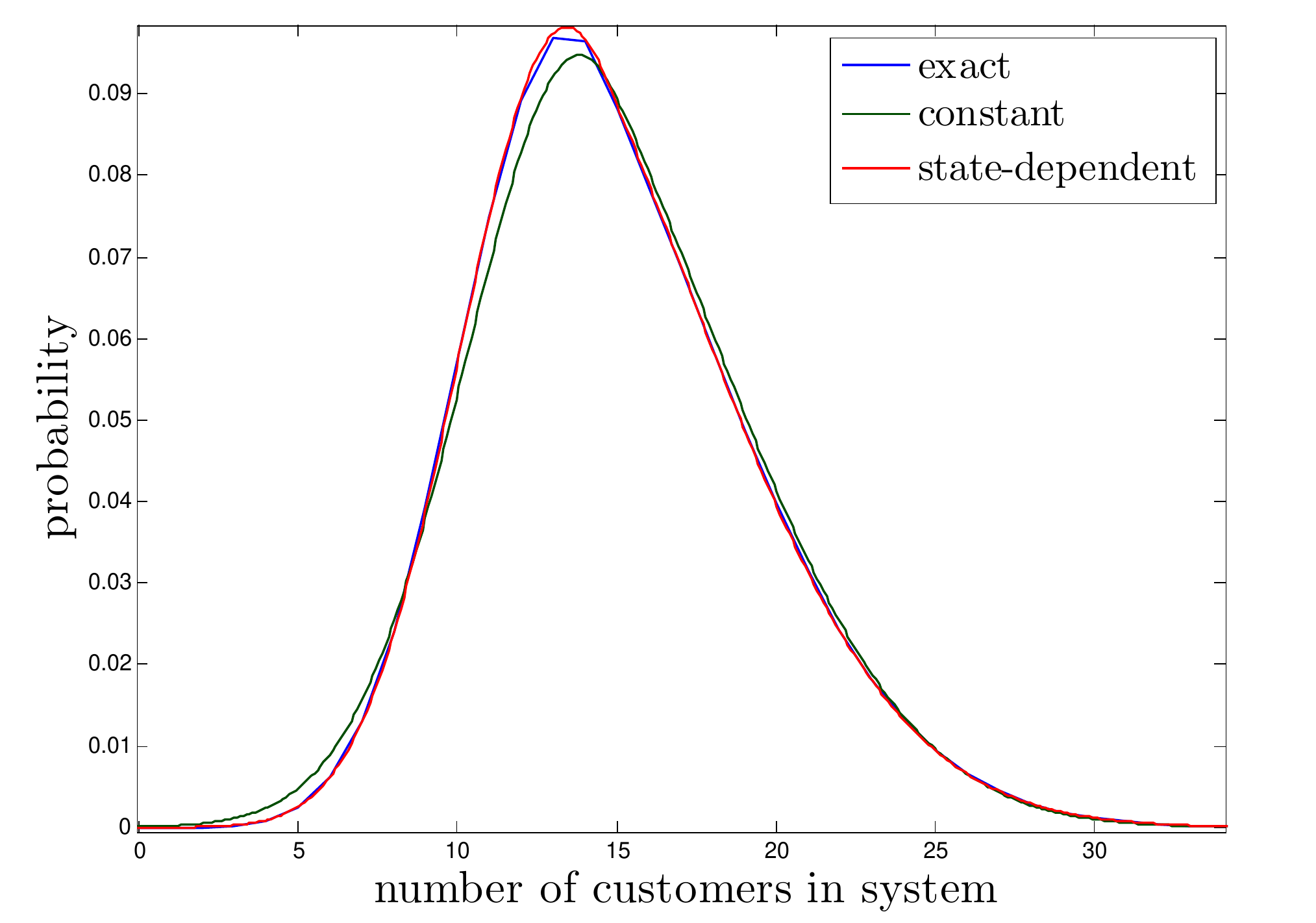}
   \end{minipage}
\begin{minipage}{0.4\textwidth}
    \includegraphics[clip, trim=0.5cm 6.5cm 0.5cm 6.5cm, width=70mm,keepaspectratio]{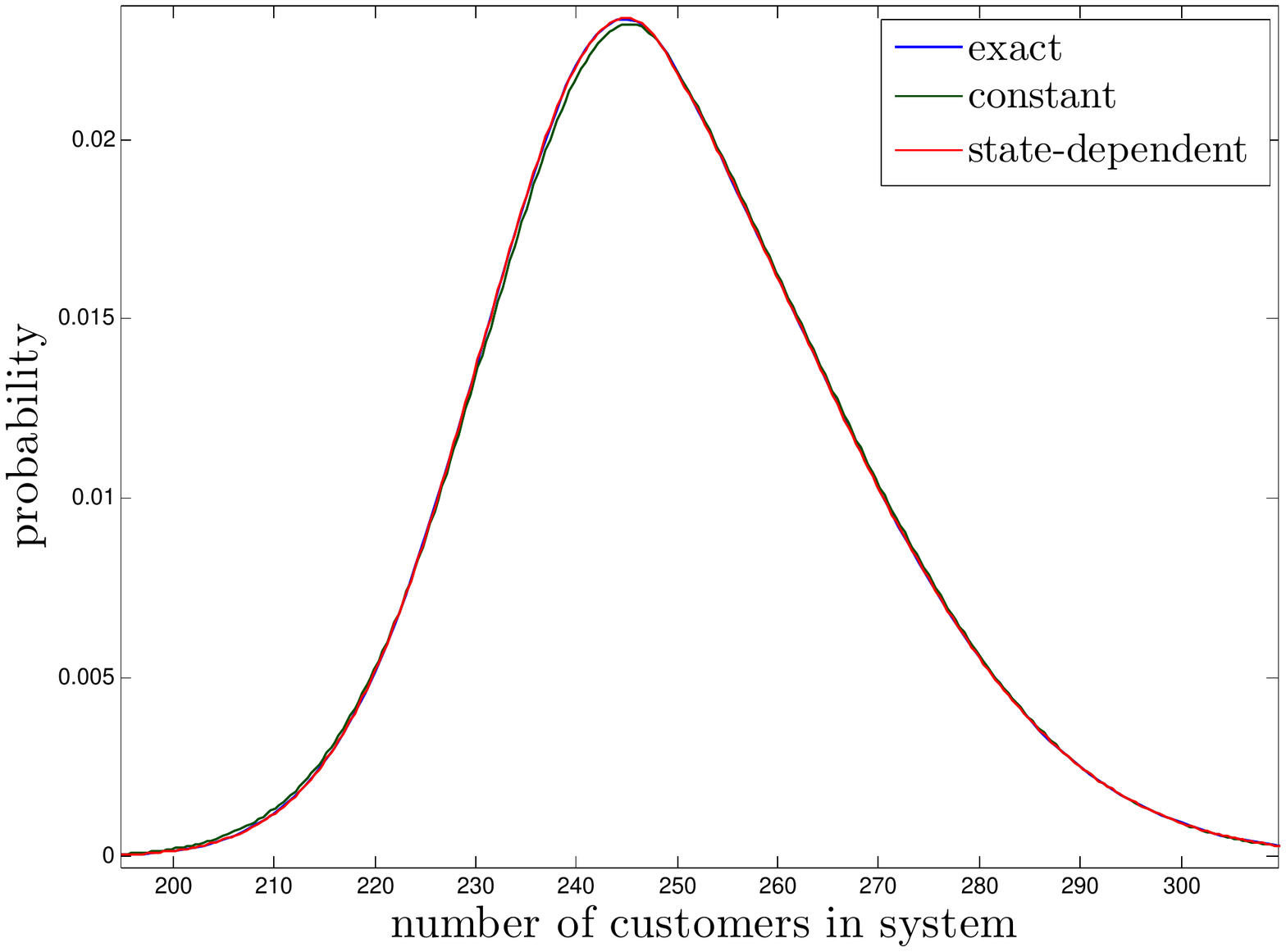}
   \end{minipage}
   \caption{In the plot on the left, $\lambda =n=15$, and $\mu=1$. In the plot on the right, $\lambda =n=250$, and $\mu=1$. In both plots, the blue line represents the probability mass function of $T$, i.e. $\Prob(T = k)$. The green and red lines plot the densities of the diffusion approximations with constant and state-dependent diffusion coefficients, respectively. The benefit of using a state-dependent diffusion coefficient is more pronounced for smaller-sized systems.}
   \label{fig:PHdens}
\end{figure}

% \begin{figure}[h]
%\centerline{\includegraphics[clip, trim=0.5cm 6.5cm 0.5cm 7cm, width=100mm,keepaspectratio]{mc2n15rho1}}
% \caption{$\lambda =n=15$, and $\mu=1$. The blue line plots the probability mass function of $T$, i.e. $\Prob(T = k)$. The green and red lines plot the densities of the diffusion approximations with constant and state-dependent diffusion coefficients, respectively.  \label{fig:n15}}
%  \end{figure}
%   \begin{figure}[h]
%\centerline{\includegraphics[clip, trim=0.5cm 6.5cm 0.5cm 6.5cm, width=100mm,keepaspectratio]{mc2n250rho1}}
% \caption{ \label{fig:n250}}
%  \end{figure}

\begin{figure}[h!]\centering
   \begin{minipage}{0.5\textwidth}
   \hspace{-1.5cm}
    \includegraphics[clip, trim=1.2cm 5.5cm 0cm 5cm, width=100mm,keepaspectratio]{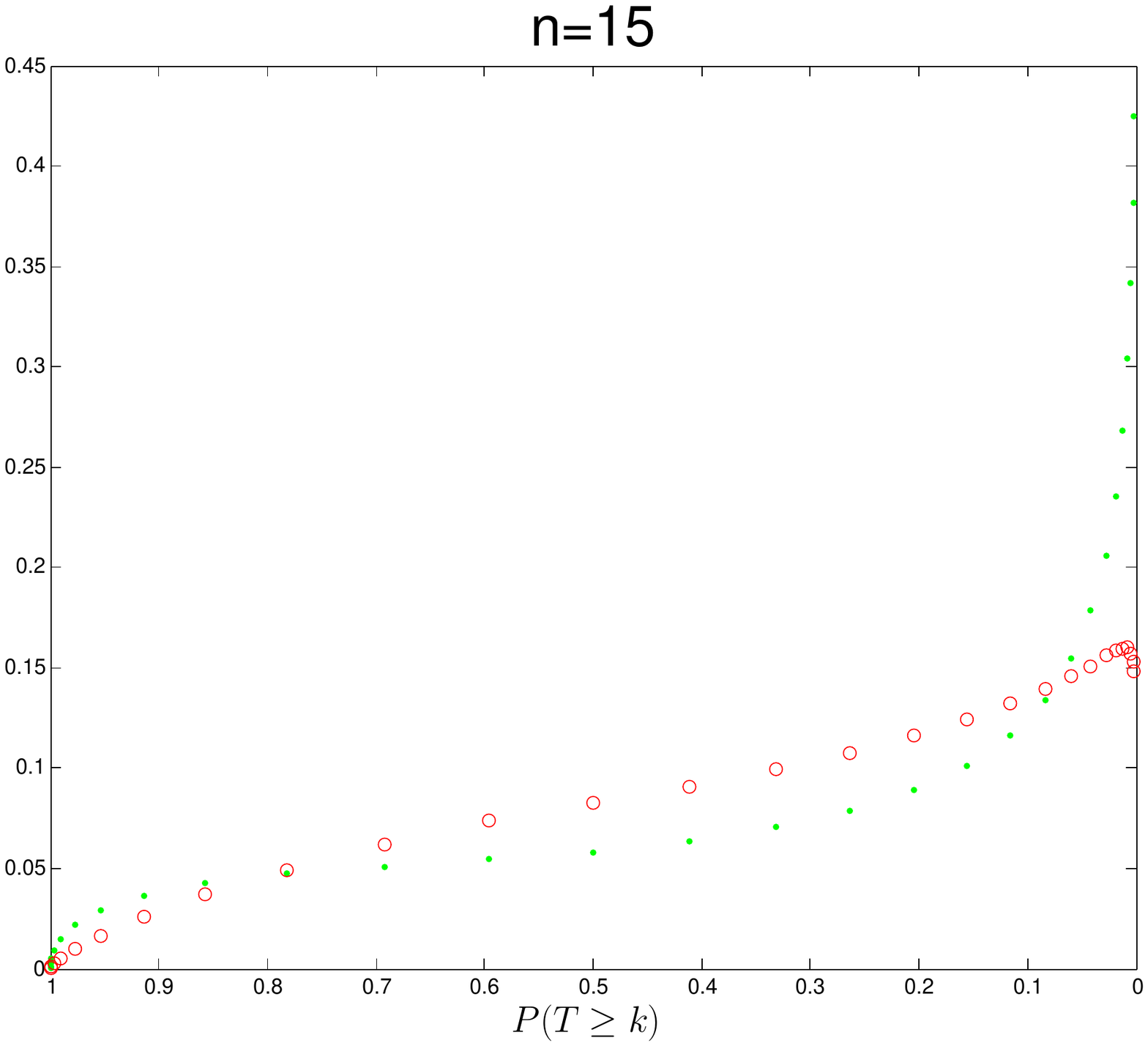}
   \end{minipage}
\begin{minipage}{0.5\textwidth}
    \hspace{-1.5cm} \includegraphics[clip, trim=.8cm 5.5cm 0.5cm 4cm, width=100mm,keepaspectratio]{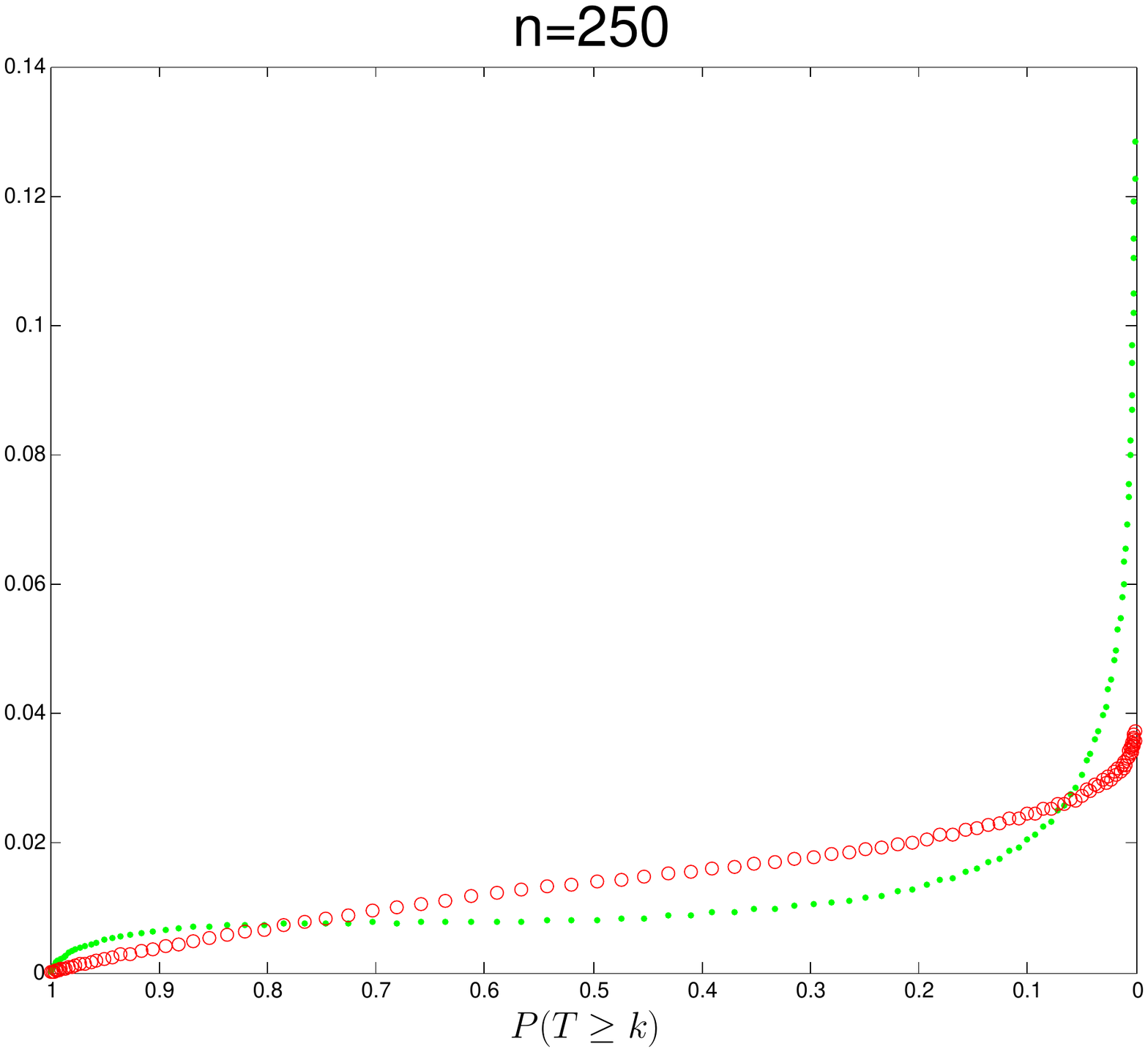}
   \end{minipage}
   \caption{In the top plot, $\lambda=n=15$, and $\mu = 1$. In the bottom plot, $\lambda=n=250$, and $\mu = 1$.  Both plots show the relative error of approximating the tail CDF $\Prob(T \geq k)$ by the two diffusion approximations. The x-axis contains $\Prob(T \geq k)$ as $k$ ranges from zero to some large value. The green dots correspond to $\frac{\Prob(W/\delta + n\geq k) - \Prob(T \geq k)}{\Prob(T\geq k)}$, the relative error of the constant diffusion coefficient approximation. The red circles correspond to $\frac{\Prob(W_S/\delta + n\geq k) - \Prob(T \geq k)}{\Prob(T\geq k)}$, the relative error from the state-dependent coefficient approximation. Observe that the state-dependent coefficient approximation is a much better choice for approximating tail events, e.g. events when $\Prob(T \geq k) \leq 0.05$.  \label{fig:PHrelerr}}
\end{figure}
\clearpage
\subsection{$M/H_2/n+M$ Model}
We now focus on the $M/H_2/n+M$ model, where the $H_2$ stands for a $2$-phase hyper-exponential distribution. The corresponding tuple of parameters $(p,\nu,P)$ is 
\begin{displaymath}
p = (p_1, p_2)^T, \quad \nu = (\nu_1,\nu_2)^T, \text{ \quad and \quad } P = 0.
\end{displaymath}
The starting service phase of each customer is random, and unlike how it was with the Coxian distribution, the process $\{X_1(t), X_2(t), t\geq 0\}$ is not a CTMC. In particular, this means that the approximation in \eqref{PH:qi} has non-zero approximation error. As a result, even though we use a state-dependent diffusion coefficient, we are unable to fully capture the first and second derivative terms in the Taylor expansion of $G_{U^{(\lambda)}}$. We also have no reason to expect faster convergence rates because the error terms corresponding to the first derivatives are a  bottleneck of order $1/\sqrt{\lambda}$. Figures~\ref{fig:PHdensH2} and \ref{fig:PHabserrH2} compare the two diffusion approximations for a system with $100$ servers. Due to the approximation error in \eqref{PH:qi}, using a state-dependent diffusion coefficient does not give us the improved accuracy we are accustomed to. In fact, we cannot conclude which approximation is better. 

To generate Figures~\ref{fig:PHdensH2} and \ref{fig:PHabserrH2}, we used the same reference density as in (3.21) and (3.23) of \cite{DaiHe2013}, a truncation rectangle of  $[-15, 40]\times [-15, 40]$, and a lattice mesh in which all finite elements are $0.5 \times 0.5$ squares; see \cite{DaiHe2013} for more details.

\begin{figure}[h]\centering
    \includegraphics[clip, trim=0.5cm 6.5cm 0.5cm 6cm, width=90mm,keepaspectratio]{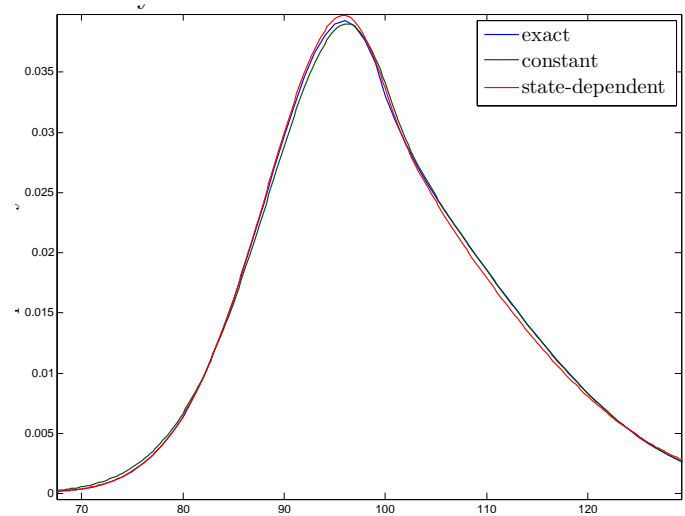}
   \caption{$2$-phase hyper-exponential distribution with $\lambda =n=100$, and $\mu=1$. The blue line represents the probability mass function of $T$, i.e. $\Prob(T = k)$. The green and red lines plot the densities of the diffusion approximations with constant and state-dependent diffusion coefficients, respectively. Using a state-dependent diffusion coefficient does not add any benefit.}
   \label{fig:PHdensH2}
\end{figure}

\begin{figure}[h!]\centering
    \hspace{-1.5cm} \includegraphics[clip, trim=.8cm 5.5cm 0.5cm 4cm, width=90mm,keepaspectratio]{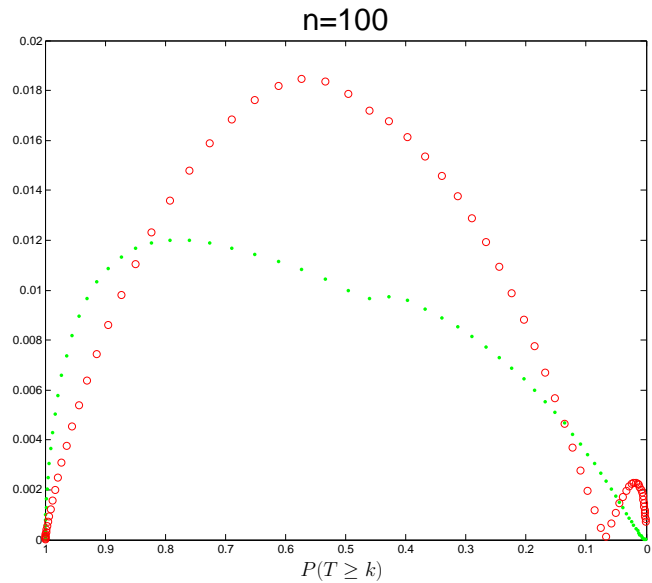}
   \caption{$2$-phase hyper-exponential distribution with $\lambda =n=100$, and $\mu=1$.  Both plots show the absolute error of approximating the tail CDF $\Prob(T \geq k)$ by the two diffusion approximations. The x-axis contains $\Prob(T \geq k)$ as $k$ ranges from zero to some large value. The green dots correspond to $\abs{\Prob(W/\delta + n\geq k) - \Prob(T \geq k)}$, the error of the constant diffusion coefficient approximation. The red circles correspond to $\abs{\Prob(W_S/\delta + n\geq k) - \Prob(T \geq k)}$, the error from the state-dependent coefficient approximation. \label{fig:PHabserrH2}}
\end{figure}

\clearpage

\section{Chapter Appendix} \label{app:proofs}

\subsection{Proof of Lemma~\ref{lemma:diffbound} (Generator Difference)} \label{app:taylorexp}
The main idea here is that $G_Y f_h(x)$ is
hidden within $G_{U^{(\lambda)}} Af_h(u)$, where the lifting operator $A$ is in (\ref{eq:lifter}). We algebraically manipulate the Taylor expansion of $G_{U^{(\lambda)}} Af_h(u)$ to make this evident. First, we first rearrange the terms in the Taylor expansion (\ref{eq:GXtaylor}) to group them by partial derivatives. Thus, $G_{U^{(\lambda)}} Af_h(u)$ equals 
\begin{eqnarray*}
&&\sum \limits_{i=1}^d \delta \partial_i f_h(x)\Big{[}p_i \lambda - \alpha q_i - \nu_i z_i + \sum \limits_{j=1}^d P_{ji}\nu_j z_j\Big{]} \\
&&+ \sum \limits_{i=1}^d \frac{\delta^2}{2}\partial_{ii}f_h(x)\Big{[}p_i \lambda + \alpha q_i + \nu_i z_i + \sum \limits_{j=1}^d P_{ji}\nu_j z_j\Big{]} - \sum \limits_{i \neq j}^d \delta^2\partial_{ij}f_h(x)\big{[}P_{ij}\nu_i z_i \big{]} \\
&&+ \sum \limits_{i=1}^d \frac{\delta^2}{2} \Big(\partial_{ii}f_h(\xi_{i}^-)-\partial_{ii}f_h(x)\Big)\Big{[}\alpha q_i + (1-\sum \limits_{j=1}^d P_{ij})\nu_i z_i\Big{]} \\
&&+ \sum \limits_{i=1}^d \frac{\delta^2}{2} \Big(\partial_{ii}f_h(\xi_{i}^+)-\partial_{ii}f_h(x)\Big)\big{[}\lambda p_i \big{]}- \sum \limits_{i \neq j}^d \delta^2 \Big(\partial_{ij}f_h(\xi_{ij})-\partial_{ij}f_h(x)\Big)\big{[}P_{ij}\nu_i z_i\big{]}\\
&&+ \sum \limits_{i=1}^d \sum \limits_{j=1}^d \frac{\delta^2}{2} \Big(\partial_{ii}f_h(\xi_{ij})-\partial_{ii}f_h(x)\Big)\Big{[} P_{ij}\nu_i z_i +  P_{ji}\nu_j z_j\Big{]} .
\end{eqnarray*}
To proceed we observe that (\ref{eq:defph}) gives us the identity
\begin{equation}\label{eq:taylorexpidentity}
-\nu_i \gamma_i n + \sum \limits_{j=1}^d P_{ji}\nu_j\gamma_j n  = -n p_i. %= -n(I-P^T)_i diag(\nu)\gamma
\end{equation}
Recall the form of $G_Y f_h(x)$ from (\ref{eq:DMgen}). From the form of $\Sigma$ in (\ref{eq:OUdiffusioncoeff}), we see that
\begin{equation} \label{eq:OUaltdiffusioncoef}
\Sigma_{ii} = 2\Big(p_i + \sum \limits_{j=1}^d P_{ji}\gamma_j \nu_j\Big), \quad \Sigma_{ij} = -(P_{ij}\nu_i \gamma_i + P_{ji}\nu_j\gamma_j) \text{ for $j\neq i$}
\end{equation}
% By (\ref{eq:defph}),
%\begin{displaymath}
%\Sigma_{ii} = 2(p_i + \sum \limits_{j=1}^d P_{ji}\gamma_j \nu_j) = p_i + \sum \limits_{j=1}^d P_{ji}\gamma_j \nu_j  +\gamma_i\nu_i.
%\end{displaymath}
 using ($\ref{eq:DMgen}$), ($\ref{eq:taylorexpidentity}$) and ($\ref{eq:OUaltdiffusioncoef}$), the difference $G_{U^{(\lambda)}} Af_h(u) - G_Y f_h(x)$ becomes

\begin{eqnarray}
&&\sum \limits_{i=1}^d \partial_i f_h(x)\Big{[}(\nu_i - \alpha-\sum \limits_{j=1}^d P_{ji}\nu_j)( \delta q_i - p_i(e^T x)^+)\Big{]}  \label{eq:errorterm} \\
&&+ \sum \limits_{i=1}^d \partial_{ii}f_h(x)\Big{[}\sum \limits_{j=1}^d P_{ji}\nu_j\gamma_j \Big{]}(n \delta^2 - 1) - \sum \limits_{i \neq j}^d \partial_{ij}f_h(x)\Big{[}P_{ij}\nu_i\gamma_i + P_{ji} \nu_j\gamma_j\Big{]}(n\delta^2 - 1) \notag \\
&&- \sum \limits_{i=1}^d \frac{\delta^2}{2}\partial_{ii}f_h(x)\Big{[}p_i (\lambda - n)- \alpha  q_i - \nu_i (z_i- \gamma_i n) - \sum \limits_{j=1}^d P_{ji}\nu_j (z_j- \gamma_j n)\Big{]} \notag \\
&&- \sum \limits_{i \neq j}^d \frac{\delta^2}{2}\partial_{ij}f_h(x)\Big{[}P_{ij}\nu_i (z_i- \gamma_i n) + P_{ji} \nu_j (z_j- \gamma_j n)\Big{]} \notag \\
&&+ \sum \limits_{i=1}^d \frac{\delta^2}{2} (\partial_{ii}f_h(\xi_{i}^-)-\partial_{ii}f_h(x))\Big{[}\alpha q_i + (1-\sum \limits_{j=1}^d P_{ij})\nu_i z_i\Big{]}\notag \\
&&  + \sum \limits_{i=1}^d \frac{\delta^2}{2} (\partial_{ii}f_h(\xi_{i}^+)-\partial_{ii}f_h(x))\Big{[}\lambda p_i \Big{]}- \sum \limits_{i \neq j}^d \delta^2 (\partial_{ij}f_h(\xi_{ij})-\partial_{ij}f_h(x))\Big{[}P_{ij}\nu_i z_i\Big{]} \notag \\
&&+ \sum \limits_{i=1}^d \sum \limits_{j=1}^d \frac{\delta^2}{2} (\partial_{ii}f_h(\xi_{ij})-\partial_{ii}f_h(x))\Big{[} P_{ij}\nu_i z_i +  P_{ji}\nu_j z_j\Big{]} .  \notag 
\end{eqnarray}
We remind the reader that our target is to prove that
  \begin{eqnarray*}
&&G_{U^{(\lambda)}} Af_h(u) - G_Y f_h(x) \nonumber\\
&=&  \sum \limits_{i=1}^d \partial_i f_h(x)\Big{[}(\nu_i - \alpha-\sum \limits_{j=1}^d P_{ji}\nu_j)( \delta q_i - p_i(e^T x)^+)\Big{]} + E(u),
  \end{eqnarray*}
  where $E(u)$ is an error term that satisfies 
\begin{displaymath}
\abs{E(u)} \leq \delta\, C(m,2) (1+\abs{x}^2)^m (1+\abs{x})^4.
\end{displaymath}
%
%  \begin{eqnarray*}
% \abs{G_{U^{(\lambda)}} Af_h(u) -
%G_Y f_h(x)}  &\le&  \sum \limits_{i=1}^d \partial_i f_h(x)\Big{[}(\nu_i - \alpha-\sum \limits_{j=1}^d P_{ji}\nu_j)( \delta q_i - p_i(e^T x)^+)\Big{]} \nonumber \\
%&& {}+ \delta\, C(m,2) (1+\abs{x}^2)^m (1+\abs{x})^4.
%  \end{eqnarray*}
We choose $E(u)$ to be all the terms in \eqref{eq:errorterm} except for the first line. We now describe how to bound $\abs{E(u)}$.
Most of the summands in (\ref{eq:errorterm}) look as follows: a term in large square brackets multiplied by some partial derivative of $f_h$. The partial derivatives are very easy to bound; we simply use (\ref{eq:gradbound2}) - (\ref{eq:gradbound4}). We wish to point out that $\xi_{i}^+$, $\xi_{i}^-$ and $\xi_{ij}$ lie within distance $2\delta$ of $x$. When $2\delta < 1$, (\ref{eq:gradbound4}) implies 
\begin{equation} \label{eq:lg4cond}
\abs{\partial_{ij}f_h(\xi)-\partial_{ij}f_h(x)} \leq 2\delta C (1+ \abs{x}^2)^m (1+ \abs{x})^3
\end{equation} 
for some constant $C>0$ (i.e.\ an extra $\delta$ term is gained). When $2\delta \geq 1$ (by Remark~\ref{rem:lambdarange} this occurs in finitely many cases), we may use (\ref{eq:gradbound3}) to obtain (\ref{eq:lg4cond}) with a redefined $C$. From here on out, we shall let $C>0$ be a generic positive constant that will change from line to line, but will always be independent of $\lambda$ and $n$.

Now we shall list the facts needed to bound all the square bracket terms in (\ref{eq:errorterm}) except for the very first one. Recall that we are operating in the Halfin-Whitt regime as defined by \eqref{eq:square-root}. Therefore, 
\begin{displaymath}
(n\delta^2 - 1) = \delta \beta \text{ and } \delta (\lambda - n) = -\beta.
\end{displaymath}
 Furthermore, it must be true that 
\begin{displaymath}
 \delta q_i \leq (e^T x)^+ \leq C \abs{x}, \label{eq:qibound}
\end{displaymath}
as the number of phase $i$ customers may never exceed the total queue size. Next, 
\begin{displaymath}
\abs{\delta(z_i - \gamma_i n)} = \abs{x_i - \delta q_i} \leq C \abs{x}
\end{displaymath}
and lastly, 
\begin{displaymath}
\abs{\delta^2 z_i} \leq \abs{\delta^2 \gamma_i n} + \abs{\delta^2 (z_i - \gamma_i n)} \leq C(1+ \abs{x}).
\end{displaymath}
It is now a simple matter to verify that the inequalities above, combined with the bounds on the partials of $f_h$ are all that it takes to achieve our desired upper bound.

\begin{addacknowledgements}
The author thanks Jim Dai, Jiekun Feng, Shuangchi He, Josh Reed and John Pike for stimulating discussions. He also thanks the participants of Applied Probability
\& Risk Seminar in Fall 2014 at Columbia University for their feedback
on this research, and the participants of the 2015 Workshop on New Directions in Stein's Method held at the Institute for Mathematical Sciences at the National University of Singapore and  they would like to thank the financial support from the Institute. This research is supported in part by NSF Grants CNS-1248117, CMMI-1335724, and
CMMI-1537795.
\end{addacknowledgements}

\appendix

\chapter{Moment Bounds} \label{app:MOMENTS}
This appendix proves all of the moment bounds used in this document. Bounds for Chapters~\ref{chap:erlangAC}, \ref{chap:dsquare} and \ref{chap:phasetype} are proved in Sections~\ref{app:CAmoment}, \ref{app:DSmoment}, and \ref{app:CTMCmomentsproof}, respectively. 

\section{Chapter~\ref{chap:erlangAC} Moment Bounds }
\label{app:CAmoment}
We first prove Lemma~\ref{lem:moment_bounds_C} in Section~\ref{app:momCproof}, establishing the moment bounds for Erlang-C model.
In Section~\ref{app:mom_A}, we  prove Lemma~\ref{lem:moment_bounds_A_under},  establishing the moment bounds for Erlang-A model.
\subsection{Erlang-C Moment Bounds}
\label{app:momCproof}
\begin{proof}[Proof of Lemma~\ref{lem:moment_bounds_C}]
We first prove \eqref{CW:xsquaredelta}, \eqref{CW:xminusdelta}, and \eqref{CW:xplus}.
Recalling the generator $G_{\tilde X}$ defined in \eqref{eq:GX}, we apply it to the function $V(x) = x^2$ to see that for $k \in \Z_+$ and $x = x_k = \delta(k - x(\infty))$, 
\begin{align}
G_{\tilde X} V(x) =&\ \lambda( 2x\delta + \delta^2) + \mu (k \wedge n)(-2x\delta + \delta^2) \notag \\
=&\ 2x\delta(\lambda - n\mu  + \mu (k - n)^-) + \mu + \delta^2 \mu (k \wedge n) \notag \\
=&\ 2x\mu ( \zeta + (x+\zeta)^-) +\mu  + \delta^2\mu (n - \frac{\lambda}{\mu} + \frac{\lambda}{\mu} - (k - n)^-) \notag \\
=&\ 2x\mu ( \zeta + (x+\zeta)^-) + \mu  -\delta \mu \zeta + \mu  - \delta\mu (x+\zeta)^- \notag \\
=&\ 1(x \leq -\zeta)\mu  \big(-2x^2 + \delta x \big) + 1(x > -\zeta) \mu \big(2x \zeta -\delta \zeta  \big) + 2\mu \notag  \\
\leq&\ 1(x \leq -\zeta) \mu \big(-\frac{3}{2}x^2 + \frac{\delta^2}{2} \big) + 1(x > -\zeta)\mu  \big(2x \zeta -\delta \zeta  \big) + 2\mu. \label{eq:gv1}
\end{align}
Instead of splitting the last two lines into the cases  $x \leq -\zeta$ and $x > -\zeta$, we could have also considered  $x < -\zeta$ and $x \geq -\zeta$ instead, and would have obtained
\begin{align}
G_{\tilde X} V(x)=&\ 1(x < -\zeta)\mu  \big(-2x^2 + \delta x \big) + 1(x \geq -\zeta) \mu \big(2x \zeta -\delta \zeta  \big) + 2\mu  \notag \\
\leq&\ 1(x < -\zeta) \mu \big(-\frac{3}{2}x^2 + \frac{\delta^2}{2} \big) + 1(x \geq -\zeta)\mu  \big(2x \zeta -\delta \zeta  \big) + 2\mu. \label{eq:gv2}
\end{align}
We take expected values on both sides of \eqref{eq:gv1} with respect to $\tilde X(\infty)$, and apply Lemma~\ref{lem:gz} to see that
\begin{align}
0 \leq& -\frac{3}{2}\mu \E \big[(\tilde X(\infty))^2 1(\tilde X(\infty) \leq -\zeta)\big] \notag \\
&+ \mu \abs{\zeta} \E \big[\big(-2\tilde X(\infty) +\delta \big)1(\tilde X(\infty) > -\zeta)  \big] + 2\mu  + \frac{\mu \delta^2}{2}.\label{eq:momineq1}
\end{align}
This implies that when $\abs{\zeta} > \delta/2$,
\begin{align*}
0 \leq& -\frac{3}{2}\mu \E \big[(\tilde X(\infty))^2 1(\tilde X(\infty) \leq -\zeta)\big] + 2\mu  + \frac{\mu \delta^2}{2},
\end{align*}
and when $\abs{\zeta} \leq \delta/2$,
\begin{align*}
0 \leq& -\frac{3}{2}\mu \E \big[(\tilde X(\infty))^2 1(\tilde X(\infty) \leq -\zeta)\big]+  2\mu  + \mu \delta^2.
\end{align*}
Therefore, 
\begin{align*}
&\E \big[(\tilde X(\infty))^2 1(\tilde X(\infty) \leq -\zeta)\big] \leq \frac{4}{3} + \frac{2\delta^2}{3}, 
\end{align*}
which proves \eqref{CW:xsquaredelta}. Jensen's inequality immediately gives us
\begin{align*}
\E \Big[\big|\tilde X(\infty)  1(\tilde X(\infty) \leq -\zeta)\big|\Big] \leq \sqrt{\E \big[(\tilde X(\infty))^2 1(\tilde X(\infty) \leq -\zeta)\big]},
\end{align*}
which proves \eqref{CW:xminusdelta}. Furthermore, \eqref{eq:momineq1} also gives us 
\begin{align*}
\E \Big[\big|\tilde X(\infty)1(\tilde X(\infty) > -\zeta)\big| \Big] \leq \frac{1}{\abs{\zeta}} + \frac{\delta^2}{4\abs{\zeta}} + \frac{\delta}{2},
\end{align*}
which is not quite \eqref{CW:xplus} because the inequality above has $1(\tilde X(\infty) > -\zeta)$ as opposed to $1(\tilde X(\infty) \geq -\zeta)$ as in \eqref{CW:xplus}. However, we can use \eqref{eq:gv2} to get the stronger bound 
\begin{align*}
\E \Big[\big|\tilde X(\infty)1(\tilde X(\infty) \geq -\zeta)\big| \Big] \leq \frac{1}{\abs{\zeta}} + \frac{\delta^2}{4\abs{\zeta}} + \frac{\delta}{2},
\end{align*}
which proves \eqref{CW:xplus}. We now prove \eqref{CW:xminuszeta}, or 
\begin{align} \label{eq:xminusproofpart1}
\E \Big[\big|\tilde X(\infty)  1(\tilde X(\infty) \leq -\zeta)\big|\Big] \leq 2\abs{\zeta}.
\end{align}
We use the triangle inequality to see that
\begin{align*}
\E \Big[\big|\tilde X(\infty)  1(\tilde X(\infty) \leq -\zeta) \big|\Big] \leq&\ \abs{\zeta} +  \E \Big[\big|\tilde X(\infty) +\zeta \big| 1(\tilde X(\infty) \leq -\zeta) \Big].
\end{align*}
The second term on the right hand side is just the expected number of idle servers, scaled by $\delta$. We now show that this expected value equals $\abs{\zeta}$. Applying the generator $G_{\tilde X}$ to the test function $f(x) = x$, one sees that for all $k \in \Z_+$ and $x = x_k = \delta(k-x(\infty))$, 
\begin{align*}
G_{\tilde X} f(x) = \delta\lambda - \delta\mu (k \wedge n) = \mu \big[\zeta + (x + \zeta)^-\big].
\end{align*}
Taking expected values with respect to $\tilde X(\infty)$ on both sides, and applying Lemma~\ref{lem:gz}, 
we arrive at
\begin{align} \label{eq:idle_expect}
\E \Big[\big|(\tilde X(\infty) +\zeta) 1(\tilde X(\infty) \leq -\zeta)\big| \Big] = \abs{\zeta},
\end{align}
which proves \eqref{CW:xminuszeta}.

We move on to prove \eqref{CW:idle_prob}, or 
\begin{align} \label{eq:inlineidleprob}
\Prob(\tilde X(\infty) \leq -\zeta) \leq (2+\delta)\abs{\zeta}.
\end{align}
Let $I$ be the unscaled expected number of idle servers. Then by \eqref{eq:idle_expect},
\begin{align*}
I = \E(X(\infty) - n)^- =  \frac{1}{\delta}\E \Big[\big|(\tilde X(\infty) +\zeta) 1(\tilde X(\infty) \leq -\zeta)\big| \Big]   = \frac{1}{\delta} \abs{\zeta}.
\end{align*}
Now let $\{\pi_k\}_{k=0}^{\infty}$ be the distribution of $X(\infty)$. We want to prove an upper bound on the probability
\begin{align*}
\Prob(\tilde X(\infty) \leq -\zeta) = \sum_{k=0}^{n} \pi_k \leq \sum_{k=0}^{\lfloor n - \sqrt{R} \rfloor} \pi_k +  \sum_{k= \lceil n - \sqrt{R}\rceil}^{n} \pi_k.
\end{align*}
Observe that 
\begin{align*}
I = \sum_{k=0}^{n} (n-k) \pi_k \geq \sqrt{R}\sum_{k=0}^{\lfloor n - \sqrt{R} \rfloor}  \pi_k.
\end{align*}
Now let $k^*$ be the first index that maximizes $\{\pi_k\}_{k=0}^{\infty}$, i.e.
\begin{align*}
k^* = \inf \{k \geq 0 : \pi_k \geq \nu_j, \text{ for all $j \neq k$}\}.
\end{align*}
Then
\begin{align} 
\Prob(\tilde X(\infty) \leq -\zeta) = \sum_{k=0}^{\lfloor n - \sqrt{R} \rfloor} \pi_k +  \sum_{k= \lceil n - \sqrt{R} \rceil}^{n} \pi_k \leq &\ \frac{I}{\sqrt{R}} +  (\sqrt{R}+1) \pi_{k^*}  \notag\\
=&\ \abs{\zeta} + (\sqrt{R}+1) \pi_{k^*}.\label{eq:prob_interm}
\end{align}
Applying $G_{\tilde X}$ to the test function $f(x) = (k \wedge k^*)$, we see that for all $k \in \Z_+$ and $x = x_k = \delta(k - x(\infty))$,
\begin{align*}
G_{\tilde X} f(x) = \delta \lambda 1(k < k^*) - \delta \mu(k \wedge n) 1(k \leq k^*).
\end{align*}
Taking expected values with respect to $X(\infty)$ on both sides and applying Lemma~\ref{lem:gz}, we see that
\begin{align*}
\Prob(X(\infty) \leq k^*) = \frac{\mu}{n\mu-\lambda} \E \big[(X(\infty)-n)^-1(X(\infty) \leq k^*) \big] - \pi_{k^*} \frac{\lambda}{n\mu-\lambda} \geq 0.
\end{align*}
Using the inequality above, together with the fact that $k^* \leq n$, we see that 
\begin{align*}
\pi_{k^*} \leq&\ \frac{\mu}{\lambda} \E \big[(X(\infty)-n)^-1(X(\infty) \leq k^*) \big] \\
\leq&\ \frac{\mu}{\lambda} \E \big[(X(\infty)-n)^-1(X(\infty) \leq n) \big] = \frac{I}{R} = \frac{\abs{\zeta}}{\sqrt{R}}.
\end{align*}
The fact that $k^* \leq n$ is a consequence of $\lambda < n\mu$, and can be verified through the flow balance equations of the CTMC X. We combine the bound above with \eqref{eq:prob_interm} to arrive at \eqref{CW:idle_prob}, which concludes the proof of this lemma.

\end{proof}

\subsection{Erlang-A Moment Bounds}
\label{app:mom_A}
Recall Lemma~\ref{lem:moment_bounds_A_under} stated in Section~\ref{sec:CAapplemmas}. We outline the proof of it below.

\subsubsection{Proof Outline for Lemma~\ref{lem:moment_bounds_A_under}: The Underloaded System}\label{app:momboundA_under}
The proof of the underloaded case of Lemma~\ref{lem:moment_bounds_A_under} is very similar to that of Lemma~\ref{lem:moment_bounds_C}. Therefore, we only outline some key intermediate steps needed to obtain the results. We remind the reader that when $R \leq n$, then $\zeta \leq 0$. We first show how to establish \eqref{eq:mwuK1}, which is proved in a similar fashion to \eqref{CW:xsquaredelta} of Lemma~\ref{lem:moment_bounds_C} -- by applying the generator $G_{\tilde X}$ to the Lyapunov function $V(x) = x^2$. The following are some useful intermediate steps for any reader wishing to produce a complete proof. The first step to prove \eqref{eq:mwuK1} is to get an analogue of \eqref{eq:gv1}. Namely, when $x \leq -\zeta$, 
\begin{align*}
G_{\tilde X} V(x) =&\ -2\mu x^2 + \mu\delta x + 2\mu \leq -\frac{3}{2}\mu x^2 + \mu \delta^2/2 + 2\mu ,
\end{align*}
and when $x \geq -\zeta$,
\begin{align}
G_{\tilde X} V(x) =&\ -2\alpha (x+\zeta)^2 + \alpha \delta(x+\zeta)  -2\mu \abs{\zeta}(x+\zeta) \notag  \\
&- 2\abs{\zeta}\alpha(x+\zeta) + \mu\abs{\zeta}(\delta - 2\abs{\zeta}) + 2\mu \notag \\
\leq &\ -\frac{3}{2} \alpha (x+\zeta)^2  -2\mu \abs{\zeta}(x+\zeta) + \delta^2\alpha/2 + \delta^2\mu/8 + 2\mu. \label{eq:genmwu1}
\end{align}
From here, we use Lemma~\ref{lem:gz} to get a statement similar to \eqref{eq:momineq1}, from which we can infer \eqref{eq:mwuK1} and by applying Jensen's inequality to \eqref{eq:mwuK1}, we get \eqref{eq:mwu1}. Observe that this procedure yields \eqref{eq:mwuK2}, \eqref{eq:mwu4}, and \eqref{eq:mwu5} as well. We now describe how to prove \eqref{eq:mwu3}, which requires only a slight modification of \eqref{eq:genmwu1}. Namely, for $x \geq -\zeta$,
\begin{align*}
G_{\tilde X} V(x) =&\ 2x\big(-\alpha (x +\zeta) + \mu \zeta \big) - \delta \big(-\alpha (x +\zeta) + \mu \zeta \big) + 2\mu.
\end{align*}
From this, we can deduce that since $x \geq -\zeta$, 
\begin{align*}
G_{\tilde X} V(x) \leq -2(\mu \wedge \alpha) x^2 - \delta \big(-\alpha (x +\zeta) + \mu \zeta \big) + 2\mu,
\end{align*}
and also
\begin{align*}
G_{\tilde X} V(x) \leq -2\mu \abs{\zeta} x - \delta \big(-\alpha (x +\zeta) + \mu \zeta \big) + 2\mu.
\end{align*}
Then Lemma~\ref{lem:gz} can be applied as before to see that both 
\begin{align}
2\mu \abs{\zeta}\E \Big[ \big|\tilde X(\infty)1(\tilde X(\infty)\geq -\zeta)\big|\Big] \text{ and }  2(\mu \wedge \alpha)\E \Big[ \big(\tilde X(\infty)\big)^2 1(\tilde X(\infty)\geq -\zeta)\Big] \label{eq:whatbounded}
\end{align}
are bounded by
\begin{align*}
2\mu + \mu \delta^2/2 - \delta \E \Big[ \big(-\alpha (\tilde X(\infty) +\zeta) + \mu \zeta \big)1(\tilde X(\infty) \geq -\zeta)\Big].
\end{align*}
Applying the generator $G_{\tilde X}$ to the test function $f(x) = x$ and taking expected values with respect to $\tilde X(\infty)$, we get $\E b(\tilde X(\infty)) = 0$, or
\begin{align} \label{eq:zerodrift}
\E \Big[ \big(-\alpha (\tilde X(\infty) +\zeta) + \mu \zeta \big)1(\tilde X(\infty) \geq -\zeta)\Big] = \mu \E \Big[\tilde X(\infty) 1(\tilde X(\infty) < -\zeta)\Big].
\end{align}
When combined with \eqref{eq:mwu1}, this implies that 
\begin{align*}
&\ 2\mu + \mu \delta^2/2 - \delta \E \Big[ \big(-\alpha (\tilde X(\infty) +\zeta) + \mu \zeta \big)1(\tilde X(\infty) \geq -\zeta)\Big] \\
\leq&\ 2\mu + \mu \delta^2/2 +\mu \delta\sqrt{\frac{1}{3}\Big(\frac{\alpha}{\mu }\delta^2 + \delta^2 + 4 \Big)},
\end{align*}
which proves \eqref{eq:mwu3}, because the quantity above is an upper bound for \eqref{eq:whatbounded}. To prove \eqref{eq:mwu2}, we manipulate \eqref{eq:zerodrift} to get 
\begin{align*} 
\E \Big[\big| (\tilde X(\infty) + \zeta) 1(\tilde X(\infty) \leq -\zeta)\big|\Big]  = \abs{\zeta} + \frac{\alpha}{\mu } \E \Big[\big| (\tilde X(\infty) + \zeta)  1(\tilde X(\infty) > -\zeta)\big|\Big],
\end{align*}
to which we can apply the triangle inequality and \eqref{eq:mwu4} to conclude \eqref{eq:mwu2}. Lastly, the proof of \eqref{eq:mwu6} is nearly identical to the proof of \eqref{CW:idle_prob} in Lemma~\ref{lem:moment_bounds_C}. The key step is to obtain an analogue of \eqref{eq:prob_interm}. 

\subsubsection{Proof Outline for Lemma~\ref{lem:moment_bounds_A_under}: The Overloaded System}\label{app:momboundA_over}
The proof of the overloaded case of Lemma~\ref{lem:moment_bounds_A_under} is also similar to that of Lemma~\ref{lem:moment_bounds_C}. Therefore, we only outline some key intermediate steps needed to obtain the results; the bounds in this lemma are not proved in the order in which they are stated. We remind the reader that when $R \geq n$, then $\zeta \geq 0$. We start by proving \eqref{eq:mwo2}. Although the left hand side of \eqref{eq:mwo2} is slightly different from \eqref{CW:xsquaredelta} of Lemma~\ref{lem:moment_bounds_C}, it is proved using the same approach -- by applying the generator $G_{\tilde X}$ to the Lyapunov function $V(x) = x^2$. The following are some useful intermediate steps for any reader wishing to produce a complete proof. The first step to prove \eqref{eq:mwo2} is to get analogue of \eqref{eq:gv1}. Namely, when  $x \leq -\zeta$,
\begin{align}
G_{\tilde X} V(x) =&\ -2\mu (x+\zeta)^2 + \mu \delta(x+\zeta) \notag \\
&+ 2(\mu +\alpha)\abs{\zeta}(x+\zeta) - 2\alpha \zeta^2 - \alpha \delta \zeta + 2\mu \notag \\
\leq &\ -2\mu (x+\zeta)^2  + 2(\mu +\alpha)\abs{\zeta}(x+\zeta)  + 2\mu, \label{eq:genmwo1}
\end{align} 
and when $x \geq -\zeta$, 
\begin{align*}
G_{\tilde X} V(x) =&\ -2\alpha x^2 + \alpha\delta x + 2\mu \leq -\frac{3}{2}\alpha x^2 + \alpha \delta^2/2 + 2\mu.
\end{align*}
From here, we use Lemma~\ref{lem:gz} to get a statement similar to \eqref{eq:momineq1}, which implies  \eqref{eq:mwo2}. Applying Jensen's inequality to \eqref{eq:mwo2} yields \eqref{eq:mwo1}. The procedure used to get \eqref{eq:mwo2} also yields \eqref{eq:mwo3}, \eqref{eq:mwoK1}, and \eqref{eq:mwo4}. 

We now describe how to prove \eqref{eq:mwo7} and \eqref{eq:mwo8}, which requires only a slight modification of \eqref{eq:genmwo1}. Namely, we use the fact that for $x \leq -\zeta$,
\begin{align*}
G_{\tilde X} V(x) =&\ 2x\big(-\mu (x +\zeta) + \alpha \zeta \big) - \delta \big(-\mu (x +\zeta) + \alpha \zeta \big) + 2\mu.
\end{align*}
From this, one can deduce that since $x \leq -\zeta$, 
\begin{align*}
G_{\tilde X} V(x) \leq -2(\mu \wedge \alpha) x^2+ 2\mu,
\end{align*}
and also
\begin{align*}
G_{\tilde X} V(x) \leq -2\alpha \abs{\zeta} \abs{x}  + 2\mu.
\end{align*}
Then Lemma~\ref{lem:gz} and Jensen's inequality can be applied as before to get both \eqref{eq:mwo7} and \eqref{eq:mwo8}.

We now prove \eqref{eq:mwo5}. Observe that 
\begin{align*}
&\E \Big[\big| \tilde X(\infty) 1(\tilde X(\infty) \geq -\zeta) \big| \Big] \\
=&\ \E \Big[\big| (\tilde X(\infty)+\zeta - \zeta )1(\tilde X(\infty) \geq -\zeta) \big| \Big] \\
\geq&\ \E \Big[\big| (\tilde X(\infty)+\zeta)1(\tilde X(\infty) > -\zeta) \big| - \zeta 1(\tilde X(\infty) > -\zeta) \Big] \\
\geq&\ \E \Big[\big| (\tilde X(\infty)+\zeta)1(\tilde X(\infty) > -\zeta) \big| \Big] - \zeta \\
=&\ \frac{\mu}{\alpha} \E \Big[\big| (\tilde X(\infty) + \zeta) 1(\tilde X(\infty) \leq -\zeta)\big|\Big],
\end{align*}
where the last equality comes from applying the generator $G_{\tilde X}$ to the function $f(x) = x$ and taking expected values with respect to $\tilde X(\infty)$ to see that $\E b(\tilde X(\infty)) = 0$, or
\begin{align} \label{eq:zerodrift2}
\E \Big[ \big(-\mu (\tilde X(\infty) +\zeta) + \alpha \zeta \big)1(\tilde X(\infty) \leq -\zeta)\Big] = \alpha \E \Big[\tilde X(\infty) 1(\tilde X(\infty) > -\zeta)\Big].
\end{align}
Therefore, 
\begin{align*}
\E \Big[\big| (\tilde X(\infty) + \zeta) 1(\tilde X(\infty) \leq -\zeta)\big|\Big] 
\leq \frac{\alpha}{\mu } \E \Big[\big| \tilde X(\infty) 1(\tilde X(\infty) \geq -\zeta) \big| \Big],
\end{align*}
and we can invoke  \eqref{eq:mwo1}  to conclude \eqref{eq:mwo5}.

We now prove \eqref{eq:mwo10}, which requires additional arguments that we have not used in the proof of Lemma~\ref{lem:moment_bounds_C}. We assume for now that 
\begin{align} \label{eq:temp_assumption}
\lambda \leq n\mu + \frac{1}{2}\sqrt n \mu.
\end{align} 
Fix $\gamma \in (0, 1/2)$, and define
\begin{align}
\label{eq:j12}
J_1=\sum_{k=0}^{\lfloor n-\gamma \sqrt R \rfloor}\pi_k, \quad J_2=\sum^n_{k=\lceil n-\gamma \sqrt R\rceil }\pi_k,
\end{align}
where $\{\pi_k\}_{k=0}^{\infty}$ is the distribution of $X(\infty)$. We note that by \eqref{eq:temp_assumption}, 
\begin{align*}
n/\sqrt{R} \geq \sqrt{R}-\frac{1}{2}\sqrt{n/R} \geq \sqrt{R}-1/2 \geq 1/2,
\end{align*}
which implies that $n-\gamma \sqrt R > 0$. Then 
\begin{align*}
\Prob(\tilde X(\infty) \leq -\zeta) = \Prob(X(\infty) \leq n) \leq J_1 + J_2.
\end{align*} 
To bound $J_1$ we observe that 
\begin{align*}
\E\left[\left|\widetilde X(\infty)+\zeta\right|1_{\{\widetilde X(\infty)\leq-\zeta\}}\right] = \frac{1}{\sqrt{R}} \sum_{k=0}^{n} (n-k) \pi_k \geq \gamma\sum_{k=0}^{\lfloor n-\gamma \sqrt R \rfloor}  \pi_k = \gamma J_1.
\end{align*}
Combining \eqref{eq:mwo3}--\eqref{eq:mwo5}, we conclude that
\begin{align}
J_1 \leq&\ \frac{1}{\gamma}\frac{2}{\sqrt 3}\Big(\frac{\delta^2}{4}+1\Big)\Big(\frac{1}{\zeta}\wedge \sqrt{\frac{\alpha}{\mu}\vee 1}\wedge \frac{\alpha}{\mu}
\sqrt{\frac{\mu}{\alpha}\vee 1}\Big)\nonumber \\
\leq&\ \frac{1}{\gamma}\frac{2}{\sqrt 3}\Big(\frac{\delta^2}{4}+1\Big)\Big(\frac{1}{\zeta}\wedge \sqrt{\frac{\alpha}{\mu}}\Big). \label{eq:J1}
\end{align}
Now to bound $J_2$, we apply $G_{\tilde X}$ to the test function $f(x) = k\wedge n$, where $x=\delta (k-x(\infty))$, and take the expectation with respect to $\tilde X(\infty)$ to see that 
\begin{gather*}
0 = -\lambda \pi_n+(\lambda - n\mu )\Prob(X(\infty)\leq n)+\mu \E \left[\left(X(\infty)-n\right)^-1_{\{X(\infty)\leq n\}}\right].
\end{gather*}
Noticing that 
\begin{align*}
\E \left[\left(X(\infty)-n\right)^-1_{\{X(\infty)\leq n\}}\right] = \frac{1}{\delta} \E\left[\left|\widetilde X(\infty)+\zeta\right|1_{\{\widetilde X(\infty)\leq-\zeta\}}\right],
\end{align*}
we arrive at
\begin{align}
\pi_n \leq  \delta \frac{2}{\sqrt 3}\Big(\frac{\delta^2}{4}+1\Big) \left(\frac{1}{\zeta}\wedge \sqrt{\frac{\alpha}{\mu}}\right)+\frac{\lambda -n\mu}{\lambda}\Prob(X(\infty)\leq n).\label{eq:pin}
\end{align}
The flow balance equations
\begin{gather*}
\lambda \pi_{k-1} = k\mu \pi_k,\quad k=1,2,\cdots,n
\end{gather*}
imply that $\pi_0<\pi_1<\cdots<\pi_{n-2}<\pi_{n-1}\leq \pi_n$, and therefore
\begin{align}
&\ J_2 \leq (\gamma \sqrt R + 1)\pi_n \notag  \\
\leq&\ (\gamma \sqrt R + 1)\Big[\delta \frac{2}{\sqrt 3}\Big(\frac{\delta^2}{4}+1\Big) \left(\frac{1}{\zeta}\wedge \sqrt{\frac{\alpha}{\mu}}\right)+\frac{\lambda -n\mu}{\lambda}\Prob(X(\infty)\leq n)\Big] \nonumber\\
=&\ (\gamma +\delta)\frac{2}{\sqrt 3}\Big(\frac{\delta^2}{4}+1\Big) \left(\frac{1}{\zeta}\wedge \sqrt{\frac{\alpha}{\mu}}\right)\notag\\
&+(\gamma \sqrt R + 1)\frac{\lambda -n\mu}{\lambda}J_1 +(\gamma \sqrt R + 1)\frac{\lambda -n\mu}{\lambda}J_2 \label{eq:J2}
\end{align}
We use \eqref{eq:temp_assumption}, the fact that $\gamma \in (0,1/2)$, and that $R \geq n \geq 1$ to see that
\begin{align*}
(\gamma \sqrt R + 1)\frac{\lambda -n\mu}{\lambda} \leq (\gamma \sqrt R + 1)\frac{\sqrt{n}}{2R} \leq \frac{1}{2}(\gamma + 1/\sqrt{R}) = \frac{1}{2}(\gamma + 1)< 3/4.
\end{align*} 
Then by rearranging terms in (\ref{eq:J2}) and applying \eqref{eq:J1} we conclude that
\begin{align*}
\frac{1}{4} J_2 \leq&\  (\gamma +\delta)\frac{2}{\sqrt 3}\Big(\frac{\delta^2}{4}+1\Big) \left(\frac{1}{\zeta}\wedge \sqrt{\frac{\alpha}{\mu}}\right)+\frac{3}{4} \frac{1}{\gamma}\frac{2}{\sqrt 3}\Big(\frac{\delta^2}{4}+1\Big)\Big(\frac{1}{\zeta}\wedge \sqrt{\frac{\alpha}{\mu}}\Big) \\
=&\ \Big(\gamma +\delta+\frac{3}{4} \frac{1}{\gamma} \Big) \frac{2}{\sqrt 3}\Big(\frac{\delta^2}{4}+1\Big) \left(\frac{1}{\zeta}\wedge \sqrt{\frac{\alpha}{\mu}}\right).
\end{align*}
Hence, we have just shown that under assumption \eqref{eq:temp_assumption}, 
\begin{align*}
\Prob(\tilde X(\infty)\leq -\zeta) \leq J_1 + J_2 \leq&\ \frac{1}{\gamma}\frac{2}{\sqrt 3}\Big(\frac{\delta^2}{4}+1\Big)\Big(\frac{1}{\zeta}\wedge \sqrt{\frac{\alpha}{\mu}}\Big) \notag \\
&+ 4\Big(\gamma +\delta+\frac{3}{4} \frac{1}{\gamma} \Big) \frac{2}{\sqrt 3}\Big(\frac{\delta^2}{4}+1\Big) \left(\frac{1}{\zeta}\wedge \sqrt{\frac{\alpha}{\mu}}\right) \notag \\
\leq&\ (3+\delta)\frac{8}{\sqrt 3}\Big(\frac{\delta^2}{4}+1\Big) \left(\frac{1}{\zeta}\wedge \sqrt{\frac{\alpha}{\mu}}\right),
\end{align*}
where to get the last inequality we fixed $\gamma \in (0,1/2)$ that solves $\gamma + 1/\gamma = 3$. 

We now wish to establish the same result without assumption \eqref{eq:temp_assumption}, i.e.\ when $\lambda > n\mu +\frac{1}{2} \sqrt{n}\mu$. For this, we rely on the following comparison result. Fix $n,\mu$ and $\alpha$ and let $X^{(\lambda)}(\infty)$ be the steady-state customer count in an Erlang-A system with arrival rate $\lambda$, service rate $\mu$, number of servers $n$, and abandonment rate $\alpha$. Then for any $0 < \lambda_1 < \lambda_2$, 
\begin{align}
\Prob(X^{(\lambda_2)}(\infty) \leq n) \leq \Prob(X^{(\lambda_1)}(\infty) \leq n). \label{eq:compare}
\end{align}
This says that with all other parameters being held fixed, an Erlang-A system with a higher arrival rate is less likely to have idle servers. For a simple proof involving a coupling argument, see page 163 of \cite{Lind1992}.

%Intuitively, this statement is obvious. To prove it, one can use the flow balance equations to see that
%\begin{align*}
%\Prob(X^{(\lambda)}(\infty) \leq n) =&\ \frac{\sum_{j=0}^{n} \frac{(\lambda/\mu)^j}{j!}}{\sum_{j=0}^{n} \frac{(\lambda/\mu)^j}{j!}+ \sum_{j=n+1}^{\infty}\frac{(\lambda/\mu)^n}{n!}\prod_{k=n+1}^{j} \frac{\lambda}{n\mu + (k-n)\alpha}} \\
%=&\ \frac{1}{1 + \frac{\sum_{j=n+1}^{\infty}\frac{(\lambda/\mu)^n}{n!}\prod_{k=n+1}^{j} \frac{\lambda}{n\mu + (k-n)\alpha}}{\sum_{j=0}^{n} \frac{(\lambda/\mu)^j}{j!}}}.
%\end{align*}
%Now 
%\begin{align*}
%\frac{\sum_{j=n+1}^{\infty}\frac{(\lambda/\mu)^n}{n!}\prod_{k=n+1}^{j} \frac{\lambda}{n\mu + (k-n)\alpha}}{\sum_{j=0}^{n} \frac{(\lambda/\mu)^j}{j!}} =&\ \frac{\sum_{j=n+1}^{\infty}\frac{(1/\mu)^n}{n!}\prod_{k=n+1}^{j} \frac{\lambda}{n\mu + (k-n)\alpha}}{\sum_{j=0}^{n}  \frac{1}{\lambda^{n-j}} \frac{(1/\mu)^j}{j!}},
%\end{align*}
%is an increasing function of $\lambda$, which implies \eqref{eq:compare}. 

Therefore, for $\lambda > n\mu +\frac{1}{2} \sqrt{n}\mu$,
\begin{align*}
 &\ \Prob(X^{(\lambda)}(\infty)\leq n) \leq \Prob(X^{(n\mu +\frac{1}{2} \sqrt{n}\mu )}(\infty)\leq n) \\
 \leq&\ (3+\delta)\frac{8}{\sqrt 3}\Big(\frac{\delta^2}{4}+1\Big) \left(\frac{1}{\zeta^{(n\mu +\frac{1}{2} \sqrt{n}\mu )}}\wedge \sqrt{\frac{\alpha}{\mu}}\right)
\end{align*}
where $\zeta^{(n\mu +\frac{1}{2} \sqrt{n}\mu )}$ is the $\zeta$ corresponding to $X^{(n\mu +\frac{1}{2} \sqrt{n}\mu )}(\infty)$, and satisfies 
\begin{align*}
\frac{1}{\zeta^{(n\mu +\frac{1}{2} \sqrt{n}\mu )}} =   \frac{2\alpha}{\mu }\sqrt{\frac{n + \sqrt{n}/2}{n} } \leq  \frac{2\alpha}{\mu }\sqrt{\frac{3}{2}}.
\end{align*}
This concludes the proof of \eqref{eq:mwo10}.

\section{Chapter~\ref{chap:dsquare} Moment Bounds}
\label{app:DSmoment}
In this section we prove Lemmas~\ref{lem:lastbounds} and \ref{lem:pibounds}. To do so, we rely on the moment bounds in Lemma~\ref{lem:moment_bounds_C} (from Section~\ref{sec:CAmomentbounds}). However, the bounds from that lemma are not sufficient, and the following additional bounds are needed.
\begin{lemma}
\label{lem:xtramom}
 For all $n \geq 1, \lambda > 0$, and $\mu > 0$ satisfying $0 < R < n $,
\begin{align}
\EE \Big[(\tilde X(\infty))^2 1(\tilde X(\infty) \leq -\zeta) \Big] \leq&\  \big( 5 + \delta (1+\delta/2) \big) \zeta^2 + (2+\delta)\abs{\zeta} \label{eq:xsquarezeta} \\
\EE \Big[(\tilde X(\infty))^2 1(\tilde X(\infty) \geq -\zeta) \Big] \leq&\  \delta^2 +8 + \frac{4}{\abs{\zeta}} \Big( \frac{1}{\abs{\zeta}} + \frac{\delta^2}{4\abs{\zeta}} + \frac{\delta}{2}\Big) +\frac{2(2\delta + \delta^3)}{3\abs{\zeta}}. \label{eq:xsquareplus}
\end{align}
\end{lemma}

\begin{proof}[Proof of Lemma~\ref{lem:xtramom}]
We first prove \eqref{eq:xsquarezeta}, or 
\begin{align*}
\EE \Big[(\tilde X(\infty))^2 1(\tilde X(\infty) \leq -\zeta) \Big] \leq  \big( 5 + \delta (1+\delta/2) \big) \zeta^2 + (2+\delta)\abs{\zeta}.
\end{align*}
Let $x$ be of the form $x = \delta(k - R)$, where $k \in \Z_+$. Applying $G_{\tilde X}$ to the function $f(x) = \big[ \delta(k-n)^-\big]^2 =  \big[ (x+\zeta)^-\big]^2$, and observing that $1(k \leq n) = 1(x \leq -\zeta)$,  we get
\begin{align*}
G_{\tilde X} f(x) =&\ \lambda 1(x \leq -\zeta - \delta)\big( 2\delta(x + \zeta) + \delta^2\big) \notag \\
&+ \mu (k \wedge n)1(x \leq -\zeta)\big( -2\delta(x + \zeta)  + \delta^2\big) \notag \\
=&\ \lambda 1(x \leq -\zeta)\big( 2\delta(x + \zeta) + \delta^2\big) + \mu k 1(x \leq -\zeta)\big( -2\delta(x + \zeta)  + \delta^2\big) \notag \\
&- \delta^2\lambda 1(x = -\zeta) \notag \\
=&\ 1(x \leq -\zeta)\Big( 2\delta(x + \zeta)(\lambda - \mu k) + \delta^2(\lambda + \mu k)  \Big) - \delta^2\lambda 1(x = -\zeta)  \notag \\
=&\ 1(x \leq -\zeta)\Big( 2\delta(x + \zeta)(\lambda -  \mu n) + 2\mu \delta(x + \zeta)(n-k) + \delta^2(\lambda +\mu k)  \Big) \\
& - \delta^2\lambda 1(x = -\zeta) \\
=&\ 1(x \leq -\zeta)\Big( 2\mu (x + \zeta)\zeta - 2\mu (x + \zeta)^2 + \delta^2(\lambda + \mu k)  \Big)  - \delta^2\lambda 1(x = -\zeta).
\end{align*}
Taking expected values on both sides and applying Lemma~\ref{lem:gz}, we see that 
\begin{align*}
&\ \EE \big[(\tilde X(\infty)+\zeta)^2 1(\tilde X(\infty) \leq -\zeta) \big]\\
\leq&\ \zeta \EE \big[(\tilde X(\infty)+\zeta) 1(\tilde X(\infty) \leq -\zeta) \big] +  \frac{\delta^2}{2\mu}\Prob(\tilde X(\infty) \leq -\zeta)( \lambda +  \mu n)\\
=&\ \zeta \EE \big[(\tilde X(\infty)+\zeta) 1(\tilde X(\infty) \leq -\zeta) \big] + \frac{1}{2}\Prob(\tilde X(\infty) \leq -\zeta)(1 + \delta^2 n).
\end{align*} 
Recall \eqref{eq:idle_expect}, which tells us that $\EE \big[(\tilde X(\infty)+\zeta) 1(\tilde X(\infty) \leq -\zeta) \big] = \zeta$, to see that
\begin{align*}
\EE \big[(\tilde X(\infty)+\zeta)^2 1(\tilde X(\infty) \leq -\zeta) \big] \leq&\  \zeta^2 + \frac{1}{2}\Prob(\tilde X(\infty) \leq -\zeta)(1 + \delta^2 n)\\
\leq&\ \zeta^2 + \frac{2+\delta}{2} \abs{\zeta}(1 + \delta^2 n),
\end{align*}
where we used \eqref{CW:idle_prob} to get the last inequality. Since $\abs{\zeta} = \delta (n - \frac{\lambda}{\mu })$,
\begin{align*}
\delta^2 n  =  \delta^2 \frac{\abs{\zeta}}{\delta} + \delta^2 \frac{\lambda}{\mu } = \delta \abs{\zeta} + 1,
\end{align*}
and hence,
\begin{align*}
\EE \big[(\tilde X(\infty) + \zeta)^2 1(\tilde X(\infty) \leq -\zeta) \big] \leq&\  \zeta^2 + \frac{2+\delta}{2}\abs{\zeta}(2 + \delta \abs{\zeta}).
\end{align*}
By expanding the square inside the expected value on the left hand side and using \eqref{CW:xminuszeta}, we see that
\begin{align*}
& \EE \big[(\tilde X(\infty))^2 1(\tilde X(\infty) \leq -\zeta) \big] \\
\leq&\  \zeta^2 + \frac{2+\delta}{2}\abs{\zeta}(2 + \delta \abs{\zeta}) + 2\abs{\zeta}\EE\Big[ \big| \tilde X(\infty) 1(\tilde X(\infty) \leq -\zeta)\big| \Big]\\
\leq&\ 5\zeta^2 + \frac{2+\delta}{2}\abs{\zeta}(2 + \delta \abs{\zeta}) \\
=&\ \big( 5 + \delta (1+\delta/2) \big) \zeta^2 + (2+\delta)\abs{\zeta}.
\end{align*}
This proves \eqref{eq:xsquarezeta}. Now we prove \eqref{eq:xsquareplus}, or 
\begin{align*}
\EE \Big[(\tilde X(\infty))^2 1(\tilde X(\infty) \geq -\zeta) \Big] \leq \delta^2 +8 + \frac{4}{\abs{\zeta}} \Big( \frac{1}{\abs{\zeta}} + \frac{\delta^2}{4\abs{\zeta}} + \frac{\delta}{2}\Big) +\frac{2(2\delta + \delta^3)}{3\abs{\zeta}}.
\end{align*}
Let $x$ be of the form $x = \delta(k - R)$, where $k \in \Z_+$. 
Recall from \eqref{DS:abk} that 
\begin{align*}
b(x) = \mu \big[\zeta + (x+\zeta)^- \big] = \delta (\lambda - \mu(k \wedge n)).
\end{align*}
Set $a = \delta\big(\lfloor R \rfloor - R\big) < 0$, and consider the function $f(x) = x^3 1( x \geq a+\delta)$. Then 
\begin{align}
G_{\tilde X} f(x) =&\ \lambda 1( x \geq a+\delta) \big((x+\delta)^3 - x^3 \big) +\lambda 1( x = a) (x+\delta)^3  \notag \\
&+ \mu (k \wedge n)1( x > a+\delta) \big((x-\delta)^3 - x^3\big) + \mu (k \wedge n)1( x = a+\delta) ( - x^3) \notag \\
=&\ 1( x \geq a+\delta) \Big[ \lambda  \big((x+\delta)^3 - x^3 \big) + \mu (k \wedge n) \big((x-\delta)^3 - x^3\big)\Big] \notag \\
&+ \lambda 1( x = a) (x+\delta)^3 - \mu (k \wedge n)1( x = a+\delta) (x-\delta)^3. \label{eq:gfcube}
\end{align}
Suppose $x \geq a+\delta$. Using the fact that $1( x = a+\delta) = 1( k = \lfloor R \rfloor + 1)$, we see that
\begin{align}
&G_{\tilde X} f(x) \notag \\
=&\ \lambda (3\delta x^2 + 3\delta^2 x + \delta^3) + \mu (k \wedge n) (-3\delta x^2 + 3\delta^2 x - \delta^3) - a^3 \mu (k \wedge n)1( x = a+\delta) \notag \\
=&\ 3\delta x^2(\lambda - \mu(k \wedge n)) + 3\delta^2 x(\lambda + \mu(k \wedge n)) + \delta^3(\lambda - \mu(k \wedge n))  \notag \\ 
& - a^3 \mu \big((\lfloor R \rfloor + 1) \wedge n\big) 1( x = a+\delta) \notag  \\
=&\ 3x^2b(x) + 3\delta^2 x\big(2\lambda - (\lambda - \mu(k \wedge n))\big) + \delta^2b(x) \notag \\
&- a^3 \mu \big((\lfloor R \rfloor + 1) \wedge n\big) 1( x = a+\delta) \notag  \\
=&\ 3x^2b(x) + 6\mu x - 3\delta xb(x) + \delta^2b(x)- a^3 \mu \big((\lfloor R \rfloor + 1) \wedge n\big) 1( x = a+\delta). \label{eq:gtildx}
\end{align}
When $x \in [a+\delta,  -\zeta)$ (which is the empty interval if $\lfloor R \rfloor + 1 = n$), then $b(x) = -\mu x$, and
\begin{align}
G_{\tilde X} f(x) =&\ -3\mu x^3 + 6\mu x + 3\delta \mu x^2 - \delta^2\mu x - a^3 \mu \big((\lfloor R \rfloor + 1) \wedge n\big) 1( x = a+\delta) \notag \\
\leq&\ -3\mu \big(x^3 -\delta x^2 + \frac{1}{3}\delta^2 x \big) + 6\mu x + \delta^3 \mu (\lfloor R \rfloor + 1) \notag  \\
\leq&\ -3\mu \big(x^3 -\delta x^2 + \frac{1}{3}\delta^2 x \big) + 6\mu x + \delta \mu  + \delta^3 \mu  \notag \\
\leq&\ 6\mu x+ \delta \mu  + \delta^3 \mu, \label{eq:gtildx1}
\end{align}
where in the first inequality we used the fact that $\abs{a} \leq \delta$, and in the last inequality we used the fact that $g(x) := x^3 -\delta x^2 + \frac{1}{3}\delta^2 x \geq 0$ for all $x \geq 0$, which is true because $g(0) = 0$ and $g'(x) \geq 0$ for all $x \in \R$. Now when $x \geq -\zeta$, then $b(x) = -\mu \abs{\zeta}$, and using \eqref{eq:gtildx} we see that
\begin{align}
G_{\tilde X} f(x) =&\ -3x^2\mu \abs{\zeta} + 6\mu x + 3\delta x\mu \abs{\zeta} - \delta^2\mu \abs{\zeta} - a^3 \mu \big((\lfloor R \rfloor + 1) \wedge n\big) 1( x = a+\delta) \notag \\
\leq&\ -3\mu \abs{\zeta}\big(x^2-\delta x\big) + 6\mu x - a^3 \mu \big((\lfloor R \rfloor + 1) \wedge n\big) 1( x = a+\delta)  \notag \\
\leq&\ -3\mu \abs{\zeta}\big(x^2-\delta x\big) + 6\mu x + \delta \mu  + \delta^3 \mu  \notag \\
\leq&\ -3\mu \abs{\zeta}\big(\frac{1}{2}x^2-\frac{1}{2}\delta^2\big) + 6\mu x + \delta \mu  + \delta^3 \mu , \label{eq:gtildx2}
\end{align}
Combining \eqref{eq:gtildx1} and \eqref{eq:gtildx2} with \eqref{eq:gfcube}, we have just shown that 
\begin{align*}
G_{\tilde X} f(x) \leq&\ -\frac{3}{2}\mu \abs{\zeta}\big(x^2-\delta^2\big)1(x \geq -\zeta)+ 6\mu x 1( x \geq a+\delta) \\
&+ \delta \mu  + \delta^3 \mu + \lambda 1( x = a) (x+\delta)^3.
\end{align*}
Taking expected values on both sides above, and applying Lemma~\ref{lem:gz}, we see that 
\begin{align*}
&\frac{3}{2}\mu \abs{\zeta} \EE\Big[ (\tilde X(\infty))^2 1(\tilde X(\infty) \geq -\zeta) \Big] \\
\leq&\ \frac{3}{2}\mu \abs{\zeta}\delta^2 + 6\mu \EE \Big[\big| \tilde X(\infty) \big| \Big]+ \delta \mu  + \delta^3 \mu + \lambda (a+\delta)^3,
\end{align*}
and since $\lambda (a+\delta)^3 \leq \lambda \delta^3 = \mu \delta$, we have
\begin{align*}
\EE\Big[ (\tilde X(\infty))^2 1(\tilde X(\infty) \geq -\zeta) \Big] \leq&\ \delta^2 + \frac{4}{\abs{\zeta}} \EE \Big[\big| \tilde X(\infty) \big| \Big] +\frac{2(2\delta + \delta^3)}{3\abs{\zeta}}.
\end{align*}
Using the moment bounds in \eqref{CW:xminuszeta} and \eqref{CW:xplus}, we conclude that 
\begin{align*}
\EE\Big[ (\tilde X(\infty))^2 1(\tilde X(\infty) \geq -\zeta) \Big] \leq&\ \delta^2 +8 + \frac{4}{\abs{\zeta}} \Big( \frac{1}{\abs{\zeta}} + \frac{\delta^2}{4\abs{\zeta}} + \frac{\delta}{2}\Big) +\frac{2(2\delta + \delta^3)}{3\abs{\zeta}},
\end{align*}
which proves \eqref{eq:xsquareplus}.
\end{proof}
\noindent We now prove Lemma~\ref{lem:lastbounds}, followed by a proof of Lemma~\ref{lem:pibounds}.

\begin{proof}[Proof of Lemma~\ref{lem:lastbounds}]
Observe that \eqref{CW:idle_prob} is identical to \eqref{eq:idle_prob_v2}. Now assume that $\delta \leq 1$. We begin by proving \eqref{eq:lb1}. Using the moment bounds in \eqref{CW:xminusdelta} and \eqref{CW:xminuszeta}, we see that
\begin{align*}
(1 + 1/\abs{\zeta})\EE \Big[\big|\tilde X(\infty)  1(\tilde X(\infty) \leq -\zeta) \big|\Big]\leq &\ (1 + 1/\abs{\zeta})\bigg(2\abs{\zeta} \wedge \sqrt{\frac{4}{3} + \frac{2\delta^2}{3} }\bigg) \\
\leq &\ \Big(\sqrt{\frac{4}{3} + \frac{2\delta^2}{3} } + 2\Big) \leq \big(\sqrt{2} + 2\big).
\end{align*}
Next we prove \eqref{eq:lb2}. Using the moment bounds in \eqref{CW:xsquaredelta} and \eqref{eq:xsquarezeta}, we see that
\begin{align*}
&\ (1 + 1/\abs{\zeta})\EE \Big[(\tilde X(\infty))^2  1(\tilde X(\infty) \leq -\zeta) \Big]\\
\leq &\ (1 + 1/\abs{\zeta})\bigg(\Big(\big( 5 + \delta (1+\delta/2) \big) \zeta^2 + (2+\delta)\abs{\zeta} \Big)\wedge \Big(\frac{4}{3} + \frac{2\delta^2}{3}\Big)\bigg) \\
\leq &\ 2 + \bigg(\big(6.5\abs{\zeta} +  3\big) \wedge \frac{2}{\abs{\zeta}}\bigg) \leq 2 + 7,
\end{align*}
where to get the last inequality we considered separately the cases when $\abs{\zeta} \leq 1/2$ and $\abs{\zeta} \geq 1/2$. To prove \eqref{eq:lb3}, we use the moment bound \eqref{CW:xplus} to get
\begin{align*}
\abs{\zeta} \Prob( \tilde X(\infty) \geq -\zeta) \leq&\ \abs{\zeta} \wedge \EE \Big[\big| \tilde X(\infty) 1(\tilde X(\infty) \geq -\zeta)\big| \Big]\notag \\
 \leq&\ \abs{\zeta} \wedge \Big(\frac{1}{\abs{\zeta}} + \frac{\delta^2}{4\abs{\zeta}} + \frac{\delta}{2} \Big) \notag \\
\leq&\ 2, 
\end{align*}
where to get the last inequality we considered separately the cases where $\abs{\zeta} \leq 1$ and $\abs{\zeta} \geq 1$. The proof of \eqref{eq:lb4} is similar. We use the moment bound \eqref{eq:xsquareplus} to see that 
\begin{align*}
\zeta^2 \Prob( \tilde X(\infty) \geq -\zeta) \leq&\ \zeta^2 \wedge \EE \Big[( \tilde X(\infty))^2 1(\tilde X(\infty) \geq -\zeta) \Big]\\
 \leq&\ \zeta^2 \wedge \Big(\delta^2 +8 + \frac{4}{\abs{\zeta}} \Big( \frac{1}{\abs{\zeta}} + \frac{\delta^2}{4\abs{\zeta}} + \frac{\delta}{2}\Big) +\frac{2(2\delta + \delta^3)}{3\abs{\zeta}} \Big)  \\
 \leq&\ \zeta^2 \wedge \Big(1 +8 + \frac{4}{\abs{\zeta}} \Big( \frac{1}{\abs{\zeta}} + \frac{1}{4\abs{\zeta}} + \frac{1}{2}\Big) +\frac{2}{\abs{\zeta}} \Big)  \\
\leq&\ 20,
\end{align*}
where to get the last inequality we considered separately the cases where $\abs{\zeta} \leq 1$ and $\abs{\zeta} \geq 1$.  This concludes the proof of Lemma~\ref{lem:lastbounds}.
\end{proof}

\begin{proof}[Proof of Lemma~\ref{lem:pibounds}]
We first prove \eqref{DS:pi0}. From \eqref{eq:idle_prob_v2}, we know that 
\begin{align*}
\Prob(\tilde X(\infty) \leq -\zeta) = \Prob(X(\infty) \leq n) = \sum_{k=0}^{n} \pi_k  \leq (2+\delta)\abs{\zeta}.
\end{align*}
From the flow balance equations, one can see that $\pi_{\lfloor R \rfloor}$ maximizes $\{\pi_k\}_{k=0}^{\infty}$. Now when  $\abs{\zeta} \leq 1$,
\begin{align*}
\abs{\zeta} = \delta (n - R) = \frac{1}{\sqrt{R}}(n - R) \leq 1,
\end{align*}
which implies that 
\begin{align*}
R \geq n - \sqrt{R} \geq n - \sqrt{n},
\end{align*}
where in the last inequality we used $R < n$. We use this inequality together with the fact that $\pi_0 \leq \pi_1 \leq \ldots \leq \pi_{\lfloor R \rfloor}$, which can be verified from the flow balance equations, to see that
\begin{align*}
(2+\delta)\abs{\zeta} \geq \sum_{k=0}^{n} \pi_k \geq \sum_{k=0}^{\lfloor R \rfloor} \pi_k \geq \pi_{0}\lfloor R \rfloor \geq \pi_{0}\lfloor n-\sqrt{n} \rfloor.
\end{align*}
Hence, for $n \geq 4$, 
\begin{align*}
\pi_0 \leq \frac{n}{\lfloor n-\sqrt{n} \rfloor} \frac{(2+\delta)\abs{\zeta}}{n} \leq \frac{n}{n-\sqrt{n} - 1} \frac{(2+\delta)\abs{\zeta}}{R} \leq  \frac{4(2+\delta)\abs{\zeta}}{R} = 4(2+\delta)\delta^2 \abs{\zeta}.
\end{align*}
To conclude the proof of \eqref{DS:pi0} we need to verify the bound above holds for $n < 4$, but this is simple to do. Observe that for $n < 4$, 
\begin{align*}
\pi_0 \leq P(X(\infty) \leq n) \leq (2+\delta)\abs{\zeta}  = (2+\delta)R\delta^2\abs{\zeta}  \leq (2+\delta)n \delta^2 \abs{\zeta} \leq 4(2+\delta)\delta^2\abs{\zeta}.
\end{align*}
This proves \eqref{DS:pi0}, and we move on to prove \eqref{DS:pin}. From the flow balance equations corresponding to the CTMC $X$, it is easy to see that 
\begin{align*}
\pi_n = \frac{\frac{R^n}{n!}}{\sum_{k=0}^{n-1} \frac{R^k}{k!} + \frac{R^n}{n!} \frac{1}{1-R/n}} \leq 1 - \frac{R}{n} = \frac{n-R}{n} = \frac{R}{n} \frac{n-R}{R} = \frac{R}{n}\delta \abs{\zeta} \leq \delta \abs{\zeta}.
\end{align*}
This concludes the proof of the lemma.
\end{proof}

\section{Chapter~\ref{chap:phasetype} Moment Bounds} \label{app:CTMCmomentsproof}
This section uses notation from Chapter~\ref{chap:phasetype}.
\begin{proof}[Proof of Lemma~\ref{lemma:CTMCmoments} ]
We first provide an intuitive roadmap for the proof. The goal is to show that a Lyapunov function for the diffusion process is also a Lyapunov function for the CTMC; this has two parts to it. In the first part of this proof, we compare how the two generators $G_{U^{(\lambda)}}$ and $G_Y$ act on this Lyapunov function, obtaining an upper bound for the difference $G_{U^{(\lambda)}} - G_Y$ in (\ref{eq:CTMCmomlemmaeq2}).
One notes  that the right hand side of (\ref{eq:CTMCmomlemmaeq2}) is unbounded. This is due to the difference in dimensions of the CTMC and diffusion process. To overcome this difficulty, we move on to the second part of the proof, which exploits our SSC result in Lemma~\ref{lemma:SSC} to bound the expectation of the right hand side of (\ref{eq:CTMCmomlemmaeq2}). We end up with a recursive relationship that guarantees the $2m$th moment is bounded (uniformly in $\lambda$ and $n$ satisfying \eqref{eq:square-root}) provided that the $m$th moment is. Finally, we rely on prior results obtained in \cite{DaiDiekGao2014} for a uniform bound on the first moment.

We remark that a version of this lemma was already proved \cite[Theorem 3.3]{Gurv2014} for the case where the dimension of the CTMC equals the dimension of the diffusion process. However, the difference in dimensions poses an additional technical challenge, which is overcome in the second part of this proof.

Its enough to prove (\ref{eq:CTMCunifmombound}) for the cases when $m=2^j$ for some $j\geq 0$. Furthermore, we may assume that $\lambda \geq 4$ because by Remark~\ref{rem:lambdarange}, there are only finitely many cases when $\lambda < 4$. In all those cases, $\E \abs{\tilde X^{(\lambda)}(\infty)}^m < \infty$ by (\ref{eq:CTMCmgfexist}). Throughout the proof, we shall use $C,C_1,C_2,C_3,C_4$ to denote generic positive constants that may change from line to line. They may depend on $(m,\beta,\alpha,p,\nu,P)$, but will be independent of both $\lambda$ and $n$. Define
\begin{displaymath}
V_m(x) = (1+V(x))^m,
\end{displaymath}
where $V$ is as in (\ref{eq:deflyapou}). By \cite[Remark 3.4]{Gurv2014}, $V_m$ also satisfies
\begin{displaymath}
G_Y V_m(x) \leq -C_1 V_m(x) + C_2
\end{displaymath}
as long as $V \in C^3(\R^d)$ and satisfies condition (30) of \cite{Gurv2014}, which is easy to verify. To prove the lemma, we will show that for large enough $\lambda$, $V$ satisfies 
\begin{displaymath}
\E G_{U^{(\lambda)}} AV_m(U^{(\lambda)}(\infty)) \leq - C_1 \E  V_m(\tilde X^{(\lambda)}(\infty)) + C_2,
\end{displaymath}
where $A$ is the lifting operator defined in (\ref{eq:lifter}). %and $C_1$, $C_2$ are positive constants that depend only on $(m,\beta,\alpha,p,\nu,P)$. Let $x$ and $u$ be related as in (\ref{eq:projlittle})}.
 We begin by observing
\begin{equation} \label{eq:CTMCmomlemmaeq1}
G_{U^{(\lambda)}} AV_m \leq G_{U^{(\lambda)}}AV_m - G_Y V_m + G_Y V_m \leq G_{U^{(\lambda)}}AV_m - G_Y V_m -C_1 V_m + C_2.
\end{equation}
Using (\ref{eq:errorterm}), we write $G_{U^{(\lambda)}}AV_m - G_Y V_m$ as
\begin{eqnarray*}
&&\sum \limits_{i=1}^d \partial_i V_m(x)\Big{[}(\nu_i - \alpha-\sum \limits_{j=1}^d P_{ji}\nu_j)( \delta q_i - p_i(e^T x)^+)\Big{]} \\
&& + \sum \limits_{i=1}^d \partial_{ii}V_m(x)\Big{[}\sum \limits_{j=1}^d P_{ji}\nu_j\gamma_j \Big{]}(n \delta^2 - 1) - \sum \limits_{i \neq j}^d \partial_{ij}V_m(x)\Big{[}P_{ij}\nu_i\gamma_i + P_{ji} \nu_j\gamma_j\Big{]}(n\delta^2 - 1) \notag \\
&&- \sum \limits_{i=1}^d \frac{\delta^2}{2}\partial_{ii}V_m(x)\Big{[}p_i (\lambda - n)- \alpha  q_i - \nu_i (z_i- \gamma_i n) - \sum \limits_{j=1}^d P_{ji}\nu_j (z_j- \gamma_j n)\Big{]} \notag \\
&&- \sum \limits_{i \neq j}^d \frac{\delta^2}{2}\partial_{ij}V_m(x)\Big{[}P_{ij}\nu_i (z_i- \gamma_i n) + P_{ji} \nu_j (z_j- \gamma_j n)\Big{]} \notag \\
&&+ \sum \limits_{i=1}^d \frac{\delta^2}{2} (\partial_{ii}V_m(\xi_{i}^-)-\partial_{ii}V_m(x))\Big{[}\alpha q_i + (1-\sum \limits_{j=1}^d P_{ij})\nu_i z_i\Big{]}  \notag \\
&&+ \sum \limits_{i=1}^d \frac{\delta^2}{2} (\partial_{ii}V_m(\xi_{i}^+)-\partial_{ii}V_m(x))\Big{[}\lambda p_i \Big{]}- \sum \limits_{i \neq j}^d \delta^2 (\partial_{ij}V_m(\xi_{ij})-\partial_{ij}V_m(x))\Big{[}P_{ij}\nu_i z_i\Big{]} \notag \\
&&+ \sum \limits_{i=1}^d \sum \limits_{j=1}^d \frac{\delta^2}{2} (\partial_{ii}V_m(\xi_{ij})-\partial_{ii}V_m(x))\Big{[} P_{ij}\nu_i z_i +  P_{ji}\nu_j z_j\Big{]} .  \notag 
\end{eqnarray*}
Now we wish to bound the derivatives of $V_m$. By \cite[Remark 3.4]{Gurv2014}, $V_m$ satisfies (16) and (30) of \cite{Gurv2014}, namely

\begin{equation}
\sup \limits_{\abs{y} \leq 1} \frac{V_m(x+y)}{V_m(x)} \leq C \label{eq:Vsubexp}
\end{equation}
and
\begin{equation} \label{eq:cond30gurvich}
(\abs{\partial_i V_m(x)} + \abs{\partial_{ij} V_m(x)} +\abs{\partial_{ijk} V_m(x)})(1+\abs{x}) \leq C V_m(x).
\end{equation}
For $\xi$ being one of $\xi_{i}^+$, $\xi_{i}^-$ or $\xi_{ij}$,
\begin{equation} \label{eq:VD3bound}
\abs{\partial_{ij}V_m(\xi) - \partial_{ij}V_m(x)}(1+\abs{x}) \leq \delta \abs{\partial_{iji} V_m(\eta) +\partial_{ijj}V_m(\eta) }(1+\abs{x}) \leq C \delta V_m(x),
\end{equation}
where the first inequality comes from a Taylor expansion and the second inequality follows by (\ref{eq:cond30gurvich}), the fact that $\abs{\eta - x} \leq 2\delta < 1$ and by (\ref{eq:Vsubexp}).
Following the exact same argument that we used to bound (\ref{eq:errorterm}) in the proof of Lemma~\ref{lemma:diffbound} (with (\ref{eq:cond30gurvich}) and (\ref{eq:VD3bound}) replacing the gradient bounds of $f_h$ there), we get 
\begin{displaymath}
G_{U^{(\lambda)}}AV_m - G_Y V_m \leq  C \delta V_m(x) +  C\sum \limits_{i=1}^d \abs{\partial_i V_m(x)}\Big{[}\abs{ q_i - p_i(e^T x)^+}\Big{]}.
\end{displaymath}
Differentiating $V$, we see that
\begin{displaymath}
(\nabla V(x))^T = 2(e^T x)e^T + 2 \kappa (x^T - p^T \phi(e^T x)) \tilde Q (I - p e^T \phi'(e^Tx)).
\end{displaymath} 
Combined with the fact that $0 \leq \phi'(x)\leq 1$, it is clear that
\begin{displaymath}
\abs{\partial_i V(x)} \leq C(1+\abs{x}).
\end{displaymath}
Therefore,
\begin{equation}\label{eq:CTMCmomlemmaeq2}
G_{U^{(\lambda)}}AV_m - G_Y V_m \leq  C \delta V_m(x) +  C\sum \limits_{i=1}^d mV_{m-1}(x)(1+\abs{x})\Big{[}\abs{ q_i - p_i(e^T x)^+}\Big{]} .
\end{equation}
It remains to find an appropriate bound for 
\begin{displaymath}
V_{m-1}(x) (1+\abs{x})\Big{[}\abs{ q_i - p_i(e^T x)^+}\Big{]} = \delta V_{m-1}(x)(1+\abs{x})\Bigg{[}\frac{\abs{q_i - p_i(e^T x)^+}}{\delta}\Bigg{]}.
\end{displaymath}
We have
\begin{eqnarray}
&&\delta V_{m-1}(x) (1+\abs{x})\Bigg{[}\frac{\abs{q_i - p_i(e^T x)^+}}{\delta}\Bigg{]} \notag \\
&\leq & \sqrt{\delta}V_{m-1}(x)(1+\abs{x})^2 + \sqrt{\delta}V_{m-1}(x)\Bigg{[}\frac{\abs{q_i - p_i(e^T x)^+}^2}{\delta}\Bigg{]} \notag \\
&\leq & C\sqrt{\delta}V_{m}(x) + \sqrt{\delta}V_{m-2}(x)V_2(x)+ \sqrt{\delta}V_{m-2}(x)\Bigg{[}\frac{\abs{q_i - p_i(e^T x)^+}^2}{\delta}\Bigg{]}^2 \notag \\
&\leq & C\sqrt{\delta}V_{m}(x) + \sqrt{\delta}V_{m}(x)+ \sqrt{\delta}V_{m-4}(x)V_4(x)+ \sqrt{\delta} V_{m-4}(x)\Bigg{[}\frac{\abs{q_i - p_i(e^T x)^+}^2}{\delta}\Bigg{]}^4 \notag \\ 
&\leq & \ldots  \notag \\
&\leq & C\sqrt{\delta}V_{m}(x) + \sqrt{\delta}\Bigg{[}\frac{\abs{q_i - p_i(e^T x)^+}^2}{\delta}\Bigg{]}^m, \label{eq:CTMCmomlemmaeq3}
\end{eqnarray}
where in the last inequality, we used the fact that $m = 2^j$. Using (\ref{eq:CTMCmomlemmaeq1}), (\ref{eq:CTMCmomlemmaeq2}) and (\ref{eq:CTMCmomlemmaeq3}), 
\begin{displaymath}
G_{U^{(\lambda)}} AV_m(u) \leq - V_m(x)(C_1 - \sqrt{\delta}C_3) + C_2 + \sqrt{\delta}C_4\sum \limits_{i=1}^d \Bigg{[}\frac{\abs{q_i - p_i(e^T x)^+}^2}{\delta}\Bigg{]}^m,
\end{displaymath}
where $x$ and $q$ are related to $u$ by (\ref{eq:projlittle}).
The arguments in the proof of Lemma~\ref{lemma:CTMCbar} can be used to show 
\begin{displaymath}
\E G_{U^{(\lambda)}} A V_m(U^{(\lambda)}(\infty)) = 0.
\end{displaymath} 
Therefore, for $\delta$ small enough, 
\begin{eqnarray*}
&&E\abs{\tilde X^{(\lambda)}(\infty)}^{2m}  \\
&\leq & C\E V_m(\tilde X^{(\lambda)}(\infty))\\
 &\leq & \frac{C}{(C_1 - \sqrt{\delta}C_3)}\Bigg{(}C_2 + \sqrt{\delta}C_4\sum \limits_{i=1}^d \frac{\E \abs{\delta Q^{(\lambda)}_i(\infty) - p_i(e^T \tilde X^{(\lambda)}(\infty))^+}^{2m}}{\delta^m}\Bigg{)}.
\end{eqnarray*}
By (\ref{eq:sscscaled}), it follows that 
\begin{displaymath}
\E \abs{\tilde  X^{(\lambda)}(\infty)}^{2m} \leq \frac{C}{C_1 - \sqrt{\delta}C_3} \Bigg{(}1 + \sqrt{\delta}\E [(e^T \tilde X^{(\lambda)}(\infty))^+]^{m}\Bigg{)}.
\end{displaymath}
Hence, we have a recursive relationship that guarantees 
\begin{displaymath}
\sup \limits_{\lambda > 0} \E \abs{\tilde  X^{(\lambda)}(\infty)}^{2m} < \infty
\end{displaymath}
whenever 
\begin{displaymath}
\sup \limits_{\lambda > 0} \E[(e^T\tilde  X^{(\lambda)}(\infty))^+]^m < \infty.
\end{displaymath}
To conclude, we need to verify that
\begin{displaymath}
\sup \limits_{\lambda > 0} \E[(e^T\tilde  X^{(\lambda)}(\infty))^+] < \infty,
\end{displaymath}
but this was proved in equation (5.2) of \cite{DaiDiekGao2014}. 

\end{proof}

\chapter{Gradient Bounds} \label{app:GRADIENTS}
This appendix proves all of the gradient bounds used in this document. Section~\ref{app:gradboundsgeneric} provides some generic tools for establishing gradient bounds in the case of a one-dimensional diffusion approximation, i.e.\ when the Poisson equation is an ordinary differential equation. Bounds for Chapters~\ref{chap:erlangAC}, \ref{chap:dsquare} and \ref{chap:phasetype} are proved in Sections~\ref{app:gradbounds}, \ref{app:DSgradbounds}, and \ref{app:PHgradbounds}, respectively. 

\section{The Poisson Equation for Diffusion Processes} \label{app:gradboundsgeneric}
To make this section self-contained, we begin by repeating Lemma~\ref{lem:solution}. Let $\bar a:\R \to \R_+$ and $\bar b: \R \to \R$ be continuous functions, and assume that 
\begin{align*}
\inf_{x \in \R} \bar a(x) > 0.
\end{align*} 
Assume that 
\begin{align}
\int_{-\infty}^{\infty}\frac{2}{\bar a(x)} \exp \Big({\int_{0}^{x} \frac{2 \bar b(u)}{\bar a(u)} du} \Big)dx < \infty, \label{eq:densintegrable}
\end{align}
and let $V$ be a continuous random variable with density 
\begin{align}
\frac{\frac{2}{\bar a(x)} \exp \Big({\int_{0}^{x} \frac{2 \bar b(u)}{\bar a(u)} du} \Big)}{\int_{-\infty}^{\infty}\frac{2}{\bar a(x)} \exp \Big({\int_{0}^{x} \frac{2 \bar b(u)}{\bar a(u)} du} \Big)dx}, \quad x \in \R. \label{eq:gendens}
\end{align}
\begin{lemma} \label{lem:solutionv2}
Fix $h:\R \to \R$ satisfying $\E |h(V)| < \infty$, and consider the Poisson equation 
\begin{align}
\frac{1}{2} \bar a(x) f_h''(x) + \bar b(x) f_h'(x) = \E h(V) - h(x), \quad x \in \R. \label{eq:genericpoissonv2}
\end{align}
There exists a solution $f_h(x)$ to this equation satisfying 
\begin{align}
f_h'(x) =&\ e^{-\int_{0}^{x} \frac{2\bar b(u)}{\bar a(u)}du}\int_{-\infty}^{x} \frac{2}{\bar a(y)} (\E h(V) - h(y)) e^{\int_{0}^{y} \frac{2\bar b(u)}{\bar a(u)}du} dy  \label{eq:fprimeneg} \\
=&\ -e^{-\int_{0}^{x} \frac{2\bar b(u)}{\bar a(u)}du} \int_{x}^{\infty} \frac{2}{\bar a(y)} (\E h(V) - h(y)) e^{\int_{0}^{y} \frac{2\bar b(u)}{\bar a(u)}du} dy \label{eq:fprimepos}, \\
f_h''(x) =&\  - \frac{2\bar b(x)}{\bar a(x)} f_h'(x) +  \frac{2}{\bar a(x)}\big(\E h(V) - h(x) \big). \label{eq:fpp}
\end{align}
\end{lemma}
\noindent Provided $h(x)$, $\bar a(x)$, and $\bar b(x)/\bar a(x)$ are sufficiently differentiable, $f_h(x)$ can have more than two derivatives. For example, 
\begin{align*}
f_h'''(x) =&\ -\Big(\frac{2\bar b(x)}{\bar a(x)}\Big)' f_h'(x) - \frac{2\bar b(x)}{\bar a(x)} f_h''(x) - \frac{2}{\bar a(x)}h'(x) - \frac{2\bar a'(x)}{a^2(x)}\big( \E h(V) - h(x)\big). 
\end{align*}
The biggest source of difficulty in bounding $f_h'(x)$, $f_h''(x)$, and $f_h'''(x)$, are the integrals in \eqref{eq:fprimeneg} and \eqref{eq:fprimepos}. Before describing how to bound them, we give an alternative representation of $f_h''(x)$, which is taken from the proof of \cite[Lemma 13.1]{ChenGoldShao2011}. The assumptions of the following lemma are  only slightly stronger than those in Lemma~\ref{lem:solution}. 
\begin{lemma}\label{lem:handle}
Fix $h:\R \to \R$ satisfying $\E |h(V)| < \infty$, and let $f_h'(x)$ be as in \eqref{eq:fprimeneg}--\eqref{eq:fprimepos}. Assume $\bar a(x), \bar b(x)/\bar a(x)$, and $h(x)$ are absolutely continuous. If
\begin{align}
\lim_{x \to -\infty}  \frac{2\bar b(x)}{\bar a(x)} f_h'(x)e^{\int_{0}^{x} \frac{2\bar b(u)}{\bar a(u)}du} = 0, \label{eq:mlimits}
\end{align}
then 
\begin{align}
f_h''(x) =&\ e^{-\int_{0}^{x} \frac{2\bar b(u)}{\bar a(u)}du} \int_{-\infty}^{x} \Big(\frac{2}{\bar a(y)} h'(y) - \frac{2\bar a'(y)}{a^2(y)}\big(\E h(V) - h(y) \big) \notag\\
&\hspace{5cm} - \Big(\frac{2\bar b(y)}{\bar a(y)}\Big)' f_h'(y)\Big)e^{\int_{0}^{y} \frac{2\bar b(u)}{\bar a(u)}du} dy. \label{eq:hf21}
\end{align}
Similarly, if 
\begin{align}
\lim_{x \to \infty} \frac{2\bar b(x)}{\bar a(x)} f_h'(x)e^{\int_{0}^{x} \frac{2\bar b(u)}{\bar a(u)}du} = 0, \label{eq:plimits}
\end{align}
then 
\begin{align}
f_h''(x) =&\  -e^{-\int_{0}^{x} \frac{2\bar b(u)}{\bar a(u)}du} \int_{x}^{\infty} \Big(\frac{2}{\bar a(y)} h'(y) - \frac{2\bar a'(y)}{a^2(y)}\big(\E h(V) - h(y) \big) \notag \\
& \hspace{5cm} - \Big(\frac{2\bar b(y)}{\bar a(y)}\Big)' f_h'(y)\Big)e^{\int_{0}^{y} \frac{2\bar b(u)}{\bar a(u)}du} dy. \label{eq:hf22}
\end{align}
\end{lemma}
\noindent This lemma is proved at the end of this section.  In practice, working with the representation in Lemma~\ref{lem:handle} often yields better bounds on $f_h''(x)$ than using \eqref{eq:fpp}.
Again, we see that both \eqref{eq:hf21} and \eqref{eq:hf22} contain integral term involving $e^{\int_{0}^{y} \frac{2\bar b(u)}{\bar a(u)}du}$. To  help bound these integrals, we make several assumptions on $\bar b(x)$. Assume that 
\begin{itemize}
\item[(a1)] \phantomsection \label{eq:a1} $\bar b(x)$ is a non-increasing function of $x$ 
\item[(a2)] \phantomsection  \label{eq:a2} $\bar b(x)$ has at most one zero, 
\end{itemize}
and define 
\begin{align}
x_0 = 
\begin{cases}
&-\infty, \quad \bar b(x) < 0, x \in \R \\
&\infty, \quad \bar b(x) > 0, x \in \R \\
&\text{ the zero of $\bar b(x)$}, \quad \text{ otherwise}.
\end{cases} \label{eq:xnot}
\end{align}
It may be helpful to the reader to pretend that $x_0 = 0$, which will always be the case in this thesis. The following two lemmas present some useful inequalities that will be very helpful in getting the gradient bounds that we require. Due to their generality, they may also be of independent interest. We discuss assumptions (\hyperref[eq:a1]{a1}) and (\hyperref[eq:a2]{a2}) after their statements and proofs.

\begin{lemma} \label{lem:main}
Suppose $\bar b(x)$ satisfies both  (\hyperref[eq:a1]{a1}) and (\hyperref[eq:a2]{a2}), and let $x_0$ be as in \eqref{eq:xnot}. Then 
\begin{align}
e^{-\int_{0}^{x} \frac{2\bar b(u)}{\bar a(u)}du}\int_{-\infty}^{x} \frac{2}{\bar a(y)} e^{\int_{0}^{y} \frac{2\bar b(u)}{\bar a(u)}du} dy \leq \frac{1}{\bar b(x)}, \quad x < x_0, \label{eq:main1} \\
e^{-\int_{0}^{x} \frac{2\bar b(u)}{\bar a(u)}du} \int_{x}^{\infty} \frac{2}{\bar a(y)} e^{\int_{0}^{y} \frac{2\bar b(u)}{\bar a(u)}du} dy \leq \frac{1}{\abs{\bar b(x)}}, \quad x > x_0, \label{eq:main2}
\end{align}
where we adopt the convention that $\{x : x  > \infty\} = \{x : x < -\infty\} = \emptyset$. Furthermore, if $x_0$ is finite, then for any $c_1 \in (-\infty, x_0)$ and $c_2 \in (x_0, \infty)$, 
\begin{align}
e^{-\int_{0}^{x} \frac{2\bar b(u)}{\bar a(u)}du}\int_{-\infty}^{x} \frac{2}{\bar a(y)} e^{\int_{0}^{y} \frac{2\bar b(u)}{\bar a(u)}du} dy \leq \frac{1}{\bar b(c_1)}  +  2\abs{c_1 - x_0}\sup_{y \in [c_1,x_0]} \frac{1}{\bar a(y)} , \quad x < x_0, \label{eq:near_orig_left} \\
e^{-\int_{x_0}^{x} \frac{2\bar b(u)}{\bar a(u)}du} \int_{x}^{\infty} \frac{2}{\bar a(y)} e^{\int_{x_0}^{y} \frac{2\bar b(u)}{\bar a(u)}du} dy \leq \frac{1}{\abs{\bar b(c_2)}} +  2\abs{c_2-x_0}\sup_{y \in [x_0,c_2]} \frac{1}{\bar a(y)}, \quad x > x_0. \label{eq:near_orig_right}
\end{align}
\end{lemma}
\begin{remark}
Suppose $x_0$ is finite. Since $\bar b(x_0) = 0$, the bounds in \eqref{eq:main1} and \eqref{eq:main2} lose relevance for $x$ near $x_0$. This is the reason for having \eqref{eq:near_orig_left} and \eqref{eq:near_orig_right}.
\end{remark}

\begin{proof}[Proof of Lemma~\ref{lem:main}]
We first prove \eqref{eq:main1}. The assumption that $\bar b(x)$ is decreasing implies that $\bar b(y)/\bar b(x) \geq 1$ for $y \leq x < x_0$. Therefore,
\begin{align}
e^{-\int_{0}^{x} \frac{2\bar b(u)}{\bar a(u)}du} \int_{-\infty}^{x} \frac{2}{\bar a(y)}e^{\int_{0}^{y} \frac{2\bar b(u)}{\bar a(u)}du} dy =&\  e^{-\int_{x_0}^{x} \frac{2\bar b(u)}{\bar a(u)}du} \int_{-\infty}^{x} \frac{2 }{\bar a(y)} e^{\int_{x_0}^{y} \frac{2\bar b(u)}{\bar a(u)}du} dy \notag \\
 \leq&\  e^{-\int_{x_0}^{x} \frac{2\bar b(u)}{\bar a(u)}du} \int_{-\infty}^{x} \frac{2 \bar b(y)}{\bar a(y)} \frac{1}{\bar b(x)} e^{\int_{x_0}^{y} \frac{2\bar b(u)}{\bar a(u)}du} dy \notag \\
=&\  e^{-\int_{x_0}^{x} \frac{2\bar b(u)}{\bar a(u)}du} \frac{1}{\bar b(x)} \Big(e^{\int_{x_0}^{x} \frac{2\bar b(u)}{\bar a(u)}du} - e^{-\int_{x_0}^{-\infty} \frac{2\bar b(u)}{\bar a(u)}du}\Big)  \notag \\
\leq&\ \frac{1}{\bar b(x)}. \label{eq:negpart}
\end{align}
One can justify \eqref{eq:main2} using a symmetric argument. We now prove \eqref{eq:near_orig_left}. Fix $c_1 < x_0$ and suppose $x \leq c_1$. Then \eqref{eq:negpart}, together with the fact that $\bar b(x) \geq \bar b(c_1)$ implies 
\begin{align*}
e^{-\int_{0}^{x} \frac{2\bar b(u)}{\bar a(u)}du} \int_{-\infty}^{x} \frac{2}{\bar a(y)}e^{\int_{0}^{y} \frac{2\bar b(u)}{\bar a(u)}du} dy \leq \frac{1}{\bar b(c_1)}.
\end{align*}
Now when $x \in [c_1,x_0]$,
\begin{align*}
&\ e^{-\int_{0}^{x} \frac{2\bar b(u)}{\bar a(u)}du} \int_{-\infty}^{x} \frac{2}{\bar a(y)}e^{\int_{0}^{y} \frac{2\bar b(u)}{\bar a(u)}du} dy \notag \\
=&\ \frac{e^{\int_{x_0}^{c_1} \frac{2\bar b(u)}{\bar a(u)}du}}{e^{\int_{x_0}^{x} \frac{2\bar b(u)}{\bar a(u)}du}} e^{-\int_{x_0}^{c_1} \frac{2\bar b(u)}{\bar a(u)}du} \int_{-\infty}^{c_1} \frac{2}{\bar a(y)}e^{\int_{x_0}^{y} \frac{2\bar b(u)}{\bar a(u)}du} dy + e^{-\int_{x_0}^{x} \frac{2\bar b(u)}{\bar a(u)}du} \int_{c_1}^{x} \frac{2}{\bar a(y)}e^{\int_{x_0}^{y} \frac{2\bar b(u)}{\bar a(u)}du} dy\notag  \\
\leq&\ e^{-\int_{c_1}^{x} \frac{2\bar b(u)}{\bar a(u)}du} \frac{1}{\bar b(c_1)} +  \int_{c_1}^{x} \frac{2}{\bar a(y)}e^{-\int_{y}^{x} \frac{2\bar b(u)}{\bar a(u)}du} dy \notag \\
\leq&\ \frac{1}{\bar b(c_1)} +  \sup_{y \in [c_1,x_0]} \frac{2\abs{c_1 - x_0}}{\bar a(y)},
\end{align*}
where in the last inequality we used the fact that $\bar b(u) \geq 0$ for $u \in [c_1, x]$. This proves \eqref{eq:near_orig_left}, and a symmetric argument can be used to prove \eqref{eq:near_orig_right}.
\end{proof}
\begin{remark}
The first inequality of \eqref{eq:negpart} was inspired by the well-known bound on the CDF of the normal distribution 
\begin{align*}
\int_{-\infty}^{x} e^{-y^2/2} dy \leq \int_{-\infty}^{x} \frac{y}{x} e^{-y^2/2} dy  = \frac{1}{x} e^{-x^2/2}, \quad x <  0,
\end{align*}
which is used to prove gradient bounds for the normal distribution \cite{ChenGoldShao2011, Ross2011}.
\end{remark}

\begin{lemma} \label{lem:poly}
Suppose $\bar b(x)$ satisfies both  (\hyperref[eq:a1]{a1}) and (\hyperref[eq:a2]{a2}), and $x_0$, as defined in \eqref{eq:xnot}, is finite. Suppose also that there exist $\ell \leq x_0$ and $r \geq x_0$ such that $\bar a(x) = a_{\ell}$ for $x < \ell$ and $\bar a(x) = a_{r}$ for $x > r$. Then for any $k \in \mathbb{N}$, 
\begin{align}
&e^{-\int_{0}^{x} \frac{2\bar b(u)}{\bar a(u)}du}\int_{-\infty}^{x} \frac{2\abs{y}^k}{\bar a(y)} e^{\int_{0}^{y} \frac{2\bar b(u)}{\bar a(u)}du} dy \leq \sum_{j=0}^{k} \frac{k!}{(k-j)!} \Big(\frac{a_{\ell}}{2\bar b(x)}\Big)^{j}\frac{1}{\bar b(x)} \abs{x}^{k-j}, \quad x < \ell, \label{eq:m1} \\
&e^{-\int_{0}^{x} \frac{2\bar b(u)}{\bar a(u)}du} \int_{x}^{\infty} \frac{2\abs{y}^k}{\bar a(y)} e^{\int_{0}^{y} \frac{2\bar b(u)}{\bar a(u)}du} dy \leq \sum_{j=0}^{k} \frac{k!}{(k-j)!} \Big(\frac{a_{r}}{2\abs{\bar b(x)}}\Big)^{j}\frac{1}{\abs{\bar b(x)}} x^{k-j}, \quad x > r. \label{eq:m2}
\end{align}
\end{lemma}
\begin{proof}
Fix $x < \ell \leq x_0$, then
\begin{align*}
&\ e^{-\int_{0}^{x} \frac{2\bar b(u)}{\bar a(u)}du}\int_{-\infty}^{x} \frac{2\abs{y}^k}{\bar a(y)} e^{\int_{0}^{y} \frac{2\bar b(u)}{\bar a(u)}du} dy\\
 =&\ \int_{-\infty}^{x} \frac{2\abs{y}^k}{\bar a(y)} e^{-\int_{y}^{x} \frac{2\bar b(u)}{\bar a(u)}du} dy\\
\leq&\ \int_{-\infty}^{x} \frac{2\abs{y}^k}{\bar a(y)} e^{-\bar b(x) \int_{y}^{x} \frac{2}{\bar a(u)}du} dy\\
=&\ \frac{\abs{y}^k}{\bar b(x)} e^{-\bar b(x) \int_{y}^{x} \frac{2}{\bar a(u)}du}\Big{|}_{y=-\infty}^{x} + \frac{1}{\bar b(x)}\int_{-\infty}^{x}ky^{k-1} e^{-\bar b(x) \int_{y}^{x} \frac{2}{\bar a(u)}du}dy.
\end{align*}
In the first inequality we used the fact that $\bar b(y) \geq 0$ for $y \leq x_0$, and the last equality was obtained using integration by parts. At this point we invoke the assumption that $\bar a(x) = a_{\ell}$ for $x < \ell$ to see that 
\begin{align}
&\ \frac{\abs{y}^k}{\bar b(x)} e^{-\bar b(x) \int_{y}^{x} \frac{2}{\bar a(u)}du}\Big{|}_{y=-\infty}^{x} + \frac{1}{\bar b(x)}\int_{-\infty}^{x}ky^{k-1} e^{-\bar b(x) \int_{y}^{x} \frac{2}{\bar a(u)}du}dy \notag \\
=&\ \frac{\abs{y}^k}{\bar b(x)} e^{-\frac{2\bar b(x)}{a_{\ell}} (x-y)}\Big{|}_{y=-\infty}^{x} + \frac{1}{\bar b(x)}\int_{-\infty}^{x}ky^{k-1} e^{-\frac{2\bar b(x)}{a_{\ell}} (x-y)}dy \notag \\
=&\ \frac{\abs{x}^k}{\bar b(x)} + \frac{1}{\bar b(x)}\int_{-\infty}^{x}k\abs{y}^{k-1} e^{-\frac{2\bar b(x)}{a_{\ell}} (x-y)}dy. \label{eq:assumption}
\end{align}
Continuing to use integration by parts, we arrive at \eqref{eq:m1}. The case when $x > r$ is handled symmetrically.
\end{proof}
\begin{remark}
In practice, the assumption that $\bar a(x)$ is constant for $x < \ell$ may be relaxed if we can establish some control over $\int_{y}^{x} \frac{2}{\bar a(u)}du$ in order to bound the left hand side of \eqref{eq:assumption}. Same goes for the case when $x > r$.
\end{remark}

Assumption (\hyperref[eq:a2]{a2}) is made mostly for technical convenience. It is not hard to adapt the results above to the case when $\bar b(x)$ equals zero at more than one point. Assumption  (\hyperref[eq:a1]{a1}) is quite reasonable when $\bar b(x)$ is the drift of a positive recurrent diffusion process on the real line. For the diffusion process to be positive recurrent, we expect its drift to be negative when the process is far to the right of zero, and to be positive when the diffusion is far to the left of zero; cf. the requirement in \eqref{eq:densintegrable}. To further match this intuition, assumption  (\hyperref[eq:a1]{a1}) can actually be weakened to say that $\bar b(x)$ is a non-increasing function outside some compact interval around zero, and the lemmas above could be modified accordingly to deal with this. One may compare assumption  (\hyperref[eq:a1]{a1}), and its proposed relaxation, to the assumptions in \cite[Proposition 2]{KusuTudo2012}. We conclude this section with the proof of Lemma~\ref{lem:handle}.

\begin{proof}[Proof of Lemma~\ref{lem:handle}]
Differentiating both sides of the Poisson equation \eqref{eq:genericpoissonv2} yields
\begin{align*}
f_h'''(x) = -\Big(\frac{2\bar b(x)}{\bar a(x)}\Big)' f_h'(x) - \frac{2\bar b(x)}{\bar a(x)} f_h''(x) - \frac{2}{\bar a(x)}h'(x) - \frac{2\bar a'(x)}{a^2(x)}\big( \E h(V) - h(x)\big).
\end{align*}
The derivative above exists almost everywhere  because we assumed that $\bar a(x), \bar b(x)/\bar a(x)$, and $h(x)$ are all absolutely continuous. The latter assumption also implies that $f_h''(x)e^{\int_{0}^{x} \frac{2\bar b(u)}{\bar a(u)}du}$ is absolutely continuous. Hence, for any $x, \ell \in \R$,
\begin{align*}
& f_h''(x)e^{\int_{0}^{x} \frac{2\bar b(u)}{\bar a(u)}du} - f_h''(\ell)e^{\int_{0}^{\ell} \frac{2\bar b(u)}{\bar a(u)}du} \\
=&\  \int_{\ell}^{x} \frac{2\bar b(y)}{\bar a(y)} f_h''(y)e^{\int_{0}^{y} \frac{2\bar b(u)}{\bar a(u)}du} + f_h'''(y)e^{\int_{0}^{y} \frac{2\bar b(u)}{\bar a(u)}du} dy \\
=&\ \int_{\ell}^{x} \Big(-\frac{2}{\bar a(y)} h'(y) - \frac{2\bar a'(y)}{a^2(y)}\big[\E h(V) - h(y) \big] - \Big(\frac{2\bar b(y)}{\bar a(y)}\Big)' f_h'(y)\Big)e^{\int_{0}^{y} \frac{2\bar b(u)}{\bar a(u)}du} dy.
\end{align*}
To conclude \eqref{eq:hf21},  we wish to take $\ell \to -\infty$ and show that 
\begin{align*}
\lim_{\ell \to -\infty} f_h''(\ell)e^{\int_{0}^{\ell} \frac{2\bar b(u)}{\bar a(u)}du} = 0.
\end{align*}
Observe that
\begin{align*}
f_h''(\ell)e^{\int_{0}^{\ell} \frac{2\bar b(u)}{\bar a(u)}du} = \frac{2\bar b(\ell)}{\bar a(\ell)}f_h'(\ell)e^{\int_{0}^{\ell} \frac{2\bar b(u)}{\bar a(u)}du} + (\E h(V) - h(\ell))\frac{2}{\bar a(\ell)}e^{\int_{0}^{\ell} \frac{2\bar b(u)}{\bar a(u)}du}.
\end{align*}
By assumption, $\lim_{\ell\to -\infty} \frac{2\bar b(\ell)}{\bar a(\ell)}f_h'(\ell)e^{\int_{0}^{\ell} \frac{2\bar b(u)}{\bar a(u)}du}= 0$. Furthermore, \eqref{eq:densintegrable} implies that $\lim_{\ell\to -\infty}\frac{2}{\bar a(\ell)}e^{\int_{0}^{\ell} \frac{2\bar b(u)}{\bar a(u)}du} = 0$. Lastly, our assumption that $\E |h(Y)| < \infty$ means that
\begin{align*}
\int_{-\infty}^{\infty} |h(y)| \frac{2}{\bar a(y)}e^{\int_{0}^{y} \frac{2\bar b(u)}{\bar a(u)}du},
\end{align*}
which implies that $\lim_{\ell\to -\infty}h(\ell)\frac{2}{\bar a(\ell)}e^{\int_{0}^{\ell} \frac{2\bar b(u)}{\bar a(u)}du} = 0$. This proves \eqref{eq:hf21}. To prove \eqref{eq:hf22}, we take $\ell \to \infty$ and repeat the above arguments. 
\end{proof}

\section{Gradient Bounds for Chapter~\ref{chap:erlangAC}}
\label{app:gradbounds}
In Section~\ref{app:wgradient}, we first prove Lemma~\ref{lem:gradboundsCW}, establishing the Wasserstein gradient bounds for Erlang-C model.
In Section~\ref{app:grad_A}, we state and prove Lemma~\ref{lem:gradboundsAWunder},  establishing the Wasserstein gradient bounds for Erlang-A model. In Section~\ref{app:kgradient} we prove Lemmas~\ref{lem:gradboundsCK} and \ref{lem:gradboundsAK}, establishing the Kolmogorov gradient bounds for both Erlang-C and Erlang-A models. 

%Both $a(x)$ and $b(x)$ vary between the Erlang-A and Erlang-C models, but we will always use $Y(\infty)$ to denote the random variable with density
%\begin{align*}
%\frac{\frac{2}{a(x)} \exp \Big({\int_{0}^{x} \frac{2 b(u)}{a(u)} du} \Big)}{\int_{-\infty}^{\infty}\frac{2}{a(x)} \exp \Big({\int_{0}^{x} \frac{2 b(u)}{a(u)} du} \Big)dx}, \quad x \in \R.
%\end{align*}

\subsection{Erlang-C Wasserstein Gradient Bounds} \label{app:wgradient}
Recall $b(x)$ defined in \eqref{eq:b-erlang-A}. For the remainder of Section~\ref{app:wgradient}, we set
\begin{align}
\bar a(x) = 2\mu, \quad \text{ and } \quad
\bar b(x)  = b(x) = 
\begin{cases}
-\mu x, \quad x \leq -\zeta,\\
\mu \zeta, \quad x \geq -\zeta,
\end{cases} \label{CW:ab}
\end{align} 
where $\zeta = \delta(R - n) < 0$. Observe that this $\bar b(x)$ satisfies both  (\hyperref[eq:a1]{a1}) and (\hyperref[eq:a2]{a2}), and that $x_0$ from \eqref{eq:xnot} equals zero. Furthermore,
\begin{align}\label{eq:pic}
\exp \Big({\int_{0}^{x} \frac{2 \bar b(u)}{\bar a(u)} du} \Big) = 
\begin{cases}
 e^{-\frac{1}{2}x^2}, \quad x \leq - \zeta,\\
e^{-\abs{\zeta} (x+\zeta)}e^{-\frac{1}{2}\zeta^2}, \quad x \geq -\zeta.
\end{cases} 
\end{align}
Fix $h(x) \in \lipone$; without loss of generality we assume that $h(0) = 0$. 

The following lemma presents several bounds that will be used to prove Lemma~\ref{lem:gradboundsCW}.

\begin{lemma}
\label{lem:lowlevelCW}
Let $\bar a(x)$ and $\bar b(x)$ be as in \eqref{CW:ab}. Then
\allowdisplaybreaks
\begin{align}
&e^{-\int_{0}^{x} \frac{2\bar b(u)}{\bar a(u)}du}\int_{-\infty}^{x} \frac{2}{\bar a(y)} e^{\int_{0}^{y} \frac{2\bar b(u)}{\bar a(u)}du} dy\leq 
\begin{cases}
\frac{2}{\mu }, \quad x \leq 0, \label{CW:fbound1} \\
\frac{1}{\mu} e^{\zeta^2/2}(2 + \abs{\zeta}), \quad x \in [0,-\zeta],
\end{cases}\\
&e^{-\int_{0}^{x} \frac{2\bar b(u)}{\bar a(u)}du}\int_{x}^{\infty} \frac{2}{\bar a(y)} e^{\int_{0}^{y} \frac{2\bar b(u)}{\bar a(u)}du} dy \leq 
\begin{cases}
\frac{2}{\mu} + \frac{1}{\mu \abs{\zeta}}, \quad x \in [0,-\zeta], \\
\frac{1}{\mu \abs{\zeta}}, \quad x \geq -\zeta,
\end{cases} \label{CW:fbound2} \\
&e^{-\int_{0}^{x} \frac{2\bar b(u)}{\bar a(u)}du}\int_{-\infty}^{x} \frac{2\abs{y}}{\bar a(y)} e^{\int_{0}^{y} \frac{2\bar b(u)}{\bar a(u)}du} dy\leq 
\begin{cases}
\frac{1}{\mu }, \quad x \leq 0,\\
\frac{1}{\mu }( 2e^{\frac{1}{2}\zeta^2} -1), \quad x \in [0,-\zeta],
\end{cases} \label{CW:fbound3} \\
&e^{-\int_{0}^{x} \frac{2\bar b(u)}{\bar a(u)}du}\int_{x}^{\infty} \frac{2\abs{y}}{\bar a(y)} e^{\int_{0}^{y} \frac{2\bar b(u)}{\bar a(u)}du}dy \leq 
\begin{cases}
\frac{2}{\mu} + \frac{1}{\mu \zeta^2}, \quad x \in [0,-\zeta],\\
\frac{x}{\mu \abs{\zeta}} + \frac{1}{\mu \zeta^2}, \quad x \geq -\zeta \label{CW:fbound4},
\end{cases}\\
&\E \abs{Y(\infty)} \leq \frac{1}{\abs{\zeta}} + 1. \label{CW:fbound7}
\end{align}
\end{lemma}

\begin{proof}[Proof of Lemma~\ref{lem:lowlevelCW} ]
We first prove \eqref{CW:fbound1}. When $x \leq 0$, we can choose $c_1 = -1$ in  \eqref{eq:near_orig_left} to see that
\begin{align*}
e^{-\int_{0}^{x} \frac{2\bar b(u)}{\bar a(u)}du}\int_{-\infty}^{x} \frac{2}{\bar a(y)} e^{\int_{0}^{y} \frac{2\bar b(u)}{\bar a(u)}du} dy \leq \frac{1}{\bar b(-1)} +  \sup_{y \in [-1,0]} \frac{2}{\bar a(y)} = \frac{2}{\mu}.
\end{align*}
For $x \in [0,-\zeta]$, 
\begin{align*}
&e^{-\int_{0}^{x} \frac{2\bar b(u)}{\bar a(u)}du}\int_{-\infty}^{x} \frac{2}{\bar a(y)} e^{\int_{0}^{y} \frac{2\bar b(u)}{\bar a(u)}du} dy \\
=&\ e^{-\int_{0}^{x} \frac{2\bar b(u)}{\bar a(u)}du}\int_{-\infty}^{0} \frac{2}{\bar a(y)} e^{\int_{0}^{y} \frac{2\bar b(u)}{\bar a(u)}du} dy + e^{-\int_{0}^{x} \frac{2\bar b(u)}{\bar a(u)}du}\int_{0}^{x} \frac{2}{\bar a(y)} e^{\int_{0}^{y} \frac{2\bar b(u)}{\bar a(u)}du} dy \\
=&\ \frac{e^{x^2/2}}{e^{-\int_{0}^{0} \frac{2\bar b(u)}{\bar a(u)}du}} e^{-\int_{0}^{0} \frac{2\bar b(u)}{\bar a(u)}du} \int_{-\infty}^{0} \frac{2}{\bar a(y)} e^{\int_{0}^{y} \frac{2\bar b(u)}{\bar a(u)}du} dy + e^{x^2/2}\int_{0}^{x} \frac{1}{\mu} e^{-y^2/2} dy \\
\leq&\ e^{\zeta^2/2} \frac{2}{\mu } + \frac{1}{\mu} e^{\zeta^2/2}\abs{\zeta}\\
=&\ \frac{1}{\mu} e^{\zeta^2/2}(2 + \abs{\zeta}).
\end{align*}
We now prove \eqref{CW:fbound2}. When $x \geq -\zeta$, we use \eqref{eq:main2} to see that
\begin{align*}
e^{-\int_{0}^{x} \frac{2\bar b(u)}{\bar a(u)}du}\int_{x}^{\infty} \frac{2}{\bar a(y)} e^{\int_{0}^{y} \frac{2\bar b(u)}{\bar a(u)}du} dy \leq  \frac{1}{\bar b(-\zeta)} =  \frac{1}{\mu \abs{\zeta}}.
\end{align*}
When $x \in [0,-\zeta]$, 
\begin{align}
&e^{-\int_{0}^{x} \frac{2\bar b(u)}{\bar a(u)}du}\int_{x}^{\infty} \frac{2}{\bar a(y)} e^{\int_{0}^{y} \frac{2\bar b(u)}{\bar a(u)}du} dy \notag \\
=&\ e^{-\int_{0}^{x} \frac{2\bar b(u)}{\bar a(u)}du}\int_{x}^{-\zeta} \frac{2}{\bar a(y)} e^{\int_{0}^{y} \frac{2\bar b(u)}{\bar a(u)}du} dy + \frac{e^{-\int_{0}^{x} \frac{2\bar b(u)}{\bar a(u)}du}}{e^{-\int_{0}^{-\zeta} \frac{2\bar b(u)}{\bar a(u)}du}} e^{-\int_{0}^{-\zeta} \frac{2\bar b(u)}{\bar a(u)}du} \int_{-\zeta}^{\infty} \frac{2}{\bar a(y)} e^{\int_{0}^{y} \frac{2\bar b(u)}{\bar a(u)}du} dy \notag \\
\leq&\ e^{-\int_{0}^{x} \frac{2\bar b(u)}{\bar a(u)}du}\int_{x}^{-\zeta} \frac{2}{\bar a(y)} e^{\int_{0}^{y} \frac{2\bar b(u)}{\bar a(u)}du} dy + e^{-\int_{0}^{-\zeta} \frac{2\bar b(u)}{\bar a(u)}du} \int_{-\zeta}^{\infty} \frac{2}{\bar a(y)} e^{\int_{0}^{y} \frac{2\bar b(u)}{\bar a(u)}du} dy \notag \\
\leq&\ e^{-\int_{0}^{x} \frac{2\bar b(u)}{\bar a(u)}du}\int_{x}^{-\zeta} \frac{2}{\bar a(y)} e^{\int_{0}^{y} \frac{2\bar b(u)}{\bar a(u)}du} dy + \frac{1}{\mu \abs{\zeta}}. \label{CW:same}
\end{align}
We now bound the first term on the right hand side above. When $\abs{\zeta} \geq 1$, we use \eqref{eq:near_orig_right} with $c_2 = 1$ to see that 
\begin{align*}
e^{-\int_{0}^{x} \frac{2\bar b(u)}{\bar a(u)}du}\int_{x}^{-\zeta} \frac{2}{\bar a(y)} e^{\int_{0}^{y} \frac{2\bar b(u)}{\bar a(u)}du} dy \leq&\ e^{-\int_{0}^{x} \frac{2\bar b(u)}{\bar a(u)}du}\int_{x}^{\infty} \frac{2}{\bar a(y)} e^{\int_{0}^{y} \frac{2\bar b(u)}{\bar a(u)}du} dy \\
\leq&\ \frac{1}{\abs{\bar b(1)}} + \frac{1}{\mu } = \frac{2}{\mu }.
\end{align*}
When $\abs{\zeta} \leq 1$, 
\begin{align*}
e^{-\int_{0}^{x} \frac{2\bar b(u)}{\bar a(u)}du}\int_{x}^{-\zeta} \frac{2}{\bar a(y)} e^{\int_{0}^{y} \frac{2\bar b(u)}{\bar a(u)}du} dy \leq&\ e^{-\int_{0}^{1} \frac{2\bar b(u)}{\bar a(u)}du}\int_{0}^{1} \frac{2}{\bar a(y)}dy \leq  e^{1/2} \frac{1}{\mu } \leq \frac{2}{\mu }.
\end{align*}
Therefore, for $x \in [0,-\zeta]$,
\begin{align*}
e^{-\int_{0}^{x} \frac{2\bar b(u)}{\bar a(u)}du}\int_{x}^{\infty} \frac{2}{\bar a(y)} e^{\int_{0}^{y} \frac{2\bar b(u)}{\bar a(u)}du} dy \leq \frac{2}{\mu} + \frac{1}{\mu \abs{\zeta}}.
\end{align*}
To prove \eqref{CW:fbound3}, observe that when $x \leq 0$,
\begin{align*}
e^{-\int_{0}^{x} \frac{2\bar b(u)}{\bar a(u)}du}\int_{-\infty}^{x} \frac{2\abs{y}}{\bar a(y)} e^{\int_{0}^{y} \frac{2\bar b(u)}{\bar a(u)}du} dy =&\ e^{\frac{1}{2}x^2} \int_{-\infty}^{x}\frac{-y}{\mu} e^{-\frac{1}{2}y^2} dy = \frac{1}{\mu },
\end{align*}
and when $x \in [0,-\zeta]$,
\begin{align*}
e^{-\int_{0}^{x} \frac{2\bar b(u)}{\bar a(u)}du}\int_{-\infty}^{x} \frac{2\abs{y}}{\bar a(y)} e^{\int_{0}^{y} \frac{2\bar b(u)}{\bar a(u)}du} dy =&\ e^{\frac{1}{2}x^2} \int_{-\infty}^{0}\frac{-y}{\mu } e^{-\frac{1}{2}y^2} dy + e^{\frac{1}{2}x^2} \int_{0}^{x}\frac{y}{\mu} e^{-\frac{1}{2}y^2} dy\\
=&\ \frac{1}{\mu } e^{\frac{1}{2}x^2} + \frac{1}{\mu } e^{\frac{1}{2}x^2}(1 - e^{-\frac{1}{2}x^2}).
\end{align*}
We now prove \eqref{CW:fbound4}. Since $\bar a(x) \equiv 2\mu$, we can use \eqref{eq:m2} to see that for $x \geq -\zeta$,
\begin{align*}
e^{-\int_{0}^{x} \frac{2\bar b(u)}{\bar a(u)}du}\int_{x}^{\infty} \frac{2\abs{y}}{\bar a(y)} e^{\int_{0}^{y} \frac{2\bar b(u)}{\bar a(u)}du}dy \leq \frac{x}{\abs{\bar b(x)}} + \frac{2\mu}{2\abs{\bar b(x)}} \frac{1}{\abs{\bar b(x)}} =  \frac{x}{\mu \abs{\zeta}} + \frac{1}{\mu \zeta^2}.
\end{align*}
Furthermore, for $x \in [0,-\zeta]$,
\begin{align*}
& e^{-\int_{0}^{x} \frac{2\bar b(u)}{\bar a(u)}du}\int_{x}^{\infty} \frac{2\abs{y}}{\bar a(y)} e^{\int_{0}^{y} \frac{2\bar b(u)}{\bar a(u)}du}dy \\
=&\ e^{-\int_{0}^{x} \frac{2\bar b(u)}{\bar a(u)}du}\int_{x}^{-\zeta} \frac{2\abs{y}}{\bar a(y)} e^{\int_{0}^{y} \frac{2\bar b(u)}{\bar a(u)}du}dy + \frac{e^{-\int_{0}^{x} \frac{2\bar b(u)}{\bar a(u)}du}}{e^{-\int_{0}^{-\zeta} \frac{2\bar b(u)}{\bar a(u)}du}} e^{-\int_{0}^{-\zeta} \frac{2\bar b(u)}{\bar a(u)}du}\int_{-\zeta}^{\infty} \frac{2\abs{y}}{\bar a(y)} e^{\int_{0}^{y} \frac{2\bar b(u)}{\bar a(u)}du}dy \\
=&\  e^{\frac{1}{2}x^2} \int_{x}^{-\zeta} \frac{y}{\mu} e^{-\frac{1}{2}y^2} dy + \frac{e^{-\int_{0}^{x} \frac{2\bar b(u)}{\bar a(u)}du}}{e^{-\int_{0}^{-\zeta} \frac{2\bar b(u)}{\bar a(u)}du}} e^{-\int_{0}^{-\zeta} \frac{2\bar b(u)}{\bar a(u)}du}\int_{-\zeta}^{\infty} \frac{2\abs{y}}{\bar a(y)} e^{\int_{0}^{y} \frac{2\bar b(u)}{\bar a(u)}du}dy\\
=&\ \frac{1}{\mu } e^{\frac{1}{2}x^2}(e^{-\frac{1}{2}x^2} - e^{-\frac{1}{2}\zeta^2}) + \frac{e^{-\int_{0}^{x} \frac{2\bar b(u)}{\bar a(u)}du}}{e^{-\int_{0}^{-\zeta} \frac{2\bar b(u)}{\bar a(u)}du}} e^{-\int_{0}^{-\zeta} \frac{2\bar b(u)}{\bar a(u)}du}\int_{-\zeta}^{\infty} \frac{2\abs{y}}{\bar a(y)} e^{\int_{0}^{y} \frac{2\bar b(u)}{\bar a(u)}du}dy\\
\leq&\ \frac{1}{\mu } + \frac{e^{-\int_{0}^{x} \frac{2\bar b(u)}{\bar a(u)}du}}{e^{-\int_{0}^{-\zeta} \frac{2\bar b(u)}{\bar a(u)}du}}\Big( \frac{\abs{\zeta}}{\mu \abs{\zeta}} + \frac{1}{\mu \zeta^2}\Big) \\
\leq&\ \frac{1}{\mu }+\Big( \frac{1}{\mu } + \frac{1}{\mu \zeta^2}\Big),
\end{align*}
where in the last inequality, we used the fact that $ e^{-\int_{0}^{x} \frac{2\bar b(u)}{\bar a(u)}du} \leq e^{-\int_{0}^{-\zeta} \frac{2\bar b(u)}{\bar a(u)}du}$.  This proves \eqref{CW:fbound4}, and we move on to prove \eqref{CW:fbound7}. Letting $V(x) = x^2$, and recalling the form of $G_Y$ from \eqref{CA:GY}, we consider 
\begin{align}
G_Y V(x) =&\ 2x\mu (\zeta + (x + \zeta)^-) + 2\mu \notag \\
=&\ -2\mu x^2 1(x < -\zeta) - 2x\mu \abs{\zeta} 1(x \geq -\zeta) + 2\mu. \label{eq:cgradlyap}
\end{align}
By the standard Foster-Lyapunov condition (see  \cite[Theorem 4.3]{MeynTwee1993b} for example), this implies that
\begin{align*}
2\E \big[ (Y(\infty))^2 1(Y(\infty) < -\zeta)\big] + 2\abs{\zeta} \E \big[Y(\infty)1(Y(\infty) \geq -\zeta)\big] \leq 2,
\end{align*}
and in particular, 
\begin{align*}
&\E \big[Y(\infty)1(Y(\infty) \geq -\zeta)\big] \leq \frac{1}{\abs{\zeta}}, \\
& \E \Big[\big|Y(\infty)1(Y(\infty) < -\zeta)\big|\Big] \leq  \sqrt{\E \big[ (Y(\infty))^2 1(Y(\infty) < -\zeta)\big]} \leq 1,
\end{align*}
where we applied Jensen's inequality in the second set of inequalities. This concludes the proof of Lemma~\ref{lem:lowlevelCW}. 
\end{proof}

\begin{proof}[Proof of Lemma~\ref{lem:gradboundsCW}]
Let $\bar b(x)$ and $\bar a(x)$ be as in \eqref{CW:ab}. We begin by bounding $f_h'(x)$. Observe that since $h(x) \in \lipone$ and $h(0) = 0$, then \eqref{eq:fprimeneg} and \eqref{eq:fprimepos} imply that 
\begin{align*}
f_h'(x) \leq&\ e^{-\int_{0}^{x} \frac{2\bar b(u)}{\bar a(u)}du}\int_{-\infty}^{x} \frac{2}{\bar a(y)} ( |y| + \E \abs{Y(\infty)}) e^{\int_{0}^{y} \frac{2\bar b(u)}{\bar a(u)}du} dy, \\
f_h'(x) \leq &\ e^{-\int_{0}^{x} \frac{2\bar b(u)}{\bar a(u)}du} \int_{x}^{\infty} \frac{2}{\bar a(y)} ( |y| + \E \abs{Y(\infty)}) e^{\int_{0}^{y} \frac{2\bar b(u)}{\bar a(u)}du} dy.
\end{align*}
For $x \leq -\zeta$, we apply \eqref{CW:fbound1}, \eqref{CW:fbound3}, and \eqref{CW:fbound7} to the first inequality above, and for $x \geq 0$, we apply \eqref{CW:fbound2}, \eqref{CW:fbound4}, and \eqref{CW:fbound7} to the second inequality above to see that
\begin{align}
\mu \abs{f_h'(x)} \leq&\ 1 + 2 \Big( 1 + \frac{1}{\abs{\zeta}} \Big) \leq 3 + 2/\abs{\zeta}, \quad x \leq 0, \notag \\
\mu \abs{f_h'(x)} \leq&\ \min\bigg\{ 2e^{\frac{1}{2}\zeta^2} -1 + e^{\frac{1}{2}\zeta^2}(2+\abs{\zeta})\Big( 1 + \frac{1}{\abs{\zeta}}\Big), \notag \\
&\ \hspace{1.5cm} 2 + \frac{1}{\zeta^2} + \Big(2 + \frac{1}{\abs{\zeta}} \Big)\Big( 1 + \frac{1}{\abs{\zeta}} \Big) \bigg\}, \quad x \in [0,-\zeta], \notag \\
\mu \abs{f_h'(x)} \leq&\ \frac{x}{\abs{\zeta}} + \frac{1}{\zeta^2} + \frac{1}{\abs{\zeta}} \Big( 1 + \frac{1}{\abs{\zeta}} \Big) \leq \frac{1}{\abs{\zeta}}(x + 1 + 2/\abs{\zeta}), \quad x \geq -\zeta. \label{eq:cgrad1}
\end{align}
For $x \in [0, -\zeta]$, observe that when $\abs{\zeta} \leq 1$, then 
\begin{align*}
2e^{\frac{1}{2}\zeta^2} -1 + (2+\abs{\zeta})e^{\frac{1}{2}\zeta^2}\Big( 1 + \frac{1}{\abs{\zeta}}\Big) \leq 3.3 - 1 + 5\Big( 1 + \frac{1}{\abs{\zeta}}\Big) = 7.5 + 5/\abs{\zeta},
\end{align*}
and when $\abs{\zeta} \geq 1$, then 
\begin{align*}
2 + \frac{1}{\zeta^2} + \Big(2 + \frac{1}{\abs{\zeta}} \Big)\Big( 1 + \frac{1}{\abs{\zeta}} \Big) \leq 3 + 3\Big( 1 + \frac{1}{\abs{\zeta}} \Big) = 6 + 3/\abs{\zeta}.
\end{align*}
Therefore, 
\begin{align} \label{eq:dsimple}
\abs{f_h'(x)} \leq  
\begin{cases}
\frac{1}{\mu }(7.5 + 5/\abs{\zeta}), \quad x \leq -\zeta,\\
\frac{1}{\mu }\frac{1}{\abs{\zeta}}(x + 1 + 2/\abs{\zeta}), \quad x \geq -\zeta.
\end{cases}
\end{align}
Before proceeding to bound $\abs{f_h''(x)}$ and $\abs{f_h'''(x)}$, we first note that both \eqref{eq:mlimits} and \eqref{eq:plimits} are satisfied. This is because $\bar a(x)$ is constant, $\bar b(x)$ is piecewise linear,  $f_h'(x)$ is bounded as in \eqref{eq:dsimple}, but $e^{\int_{0}^{x} \frac{2\bar b(u)}{\bar a(u)}du}$ decays exponentially fast as $x \to \infty$, and decays even faster as $x \to -\infty$. To bound $\abs{f_h''(x)}$, we use \eqref{eq:hf21} and \eqref{eq:hf22}, together with the facts that $\bar a(x)$ is constant, $h(x) \in \lipone$,  and
\begin{align*}
\bar b'(x) = -\mu 1(x < -\zeta), \quad x \in \R,
\end{align*}
to see that 
\begin{align}
\abs{f_h''(x)} \leq&\ e^{-\int_{0}^{x} \frac{2\bar b(u)}{\bar a(u)}du}\int_{-\infty}^{x} \frac{2}{\bar a(y)}\big(1 + \mu \abs{f_h'(y)} 1(y < -\zeta)\big) e^{\int_{0}^{y} \frac{2\bar b(u)}{\bar a(u)}du} dy \label{eq:d21bound}\\
\abs{f_h''(x)} \leq&\ e^{-\int_{0}^{x} \frac{2\bar b(u)}{\bar a(u)}du} \int_{x}^{\infty} \frac{2}{\bar a(y)}\big(1 + \mu \abs{f_h'(y)} 1(y < -\zeta)\big) e^{\int_{0}^{y} \frac{2\bar b(u)}{\bar a(u)}du} dy \label{eq:d22bound}.
\end{align}
We know $\abs{f_h'(x)}$ is bounded as in \eqref{eq:dsimple}. For $x \leq -\zeta$, we apply \eqref{CW:fbound1} to \eqref{eq:d21bound} and for $x \geq 0$ we apply \eqref{CW:fbound2} to \eqref{eq:d22bound} to conclude that
\begin{align} \label{eq:der2_final}
\mu \abs{f_h''(x)} \leq 
\begin{cases}
2 \big( 1 + 7.5 + 5/\abs{\zeta}\big), \quad x \leq 0,\\
\min \big\{(2+\abs{\zeta})e^{\frac{1}{2}\zeta^2}, 2 + \frac{1}{\abs{\zeta}}\big\} \big( 1 + 7.5 + 5/\abs{\zeta}\big), \quad x \in [0,-\zeta],\\
\frac{1}{\abs{\zeta}}, \quad x \geq -\zeta.
\end{cases}
\end{align}
By considering separately the cases when $\abs{\zeta} \leq 1/2$ and $\abs{\zeta} \geq 1/2$, we see that
\begin{align}
\min \Big\{(2+\abs{\zeta})e^{\frac{1}{2}\zeta^2},2 + \frac{1}{\abs{\zeta}}\Big\} \leq 4, \label{eq:arithmetic}
\end{align}
and therefore,
\begin{align} \label{eq:der2_simple_bound}
\abs{f_h''(x)} \leq 
\begin{cases}
\frac{34}{\mu }( 1 + 1/\abs{\zeta}), \quad x \leq -\zeta,\\
\frac{1}{\mu \abs{\zeta}}, \quad x \geq -\zeta.
\end{cases}
\end{align}
Lastly, we bound $\abs{f_h'''(x)}$, which exists for all $x \in \R$ where $h'(x)$ and $\bar b'(x)$ exist. Since $\bar a(x)$ is a constant and $h(x) \in \lipone$, we know from \eqref{eq:fppp}  that 
\begin{align*}
\abs{f_h'''(x)} \leq&\ \frac{1}{\mu } \big(1 + \abs{f_h''(x)\bar b(x)} + \abs{f_h'(x) \bar b'(x)}\big).
\end{align*} 
For $x \geq -\zeta$, we use the forms of $\bar b(x)$ and $\bar b'(x)$ together with the bounds on $\abs{f_h'(x)}$ and $\abs{f_h''(x)}$ in \eqref{eq:dsimple} and \eqref{eq:der2_simple_bound} to see that 
\begin{align*}
\abs{f_h'''(x)} \leq&\ \frac{1}{\mu } \Big( 1 + \mu \abs{\zeta} \frac{1}{\mu \abs{\zeta}}\Big).
\end{align*}
Although tempting, it is not sufficient to use the bound on $\abs{f_h''(x)}$ in \eqref{eq:der2_final} and the form of $\bar b(x)$  to bound $\abs{f_h''(x)\bar b(x)}$ for all $x \leq -\zeta$. Instead, we multiply both sides of \eqref{eq:d21bound} and \eqref{eq:d22bound} by $\abs{\bar b(x)}$ to see that
\begin{align}
&\norm{f_h''(x) \bar b(x)} \notag \\
 \leq&\ \big(1 + \sup_{y \leq x} \mu \abs{f_h'(y)}\big)\abs{\bar b(x)}e^{-\int_{0}^{x} \frac{2\bar b(u)}{\bar a(u)}du}\int_{-\infty}^{x} \frac{2}{\bar a(y)} e^{\int_{0}^{y} \frac{2\bar b(u)}{\bar a(u)}du} dy, \quad x \leq 0, \notag \\
&\norm{f_h''(x) \bar b(x)} \notag  \\
\leq&\ \big(1 + \sup_{y \in [x,-\zeta]} \mu \abs{f_h'(y)}\big)\abs{\bar b(x)}e^{-\int_{0}^{x} \frac{2\bar b(u)}{\bar a(u)}du}\int_{x}^{\infty} \frac{2}{\bar a(y)} e^{\int_{0}^{y} \frac{2\bar b(u)}{\bar a(u)}du} dy, \quad x \in [0,-\zeta]. \label{eq:cgradfb}
\end{align}
By invoking \eqref{eq:main1} and \eqref{eq:main2}, together with the bound on $\abs{f_h'(x)}$ from \eqref{eq:dsimple}, we conclude that
\begin{align*}
\abs{f_h'''(x)}\leq&\ \frac{1}{\mu }\big(  1 + (1 +  7.5 + 5/\abs{\zeta}) + 7.5 + 5/\abs{\zeta}\big), \quad x \leq -\zeta.
\end{align*}
Therefore, for those $x \in \R$ where $h'(x)$ and $\bar b'(x)$ exist,
\begin{align*}
\abs{f_h'''(x)} \leq
\begin{cases}
\frac{1}{\mu }( 17 + 10/\abs{\zeta}) , \quad x \leq -\zeta,\\
\frac{2}{\mu }, \quad x \geq -\zeta.
\end{cases}
\end{align*}
This concludes the proof of Lemma~\ref{lem:gradboundsCW}.
\end{proof}

\subsection{Erlang-A Wasserstein Gradient Bounds} \label{app:grad_A}
Below we prove the Erlang-A gradient bounds, which were stated in Lemma~\ref{lem:gradboundsAWunder} of Section~\ref{sec:CAapplemmas}. Their proof is similar to that of Lemma~\ref{lem:gradboundsCW}. We only outline the necessary steps needed for a proof, and emphasize all the differences with the proof of Lemma~\ref{lem:gradboundsCW}.

\subsubsection{Proof Outline for Lemma~\ref{lem:gradboundsAWunder}: The Underloaded System} \label{app:gradAWunder}
In the Erlang-A model, 
\begin{align*}
\bar b(x) = 
\begin{cases}
-\mu x, \quad x \leq -\zeta,\\
-\alpha(x+\zeta)+ \mu \zeta, \quad x \geq -\zeta,
\end{cases}
\end{align*}
and $\bar a(x) = 2\mu$. To prove Lemma~\ref{lem:gradboundsAWunder}, we need the following version of Lemma~\ref{lem:lowlevelCW}. 

\begin{lemma} \label{lem:lowlevWunder}
Consider the Erlang-A model ($\alpha > 0$) with $0 < R \leq n$.  Then there exists a constant $C$, independent of $\lambda,\mu,n$, and $\alpha$, such that
\allowdisplaybreaks
\begin{align}
&e^{-\int_{0}^{x} \frac{2\bar b(u)}{\bar a(u)}du}\int_{-\infty}^{x} \frac{2}{\bar a(y)} e^{\int_{0}^{y} \frac{2\bar b(u)}{\bar a(u)}du} dy \leq
\begin{cases}
\frac{C}{\mu }, \quad x \leq 0 , \\
\frac{C}{\mu }e^{\frac{\zeta^2}{2}}, \quad x \in [0,-\zeta],
\end{cases}\label{eq:ingredient1}\\
&e^{-\int_{0}^{x} \frac{2\bar b(u)}{\bar a(u)}du}\int_{x}^{\infty} \frac{2}{\bar a(y)} e^{\int_{0}^{y} \frac{2\bar b(u)}{\bar a(u)}du} dy \leq
\begin{cases}
\frac{C}{\mu }\big(1  +\sqrt{\frac{\mu}{\alpha}}\wedge \frac{1}{\abs{\zeta}}\big), \quad x \in [0,-\zeta], \\
\frac{C}{\mu }\big(\sqrt{\frac{\mu}{\alpha}}\wedge \frac{1}{\abs{\zeta}}\big), \quad x \geq -\zeta,
\end{cases} \label{eq:ingredient2} \\
&e^{-\int_{0}^{x} \frac{2\bar b(u)}{\bar a(u)}du}\int_{-\infty}^{x} \frac{2\abs{y}}{\bar a(y)} e^{\int_{0}^{y} \frac{2\bar b(u)}{\bar a(u)}du} dy \leq
\begin{cases}
\frac{C}{\mu }, \quad x \leq 0,\\
\frac{C}{\mu }e^{\frac{\zeta^2}{2}}, \quad x \in [0,-\zeta],
\end{cases} \label{eq:ingredient3} \\
&e^{-\int_{0}^{x} \frac{2\bar b(u)}{\bar a(u)}du}\int_{x}^{\infty} \frac{2\abs{y}}{\bar a(y)} e^{\int_{0}^{y} \frac{2\bar b(u)}{\bar a(u)}du}dy \leq
\begin{cases}
\frac{C}{\mu }\big(1+\frac{1}{\zeta^2}\big), \quad x \in [0,-\zeta],\\
\frac{C}{\mu }\big(1+\frac{\mu}{\alpha}\big), \quad x \geq -\zeta \label{eq:ingredient4},
\end{cases}\\
&\E \abs{Y(\infty)} \leq 1+\sqrt{\frac{\mu}{\alpha}}\wedge \frac{1}{\abs{\zeta}} . \label{eq:ingredient5}
\end{align}
\end{lemma}
\noindent To prove this lemma, we first observe that 
%\begin{align}
%e^{\int_{0}^{x} \frac{2\bar b(u)}{\bar a(u)}du} =  e^{-\frac{1}{2}x^2},\quad x\leq -\zeta, \label{eq:densAunder}
%\end{align}
\begin{align}
e^{\int_{0}^{x} \frac{2\bar b(u)}{\bar a(u)}du} =
\begin{cases}
e^{-\frac{1}{2}x^2},\quad x\leq -\zeta,\\
e^{-\frac{1}{2}\zeta^2} e^{\frac{\mu}{2\alpha}\zeta^2} e^{-\frac{\alpha}{2\mu}\left(x+\zeta-\frac{\mu}{\alpha}\zeta\right)^2},\quad x\geq -\zeta,
\end{cases} \label{eq:densAunder}
\end{align}
By comparing \eqref{eq:densAunder} to \eqref{eq:pic} for the region $x \leq -\zeta$, we immediately see that \eqref{eq:ingredient1} and  \eqref{eq:ingredient3} are restatements of \eqref{CW:fbound1} and  \eqref{CW:fbound3}, from Lemma~\ref{lem:lowlevelCW}, and hence have already been established. The proof of \eqref{eq:ingredient5} involves applying $G_Y$ to the Lyapunov function $V(x) = x^2$ to see that
\begin{align*}
G_Y V(x) =&\ -2\mu x^2 1(x < -\zeta) + 2\big(-\alpha x^2 + x\zeta(\mu - \alpha) \big)1(x \geq -\zeta) + 2\mu \\
\leq&\ -2\mu x^2 1(x < -\zeta) -2 (\alpha \wedge\mu ) x^2 1(x \geq -\zeta) + 2\mu,  
\end{align*}
and 
\begin{align*}
G_Y V(x) =&\ -2\mu x^2 1(x < -\zeta) + 2\big(-\alpha x(x+\zeta) - \mu \abs{\zeta}x \big)1(x \geq -\zeta) + 2\mu \\
\leq&\ -2\mu x^2 1(x < -\zeta) -2\mu \abs{\zeta} x1(x \geq -\zeta) + 2\mu.
\end{align*}
One can compare these inequalities to \eqref{eq:cgradlyap} in the proof of Lemma~\ref{lem:lowlevelCW} to see that \eqref{eq:ingredient5} follows by the Foster-Lyapunov condition.

We now go over the proofs of \eqref{eq:ingredient2} and \eqref{eq:ingredient4}. We first prove \eqref{eq:ingredient2} when $x \in [0,-\zeta]$. Just like in \eqref{CW:same} and the displays right below it, we can show that
\begin{align*}
&e^{-\int_{0}^{x} \frac{2\bar b(u)}{\bar a(u)}du}\int_{x}^{\infty} \frac{2}{\bar a(y)} e^{\int_{0}^{y} \frac{2\bar b(u)}{\bar a(u)}du} dy \notag \\
\leq&\ e^{-\int_{0}^{x} \frac{2\bar b(u)}{\bar a(u)}du}\int_{x}^{-\zeta} \frac{2}{\bar a(y)} e^{\int_{0}^{y} \frac{2\bar b(u)}{\bar a(u)}du} dy + e^{-\int_{0}^{-\zeta} \frac{2\bar b(u)}{\bar a(u)}du} \int_{-\zeta}^{\infty} \frac{2}{\bar a(y)} e^{\int_{0}^{y} \frac{2\bar b(u)}{\bar a(u)}du} dy \notag \\
\leq&\ \frac{2}{\mu } + \frac{1}{\mu \abs{\zeta}}.
\end{align*}
It can also be checked that 
\begin{align*}
e^{-\int_{0}^{-\zeta} \frac{2\bar b(u)}{\bar a(u)}du} \int_{-\zeta}^{\infty} \frac{2}{\bar a(y)} e^{\int_{0}^{y} \frac{2\bar b(u)}{\bar a(u)}du} dy =&\ e^{\frac{1}{2}(x^2-\zeta^2)}e^{\frac{\alpha}{2\mu}\left(\frac{\mu}{\alpha}\zeta\right)^2}\int^\infty_{\abs{\zeta}}\frac{1}{\mu } e^{-\frac{\alpha}{2\mu}\left(y+\zeta-\frac{\mu}{\alpha}\zeta\right)^2}dy \\
=&\ e^{\frac{1}{2}(x^2-\zeta^2)}e^{\frac{\alpha}{2\mu}\left(\frac{\mu}{\alpha}\zeta\right)^2}\int^\infty_{\frac{\mu}{\alpha}\abs{\zeta}}\frac{1}{\mu } e^{-\frac{\alpha}{2\mu}y^2}dy \\
\leq&\ e^{\frac{\alpha}{2\mu}\left(\frac{\mu}{\alpha}\zeta\right)^2}\int^\infty_{\frac{\mu}{\alpha}\abs{\zeta}}\frac{1}{\mu } e^{-\frac{\alpha}{2\mu}y^2}dy \\
\leq&\ \int^\infty_{0}\frac{1}{\mu } e^{-\frac{\alpha}{2\mu}y^2}dy \\
=&\ \frac{1}{\mu } \sqrt{\frac{\pi}{2}\frac{\mu}{\alpha}},
\end{align*}
where in the last inequality, we used the fact that for $x \geq 0$,  the function $e^{\frac{\alpha}{2\mu}x^2}\int^\infty_{x}\frac{1}{\mu } e^{-\frac{\alpha}{2\mu}y^2}dy$ is maximized at $x = 0$ (this can be checked by differentiating the function). 
%\begin{align}
%e^{-\int_{0}^{x} \frac{2\bar b(u)}{\bar a(u)}du}\int_{x}^{\infty} \frac{2}{\bar a(y)} e^{\int_{0}^{y} \frac{2\bar b(u)}{\bar a(u)}du} dy=&\ e^{\frac{1}{2}x^2} \int_x^{\abs{\zeta}} \frac{1}{\mu } e^{-\frac{1}{2}y^2}  dy
%\notag \\
%&+ e^{\frac{1}{2}(x^2-\zeta^2)}e^{\frac{\alpha}{2\mu}\left(\frac{\mu}{\alpha}\zeta\right)^2}\int^\infty_{\abs{\zeta}}\frac{1}{\mu } e^{-\frac{\alpha}{2\mu}\left(y+\zeta-\frac{\mu}{\alpha}\zeta\right)^2}dy \notag \\
%\leq&\ e^{\frac{1}{2}x^2} \int_x^{\infty} \frac{1}{\mu } e^{-\frac{1}{2}y^2}  dy + e^{\frac{\alpha}{2\mu}\left(\frac{\mu}{\alpha}\zeta\right)^2}\int^\infty_{\frac{\mu}{\alpha}\abs{\zeta}}\frac{1}{\mu } e^{-\frac{\alpha}{2\mu}y^2}dy \notag  \\
%\leq&\ \sqrt{\frac{\pi}{2}} + \sqrt{\frac{\pi}{2}\frac{\mu}{\alpha}}\wedge \frac{1}{\abs{\zeta}} \label{eq:undergrad1}
%\end{align}
%where in the first inequality we used a change of variables and the fact that $a_+/a_- = e^{-\zeta^2/2} e^{\frac{\alpha}{2\mu } (\frac{\mu }{\alpha}\zeta)^2}$, and in the last inequality we used both \eqref{eq:usefulbound} and \eqref{eq:normcdfbound}, and the fact that $e^{\frac{1}{2}(x^2-\zeta^2)} \leq 1$. 
This proves the part of \eqref{eq:ingredient2} when $x \in [0,-\zeta]$. The case when $x \geq -\zeta$ is handled similarly.  We now prove \eqref{eq:ingredient4}. When $x \in [0,-\zeta]$, 
\begin{align*}
&\ e^{-\int_{0}^{x} \frac{2\bar b(u)}{\bar a(u)}du}\int_{x}^{\infty} \frac{2\abs{y}}{\bar a(y)} e^{\int_{0}^{y} \frac{2\bar b(u)}{\bar a(u)}du}dy \\
=&\ \frac{1}{\mu } e^{\frac{1}{2}x^2} \int_{x}^{-\zeta} y e^{-\frac{1}{2}y^2}dy + \frac{1}{\mu } e^{\frac{\alpha}{2\mu } (\frac{\mu }{\alpha}\zeta)^2} e^{\frac{1}{2}(x^2-\zeta^2)}\int_{-\zeta}^{\infty} ye^{-\frac{\alpha}{2\mu}(y+\zeta-\frac{\mu}{\alpha}\zeta)^2} dy \\
=&\ \frac{1}{\mu }(1 - e^{\frac{1}{2}(x^2-\zeta^2)}) +  \frac{1}{\mu }e^{\frac{\alpha}{2\mu } (\frac{\mu }{\alpha}\zeta)^2} e^{\frac{1}{2}(x^2-\zeta^2)}\int_{-\zeta}^{\infty} ye^{-\frac{\alpha}{2\mu}(y+\zeta-\frac{\mu}{\alpha}\zeta)^2} dy \\
\leq&\ \frac{1}{\mu } + \frac{1}{\mu }e^{\frac{\alpha}{2\mu } (\frac{\mu }{\alpha}\zeta)^2}\int_{-\zeta}^{\infty} ye^{-\frac{\alpha}{2\mu} \big( (y+\zeta)^2 -2\frac{\mu }{\alpha} (y+\zeta)\zeta + (\frac{\mu }{\alpha}\zeta)^2\big)} dy \\
\leq&\ \frac{1}{\mu } + \frac{1}{\mu }\int_{-\zeta}^{\infty} ye^{(y+\zeta)\zeta} dy = \frac{1}{\mu } + \frac{1}{\mu } + \frac{1}{\mu \zeta^2},
\end{align*}
and when $x \geq -\zeta$,
\begin{align}
&e^{-\int_{0}^{x} \frac{2\bar b(u)}{\bar a(u)}du}\int_{x}^{\infty} \frac{2\abs{y}}{\bar a(y)} e^{\int_{0}^{y} \frac{2\bar b(u)}{\bar a(u)}du}dy  \notag \\
=&\  \frac{1}{\mu }e^{\frac{\alpha}{2\mu}(x+\zeta-\frac{\mu}{\alpha}\zeta)^2}\int_{x}^{\infty} ye^{-\frac{\alpha}{2\mu}(y+\zeta-\frac{\mu}{\alpha}\zeta)^2} dy  \notag \\
=&\  \frac{1}{\mu }e^{\frac{\alpha}{2\mu}(x+\zeta-\frac{\mu}{\alpha}\zeta)^2}\int_{x+\zeta-\frac{\mu}{\alpha}\zeta}^{\infty} ye^{-\frac{\alpha}{2\mu}y^2} dy  \notag \\
&+ \frac{1}{\mu }(1-\mu/\alpha)\abs{\zeta} e^{\frac{\alpha}{2\mu}(x+\zeta-\frac{\mu}{\alpha}\zeta)^2}\int_{x+\zeta-\frac{\mu}{\alpha}\zeta}^{\infty} e^{-\frac{\alpha}{2\mu}y^2} dy \notag  \\
\leq&\ \frac{1}{\mu } \frac{\mu }{\alpha} +\frac{1}{\mu } \abs{\zeta} e^{\frac{\alpha}{2\mu}(x+\zeta-\frac{\mu}{\alpha}\zeta)^2}\int_{x+\zeta-\frac{\mu}{\alpha}\zeta}^{\infty} \frac{y}{ (x + \zeta - \frac{\mu }{\alpha}\zeta)} e^{-\frac{\alpha}{2\mu}y^2} dy  \notag \\
=&\ \frac{1}{\mu }\Big(\frac{\mu }{\alpha} + \abs{\zeta} \frac{1}{\frac{\alpha}{\mu } (x + \zeta - \frac{\mu }{\alpha}\zeta)}\Big)  \leq \frac{1}{\mu }\Big( \frac{\mu }{\alpha} + 1\Big). \label{AWU:intermed}
\end{align}
We now describe how to prove Lemma~\ref{lem:gradboundsAWunder}. To prove \eqref{eq:gwu1}, we repeat the procedure used to get \eqref{eq:cgrad1}, except this time using the bounds in Lemma~\ref{lem:lowlevWunder} instead of those in Lemma~\ref{lem:lowlevelCW}. Using the resulting bounds on $f_h'(x)$, we argue that \eqref{eq:mlimits} and \eqref{eq:plimits} are true, just like we did in the proof of Lemma~\ref{lem:gradboundsCW}. We now describe how to prove \eqref{eq:gwu2}. When $x \leq 0$,  we apply \eqref{eq:gwu1} and \eqref{eq:ingredient1} to \eqref{eq:hf21}, and when $x \geq -\zeta$ we apply \eqref{eq:gwu1} and \eqref{eq:ingredient2} to \eqref{eq:hf22}. The last region, when $x \in [0,-\zeta]$, has to be handled differently depending on the size of $\abs{\zeta}$. When $\abs{\zeta} \leq 1$, we just apply \eqref{eq:gwu1} and \eqref{eq:ingredient1} to \eqref{eq:hf21}. However, when $\abs{\zeta} \geq 1$, we manipulate \eqref{eq:hf22} to see that
\begin{align}
f_h''(x) =&\ -e^{-\int_{0}^{x} \frac{2\bar b(u)}{\bar a(u)}du} \int_{x}^{-\zeta} \frac{1}{\mu }(-h'(y) + \mu f_h'(y)) e^{-\int_{0}^{y} \frac{2\bar b(u)}{\bar a(u)}du}dy  \notag  \\
&-\frac{e^{-\int_{0}^{x} \frac{2\bar b(u)}{\bar a(u)}du}}{e^{-\int_{0}^{-\zeta} \frac{2\bar b(u)}{\bar a(u)}du}} e^{-\int_{0}^{-\zeta} \frac{2\bar b(u)}{\bar a(u)}du} \int_{-\zeta}^{\infty} \frac{1}{\mu }(-h'(y) + \alpha f_h'(y)) e^{-\int_{0}^{y} \frac{2\bar b(u)}{\bar a(u)}du} dy. \label{eq:der2manip}
\end{align}
We then apply \eqref{eq:gwu1}, \eqref{eq:ingredient2}, and the fact that $\frac{e^{-\int_{0}^{x} \frac{2\bar b(u)}{\bar a(u)}du}}{e^{-\int_{0}^{-\zeta} \frac{2\bar b(u)}{\bar a(u)}du}} \leq 1$ to conclude \eqref{eq:gwu2}. The proof of \eqref{eq:gwu3} relies on \eqref{eq:fppp}, which tells us that 
\begin{align*}
\abs{f_h'''(x)} \leq&\ \frac{1}{\mu } \big[ 1 + \abs{f_h''(x)\bar b(x)} + \abs{f_h'(x) \bar b'(x)}\big].
\end{align*}
Bounding $\abs{f_h'(x) \bar b'(x)}$  only relies on \eqref{eq:gwu1}. The term $\abs{f_h''(x)\bar b(x)}$ is bounded similarly to the way it is done in Lemma~\ref{lem:gradboundsCW}; see for instance \eqref{eq:cgradfb}.  This concludes the proof outline for Lemma~\ref{lem:gradboundsAWunder} when the system is underloaded.

\subsubsection{Proof Outline for Lemma~\ref{lem:gradboundsAWunder}: The Overloaded System} \label{app:gradAWover}
For the overloaded case in Lemma~\ref{lem:gradboundsAWunder}, we again need the following version of Lemma~\ref{lem:lowlevelCW}.

\begin{lemma} \label{lem:lowlevWover}
Consider the Erlang-A model ($\alpha > 0$) with $0 < R \leq n$.  Then there exists a constant $C$, independent of $\lambda,\mu,n$, and $\alpha$, such that
\allowdisplaybreaks
\begin{align}
&e^{-\int_{0}^{x} \frac{2\bar b(u)}{\bar a(u)}du}\int_{-\infty}^{x} \frac{2}{\bar a(y)} e^{\int_{0}^{y} \frac{2\bar b(u)}{\bar a(u)}du} dy \leq 
\begin{cases}
\frac{C}{\mu }\big(1\wedge \frac{\mu}{\alpha \zeta}\big), \quad x \leq -\zeta , \\
\frac{C}{\mu }\big(1+\sqrt{\frac{\mu}{\alpha}}\wedge \zeta\big), \quad x \in [-\zeta,0],
\end{cases}\label{eq:oingredient1}\\
&e^{-\int_{0}^{x} \frac{2\bar b(u)}{\bar a(u)}du}\int_{x}^{\infty} \frac{2}{\bar a(y)} e^{\int_{0}^{y} \frac{2\bar b(u)}{\bar a(u)}du} dy  \leq
\begin{cases}
\frac{C}{\mu }\sqrt{\frac{\mu}{\alpha}} e^{\frac{\alpha}{2\mu}\zeta^2}, \quad x \in [-\zeta,0], \\
\frac{C}{\mu }\sqrt{\frac{\mu}{\alpha}}, \quad x \geq 0,
\end{cases} \label{eq:oingredient2} \\
&e^{-\int_{0}^{x} \frac{2\bar b(u)}{\bar a(u)}du}\int_{-\infty}^{x} \frac{2\abs{y}}{\bar a(y)} e^{\int_{0}^{y} \frac{2\bar b(u)}{\bar a(u)}du} dy \leq
\begin{cases}
\frac{C}{\mu }\big(1 +\zeta \wedge \frac{\mu}{\alpha}\big), \quad x \leq -\zeta,\\
\frac{C}{\mu }\big(\frac{\mu}{\alpha}+1\big), \quad x \in [-\zeta,0],
\end{cases} \label{eq:oingredient3} \\
&e^{-\int_{0}^{x} \frac{2\bar b(u)}{\bar a(u)}du}\int_{x}^{\infty} \frac{2\abs{y}}{\bar a(y)} e^{\int_{0}^{y} \frac{2\bar b(u)}{\bar a(u)}du}dy \leq
\begin{cases}
\frac{C}{\mu }\frac{\mu}{\alpha}e^{\frac{\alpha}{2\mu}\zeta^2}, \quad x \in [-\zeta,0],\\
\frac{C}{\mu }\frac{\mu}{\alpha}, \quad x \geq 0 \label{eq:oingredient4},
\end{cases}\\
&\E \abs{Y(\infty)} \leq  \sqrt{\frac{\mu}{\alpha}} + 1. \label{eq:oingredient5}
\end{align}
\end{lemma}
\noindent To prove this lemma, we first observe that $\bar a(x) = 2\mu$, 
\begin{align*}
\bar b(x) = 
\begin{cases}
-\mu (x+\zeta)+ \alpha\zeta, \quad x \leq -\zeta,\\
-\alpha x, \quad x \geq -\zeta,
\end{cases}
\end{align*}
and 
\begin{align}
e^{-\int_{0}^{x} \frac{2\bar b(u)}{\bar a(u)}du} =
\begin{cases}
e^{\frac{1}{2}\left(\frac{\alpha}{\mu}\zeta\right)^2}e^{-\frac{\alpha}{2\mu}\zeta^2}e^{-\frac{1}{2}\left(x+\zeta-\frac{\alpha}{\mu}\zeta\right)^2},\quad x\leq -\zeta,\\
e^{-\frac{\alpha}{2\mu}x^2},\quad x\geq -\zeta.
\end{cases} \label{eq:densAover}
\end{align}
Observe that  in the region $x \geq -\zeta$, the form of \eqref{eq:densAover} is very similar to the  \eqref{eq:pic} in the region $x \leq -\zeta$. Hence, one can check that the arguments needed to prove Lemma~\ref{lem:lowlevWover}'s \eqref{eq:oingredient2} and \eqref{eq:oingredient4} are nearly identical to the arguments used to prove Lemma~\ref{lem:lowlevelCW}'s \eqref{CW:fbound1} and \eqref{CW:fbound3}.

The proof of \eqref{eq:oingredient5} involves applying $G_Y$, where 
\begin{align*}
G_Y f(x) = \frac{1}{2}\bar a(x) f''(x) + \bar b(x) f'(x)
\end{align*}
to the Lyapunov function $V(x) = x^2$ to see that
\begin{align*}
G_Y V(x) =&\ -2\alpha x^2 1(x > -\zeta) + 2\big(-\mu x^2 + x\zeta(\alpha - \mu ) \big)1(x \leq -\zeta) + 2\mu \\
\leq&\ -2\alpha x^2 1(x > -\zeta) -2 (\alpha \wedge\mu ) x^2 1(x \leq -\zeta) + 2\mu.  
\end{align*}
One can compare this inequality to \eqref{eq:cgradlyap} in the proof of Lemma~\ref{lem:lowlevelCW} to see that \eqref{eq:oingredient5} follows by the Foster-Lyapunov condition. 

We now describe how to prove \eqref{eq:oingredient1} and  \eqref{eq:oingredient3}. The proof of \eqref{eq:oingredient1} uses a series of arguments similar to those in the proof of \eqref{eq:ingredient2} of Lemma~\ref{lem:lowlevWunder}. We now prove \eqref{eq:oingredient3}. When $x \leq -\zeta$, 
\begin{align}
&e^{-\int_{0}^{x} \frac{2\bar b(u)}{\bar a(u)}du}\int_{-\infty}^{x} \frac{2\abs{y}}{\bar a(y)} e^{\int_{0}^{y} \frac{2\bar b(u)}{\bar a(u)}du} dy \notag \\
 =&\  \frac{1}{\mu } e^{\frac{1}{2}(x+\zeta-\frac{\alpha}{\mu}\zeta)^2}\int_{-\infty}^{x} -ye^{-\frac{1}{2}(y+\zeta-\frac{\alpha}{\mu}\zeta)^2} dy  \notag \\
=&\  \frac{1}{\mu }e^{\frac{1}{2}(x+\zeta-\frac{\alpha}{\mu}\zeta)^2}\int_{-\infty}^{x+\zeta-\frac{\alpha}{\mu}\zeta} -ye^{-\frac{1}{2}y^2} dy \notag \\
&+\frac{1}{\mu } (1-\alpha/\mu )\zeta e^{\frac{1}{2}(x+\zeta-\frac{\alpha}{\mu}\zeta)^2}\int_{-\infty}^{x+\zeta-\frac{\alpha}{\mu}\zeta} e^{-\frac{1}{2}y^2} dy \notag  \\
\leq&\ \frac{1}{\mu } + \frac{\zeta }{\mu }\Big( \sqrt{\frac{\pi}{2}} \wedge  \frac{1}{ \frac{\alpha}{\mu }\zeta -x - \zeta } \Big) \leq 1 +  \sqrt{\frac{\pi}{2}}\zeta  \wedge \frac{\mu }{\alpha}, \label{eq:overgrad1}
\end{align}
where the second last inequality uses logic similar to what was used in \eqref{AWU:intermed}. For $x \in [-\zeta, 0]$, 
\begin{align*}
&\ e^{-\int_{0}^{x} \frac{2\bar b(u)}{\bar a(u)}du}\int_{-\infty}^{x} \frac{2\abs{y}}{\bar a(y)} e^{\int_{0}^{y} \frac{2\bar b(u)}{\bar a(u)}du} dy \\
=&\  \frac{1}{\mu } e^{-\frac{\alpha}{2\mu } \zeta^2} e^{\frac{1}{2} (\frac{\alpha}{\mu }\zeta)^2} e^{\frac{\alpha}{2\mu }x^2}\int_{-\infty}^{-\zeta} -ye^{-\frac{1}{2}(y+\zeta-\frac{\alpha}{\mu}\zeta)^2} dy +  \frac{1}{\mu }e^{\frac{\alpha}{2\mu }x^2}\int_{-\zeta}^{x} -ye^{\frac{\alpha}{2\mu }y^2} dy.
\end{align*}
Repeating arguments from \eqref{eq:overgrad1}, we can show that the first term above satisfies 
\begin{align*}
\frac{1}{\mu } e^{-\frac{\alpha}{2\mu } \zeta^2} e^{\frac{1}{2} (\frac{\alpha}{\mu }\zeta)^2} e^{\frac{\alpha}{2\mu }x^2}\int_{-\infty}^{-\zeta} -ye^{-\frac{1}{2}(y+\zeta-\frac{\alpha}{\mu}\zeta)^2} dy \leq \frac{1}{\mu }e^{\frac{\alpha}{2\mu } (x^2-\zeta^2)}\Big(1 + \frac{\mu }{\alpha}\Big),
\end{align*}
and by computing the second term explicitly, we conclude that 
\begin{align*}
e^{-\int_{0}^{x} \frac{2\bar b(u)}{\bar a(u)}du}\int_{-\infty}^{x} \frac{2\abs{y}}{\bar a(y)} e^{\int_{0}^{y} \frac{2\bar b(u)}{\bar a(u)}du} dy  \leq&\ \frac{1}{\mu }e^{\frac{\alpha}{2\mu } (x^2-\zeta^2)}\Big(1 + \frac{\mu }{\alpha}\Big) + \frac{1}{\mu }\frac{\mu }{\alpha} \big( 1 - e^{\frac{\alpha}{2\mu } (x^2-\zeta^2)}\big) \\
\leq&\ \frac{1}{\mu }\big(1 + \frac{\mu }{\alpha}\big),
\end{align*}
which proves \eqref{eq:oingredient3}.

Having argued Lemma~\ref{lem:lowlevWover}, we now use it to prove the bounds in \eqref{eq:gwo1}--\eqref{eq:gwo42}. To prove \eqref{eq:gwo1}, we repeat the procedure used to get \eqref{eq:cgrad1}, except this time using the bounds in Lemma~\ref{lem:lowlevWover} instead of those in Lemma~\ref{lem:lowlevelCW}. Using the resulting bounds on $f_h'(x)$, we argue that \eqref{eq:mlimits} and \eqref{eq:plimits} are true, just like we did in the proof of Lemma~\ref{lem:gradboundsCW}.  We now describe how to prove \eqref{eq:gwo2}. When $x \leq -\zeta$,  we apply \eqref{eq:gwo1} and \eqref{eq:oingredient1} to \eqref{eq:hf21}.  When $x \geq -\zeta$, instead of using the expressions for $f_h''(x)$ in \eqref{eq:hf21} and \eqref{eq:hf22} like we would usually do, we instead apply \eqref{eq:gwo1} to the bound 
\begin{align*}
\abs{f_h''(x)} \leq \frac{1}{\mu }\abs{f_h'(x)} \abs{\bar b(x)} + \frac{1}{\mu }\big(\abs{x} + \E \abs{Y(\infty)} \big), \quad x \in \R,
\end{align*}
which follows by rewriting the Poisson equation \eqref{CA:poisson} and using the Lipschitz property of $h(x)$. We now prove \eqref{eq:gwo3}--\eqref{eq:gwo42}. We recall \eqref{eq:fppp} to see that
\begin{align*}
\abs{f_h'''(x)} \leq&\ \frac{1}{\mu } \big[ 1 + \abs{f_h''(x)\bar b(x)} + \abs{f_h'(x) \bar b'(x)}\big].
\end{align*}
Bounding $\abs{f_h'(x) \bar b'(x)}$ is simple, and only relies on \eqref{eq:gwo1}. The other term, $\abs{f_h''(x)\bar b(x)}$, is bounded as follows. To prove \eqref{eq:gwo3}, i.e.\ when $x \leq -\zeta$, the term $\abs{f_h''(x)\bar b(x)}$ is bounded similarly to the way it is done in Lemma~\ref{lem:gradboundsCW}; see for instance \eqref{eq:cgradfb}. When $x \geq -\zeta$ then 
\begin{align*}
\abs{f_h''(x)b(x)} = \alpha \abs{x}\abs{f_h''(x)},
\end{align*}
and the difference between \eqref{eq:gwo41} and \eqref{eq:gwo42} lies in the way that the quantity above is bounded. To get \eqref{eq:gwo41}, we simply apply the bounds on $f_h''(x)$ from \eqref{eq:gwo2} to the right hand side above. 

To prove \eqref{eq:gwo42}, we will first argue that
\begin{align}
\abs{f_h'''(x)} \leq &\ 
\begin{cases}
\frac{C}{\mu}\Big(\frac{\alpha}{\mu}+\sqrt{\frac{\alpha}{\mu}}+1\Big)
+ \frac{C}{\mu}\Big(\frac{\alpha}{\mu}+\sqrt{\frac{\alpha}{\mu}}+1\Big)^2 \abs{x},\quad x\in [-\zeta,0],\\
\frac{C}{\mu}\Big(\frac{\alpha}{\mu}+\sqrt{\frac{\alpha}{\mu}}+1\Big),\quad x\geq 0,
\end{cases} \label{eq:altern_third}
\end{align}
where $C$ is some positive constant independent of everything else; this will imply \eqref{eq:gwo42}. The only difference between the proof of \eqref{eq:altern_third} and the  bound on $f_h'''(x)$ in \eqref{eq:gwo41} is in how $\abs{f_h''(x)b(x)}$ is bounded; we now describe the different way to bound $\abs{f_h''(x)b(x)}$. When $x \geq 0$, we bound $\abs{f_h''(x)\bar b(x)}$ just like we did in Lemma~\ref{lem:gradboundsCW}; see for instance \eqref{eq:cgradfb}. When $x \in [-\zeta,0]$, we want to prove that
\begin{align}
\abs{f_h''(x)} \leq&\ \frac{C}{\mu}\Big(\frac{\alpha}{\mu}+\sqrt{\frac{\alpha}{\mu}}+1\Big)\Big(1+\sqrt{\frac{\mu}{\alpha}}\Big) + \frac{C}{\mu } \Big( \zeta \wedge \frac{\mu^2}{\alpha^2	\zeta}\Big), \label{eq:alternativesgb}
\end{align}
which, after considering separately the cases when $\zeta \leq \mu /\alpha$ and $\zeta \geq \mu /\alpha$, implies that 
\begin{align*}
\abs{f_h''(x)} \leq&\ \frac{C}{\mu}\Big(\frac{\alpha}{\mu}+\sqrt{\frac{\alpha}{\mu}}+1\Big)\Big(1+\sqrt{\frac{\mu}{\alpha}}\Big)
+ \frac{C}{\alpha}.
\end{align*}
We can then use this fact to bound $\abs{f_h''(x)b(x)} =  \alpha \abs{x}\abs{f_h''(x)}$. To prove \eqref{eq:alternativesgb} for $\zeta \leq \sqrt{\mu/\alpha}$, we bound \eqref{eq:hf22} using \eqref{eq:gwo1} and \eqref{eq:oingredient2}. To prove \eqref{eq:alternativesgb} for $\zeta \geq \sqrt{\mu/\alpha}$, we bound \eqref{eq:hf21} using \eqref{eq:gwo1} and \eqref{eq:oingredient1}. We point out that to bound \eqref{eq:hf21} we need to perform a manipulation similar to the one in \eqref{eq:der2manip}.  This concludes the proof outline for the overloaded case.

\subsection{Kolmogorov Gradient Bounds: Proof of Lemmas~\ref{lem:gradboundsCK} and \ref{lem:gradboundsAK}} \label{app:kgradient}
Let $\bar a(x)$ and $\bar b(x)$ be as in \eqref{CW:ab}. Fix $a \in \R$ and let $h(x) = 1_{(-\infty,a]}(x)$. The the  Poisson equation is 
\begin{align*}
\bar b(x) f_a'(x) + \frac{1}{2}\bar a(x) f_a''(x) = F_Y(a) - 1_{(-\infty,a]}(x),
\end{align*}
where $F_Y(x) = \Prob(Y(\infty) \leq x)$. Since $1_{(-\infty,a]}(x)$ is discontinuous, any solution to the Poisson equation will have a discontinuity in its second derivative, which makes the gradient bounds for it differ from the Wasserstein setting. 

 Together, \eqref{eq:fprimeneg} and \eqref{eq:fprimepos} both imply that 
\begin{align*}
\abs{f_a'(x)} \leq  e^{-\int_{0}^{x} \frac{2\bar b(u)}{\bar a(u)}du}\min \Big\{ \int_{-\infty}^{x} \frac{2}{\bar a(y)}e^{\int_{0}^{y} \frac{2\bar b(u)}{\bar a(u)}du} dy, \int_{x}^{\infty}\frac{2}{\bar a(y)} e^{\int_{0}^{y} \frac{2\bar b(u)}{\bar a(u)}du} dy \Big\}.  
\end{align*}
Furthermore, 
\begin{align*}
f_a''(x) = \frac{1}{\mu } \big( F_Y(a) - 1_{(-\infty,a]}(x) - \bar b(x) f_a'(x)\big).
\end{align*}
We now prove the Kolmogorov gradient bounds for the Erlang-C model.
\begin{proof}[Proof of Lemma~\ref{lem:gradboundsCK}]
First of all, by \eqref{CW:fbound1} and \eqref{CW:fbound2}, 
\begin{align}
\mu \abs{f_a'(x)} \leq 
\begin{cases}
2, \quad x \leq 0, \\
\min \big\{(2+\abs{\zeta})e^{\frac{1}{2}\zeta^2},2 + \frac{1}{\abs{\zeta}} \big\}, \quad x \in [0,-\zeta], \\
\frac{1}{\abs{\zeta}}, \quad x \geq -\zeta,
\end{cases} \label{eq:kolmfp}
\end{align}
and \eqref{eq:arithmetic} implies that
\begin{align*}
\min\Big\{(2+\abs{\zeta})e^{\frac{1}{2}\zeta^2},2 + \frac{1}{\abs{\zeta}} \Big\} \leq 4,
\end{align*}
which proves the bounds for $f_a'(x)$. Second, \eqref{eq:main1} and \eqref{eq:main2} imply that for all $x \in \R$, 
\begin{align}
&\abs{f_a''(x)}\notag  \\
\leq&\ \frac{1}{\mu } \bigg( 1 + \abs{\bar b(x)} e^{-\int_{0}^{x} \frac{2\bar b(u)}{\bar a(u)}du}\min \Big\{ \int_{-\infty}^{x} \frac{2}{\bar a(y)}e^{\int_{0}^{y} \frac{2\bar b(u)}{\bar a(u)}du} dy, \int_{x}^{\infty}\frac{2}{\bar a(y)} e^{\int_{0}^{y} \frac{2\bar b(u)}{\bar a(u)}du} dy \Big\}\bigg) \notag \\
 \leq&\ 2/\mu, \label{eq:kolmfpp}
\end{align}
where $f_a''(x)$ is understood to be the left derivative at the point $x = a$.

\end{proof}

\begin{proof}[Proof of Lemma~\ref{lem:gradboundsAK}]
The proof of this lemma is almost identical to the proof of Lemma~\ref{lem:gradboundsCK}. Its not hard to check that \eqref{eq:kolmfpp} holds for the Erlang-A model as well. To prove the bounds on $f_a'(x)$, we obtain inequalities similar to \eqref{eq:kolmfp} by using analogues of \eqref{CW:fbound1} and \eqref{CW:fbound2} from Lemmas~\ref{lem:lowlevWunder} and \ref{lem:lowlevWover}.  These inequalities will imply \eqref{eq:ACuder1} and \eqref{eq:ACoder1} once we consider in them separately the cases when $\abs{\zeta} \leq 1$ and $\abs{\zeta} \geq 1$.

\end{proof}

%\section{Gradient Bounds for Chapter~\ref{chap:dsquare}}
\section{Gradient Bounds for Chapters~\ref{chap:dsquare} and \ref{chap:md}}
\label{app:DSgradbounds}
In the setting of Chapter~\ref{chap:dsquare} and \ref{chap:md},
\begin{align}
&\bar a(x)  = 
\begin{cases}
\mu , \quad x \leq -1/\delta, \\
\mu (2 + \delta x), \quad x \in [-1/\delta, -\zeta], \\
\mu (2 + \delta \abs{\zeta}), \quad x \geq -\zeta,
\end{cases}
\quad \text{ and } \quad
&\bar b(x) = 
\begin{cases}
-\mu x, \quad x \leq -\zeta,\\
\mu \zeta, \quad x \geq -\zeta,
\end{cases} \label{DS:ab}
\end{align} 
where $\zeta = \delta(R - n) < 0$. Observe that $\bar b(x)$ satisfies both  (\hyperref[eq:a1]{a1}) and (\hyperref[eq:a2]{a2}), and that $x_0$ from \eqref{eq:xnot} equals zero. Furthermore, 
\begin{align}
&\exp \Big({\int_{0}^{x} \frac{2 \bar b(u)}{\bar a(u)} du} \Big) \notag  \\
=& 
\begin{cases}
\exp\big(\frac{1}{\delta^2}+ \frac{2}{\delta^2} - \frac{4}{\delta^2} \log(2)\big)\exp(-x^2), \quad x \leq -1/\delta,\\
\exp\big(- \frac{4}{\delta^2} \log(2)\big)\exp\Big[\frac{4}{\delta^2}\log(2 + \delta x) - \frac{2\delta x}{\delta^2}\Big], \quad x \in [-1/\delta, -\zeta], \\
\exp\big(- \frac{4}{\delta^2} \log(2) + \frac{2}{\delta^2}(2\log(2 + \delta \abs{\zeta}) - \delta \abs{\zeta} ) + \frac{2\zeta^2}{2 + \delta \abs{\zeta}} \big)\exp\big(\frac{-2\abs{\zeta} x }{2+\delta \abs{\zeta}}\big), \quad x \geq -\zeta.
\end{cases}\label{DS:pdef}
\end{align}
The following lemma presents several bounds that will be used to prove Lemma~\ref{lem:gb} and \ref{lem:MDgradient_bounds}.
\begin{lemma} \label{lem:lowlevelbounds}
\allowdisplaybreaks
Let $\bar a(x)$ and $\bar b(x)$ be as in \eqref{DS:ab}. Then
\begin{align}
&e^{-\int_{0}^{x} \frac{2\bar b(u)}{\bar a(u)}du}\int_{-\infty}^{x} \frac{2}{\bar a(y)} e^{\int_{0}^{y} \frac{2\bar b(u)}{\bar a(u)}du} dy\leq 
\begin{cases}
\frac{3}{ \mu }, \quad x \leq 0,\\
\frac{1}{\mu}e^{\zeta^2} (3 + \abs{\zeta}), \quad x \in [0,-\zeta],
\end{cases} \label{DS:fbound1}\\
&e^{-\int_{0}^{x} \frac{2\bar b(u)}{\bar a(u)}du}\int_{x}^{\infty} \frac{2}{\bar a(y)} e^{\int_{0}^{y} \frac{2\bar b(u)}{\bar a(u)}du} dy \leq 
\begin{cases}
\frac{1}{\mu } \Big( 2 + \frac{1}{\abs{\zeta}} \Big), \quad x \in [0,-\zeta], \\
\frac{1}{\mu \abs{\zeta}}, \quad x \geq -\zeta,
\end{cases} \label{DS:fbound2} \\
&e^{-\int_{0}^{x} \frac{2\bar b(u)}{\bar a(u)}du}\int_{-\infty}^{x} \frac{2\abs{y}}{\bar a(y)} e^{\int_{0}^{y} \frac{2\bar b(u)}{\bar a(u)}du} dy\leq 
\begin{cases}
\frac{1}{\mu }, \quad x \leq 0,\\
\frac{2}{\mu } e^{\frac{\zeta^2}{2}}, \quad x \in [0,-\zeta],
\end{cases} \label{DS:fbound3} \\
&e^{-\int_{0}^{x} \frac{2\bar b(u)}{\bar a(u)}du}\int_{x}^{\infty} \frac{2\abs{y}}{\bar a(y)} e^{\int_{0}^{y} \frac{2\bar b(u)}{\bar a(u)}du}dy \leq 
\begin{cases}
\frac{2}{\mu} + \frac{1}{\mu \zeta^2} + \frac{\delta}{2\mu \abs{\zeta}}, \quad x \in [0,-\zeta],\\
\frac{x}{\mu \abs{\zeta} } + \frac{1}{\mu \zeta^2} + \frac{\delta}{2\mu \abs{\zeta}}, \quad x \geq -\zeta,
\end{cases} \label{DS:fbound4} \\
&\EE \big|Y(\infty)\big| \leq \sqrt{\delta^2 + 2} + \sqrt{2\delta^2 + 4} + \frac{2+\delta^2}{\abs{\zeta}} + \delta. \label{DS:fbound7}
\end{align}
\end{lemma}
\begin{proof}[Proof of Lemma~\ref{lem:lowlevelbounds}]
To prove this lemma we verify \eqref{DS:fbound1}--\eqref{DS:fbound7} one at a time. We now prove \eqref{DS:fbound1}. Using \eqref{eq:near_orig_left}  with $c_1 = -1$, we see that for $x \leq 0$, 
\begin{align}
e^{-\int_{0}^{x} \frac{2\bar b(u)}{\bar a(u)}du}\int_{-\infty}^{x} \frac{2}{\bar a(y)} e^{\int_{0}^{y} \frac{2\bar b(u)}{\bar a(u)}du} dy\leq \frac{1}{\bar b(-1)}  +  \sup_{y \in [-1, 0]} \frac{2}{\bar a(y)} \leq \frac{1}{\mu } + \frac{2}{\mu} = \frac{3}{\mu}. \label{DS:interm0}
\end{align}
For  $x \in [0,-\zeta]$,
\begin{align}
& e^{-\int_{0}^{x} \frac{2\bar b(u)}{\bar a(u)}du}\int_{-\infty}^{x} \frac{2}{\bar a(y)} e^{\int_{0}^{y} \frac{2\bar b(u)}{\bar a(u)}du} dy \notag  \\
=&\  \frac{e^{-\int_{0}^{x} \frac{2\bar b(u)}{\bar a(u)}du}}{e^{-\int_{0}^{0} \frac{2\bar b(u)}{\bar a(u)}du}} e^{-\int_{0}^{0} \frac{2\bar b(u)}{\bar a(u)}du}\int_{-\infty}^{0} \frac{2}{\bar a(y)} e^{\int_{0}^{y} \frac{2\bar b(u)}{\bar a(u)}du} dy + e^{-\int_{0}^{x} \frac{2\bar b(u)}{\bar a(u)}du}\int_{0}^{x} \frac{2}{\bar a(y)} e^{\int_{0}^{y} \frac{2\bar b(u)}{\bar a(u)}du} dy \notag \\
\leq&\ e^{-\int_{0}^{x} \frac{2\bar b(u)}{\bar a(u)}du}\frac{3}{\mu } +  e^{-\int_{0}^{x} \frac{2\bar b(u)}{\bar a(u)}du}\int_{0}^{-\zeta} \frac{2}{\bar a(y)} e^{\int_{0}^{y} \frac{2\bar b(u)}{\bar a(u)}du} dy \notag  \\
\leq&\ e^{-\int_{0}^{x} \frac{2\bar b(u)}{\bar a(u)}du}\frac{3}{\mu } +  e^{-\int_{0}^{x} \frac{2\bar b(u)}{\bar a(u)}du}\abs{\zeta} \frac{1}{\mu }, \label{DS:interm1}
\end{align}
where in the second last inequality we used \eqref{DS:interm0}, and in the last inequality we used the fact that $e^{\int_{0}^{y} \frac{2\bar b(u)}{\bar a(u)}du} \leq 1$ and $\bar a(y) \geq 2\mu$ for $y \in [0,-\zeta]$. From \eqref{DS:pdef}, we know that 
\begin{align*}
e^{-\int_{0}^{x} \frac{2\bar b(u)}{\bar a(u)}du} =&\  \exp\Big( -\frac{4}{\delta^2} \big(\log(2 + \delta x) -\log(2) -  \delta x/2 \big)  \Big), \quad x \in [0,-\zeta].
\end{align*}
Using Taylor expansion,
\begin{align}
\log(2 + y) = \log(2) + \frac{1}{2}y - \frac{1}{2} \frac{y^2}{(2 + \xi(y))^2}, \quad y \in (-2, \infty), \label{eq:logtaylor}
\end{align}
where $\xi(y)$ is some point between $0$ and $y$. Therefore, for $x \in [0,-\zeta]$,
\begin{align}
e^{-\int_{0}^{x} \frac{2\bar b(u)}{\bar a(u)}du} =&\ \exp\Big( -\frac{4}{\delta^2} \big(\log(2 + \delta x) -\log(2) -  \delta x/2 \big)  \Big)  \notag \\
=&\  \exp\Big( \frac{4}{\delta^2} \frac{1}{2} \frac{\delta^2 x^2}{(2 + \xi(\delta x))^2}  \Big)  \notag \\
\leq&\ \exp\Big( \frac{ x^2}{2}  \Big), \label{DS:taylorbound}
\end{align}
where in the last inequality we used the fact that $\xi(\delta x) \geq 0$ for $x \geq 0$. Combining this with \eqref{DS:interm1}, we conclude that 
\begin{align*}
e^{-\int_{0}^{x} \frac{2\bar b(u)}{\bar a(u)}du}\int_{-\infty}^{x} \frac{2}{\bar a(y)} e^{\int_{0}^{y} \frac{2\bar b(u)}{\bar a(u)}du} dy\leq&\ \frac{e^{\zeta^2} }{\mu}(3 + \abs{\zeta}), \quad x \in [0,-\zeta],
\end{align*}
which proves \eqref{DS:fbound1}. We now prove \eqref{DS:fbound2}. When $x \geq -\zeta$, \eqref{eq:main2} implies that 
\begin{align*}
e^{-\int_{0}^{x} \frac{2\bar b(u)}{\bar a(u)}du}\int_{x}^{\infty} \frac{2}{\bar a(y)} e^{\int_{0}^{y} \frac{2\bar b(u)}{\bar a(u)}du} dy \leq \frac{1}{\abs{\bar b(x)}} =  \frac{1}{\mu \abs{\zeta}}.
\end{align*}
When $x \in [0,-\zeta]$, we can repeat the procedure in \eqref{CW:same} to see that 
\begin{align*}
e^{-\int_{0}^{x} \frac{2\bar b(u)}{\bar a(u)}du}\int_{x}^{\infty} \frac{2}{\bar a(y)} e^{\int_{0}^{y} \frac{2\bar b(u)}{\bar a(u)}du} dy \leq&\ e^{-\int_{0}^{x} \frac{2\bar b(u)}{\bar a(u)}du}\int_{x}^{-\zeta} \frac{2}{\bar a(y)} e^{\int_{0}^{y} \frac{2\bar b(u)}{\bar a(u)}du} dy + \frac{1}{\mu \abs{\zeta}}.
\end{align*}
We now bound the first term on the right hand side above. When $\abs{\zeta} \geq 1$, we use \eqref{eq:near_orig_right} with $c_2 = 1$ to see that 
\begin{align*}
e^{-\int_{0}^{x} \frac{2\bar b(u)}{\bar a(u)}du}\int_{x}^{-\zeta} \frac{2}{\bar a(y)} e^{\int_{0}^{y} \frac{2\bar b(u)}{\bar a(u)}du} dy \leq&\ e^{-\int_{0}^{x} \frac{2\bar b(u)}{\bar a(u)}du}\int_{x}^{\infty} \frac{2}{\bar a(y)} e^{\int_{0}^{y} \frac{2\bar b(u)}{\bar a(u)}du} dy \\
\leq&\ \frac{1}{\abs{\bar b(1)}} + \frac{1}{\mu } = \frac{2}{\mu }.
\end{align*}
When $\abs{\zeta} \leq 1$, \eqref{DS:taylorbound} implies that 
\begin{align*}
e^{-\int_{0}^{x} \frac{2\bar b(u)}{\bar a(u)}du}\int_{x}^{-\zeta} \frac{2}{\bar a(y)} e^{\int_{0}^{y} \frac{2\bar b(u)}{\bar a(u)}du} dy \leq&\ e^{-\int_{0}^{-\zeta} \frac{2\bar b(u)}{\bar a(u)}du}\int_{0}^{1} \frac{2}{\bar a(y)}dy \leq  e^{\zeta^2/2} \frac{1}{\mu } \leq \frac{2}{\mu }.
\end{align*}
Therefore, for $x \in [0,-\zeta]$,
\begin{align*}
e^{-\int_{0}^{x} \frac{2\bar b(u)}{\bar a(u)}du}\int_{x}^{\infty} \frac{2}{\bar a(y)} e^{\int_{0}^{y} \frac{2\bar b(u)}{\bar a(u)}du} dy \leq \frac{2}{\mu} + \frac{1}{\mu \abs{\zeta}},
\end{align*}
which proves \eqref{DS:fbound2}. We now prove \eqref{DS:fbound3}. For $x \leq 0$,
 \begin{align*}
 e^{-\int_{0}^{x} \frac{2\bar b(u)}{\bar a(u)}du}\int_{-\infty}^{x} \frac{2\abs{y}}{\bar a(y)} e^{\int_{0}^{y} \frac{2\bar b(u)}{\bar a(u)}du} dy =&\  e^{-\int_{0}^{x} \frac{2\bar b(u)}{\bar a(u)}du}\int_{-\infty}^{x}\frac{1}{\mu } \frac{2\bar b(y)}{\bar a(y)} e^{\int_{0}^{y} \frac{2\bar b(u)}{\bar a(u)}du} dy \\
=&\ \frac{1}{\mu }\Big(1 - e^{-\int_{-\infty}^{x} \frac{2\bar b(u)}{\bar a(u)}du}\Big)\\
\leq&\  \frac{1}{\mu}.
 \end{align*}
When $x \in [0,-\zeta]$,  
\begin{align*}
&e^{-\int_{0}^{x} \frac{2\bar b(u)}{\bar a(u)}du}\int_{-\infty}^{x} \frac{2\abs{y}}{\bar a(y)} e^{\int_{0}^{y} \frac{2\bar b(u)}{\bar a(u)}du} dy \\
=&\  e^{-\int_{0}^{x} \frac{2\bar b(u)}{\bar a(u)}du}\int_{-\infty}^{0}\frac{-2y}{\bar a(y)} e^{\int_{0}^{y} \frac{2\bar b(u)}{\bar a(u)}du} dy + e^{-\int_{0}^{x} \frac{2\bar b(u)}{\bar a(u)}du}\int_{0}^{x}\frac{2y}{\bar a(y)} e^{\int_{0}^{y} \frac{2\bar b(u)}{\bar a(u)}du} dy \\
=&\  e^{-\int_{0}^{x} \frac{2\bar b(u)}{\bar a(u)}du}\int_{-\infty}^{0}\frac{1}{\mu } \frac{2\bar b(y)}{\bar a(y)} e^{\int_{0}^{y} \frac{2\bar b(u)}{\bar a(u)}du} dy - e^{-\int_{0}^{x} \frac{2\bar b(u)}{\bar a(u)}du}\int_{0}^{x}\frac{1}{\mu } \frac{2\bar b(y)}{\bar a(y)} e^{\int_{0}^{y} \frac{2\bar b(u)}{\bar a(u)}du} dy \\
=&\ \frac{1}{\mu }e^{-\int_{0}^{x} \frac{2\bar b(u)}{\bar a(u)}du} \bigg( \Big(e^{\int_{0}^{0} \frac{2\bar b(u)}{\bar a(u)}du} - e^{\int_{0}^{-\infty} \frac{2\bar b(u)}{\bar a(u)}du}\Big) - \Big(e^{\int_{0}^{x} \frac{2\bar b(u)}{\bar a(u)}du} - e^{\int_{0}^{0} \frac{2\bar b(u)}{\bar a(u)}du}\Big) \bigg) \\
\leq&\ \frac{2}{\mu }e^{-\int_{0}^{x} \frac{2\bar b(u)}{\bar a(u)}du}  \leq  \frac{2}{\mu}e^{\zeta^2/2},
\end{align*}
where in the last inequality we used \eqref{DS:taylorbound}. This proves \eqref{DS:fbound3}, and now we prove \eqref{DS:fbound4}. Fix $x \in [0,-\zeta]$.

We now prove \eqref{CW:fbound4}. Since $\bar a(x) = \mu(2+\delta \abs{\zeta})$ for $x \geq -\zeta$, we can use \eqref{eq:m2} to see that for $x \geq -\zeta$,
\begin{align*}
e^{-\int_{0}^{x} \frac{2\bar b(u)}{\bar a(u)}du}\int_{x}^{\infty} \frac{2\abs{y}}{\bar a(y)} e^{\int_{0}^{y} \frac{2\bar b(u)}{\bar a(u)}du}dy \leq&\ \frac{x}{\abs{\bar b(x)}} + \frac{\mu (2 + \delta \abs{\zeta})}{2\abs{\bar b(x)}} \frac{1}{\abs{\bar b(x)}} \\
=&\  \frac{x}{\mu\abs{\zeta}} +   \frac{2 + \delta\abs{\zeta}}{2\abs{\zeta}}\frac{1}{\mu\abs{\zeta}}.
\end{align*}
Furthermore, for $x \in [0,-\zeta]$,
\begin{align*}
& e^{-\int_{0}^{x} \frac{2\bar b(u)}{\bar a(u)}du}\int_{x}^{\infty} \frac{2\abs{y}}{\bar a(y)} e^{\int_{0}^{y} \frac{2\bar b(u)}{\bar a(u)}du}dy \\
=&\  - e^{-\int_{0}^{x} \frac{2\bar b(u)}{\bar a(u)}du}\int_{x}^{-\zeta}\frac{1}{\mu } \frac{2\bar b(y)}{\bar a(y)} e^{\int_{0}^{y} \frac{2\bar b(u)}{\bar a(u)}du} dy \\
&+ \frac{e^{-\int_{0}^{x} \frac{2\bar b(u)}{\bar a(u)}du}}{e^{-\int_{0}^{-\zeta} \frac{2\bar b(u)}{\bar a(u)}du}} e^{-\int_{0}^{-\zeta} \frac{2\bar b(u)}{\bar a(u)}du}\int_{-\zeta}^{\infty} \frac{2\abs{y}}{\bar a(y)} e^{\int_{0}^{y} \frac{2\bar b(u)}{\bar a(u)}du}dy \\
\leq&\  - e^{-\int_{0}^{x} \frac{2\bar b(u)}{\bar a(u)}du}\int_{x}^{-\zeta}\frac{1}{\mu } \frac{2\bar b(y)}{\bar a(y)} e^{\int_{0}^{y} \frac{2\bar b(u)}{\bar a(u)}du} dy + e^{-\int_{0}^{-\zeta} \frac{2\bar b(u)}{\bar a(u)}du}\int_{-\zeta}^{\infty} \frac{2\abs{y}}{\bar a(y)} e^{\int_{0}^{y} \frac{2\bar b(u)}{\bar a(u)}du}dy\\
=&\ \frac{1}{\mu }e^{-\int_{0}^{x} \frac{2\bar b(u)}{\bar a(u)}du} \Big(-e^{\int_{0}^{-\zeta} \frac{2\bar b(u)}{\bar a(u)}du} + e^{\int_{0}^{x} \frac{2\bar b(u)}{\bar a(u)}du}\Big) + e^{-\int_{0}^{-\zeta} \frac{2\bar b(u)}{\bar a(u)}du}\int_{-\zeta}^{\infty} \frac{2\abs{y}}{\bar a(y)} e^{\int_{0}^{y} \frac{2\bar b(u)}{\bar a(u)}du}dy\\
\leq&\ \frac{1}{\mu } + \Big( \frac{\abs{\zeta}}{\mu \abs{\zeta}} +  \frac{2 + \delta\abs{\zeta}}{2\abs{\zeta}}\frac{1}{\mu\abs{\zeta}}\Big),
\end{align*}
where in the first inequality, we used the fact that $ e^{-\int_{0}^{x} \frac{2\bar b(u)}{\bar a(u)}du} \leq e^{-\int_{0}^{-\zeta} \frac{2\bar b(u)}{\bar a(u)}du}$. 
This proves \eqref{DS:fbound4}, and we move on to verify \eqref{DS:fbound7}. Consider the Lyapunov function $V(x) = x^2$, and recall the form of $G_Y$ from \eqref{eq:GY} to see that 
\begin{align*}
G_Y V(x) =&\ 2x\mu (\zeta + (x + \zeta)^-) + 2\mu \Big(1 + 1(x > -1/\delta)\big(1 -  \delta(\zeta + (x + \zeta)^-)\big)\Big).
\end{align*}
Now when $x < -\zeta$, 
\begin{align*}
G_Y V(x) =&\ -2\mu x^2 + 2\mu \big(1 + 1(x > -1/\delta)(1 +\delta x) \big)\\
\leq&\ -2\mu x^2  + 2\mu \delta x 1\big( x \in [0, -\zeta)\big)+ 4\mu \\
=&\ -2\mu x^2 1(x < 0) - 2\mu \big(x^2 - \delta x \big) 1\big( x \in [0, -\zeta)\big) + 4\mu \\
\leq&\ -2\mu x^2 1(x < 0) - \mu \big(x^2 - \delta^2 \big) 1\big( x \in [0, -\zeta)\big) + 4\mu \\
\leq&\ -2\mu x^2 1(x < 0) - \mu x^2 1\big( x \in [0, -\zeta)\big) + \mu \delta^2 + 4\mu,
\end{align*}
and when $x \geq -\zeta$, 
\begin{align*}
G_Y V(x) =&\ - 2x\mu \abs{\zeta} + 2 \delta \mu \abs{\zeta} + 4\mu \\
=&\ - 2\mu \abs{\zeta} ( x - \delta) 1( \abs{\zeta} < \delta) - 2\mu \abs{\zeta} ( x - \delta) 1( \abs{\zeta} \geq \delta)  + 4\mu\\
\leq&\ - 2\mu \abs{\zeta}x 1( \abs{\zeta} < \delta)   + 2\mu \delta^21( \abs{\zeta} < \delta) - 2\mu \abs{\zeta} ( x - \delta) 1( \abs{\zeta} \geq \delta) +  4\mu.
\end{align*}
Therefore, 
\begin{align*}
G_Y V(x) \leq&\ -2\mu x^2 1(x < 0) -\mu x^2 1(x \in [0,-\zeta)) \\
  & - 2\mu \abs{\zeta}x 1( \abs{\zeta} < \delta)1(x \geq -\zeta) - 2\mu \abs{\zeta} ( x - \delta) 1( \abs{\zeta} \geq \delta)1(x \geq -\zeta)\\
  & + 2\mu \delta^21( \abs{\zeta} < \delta) 1(x \geq -\zeta)  + \mu \delta^2 1(x < -\zeta) + 4\mu,
\end{align*}
i.e.\ $G_Y V(x)$ satisfies 
\begin{align*}
G_Y V(x) \leq -f(x) + g(x),
\end{align*}
where $f(x)$ and $g(x)$ are functions from $\R \to \R_+$. By the standard Foster-Lyapunov condition (see for example \cite[Theorem 4.3]{MeynTwee1993b}), this implies that
\begin{align*}
\EE f(Y(\infty)) \leq \EE g(Y(\infty)),
\end{align*}
or
\begin{align*}
&\ 2\EE \big[ (Y(\infty))^2 1(Y(\infty) < 0)\big] + \EE \big[ (Y(\infty))^2 1(Y(\infty) \in [0,-\zeta))\big] \\
&+ 2\abs{\zeta} \EE \big[Y(\infty)1(Y(\infty) \geq -\zeta)\big] 1( \abs{\zeta} < \delta)\\
&+ 2\abs{\zeta} \EE \big[(Y(\infty)-\delta )1(Y(\infty) \geq -\zeta)\big] 1(\abs{\zeta} \geq \delta) \\
 \leq&\ 2\delta^2 + 4,
\end{align*}
from which we can see that
\begin{align*}
\EE \big[Y(\infty)1(Y(\infty) \geq -\zeta)\big] \leq&\ \frac{\delta^2}{\abs{\zeta}} +\frac{2}{\abs{\zeta}} +  \delta.
\end{align*}
Furthermore, by invoking Jensen's inequality we see that
\begin{align*}
\EE \Big[ \big|Y(\infty) 1(Y(\infty) < 0)\big|\Big] \leq&\  \sqrt{\EE \big[ (Y(\infty))^2 1(Y(\infty) < 0)\big]} \\
\leq&\ \sqrt{\delta^2 + 2}, \\
\EE \Big[ \big|Y(\infty) 1(Y(\infty) \in [0,-\zeta))\big|\Big] \leq&\  \sqrt{\EE \big[ (Y(\infty))^2 1(Y(\infty) \in [0,-\zeta))\big]} \\
\leq&\ \sqrt{2\delta^2 + 4}.
\end{align*}
Hence 
\begin{align*}
\EE \big[ \big|Y(\infty)\big|\big] =&\ \EE \Big[ \big|Y(\infty) 1(Y(\infty) < 0)\big|\Big] + \EE \Big[ \big|Y(\infty) 1(Y(\infty) \in [0,-\zeta))\big|\Big] \\
&+ \EE \big[Y(\infty)1(Y(\infty) \geq -\zeta)\big]\\
\leq&\ \sqrt{\delta^2 + 2} + \sqrt{2\delta^2 + 4} + \frac{2+\delta^2}{\abs{\zeta}} + \delta.
\end{align*}
This proves \eqref{DS:fbound7} and concludes the proof of this lemma.
\end{proof}

\noindent We are now ready to prove Lemma~\ref{lem:gb} and \ref{lem:MDgradient_bounds}.
\subsection{Proof of Lemma~\ref{lem:gb} ($W_2$ Bounds)}
\begin{proof}[Proof of Lemma~\ref{lem:gb} ]
Recall our assumption that $R \geq 1$, or equivalently, $\delta \leq 1$. Throughout the proof we use $C > 0$ to denote a generic constant that does not depend on $\lambda,n$, and $\mu$, and may change from line to line. We begin by bounding $f_h'(x)$. Observe that since $h(x) \in W_2$ and $h(0) = 0$, then \eqref{eq:fprimeneg} and \eqref{eq:fprimepos} imply that 
\begin{align*}
f_h'(x) \leq&\ e^{-\int_{0}^{x} \frac{2\bar b(u)}{\bar a(u)}du}\int_{-\infty}^{x} \frac{2}{\bar a(y)} ( |y| + \E \abs{Y(\infty)}) e^{\int_{0}^{y} \frac{2\bar b(u)}{\bar a(u)}du} dy,  \\
f_h'(x) \leq &\ e^{-\int_{0}^{x} \frac{2\bar b(u)}{\bar a(u)}du} \int_{x}^{\infty} \frac{2}{\bar a(y)} ( |y| + \E \abs{Y(\infty)}) e^{\int_{0}^{y} \frac{2\bar b(u)}{\bar a(u)}du} dy.
\end{align*}
We apply \eqref{DS:fbound1}, \eqref{DS:fbound3}, and \eqref{DS:fbound7} to the first inequality above when $x \leq -\zeta$ to see that
\begin{align*}
\mu \abs{f_h'(x)} \leq&\ C\Big( 1 +\frac{1}{\abs{\zeta}} \Big), \quad x \leq 0,\notag \\
\mu \abs{f_h'(x)} \leq&\  2e^{\frac{1}{2}\zeta^2} + e^{\zeta^2} (3 + \abs{\zeta})\EE \big|Y(\infty)\big|, \quad x \in [0,-\zeta],
\end{align*}
and apply \eqref{DS:fbound2}, \eqref{DS:fbound4}, and \eqref{DS:fbound7} to the second inequality when $x \geq 0$ to see that
\begin{align*}
\mu \abs{f_h'(x)} \leq&\ 2 + \frac{1}{\zeta^2} + \frac{\delta}{2\abs{\zeta}} + \Big(2 +  \frac{1}{\abs{\zeta}} \Big)\EE \big|Y(\infty)\big|, \quad x \in [0,-\zeta], \notag \\
\mu \abs{f_h'(x)} \leq&\ \frac{C}{\abs{\zeta}} \Big(x + 1 + \frac{1}{\abs{\zeta}}\Big), \quad x \geq -\zeta. 
\end{align*}
Above, there are two possible bounds on $\mu \abs{f_h'(x)}$ when $x \in [0, -\zeta]$. By considering separately the cases when $\abs{\zeta} \leq 1$ and $\abs{\zeta} \geq 1$, and using \eqref{DS:fbound7} to bound $\EE \big| Y(\infty) \big|$, we conclude that
\begin{align*}
\mu \abs{f_h'(x)} \leq&\ C\Big( 1 +\frac{1}{\abs{\zeta}} \Big), \quad x \in [0,-\zeta].
\end{align*}
Therefore,
\begin{align}
\abs{f_h'(x)} 
\leq
\begin{cases}
\frac{C}{\mu }\Big( 1 +\frac{1}{\abs{\zeta}} \Big), \quad x \leq -\zeta,\\
\frac{C}{\mu \abs{\zeta}}\Big(x + 1 + \frac{1}{\abs{\zeta}}\Big), \quad x \geq -\zeta,
\end{cases} \label{eq:inlineder1}
\end{align}
which proves \eqref{DS:WCder1}. Using \eqref{DS:ab}, \eqref{DS:pdef}, and \eqref{eq:inlineder1}, the reader can verify that \eqref{eq:mlimits} and \eqref{eq:plimits} are satisfied, which allows us to use the two forms of $f_h''(x)$ in \eqref{eq:hf21} and \eqref{eq:hf22}. We now bound $\abs{f_h''(x)}$. Since $h(0) = 0$ and $h(x) \in W_2$, we know that $\abs{h(x)} \leq \abs{x}$ and $\abs{h'(x)} \leq 1$ for all $x \in \R$. From \eqref{eq:hf21} and \eqref{eq:hf22}, it follows that
\begin{align}
\abs{f_h''(x)} \leq&\ e^{-\int_{0}^{x} \frac{2\bar b(u)}{\bar a(u)}du} \int_{-\infty}^{x} \Big(\frac{2}{\bar a(y)} +  \frac{2\abs{\bar a'(y)y}}{a^2(y)}+ \frac{2\abs{\bar a'(y)}}{a^2(y)}\EE \big| Y(\infty)\big| \notag  \\
& \hspace{5cm} + \abs{\Big(\frac{2\bar b(y)}{\bar a(y)}\Big)' f_h'(y)}\Big)e^{\int_{0}^{y} \frac{2\bar b(u)}{\bar a(u)}du}dy, \label{eq:der2bound1} \\
\abs{f_h''(x)} \leq&\ e^{-\int_{0}^{x} \frac{2\bar b(u)}{\bar a(u)}du} \int_{x}^{\infty} \Big(\frac{2}{\bar a(y)} +  \frac{2\abs{\bar a'(y)y}}{a^2(y)} + \frac{2\abs{\bar a'(y)}}{a^2(y)}\EE \big| Y(\infty)\big| \notag \\
& \hspace{5cm} + \abs{\Big(\frac{2\bar b(y)}{\bar a(y)}\Big)' f_h'(y)}\Big)e^{\int_{0}^{y} \frac{2\bar b(u)}{\bar a(u)}du} dy. \label{eq:der2bound2}
\end{align}
We now bound the terms inside the integrals above. By definition of $\bar a(x)$ in \eqref{DS:ab}, we see that
\begin{align}
\bar a'(x) = \mu \delta 1(x \in (-1/\delta, -\zeta]), \label{eq:ap}
\end{align}
where $\bar a'(x)$ is interpreted as the left derivative for $x = -1/\delta$ and $x = -\zeta$. Therefore,
\begin{align}
\frac{\abs{\bar a'(x)x}}{\bar a(x)} =&\ \frac{\mu \delta \abs{x}}{\mu (2 + \delta x)}1(x \in (-1/\delta, -\zeta]) \leq 1(x \in (-1/\delta, -\zeta]), \label{eq:ap1} \\
\EE \big| Y(\infty) \big| \frac{\abs{\bar a'(x)}}{\bar a(x)} =&\ \EE \big| Y(\infty) \big|\frac{\mu \delta}{\mu (2 + \delta x)}1(x \in (-1/\delta, -\zeta]) \notag \\
 \leq&\ \delta C\Big(1 + \frac{1}{\abs{\zeta}}\Big) 1(x \in (-1/\delta, -\zeta]), \label{eq:ap2}
\end{align}
where in the last inequality we used \eqref{DS:fbound7} and the fact that $\delta \leq 1$ to bound $\EE \big| Y(\infty) \big|$. Furthermore, 
\begin{align}
\frac{2\bar b(x)}{\bar a(x)} = 
\begin{cases}
-2x, \quad x \leq -1/\delta,\\
\frac{-2x}{2 +\delta x}, \quad x \in [-1/\delta, -\zeta], \\
\frac{2\zeta}{2 + \delta\abs{\zeta}}, \quad x \geq -\zeta,
\end{cases}
\quad
\Big(\frac{2\bar b(x)}{\bar a(x)}\Big)' = 
\begin{cases}
-2, \quad x \leq -1/\delta,\\
\frac{-4}{(2+\delta x)^2}, \quad x \in (-1/\delta, -\zeta], \\
0, \quad x > -\zeta,
\end{cases} \label{eq:rform}
\end{align}
where $\Big(\frac{2\bar b(x)}{\bar a(x)}\Big)'$ is interpreted as the left derivative at the points $x = -1/\delta$ and $x = -\zeta$. Combining \eqref{eq:rform} with the bound on $f_h'(x)$ in  \eqref{eq:inlineder1}, we get
\begin{align}
\abs{\Big(\frac{2\bar b(x)}{\bar a(x)}\Big)'f_h'(x)} =&\ 2\abs{f_h'(x)}1(x \leq  -1/\delta) + \frac{4}{(2+\delta x)^2}\abs{f_h'(x)}1(x \in  (-1/\delta, -\zeta]) \notag \\
\leq&\ 2\abs{f_h'(x)}1(x \leq  -1/\delta) + \frac{4}{2+\delta x}\abs{f_h'(x)}1(x \in  (-1/\delta, -\zeta]) \notag \\
\leq&\  \frac{C}{1 + 1(x \in  (-1/\delta, -\zeta])(1+ \delta x  )} \frac{1}{\mu } \Big(1 +  \frac{1}{\abs{\zeta}}\Big) 1(x \leq -\zeta)\notag \\
=&\  \frac{C}{\bar a(x)} \Big(1 +  \frac{1}{\abs{\zeta}}\Big) 1(x \leq -\zeta). \label{eq:sabound2}
\end{align}
Therefore, when $x \leq -\zeta$ we apply the bounds in \eqref{eq:ap1}, \eqref{eq:ap2}, and \eqref{eq:sabound2} to \eqref{eq:der2bound1} to see that
\begin{align}
\abs{f_h''(x)} \leq&\ Ce^{-\int_{0}^{x} \frac{2\bar b(u)}{\bar a(u)}du} \int_{-\infty}^{x} \frac{1}{\bar a(y)} \Big(1 +  1(y \in (-1/\delta, -\zeta]) \notag \\
& \hspace{5cm} + \delta \Big(1 + \frac{1}{\abs{\zeta}}\Big) 1(y \in (-1/\delta, -\zeta]) \notag \\
& \hspace{5cm} + \Big(1 +  \frac{1}{\abs{\zeta}}\Big) 1(y \leq -\zeta)\Big)e^{\int_{0}^{y} \frac{2\bar b(u)}{\bar a(u)}du}dy \notag  \\
\leq&\ Ce^{-\int_{0}^{x} \frac{2\bar b(u)}{\bar a(u)}du}\int_{-\infty}^{x} \frac{1}{\bar a(y)}\Big(1 +  \frac{1}{\abs{\zeta}}\Big)e^{\int_{0}^{y} \frac{2\bar b(u)}{\bar a(u)}du} dy, \quad x \leq -\zeta \label{eq:fppboundprocessed1}
\end{align}
and when $x \geq 0$ we apply the same bounds to \eqref{eq:der2bound2} to see that 
\begin{align}
\abs{f_h''(x)} \leq&\ Ce^{-\int_{0}^{x} \frac{2\bar b(u)}{\bar a(u)}du} \int_{x}^{\infty} \frac{1}{\bar a(y)} \Big(1 +  1(y \in (-1/\delta, -\zeta]) \notag \\
& \hspace{5cm} + \delta \Big(1 + \frac{1}{\abs{\zeta}}\Big) 1(y \in (-1/\delta, -\zeta]) \notag \\
& \hspace{5cm} + \Big(1 +  \frac{1}{\abs{\zeta}}\Big) 1(y \leq -\zeta)\Big)e^{\int_{0}^{y} \frac{2\bar b(u)}{\bar a(u)}du}dy \notag  \\
\leq&\ Ce^{-\int_{0}^{x} \frac{2\bar b(u)}{\bar a(u)}du}\int_{x}^{\infty} \frac{1}{\bar a(y)}\Big(1 +  \frac{1}{\abs{\zeta}}\Big)e^{\int_{0}^{y} \frac{2\bar b(u)}{\bar a(u)}du} dy, \quad x \geq 0. \label{eq:fppboundprocessed2}
\end{align} 
We apply \eqref{DS:fbound1} to \eqref{eq:fppboundprocessed1} and \eqref{DS:fbound2} to \eqref{eq:fppboundprocessed2} to get
\begin{align*}
\abs{f_h''(x)} \leq&\ 
\begin{cases}
\frac{C}{\mu }\Big(1 +  \frac{1}{\abs{\zeta}}\Big), \quad x \leq 0,\\
\min \Big\{e^{\zeta^2/2} (3 + \abs{\zeta}) , 2 + \frac{1}{\abs{\zeta}} \Big\}\frac{C}{\mu }\Big(1 +  \frac{1}{\abs{\zeta}}\Big), \quad x \in [0,-\zeta],\\
\frac{C}{\mu \abs{\zeta}}, \quad x \geq -\zeta,
\end{cases} 
\end{align*}
and by considering separately the cases when $\abs{\zeta} \leq 1$ and $\abs{\zeta} \geq 1$, we conclude that
\begin{align}
\abs{f_h''(x)} \leq&\ 
\begin{cases}
\frac{C}{\mu }\Big(1 +  \frac{1}{\abs{\zeta}}\Big), \quad x \leq -\zeta,\\
\frac{C}{\mu \abs{\zeta}}, \quad x \geq -\zeta,
\end{cases} \label{eq:inlineder2}
\end{align}
which proves \eqref{DS:WCder2}. 

Now we prove \eqref{DS:WCder3}. Recall the form of $f_h'''(x)$ from \eqref{eq:fppp}, which together with the facts that $\abs{h(x)} \leq \abs{x}$ and $\abs{h'(x)} \leq 1$ implies that for all $x \in \R$, 
\begin{align*}
\abs{f_h'''(x)} \leq&\ \abs{\Big(\frac{2\bar b(x)}{\bar a(x)}\Big)' f_h'(x)} + \abs{\frac{2\bar b(x)}{\bar a(x)} f_h''(x)} + \frac{2}{\bar a(x)} + \frac{2\abs{\bar a'(x)}}{a^2(x)}\Big( \abs{x} + \EE \big| Y(\infty)\big| \Big),
\end{align*}
where $f_h'''(x)$ is interpreted as the left derivative at the points $x = -1/\delta$ and $x = -\zeta$. We apply the bound on $\abs{\Big(\frac{2\bar b(x)}{\bar a(x)}\Big)' f_h'(x)}$ from \eqref{eq:sabound2}, the bounds on $\abs{\bar a'(x) x}/\bar a(x)$ and $\EE \big| Y(\infty)\big|\abs{\bar a'(x)}/\bar a(x)$ from \eqref{eq:ap1} and \eqref{eq:ap2}, and the fact that $1/\bar a(x) \leq 1/\mu$ for all $x \in \R$ to see that 
\begin{align*}
\abs{f_h'''(x)} \leq&\ \frac{C}{\mu }\Big( 1 + \frac{1}{\abs{\zeta}} \Big) 1(x \leq -\zeta) + \frac{C}{\mu }1(x > -\zeta)  + \abs{\frac{2\bar b(x)}{\bar a(x)} f_h''(x)}.
\end{align*}
It remains to bound $\abs{\frac{2\bar b(x)}{\bar a(x)} f_h''(x)}$, but this term does not pose much added difficulty. Indeed, one can multiply both sides of \eqref{eq:fppboundprocessed1} and \eqref{eq:fppboundprocessed2} by $\abs{\frac{2\bar b(x)}{\bar a(x)}}$ and invoke \eqref{eq:main1} and \eqref{eq:main2} to arrive at 
\begin{align*}
\abs{\frac{2\bar b(x)}{\bar a(x)} f_h''(x)}  \leq&\ 
\begin{cases}
\frac{C}{\mu }\Big(1 +  \frac{1}{\abs{\zeta}}\Big), \quad x \leq -\zeta,\\
\frac{C}{\mu }, \quad x \geq -\zeta.
\end{cases}
\end{align*}
This proves \eqref{DS:WCder3} and concludes the proof of this lemma.
\end{proof}

\subsection{Proof of Lemma~\ref{lem:eterm} ($W_2$ Fourth Derivative)}
\label{app:eterm}
This section is devoted to proving Lemma~\ref{lem:eterm}.  In this entire section, we reserve the variable $x$ to be of the form $x = x_k = \delta (k - R)$, where $k \in \Z_+$. Let $a(x)$ and $b(x)$ be as in \eqref{DS:adef} and \eqref{DS:bdef}, respectively, and let $r(x) = 2b(x)/a(x)$, whose form can be found in \eqref{eq:rform}. The form of $f_h'''(x)$ in \eqref{eq:fppp} implies that for any $y \in \R$, 
\begin{align}
&\ \abs{f_h'''(y)-f_h'''(x-)} \notag \\
\leq&\ \abs{r'(y)-r'(x-)} \abs{f_h'(y)} + \abs{r'(x-)}\abs{f_h'(x)-f_h'(y)}  \notag \\
&+ \abs{r(y)-r(x)} \abs{f_h''(y)} + \abs{r(x)}\abs{f_h''(x)-f_h''(y)}  \notag \\
&+ \abs{2/a(x) - 2/a(y)} \abs{h'(y)} + \abs{2/a(x)}\abs{h'(x-)-h'(y)} \notag \\
&+ \abs{\frac{2a'(x-)}{a^2(x)} - \frac{2a'(y)}{a^2(y)}} \Big(\abs{h(y)} + \abs{\EE h(Y(\infty))}\Big) + \abs{\frac{2a'(x-)}{a^2(x)}}\abs{h(x)-h(y)}. \label{eq:lipf}
\end{align}
We first state a few auxiliary lemmas that will help us prove Lemma~\ref{lem:eterm}. These lemmas are proved at the end of this section. The first lemma deals with the case when $y \in (x-\delta, x)$.

\begin{lemma} \label{lem:llminus}
Fix $h(x) \in W_2$ with $h(0)=0$, and let $f_h(x)$ be the solution to the Poisson equation \eqref{eq:poisson} that satisfies the conditions of Lemma~\ref{lem:gb}. There exists a constant $C>0$ (independent of $\lambda, n$, and $\mu$), such that for all $x = x_k = \delta (k - R)$ with $k \in \Z_+$, all $y \in (x-\delta, x)$, and all $n \geq 1, \lambda > 0$, and $\mu > 0$ satisfying $1 \leq R < n $,
\begingroup
\allowdisplaybreaks
\begin{align}
&\ \abs{r'(y)-r'(x-)}\abs{f_h'(y)} + \abs{r'(x-)}\abs{f_h'(x)-f_h'(y)} \leq \frac{C\delta}{\mu }  \Big(1 +  \frac{1}{\abs{\zeta}}\Big)1(x \leq -\zeta), \label{eq:ebound1} \\
&\ \abs{r(y)-r(x)} \abs{f_h''(y)} + \abs{r(x)}\abs{f_h''(x)-f_h''(y)} \notag \\
&\hspace{0.5cm} \leq \frac{C\delta }{\mu }\bigg[(1+\abs{x})\Big(1 +  \frac{1}{\abs{\zeta}}\Big)  1( x\leq -\zeta) + \abs{\zeta}1(x \geq -\zeta + \delta) \bigg], \label{eq:ebound2} \\
&\ \abs{2/a(x) - 2/a(y)} \abs{h'(y)} + \abs{2/a(x)}\abs{h'(x-)-h'(y)}\leq  \frac{C\delta}{\mu}, \label{eq:ebound3} \\
&\ \abs{\frac{2a'(x-)}{a^2(x)} - \frac{2a'(y)}{a^2(y)}} \abs{\EE h(Y(\infty))} +  \abs{\frac{2a'(x-)}{a^2(x)}}\abs{h(x)-h(y)} \notag \\
&\hspace{0.5cm} \leq \frac{C\delta}{\mu }\Big(1 + \frac{1}{\abs{\zeta}} \Big) 1(x \in [-1/\delta + \delta, -\zeta]) \label{eq:ebound4}\\
&\ \abs{\frac{2a'(x-)}{a^2(x)} - \frac{2a'(y)}{a^2(y)}} \abs{h(y)} \leq \frac{C\delta}{\mu } 1(x \in [-1/\delta + \delta, -\zeta])  \label{eq:ebound5}
\end{align}
\endgroup
\end{lemma}
The second lemma deals with the case when $y \in (x, x+\delta)$.
\begin{lemma} \label{lem:llplus}
Consider the same setup as in Lemma~\ref{lem:llminus}, but this time let $y \in (x, x + \delta)$. Then
%Fix $h(x) \in W_2$ with $h(0)=0$, and let $f_h(x)$ be a solution to the Poisson equation \eqref{eq:poisson} that satisfies the conditions of Lemma~\ref{lem:gb}. Recall that $a(x)$ and $r(x)$ are given by \eqref{eq:abdef} and \eqref{eq:rform}, respectively. Then there exists a constant $C>0$ (independent of $\lambda, n$, and $\mu$), such that for all $x = x_k = \delta (k - R)$, where $k \in \Z_+$, all $y \in (x, x + \delta)$, and all $\lambda, n$, and $\mu$ satisfying $n \geq 1$ and $0 < \lambda < n\mu $, 
\begin{align}
&\ \abs{r'(y)-r'(x-)}\abs{f_h'(y)} + \abs{r'(x-)}\abs{f_h'(x)-f_h'(y)} \notag \\
& \hspace{0.5cm} \leq \frac{C\delta}{\mu } \Big[ \Big(1 +  \frac{1}{\abs{\zeta}}\Big) 1(x \leq -\zeta-\delta)+ \frac{1}{\delta} \Big(1 +  \frac{1}{\abs{\zeta}}\Big)1(x\in \{-1/\delta, -\zeta\}) \Big], \label{eq:epbound1} \\
&\ \abs{r(y)-r(x)} \abs{f_h''(y)} + \abs{r(x)}\abs{f_h''(x)-f_h''(y)} \notag \\
&\hspace{0.5cm} \leq \frac{C\delta }{\mu }\bigg[(1+\abs{x})\Big(1 +  \frac{1}{\abs{\zeta}}\Big)  1( x\leq -\zeta-\delta) + \abs{\zeta}1(x \geq -\zeta ) \bigg], \label{eq:epbound2} \\
&\ \abs{2/a(x) - 2/a(y)} \abs{h'(y)} + \abs{2/a(x)}\abs{h'(x-)-h'(y)}\leq  \frac{C\delta}{\mu}, \label{eq:epbound3} \\
&\ \abs{\frac{2a'(x-)}{a^2(x)} - \frac{2a'(y)}{a^2(y)}} \abs{\EE h(Y(\infty))} +  \abs{\frac{2a'(x-)}{a^2(x)}}\abs{h(x)-h(y)} \notag \\
&\hspace{0.5cm} \leq \frac{C\delta}{\mu }\Big(1 + \frac{1}{\abs{\zeta}} \Big)1(x \in [-1/\delta, -\zeta]) \label{eq:epbound4}\\
&\ \abs{\frac{2a'(x-)}{a^2(x)} - \frac{2a'(y)}{a^2(y)}} \abs{h(y)} \leq \frac{C\delta}{\mu} \Big[ 1(x \in [-1/\delta + \delta, -\zeta - \delta]) + \frac{1}{\delta} 1(x \in \{-1/\delta, -\zeta\})\Big]  \label{eq:epbound5}
\end{align}
\end{lemma}
With these two lemmas, the proof of Lemma~\ref{lem:eterm} becomes trivial.
\begin{proof}[Proof of Lemma~\ref{lem:eterm}]
When $y \in (x-\delta, x)$, we just apply \eqref{eq:ebound1}--\eqref{eq:ebound5} from Lemma~\ref{lem:llminus} to \eqref{eq:lipf} to get \eqref{eq:eboundleft}. Similarly, for $y \in (x,x+\delta)$ we apply \eqref{eq:epbound1}--\eqref{eq:epbound5} of Lemma~\ref{lem:llplus} to \eqref{eq:lipf} to get \eqref{eq:eboundright}. This concludes the proof of Lemma~\ref{lem:eterm}.
\end{proof}

\subsubsection{Proof of Lemma~\ref{lem:llminus}}
\begin{proof}[Proof of Lemma~\ref{lem:llminus}]
Fix $k \in \Z_+$, let $x = x_k = \delta(k - R)$, and fix $y \in (x-\delta, x)$. Throughout the proof we use $C> 0$ to denote a generic constant that may change from line to line, but does not depend on $\lambda, n$, and $\mu$. To prove this lemma we verify \eqref{eq:ebound1}--\eqref{eq:ebound5}, starting with \eqref{eq:ebound1}. Using the form of $r'(x) = \Big(\frac{2b(x)}{a(x)}\Big)' $ in \eqref{eq:rform}, we see that
\begin{align*}
\abs{r'(y)-r'(x-)} = \abs{r'(y)-r'(x-)} 1(x \in [-1/\delta + \delta, -\zeta]).
\end{align*} 
Furthermore, $r''(u)$ exists for all $u \in (-1/\delta, -\zeta)$, and from \eqref{eq:rform} one can see that
\begin{align*}
r''(u) = \frac{8\delta}{(2 + \delta u)^3} \leq 8\delta, \quad u \in (-1/\delta, -\zeta).
\end{align*}
Therefore, 
\begin{align}
\abs{r'(y)-r'(x-)}\abs{f_h'(y)} \leq&\ \abs{f_h'(y)}1(x \in [-1/\delta + \delta, -\zeta]) \int_{x-\delta}^{x} \abs{r''(u)} du \notag\\
 \leq&\ \frac{C\delta^2}{\mu }\Big(1 +  \frac{1}{\abs{\zeta}}\Big) 1(x \in [-1/\delta + \delta, -\zeta ]), \label{inl:e1}
\end{align}
where in the last inequality we used the gradient bound \eqref{eq:WCder1}. Furthermore, we observe that
\begin{align*}
\abs{r'(x-)} \leq&\ 4 \times 1(x \leq -\zeta),\\
\abs{f_h'(x) - f_h'(y)} \leq&\ \int_{x-\delta}^{x} \abs{f_h''(u)} du \leq \frac{C\delta}{\mu }\Big[\Big(1 +  \frac{1}{\abs{\zeta}}\Big)1(x \leq -\zeta) + \frac{1}{\abs{\zeta}} 1(x \geq -\zeta+\delta)\Big],
\end{align*}
where in the first line we used the form of $r'(x)$ from \eqref{eq:rform}, and in the second line we used the gradient bound \eqref{eq:WCder2}. Recalling that $\delta \leq 1$, we conclude that
\begin{align*}
&\ \abs{r'(y)-r'(x-)}\abs{f_h'(y)} + \abs{r'(x-)}\abs{f_h'(x)-f_h'(y)} \leq \frac{C\delta}{\mu }  \Big(1 +  \frac{1}{\abs{\zeta}}\Big)1(x \leq -\zeta).
\end{align*}
This proves \eqref{eq:ebound1}, and we move on to show \eqref{eq:ebound2}. Observe that
\begin{align}
\abs{r(x)} \leq&\ 2\abs{x}1(x \leq -\zeta) + \abs{\zeta} 1(x \geq -\zeta+\delta), \notag \\
\abs{r(x) - r(y)} \leq&\ \int_{x-\delta}^{x} \abs{r'(u)} du \leq 4\delta 1(x \leq -\zeta), \notag \\
\abs{f_h''(y)} \leq&\ \frac{C}{\mu }\Big[\Big(1 +  \frac{1}{\abs{\zeta}}\Big)1(x \leq -\zeta) + \frac{1}{\abs{\zeta}} 1(x \geq -\zeta+\delta)\Big], \notag \\
\abs{f_h''(x) - f_h''(y)} \leq&\ \int_{x-\delta}^{x} \abs{f_h'''(u)} du \leq \frac{C\delta }{\mu }\Big[\Big(1 + \frac{1}{\abs{\zeta}}\Big)1(x \leq -\zeta) +  1(x \geq -\zeta+\delta)\Big], \label{inl:e2}
\end{align}
where the first two lines above are obtained using the form of $r(x)$ in \eqref{eq:rform}, and in the last two lines we used the gradient bounds \eqref{eq:WCder2} and \eqref{eq:WCder3}. Combining the bounds above proves \eqref{eq:ebound2}, and we move on to prove \eqref{eq:ebound3}. Observe that
\begin{align}
\abs{2/a(x)} \leq&\ 2/\mu, \notag \\
\abs{2/a(x) - 2/a(y)} \leq&\ 2\int_{x-\delta}^{x} \abs{\frac{a'(u)}{a^2(u)}} du \leq \frac{2\delta}{\mu} 1(x \in [-1/\delta + \delta, -\zeta]), \notag\\
\abs{h'(x-)} \leq&\ 1, \quad \text{ and } \quad \abs{h'(x-) - h'(y)} \leq \norm{h''} \abs{x-y} \leq \delta, \label{inl:e3}
\end{align}
where in the first two lines we used the forms of $a(x)$ and $a'(x)$ from \eqref{DS:adef} and \eqref{eq:ap}, and in the last line we used the fact that $h(x) \in W_2$. Combining these bounds proves \eqref{eq:ebound3}, and we move on to prove \eqref{eq:ebound4}. Observe that 
\begin{align*}
\abs{\frac{2a'(x-)}{a^2(x)}} =&\  \abs{\frac{2\delta}{\mu (2+\delta x)^2}} 1(x \in [-1/\delta + \delta, -\zeta]) \leq \frac{2\delta}{\mu}1(x \in [-1/\delta + \delta, -\zeta]),\\
\abs{\frac{2a'(y)}{a^2(y)}} \leq&\ \frac{2\delta}{\mu}1(x \in [-1/\delta + \delta, -\zeta]),\\
\abs{\EE h(Y(\infty))} \leq&\ \EE \big| Y(\infty)\big|, \quad \text{ and } \quad \abs{h(x)-h(y)} \leq \norm{h'} \abs{x-y} \leq \delta,
\end{align*} 
where in the first line we used the forms of $a(x)$ and $a'(x)$ from \eqref{DS:adef} and \eqref{eq:ap}, and in the last line we used the fact that $h(x) \in W_2$. We use the bounds above together with \eqref{DS:fbound7} and the fact that $\delta \leq 1$ to see that
\begin{align}
&\ \abs{\frac{2a'(x-)}{a^2(x)} - \frac{2a'(y)}{a^2(y)}} \abs{\EE h(Y(\infty))} + \abs{\frac{2a'(x-)}{a^2(x)}}\abs{h(x)-h(y)} \notag \\
\leq&\ \abs{\frac{2a'(x-)}{a^2(x)}}\EE \big| Y(\infty)\big| + \abs{\frac{2a'(y)}{a^2(y)}} \EE \big| Y(\infty)\big| + \frac{2\delta^2}{\mu}1(x \in [-1/\delta + \delta, -\zeta]) \notag \\
\leq&\ \frac{C\delta}{\mu}\Big(1 + \frac{1}{\abs{\zeta}} \Big)1(x \in [-1/\delta + \delta, -\zeta]) + \frac{2\delta^2}{\mu}1(x \in [-1/\delta + \delta, -\zeta]) \notag \\
\leq&\ \frac{C\delta}{\mu }\Big(1 + \frac{1}{\abs{\zeta}} \Big)1(x \in [-1/\delta + \delta, -\zeta]), \label{inl:e4}
\end{align}
which proves \eqref{eq:ebound4}. Lastly we show \eqref{eq:ebound5}. Observe that
\begin{align*}
\abs{\frac{2a'(x-)}{a^2(x)} - \frac{2a'(y)}{a^2(y)}} = \abs{\frac{2a'(x-)}{a^2(x)} - \frac{2a'(y)}{a^2(y)}} 1(x \in [-1/\delta + \delta, -\zeta]),
\end{align*} 
and that the derivative of $2a'(u-)/a^2(u)$ exists for all $u \in (-1/\delta, -\zeta)$ and satisfies 
\begin{align*}
\abs{\bigg(\frac{2a'(u)}{a^2(u)} \bigg)'} = \frac{4\delta^2}{\mu (2 + \delta u)^3}, \quad u \in (-1/\delta, -\zeta).
\end{align*}
Recalling that $\abs{h(y)} \leq \abs{y}$, we see that
\begin{align}
\abs{\frac{2a'(x-)}{a^2(x)} - \frac{2a'(y)}{a^2(y)}} \abs{h(y)} \leq&\ 1(x \in [-1/\delta + \delta, -\zeta]) \int_{x-\delta}^{x} \abs{y}\abs{\bigg(\frac{2a'(u)}{a^2(u)} \bigg)'} du  \notag \\
=&\ 1(x \in [-1/\delta + \delta, -\zeta]) \int_{x-\delta}^{x} \frac{4\delta}{\mu (2 + \delta u)^2} \abs{\frac{\delta y}{(2 + \delta u)}} du \notag  \\
\leq&\ 1(x \in [-1/\delta + \delta, -\zeta]) \frac{4\delta}{\mu }\int_{x-\delta}^{x}  \abs{\frac{\delta y}{ (2 + \delta u)}} du \notag \\
\leq&\ 1(x \in [-1/\delta + \delta, -\zeta]) \frac{4\delta}{\mu } \delta(\delta^2+1), \label{inl:e5}
\end{align}
where to obtain the last inequality, we used the fact that $\abs{y-u} \leq \delta$ and $\delta u \geq -1$ to see that
\begin{align*}
\abs{\frac{\delta y}{ (2 + \delta u)}} = \abs{\frac{\delta (y-u) + \delta u}{ (2 + \delta u)}} \leq \delta^2 + \abs{\frac{\delta u}{ 2 + \delta u}} \leq \delta^2 + 1.
\end{align*}
Recalling that $\delta \leq 1$ establishes \eqref{eq:ebound5}, and concludes the proof of this lemma.
\end{proof}

\subsubsection{Proof of Lemma~\ref{lem:llplus}}
\begin{proof}[Proof of Lemma~\ref{lem:llplus}]
Fix $k \in \Z_+$, let $x = x_k = \delta(k - R)$, and fix $y \in (x,x+\delta)$. Throughout the proof we use $C> 0$ to denote a generic constant that may change from line to line, but does not depend on $\lambda, n$, and $\mu$. The proof for this lemma is very similar to the proof of Lemma~\ref{lem:llminus}. In most cases, the only adjustment necessary to the proof is to consider cases when $x \leq -\zeta - \delta$ and $x \geq -\zeta$, instead of $x \leq -\zeta$ and $x \geq -\zeta + \delta$.
We now verify \eqref{eq:epbound1}--\eqref{eq:epbound5} in order, starting with \eqref{eq:epbound1}. Using the form of $r'(x)$ in \eqref{eq:rform}, we see that
\begin{align*}
\abs{r'(y)-r'(x-)} =&\ \abs{r'(y)-r'(x-)} 1(x \in [-1/\delta + \delta, -\zeta  - \delta]) \\
&+ (\abs{r'(y)}+2) 1(x = -1/\delta) + \abs{r'(x-)} 1(x = -\zeta).
\end{align*} 
Therefore, 
\begin{align*}
&\ \abs{r'(y)-r'(x-)}\abs{f_h'(y)}\\
=&\ \abs{r'(y)-r'(x-)} \abs{f_h'(y)} 1(x \in [-1/\delta + \delta, -\zeta  - \delta]) \\
&+ (\abs{r'(y)}+2)\abs{f_h'(y)}  1(x = -1/\delta) + \abs{r'(x-)} \abs{f_h'(y)}1(x = -\zeta) \\
 \leq&\ \frac{C\delta^2}{\mu }\Big(1 +  \frac{1}{\abs{\zeta}}\Big) 1(x \in [-1/\delta + \delta, -\zeta - \delta]) + \frac{C}{\mu }\Big(1 +  \frac{1}{\abs{\zeta}}\Big) 1(x = -1/\delta) \\
 &+ \abs{r'(x-)} \abs{f_h'(y)}1(x = -\zeta),
\end{align*}
where in the last inequality, the first term is obtained just like in \eqref{inl:e1}, and the second term comes from the gradient bound \eqref{eq:WCder1} and the fact that $\abs{r'(y)} \leq 4$, which can be seen from \eqref{eq:rform}. Now using the gradient bounds \eqref{eq:WCder1} and \eqref{eq:WCder2}, together with the facts that $\abs{r'(\abs{\zeta}-)} \leq 4$ and $\delta \leq 1$, we see that
\begin{align*}
&\ \abs{r'(x-)} \abs{f_h'(y)}1(x = -\zeta) \\
\leq&\ \abs{r'(x-)} \abs{f_h'(x)}1(x = -\zeta) + \abs{r'(x-)} \abs{f_h'(x) - f_h'(y)}1(x = -\zeta) \\
\leq&\ \frac{C}{\mu }\Big(1 +  \frac{1}{\abs{\zeta}}\Big) 1(x = -\zeta) + \abs{r'(x-)} 1(x = -\zeta) \int_{-\zeta}^{-\zeta + \delta} \abs{f_h''(u)} du \\
\leq&\ \frac{C}{\mu }\Big(1 +  \frac{1}{\abs{\zeta}}\Big) 1(x = -\zeta),
\end{align*}
and therefore
\begin{align*}
\abs{r'(y)-r'(x-)}\abs{f_h'(y)} \leq&\ \frac{C\delta}{\mu }\Big(1 +  \frac{1}{\abs{\zeta}}\Big) 1(x \in [-1/\delta + \delta, -\zeta - \delta]) \\
&+ \frac{C}{\mu }\Big(1 +  \frac{1}{\abs{\zeta}}\Big) 1(x \in \{-1/\delta, -\zeta\}).
\end{align*}
Furthermore, 
\begin{align*}
\abs{r'(x-)}\abs{f_h'(x)-f_h'(y)} \leq&\ \abs{r'(x-)} \int_{x}^{x+\delta} \abs{f_h''(u)} du\\
\leq&\ \frac{C\delta}{\mu }\Big(1 +  \frac{1}{\abs{\zeta}}\Big)1(x \leq -\zeta-\delta) + \frac{C\delta}{\mu \abs{\zeta}}1(x = -\zeta) ,
\end{align*}
where in the second inequality we used that $\abs{r'(x)} \leq 4$ and the gradient bound in \eqref{eq:WCder2}. Recalling that $\delta \leq 1$, we can combine the bounds above to see that 
\begin{align*}
&\ \abs{r'(y)-r'(x-)}\abs{f_h'(y)} + \abs{r'(x-)}\abs{f_h'(x)-f_h'(y)} \notag \\
\leq&\ \frac{C\delta}{\mu } \Big[ \Big(1 +  \frac{1}{\abs{\zeta}}\Big) 1(x \leq -\zeta-\delta)+ \frac{1}{\delta} \Big(1 +  \frac{1}{\abs{\zeta}}\Big)1(x\in \{-1/\delta, -\zeta\}) \Big],
\end{align*}
which proves \eqref{eq:epbound1}.

 The proofs for \eqref{eq:epbound2}, \eqref{eq:epbound3}, and \eqref{eq:epbound4}, are nearly identical to the proofs of \eqref{eq:ebound2}, \eqref{eq:ebound3}, and \eqref{eq:ebound4} from Lemma~\ref{lem:llminus}, respectively, and we do not repeat them here. The only differences to note is that \eqref{eq:epbound2} is separated into the cases $x \leq -\zeta - \delta$ and $x \geq -\zeta$, as opposed to \eqref{eq:ebound2} which has $x \leq -\zeta$ and $x \geq -\zeta + \delta$. Likewise, \eqref{eq:epbound4} contains $1(x \in [-1/\delta, -\zeta])$, whereas \eqref{eq:ebound4} contains $1(x \in [-1/\delta + \delta, -\zeta])$.

 Lastly we prove \eqref{eq:epbound5}. From the form of $a'(x)$ in \eqref{eq:ap}, we see that
\begin{align*}
\abs{\frac{2a'(x-)}{a^2(x)} - \frac{2a'(y)}{a^2(y)}} =&\ \abs{\frac{2a'(x-)}{a^2(x)} - \frac{2a'(y)}{a^2(y)}} 1(x \in [-1/\delta + \delta, -\zeta-\delta]) \\
&+ \abs{\frac{2a'(y)}{a^2(y)}} 1(x = -1/\delta) + \abs{\frac{2a'(x-)}{a^2(x-)}} 1(x = -\zeta).
\end{align*}  
We can repeat the argument from \eqref{inl:e5} to get
\begin{align*}
&\abs{\frac{2a'(x-)}{a^2(x)} - \frac{2a'(y)}{a^2(y)}} \abs{h(y)} \\
 \leq&\ \frac{4\delta}{\mu } \delta(\delta^2+1)1(x \in [-1/\delta + \delta, -\zeta - \delta])  +  \abs{\frac{2a'(y)}{a^2(y)}} \abs{y}1(x=-1/\delta)\\
 & +  \abs{\frac{2a'(x-)}{a^2(x-)}}\abs{y}1(x = -\zeta).
\end{align*}
Then using \eqref{eq:ap1}, the form of $a'(x)$ in \eqref{eq:ap}, and the fact that $a(x) \leq 1/\mu$, we can bound the term above by
\begin{align*}
&\ \frac{C\delta}{\mu }1(x \in [-1/\delta + \delta, -\zeta - \delta]) +  \frac{2}{\abs{a(y)}} \abs{\frac{ya'(y)}{a(y)}}1(x=-1/\delta)  \\
&+  \frac{2}{\abs{a(x-)}} \abs{\frac{(\abs{x}+\delta)a'(x-)}{a(x-)}}1(x = -\zeta) \\
\leq&\ \frac{C\delta}{\mu }1(x \in [-1/\delta + \delta, -\zeta - \delta]) + \frac{C}{\mu }1(x=-1/\delta) + \frac{C}{\mu}1(x = -\zeta).
\end{align*}
Hence,
\begin{align*}
\abs{\frac{2a'(x-)}{a^2(x)} - \frac{2a'(y)}{a^2(y)}} \abs{h(y)} \leq&\ \frac{C\delta}{\mu} \Big[ 1(x \in [-1/\delta + \delta, -\zeta - \delta]) + \frac{1}{\delta} 1(x \in \{-1/\delta, -\zeta\})\Big],
\end{align*}
which proves \eqref{eq:epbound5} and concludes the proof of this lemma.
\end{proof}
\subsection{Proof of Lemma~\ref{lem:MDgradient_bounds} (Kolmogorov Bounds)} \label{app:MDgb}
\begin{proof}[Proof of Lemma~\ref{lem:MDgradient_bounds}]
From \eqref{eq:fprimeneg} and \eqref{eq:fprimepos}
Its not hard to check that
\begin{align*}
f_z'(w) = 
\begin{cases}
\Prob(Y_S \geq z)e^{-\int_{0}^{w} \frac{2\bar b(u)}{\bar a(u)}du}\int_{-\infty}^{w} \frac{2}{\bar a(y)}e^{\int_{0}^{y} \frac{2\bar b(u)}{\bar a(u)}du} dy , \quad w \leq z, \\
 \Prob(Y_S \leq z)e^{-\int_{0}^{w} \frac{2\bar b(u)}{\bar a(u)}du} \int_{w}^{\infty} \frac{2}{\bar a(y)} e^{\int_{0}^{y} \frac{2\bar b(u)}{\bar a(u)}du} dy , \quad w \geq z,
\end{cases}
\end{align*}
In fact, for $w \geq z \geq -\zeta$, 
\begin{align*}
f_z'(w) =&\ e^{-\int_{0}^{-\zeta} \frac{2\bar b(y)}{\bar a(y)}dy}e^{-(w+\zeta) \frac{2\bar b(-\zeta)}{\bar a(-\zeta)}} \Prob(Y_S \leq z)\int_{w}^{\infty} \frac{2}{\bar a(y)} e^{\int_{0}^{-\zeta} \frac{2\bar b(u)}{\bar a(u)}du}e^{(y+\zeta) \frac{2\bar b(-\zeta)}{\bar a(-\zeta)}} dy \\
=&\ e^{-w \frac{2\bar b(-\zeta)}{\bar a(-\zeta)}} \frac{2\Prob(Y_S \leq z)}{\bar a(-\zeta)}\int_{w}^{\infty}  e^{y \frac{2\bar b(-\zeta)}{\bar a(-\zeta)}} dy \\
=&\ -\frac{\Prob(Y_S \leq z)}{\bar b(-\zeta)},
\end{align*}
and hence, $f_z''(w) = 0$ for $w \geq z$. Applying \eqref{DS:fbound1} to the form of $f_z'(w)$ tells us that for $x \leq -\zeta$, 
\begin{align*}
\abs{f_z'(w)} \leq \frac{1}{\mu }e^{\zeta^2}(3+\abs{\zeta}),
\end{align*}
which proves \eqref{MD:gb1}. To prove the rest of the bounds on $\abs{f_z''(w)}$, we differentiate $f_z'(w)$ to see that for $w \leq z$, 
\begin{align}
\frac{1}{\Prob(Y_S \geq z)}f_z''(w) = -\frac{2\bar b(w)}{\bar a(w)}e^{-\int_{0}^{w} \frac{2\bar b(u)}{\bar a(u)}du}\int_{-\infty}^{w} \frac{2}{\bar a(y)}e^{\int_{0}^{y} \frac{2\bar b(u)}{\bar a(u)}du} dy + \frac{2}{\bar a(w)}. \label{MD:fpp}
\end{align}
We claim that the right hand side above is bounded by $2/\bar a(w) \leq 2/\mu$ when $w \leq 0$. This is true for $w = 0$ because $\bar b(0) = 0$. For $w < 0$, we use \eqref{eq:main1} to see that 
\begin{align*}
 &\ \frac{2\bar b(w)}{\bar a(w)}e^{-\int_{0}^{w} \frac{2\bar b(u)}{\bar a(u)}du}\int_{-\infty}^{w} \frac{2}{\bar a(y)}e^{\int_{0}^{y} \frac{2\bar b(u)}{\bar a(u)}du} dy \leq \frac{2}{\bar a(w)}.
\end{align*}
Combining this with the fact that $\bar b(w) > 0$ for $w < 0$ verifies our claim. When $w \in [0,-\zeta]$, we apply \eqref{DS:fbound1} and the fact that $\abs{\bar b(w)} = \mu w$ to \eqref{MD:fpp} to conclude that
\begin{align*}
\frac{1}{\Prob(Y_S \geq z)}\abs{f_z''(w)} \leq&\  \frac{2}{\bar a(w)}\Big(\mu w\frac{1}{\mu}e^{\zeta^2} (3 + \abs{\zeta})+1\Big)\\
\leq&\  \frac{C}{\mu}e^{\zeta^2}(w(1+\abs{\zeta})+1),
\end{align*}
where in the last inequality we used the fact that $2/\bar a(w) \leq 2/\mu$. This proves \eqref{MD:gb3}, and it remains to prove the bound on $\abs{f_z''(w)}$ in the case when $w \in [-\zeta, z]$. Observe that 
\begin{align*}
&\ e^{-\int_{0}^{w} \frac{2\bar b(u)}{\bar a(u)}du}\int_{-\infty}^{w} \frac{2}{\bar a(y)}e^{\int_{0}^{y} \frac{2\bar b(u)}{\bar a(u)}du} dy \\
=&\ e^{-\int_{0}^{-\zeta} \frac{2\bar b(u)}{\bar a(u)}du}e^{-\int_{-\zeta}^{w} \frac{2\bar b(u)}{\bar a(u)}du}\int_{-\infty}^{-\zeta} \frac{2}{\bar a(y)}e^{\int_{0}^{y} \frac{2\bar b(u)}{\bar a(u)}du} dy \\
&+ e^{-\int_{0}^{-\zeta} \frac{2\bar b(u)}{\bar a(u)}du}e^{-\int_{-\zeta}^{w} \frac{2\bar b(u)}{\bar a(u)}du}\int_{-\zeta}^{w} \frac{2}{\bar a(y)}e^{\int_{0}^{-\zeta} \frac{2\bar b(u)}{\bar a(u)}du}e^{\int_{-\zeta}^{y} \frac{2\bar b(u)}{\bar a(u)}du} dy\\
\leq&\ e^{-\int_{-\zeta}^{w} \frac{2\bar b(u)}{\bar a(u)}du}\frac{1}{\mu}e^{\zeta^2} (3 + \abs{\zeta})+ e^{-\int_{-\zeta}^{w} \frac{2\bar b(u)}{\bar a(u)}du}\int_{-\zeta}^{w} \frac{2}{\bar a(y)}e^{\int_{-\zeta}^{y} \frac{2\bar b(u)}{\bar a(u)}du} dy\\
=&\ e^{-(w+\zeta)\frac{2\bar b(-\zeta)}{\bar a(-\zeta)}}\frac{1}{\mu}e^{\zeta^2} (3 + \abs{\zeta})+ e^{-(w+\zeta)\frac{2\bar b(-\zeta)}{\bar a(-\zeta)}}\int_{-\zeta}^{w} \frac{2}{\bar a(-\zeta)}e^{(y+\zeta)\frac{2\bar b(-\zeta)}{\bar a(-\zeta)}} dy\\
=&\ e^{-(w+\zeta)\frac{2\bar b(-\zeta)}{\bar a(-\zeta)}}\frac{1}{\mu}e^{\zeta^2} (3 + \abs{\zeta})+ \frac{1}{\bar b(-\zeta)} \Big(1 - e^{-(w+\zeta)\frac{2\bar b(-\zeta)}{\bar a(-\zeta)}}\Big)\\
\leq&\ e^{-(w+\zeta)\frac{2\bar b(-\zeta)}{\bar a(-\zeta)}}\frac{1}{\mu}e^{\zeta^2} (3 + \abs{\zeta})+ \frac{1}{\abs{\bar b(-\zeta)}} e^{-(w+\zeta)\frac{2\bar b(-\zeta)}{\bar a(-\zeta)}}.
\end{align*}
We combine the above inequality with \eqref{MD:fpp} to see that for $w \in [-\zeta, z]$, 
and therefore 
\begin{align*}
&\frac{1}{\Prob(Y_S \geq z)}|f_z''(w)| \\
=&\ \bigg| -\frac{2\bar b(-\zeta)}{\bar a(-\zeta)}e^{-\int_{0}^{w} \frac{2\bar b(u)}{\bar a(u)}du}\int_{-\infty}^{w} \frac{2}{\bar a(y)}e^{\int_{0}^{y} \frac{2\bar b(u)}{\bar a(u)}du} dy + \frac{2}{\bar a(-\zeta)} \bigg| \\
\leq&\  \abs{\frac{2\bar b(-\zeta)}{\bar a(-\zeta)}}\Big(e^{-(w+\zeta)\frac{2\bar b(-\zeta)}{\bar a(-\zeta)}}\frac{1}{\mu}e^{\zeta^2} (3 + \abs{\zeta})+ \frac{1}{\abs{\bar b(-\zeta)}} e^{-(w+\zeta)\frac{2\bar b(-\zeta)}{\bar a(-\zeta)}}\Big)+ \frac{1}{\mu } \\
\leq&\ \frac{1}{\mu}e^{-(w+\zeta)\frac{2\bar b(-\zeta)}{\bar a(-\zeta)}}e^{\zeta^2} (3\abs{\zeta} + \zeta^2)+ \frac{1}{\mu} e^{-(w+\zeta)\frac{2\bar b(-\zeta)}{\bar a(-\zeta)}}+ \frac{1}{\mu } \\
\leq&\ \frac{1}{\mu}e^{w\frac{2\abs{\bar b(-\zeta)}}{\bar a(-\zeta)}}e^{\zeta^2} (3\abs{\zeta} + \zeta^2)+ \frac{1}{\mu} e^{w\frac{2\abs{\bar b(-\zeta)}}{\bar a(-\zeta)}}+ \frac{1}{\mu }\\
\leq&\ \frac{1}{\mu}e^{w\frac{2\abs{\bar b(-\zeta)}}{\bar a(-\zeta)}}e^{\zeta^2} (1+3\abs{\zeta} + \zeta^2)+ \frac{1}{\mu },
\end{align*}
where in the first inequality we used the fact that $a(w) \geq 2\mu$ for $w \in [-\zeta, z]$. This proves \eqref{MD:gb4} once we recall that $2\abs{\bar b(-\zeta)}/\bar a(-\zeta) = 2\abs{\zeta}/(2+\delta \abs{\zeta})$. 
\end{proof}

\section{Gradient Bounds for Chapter~\ref{chap:phasetype} } \label{app:PHgradbounds}
In this section, we prove  Lemma~\ref{lemma:poisson}. We adopt the notation from Chapter~\ref{chap:phasetype}. Before proving the lemma, we introduce an important common quadratic Lyapunov
function from \cite{DiekGao2013}.  This Lyapunov function
plays a key role in the proof of this lemma. As in (5.24) of \cite{DiekGao2013}, for
$x \in \R^d$, define
\begin{equation} \label{eq:deflyapou}
V(x) = (e^Tx)^2 + \kappa[x-p \phi(e^Tx)]' M [x-p \phi(e^Tx)],
\end{equation}
where $\kappa>0$ is some constant, $M$ is some $d\times d$ positive definite matrix, and the function $\phi$ is a smooth approximation to $x \longmapsto x^+$ and is defined by
\begin{displaymath}
\phi(x) =  \begin{cases}
    x, & \text{if $x \geq 0$},\\
    -\frac{1}{2}\epsilon, & \text{if $x \leq -\epsilon$},\\
    \text{smooth}, & \text{if $-\epsilon < x < 0$}.
  \end{cases}
\end{displaymath}
In (5.24) of \cite{DiekGao2013}, the authors use $\tilde Q$ to represent the positive definite matrix that we called $M$ in \eqref{eq:deflyapou}. We use $M$ instead of $\tilde Q$ on purpose, to avoid any potential confusion with the queue size $Q(t)$. 
For our purposes, ``smooth" means that $\phi$ can be anything as long as $\phi \in C^3(\R^d)$.  We require that the ``smooth" part of $\phi$ also satisfies $-\frac{1}{2}\epsilon < \phi (x) < x$ and $0 \leq {\phi}'(x)\leq 1$. For example, $\phi$ can be taken to be a polynomial of sufficiently high degree on $(-\epsilon, 0)$ and this will satisfy our requirements.  The vector $p$ is as in (\ref{eq:defou}). The constant $\kappa$ and matrix $M$ are chosen just as in \cite{DiekGao2013}; their exact values are not important to us. In their paper, they show that $V(x)$ satisfies
\begin{displaymath}
G_YV(x) \leq -c_1V(x) + c_2 \quad \text{ for all } x\in \R^d
\end{displaymath}
for some positive constants $c_1$,$c_2$; this result requires $\alpha > 0$, i.e.\ a strictly positive abandonment rate. Before proceeding to the proof of Lemma~\ref{lemma:poisson}, we state two bounds on $V(x)$ that shall be useful in the future. For some constant $C>0$,
\begin{eqnarray} 
V(x) \leq C(1+ \abs{x}^2), \label{eq:lyapupperbound}\\
\abs{x}^2 \leq C(1+V(x)). \label{eq:lyaplowbound}
\end{eqnarray}
The first is immediate from the form of $V(x)$, while the second is proved in \cite{DiekGao2013}.

\begin{proof}[Proof of Lemma~\ref{lemma:poisson}]
Without loss of generality, we may assume that $h(0) = 0$, otherwise one may consider $h(x) -h(0)$. This lemma is essentially a restatement of equation (22) and equation (40) from the discussion that follows after \cite[Theorem 4.1]{Gurv2014}. We verify that (22) and (40) are applicable in our case by first confirming that we have a function satisfying assumption 3.1 of \cite{Gurv2014}. Recalling the definition of $V(x)$ from (\ref{eq:deflyapou}), when $\phi$ is taken to be a polynomial (of sufficiently high degree to guarantee $V(x) \in C^3(\R^d)$), the function 
\begin{displaymath}
1+V(x)
\end{displaymath}
satisfies assumption 3.1. To verify condition (17) of Assumption 3.1, one observes that 
\begin{displaymath}
X^{(\lambda)}(t) \leq X^{(\lambda)}(0) + n + A^{(\lambda)}(t),
\end{displaymath}
where $A^{(\lambda)}(t)$ is the total number of arrivals to the system by time $t$
and it is a Poisson random variable with mean $\lambda t$ for each
$t\ge 0$. The properties of Poisson processes then yield (17). By
\cite[Remark 3.4]{Gurv2014},
\begin{displaymath}
C(1+V(x))^m
\end{displaymath}
also satisfies assumption 3.1 for any constant $C>0$. Since we require that $\abs{h(x)} \leq \abs{x}^m$, by (\ref{eq:lyaplowbound}) we have 
\begin{displaymath}
\abs{h(x) - \E h(Y(\infty))} \leq \abs{x}^m + \E \abs{Y(\infty)}^m \leq C_m (1+V(x))^m.
\end{displaymath}
The finiteness of $\E \abs{Y(\infty)}^m$ is guaranteed because one of the conditions of assumption 3.1 is that
\begin{displaymath}
G_Y (1+V(x))^m \leq -c_1 (1+V(x))^m + c_2
\end{displaymath}
for some positive constants $c_1$ and $c_2$.\ Therefore, equation (22) gives us (\ref{eq:gradbound1}) and equation (40) gives us (\ref{eq:gradbound2}) and (\ref{eq:gradbound3}). We get (\ref{eq:gradbound4}) by observing that in the discussion preceding (40), everything still holds if we replace $B_x(\bar l / \sqrt{n})$ by an open ball of radius $1$ centered at $x$. We wish to point out that the constants in (40) and (22) do not depend on the choice of function $h(x)$.
\end{proof}

\bibliography{dai20170325}

\def\cprime{$'$} \def\cprime{$'$} \def\cprime{$'$} \def\cprime{$'$}
  \def\cprime{$'$} \def\cprime{$'$} \def\cprime{$'$}
\begin{thebibliography}{92}
\expandafter\ifx\csname natexlab\endcsname\relax\def\natexlab#1{#1}\fi
\expandafter\ifx\csname url\endcsname\relax
  \def\url#1{\texttt{#1}}\fi
\expandafter\ifx\csname urlprefix\endcsname\relax\def\urlprefix{URL }\fi
\providecommand{\eprint}[2][]{\url{#2}}

\bibitem[{Aksin et~al.({2007})Aksin, Armony and Mehrotra}]{AksiArmoMehr2007}
\textsc{Aksin, Z.}, \textsc{Armony, M.} and \textsc{Mehrotra, V.} ({2007}).
\newblock {The modern call center: a multi-disciplinary perspective on
  operations management research}.
\newblock \textit{Production and Operations Management}, \textbf{{16}}
  {665--688}.

\bibitem[{Armony et~al.(2011)Armony, Israelit, Mandelbaum, Marmor, Tseytlin and
  Yom-Tov}]{Armoetal2011}
\textsc{Armony, M.}, \textsc{Israelit, S.}, \textsc{Mandelbaum, A.},
  \textsc{Marmor, Y.~N.}, \textsc{Tseytlin, Y.} and \textsc{Yom-Tov, G.~B.}
  (2011).
\newblock Patient flow in hospitals: A data-based queueing-science perspective.
\newblock \textit{working paper}.
\newblock
  \urlprefix\url{http://www.stern.nyu.edu/om/faculty/armony/Patient%20flow%20main.pdf}.

\bibitem[{Asmussen(2003)}]{Asmu2003}
\textsc{Asmussen, S.} (2003).
\newblock \textit{Applied probability and queues}, vol.~51 of
  \textit{Applications of Mathematics (New York)}.
\newblock 2nd ed. Springer-Verlag, New York.
\newblock Stochastic Modelling and Applied Probability.

\bibitem[{Atar(2012)}]{Atar2012}
\textsc{Atar, R.} (2012).
\newblock A diffusion regime with nondegenerate slowdown.
\newblock \textit{Operations Research}, \textbf{60} 490--500.
\newblock \urlprefix\url{http://dx.doi.org/10.1287/opre.1110.1030}.

\bibitem[{Barbour(1990)}]{Barb1990}
\textsc{Barbour, A.} (1990).
\newblock Stein's method for diffusion approximations.
\newblock \textit{Probability Theory and Related Fields}, \textbf{84} 297--322.
\newblock \urlprefix\url{http://dx.doi.org/10.1007/BF01197887}.

\bibitem[{Barbour and Brown(1992)}]{BarbBrow1992}
\textsc{Barbour, A.} and \textsc{Brown, T.} (1992).
\newblock Stein's method and point process approximation.
\newblock \textit{Stochastic Processes and their Applications}, \textbf{43} 9
  -- 31.
\newblock \urlprefix\url{http://dx.doi.org/10.1016/0304-4149(92)90073-Y}.

\bibitem[{Barbour and Xia(2006)}]{BarbXia2006}
\textsc{Barbour, A.} and \textsc{Xia, A.} (2006).
\newblock On {S}tein's factors for {P}oisson approximation in {W}asserstein
  distance.
\newblock \textit{Bernoulli}, \textbf{12} 943--954.
\newblock \urlprefix\url{http://dx.doi.org/10.3150/bj/1165269145}.

\bibitem[{Barbour(1988)}]{Barb1988}
\textsc{Barbour, A.~D.} (1988).
\newblock Stein's method and {P}oisson process convergence.
\newblock \textit{Journal of Applied Probability}, \textbf{25} pp. 175--184.
\newblock \urlprefix\url{http://www.jstor.org/stable/3214155}.

\bibitem[{Bell and Williams(2005)}]{BellWill2005}
\textsc{Bell, S.~L.} and \textsc{Williams, R.~J.} (2005).
\newblock Dynamic scheduling of a parallel server system in heavy traffic with
  complete resource pooling: asymptotic optimality of a threshold policy.
\newblock \textit{Electronic Journal of Probability}, \textbf{10} 1044--1115.
\newblock \urlprefix\url{http://projecteuclid.org/euclid.ejp/1464816834}.

\bibitem[{Blanchet and Glynn(2007)}]{BlanGlyn2007}
\textsc{Blanchet, J.} and \textsc{Glynn, P.} (2007).
\newblock Uniform renewal theory with applications to expansions of random
  geometric sums.
\newblock \textit{Advances in Applied Probability}, \textbf{39} 1070--1097.

\bibitem[{Borovkov(1964)}]{Boro1964}
\textsc{Borovkov, A.} (1964).
\newblock Some limit theorems in the theory of mass service, {I}.
\newblock \textit{Theory of Probability and its Applications}, \textbf{9}
  550--565.

\bibitem[{Borovkov(1965)}]{Boro1965}
\textsc{Borovkov, A.} (1965).
\newblock Some limit theorems in the theory of mass service, {II}.
\newblock \textit{Theory of Probability and its Applications}, \textbf{10}
  375--400.

\bibitem[{Bramson(1998)}]{Bram1998a}
\textsc{Bramson, M.} (1998).
\newblock State space collapse with application to heavy traffic limits for
  multiclass queueing networks.
\newblock \textit{Queueing Systems}, \textbf{30} 89--140.
\newblock \urlprefix\url{http://dx.doi.org/10.1023/A:1019160803783}.

\bibitem[{Braverman and Dai(2017)}]{BravDai2017}
\textsc{Braverman, A.} and \textsc{Dai, J.~G.} (2017).
\newblock Stein's method for steady-state diffusion approximations of
  ${M}/\mathit{Ph}/n+{M}$ systems.
\newblock \textit{Ann. Appl. Probab.}, \textbf{27} 550--581.

\bibitem[{{Braverman} et~al.(2016){Braverman}, {Dai} and
  {Feng}}]{BravDaiFeng2016}
\textsc{{Braverman}, A.}, \textsc{{Dai}, J.~G.} and \textsc{{Feng}, J.} (2016).
\newblock {Stein's method for steady-state diffusion approximations: an
  introduction through the Erlang-A and Erlang-C models}.
\newblock \textit{Stochastic Systems}, \textbf{6} 301--366.
\newblock
  \urlprefix\url{http://www.i-journals.org/ssy/viewarticle.php?id=212&layout=abstract}.

\bibitem[{Brown and Xia(2001)}]{BrowXia2001}
\textsc{Brown, T.~C.} and \textsc{Xia, A.} (2001).
\newblock Stein's method and birth-death processes.
\newblock \textit{Ann. Probab.}, \textbf{29} 1373--1403.
\newblock \urlprefix\url{http://dx.doi.org/10.1214/aop/1015345606}.

\bibitem[{Budhiraja and Lee(2009)}]{BudhLee2009}
\textsc{Budhiraja, A.} and \textsc{Lee, C.} (2009).
\newblock Stationary distribution convergence for generalized {Jackson}
  networks in heavy traffic.
\newblock \textit{Mathematics of Operations Research}, \textbf{34} 45--56.

\bibitem[{Chang et~al.(2016)Chang, Shao and Zhou}]{ChangShaoZhou2016}
\textsc{Chang, J.}, \textsc{Shao, Q.-M.} and \textsc{Zhou, W.-X.} (2016).
\newblock Cram\'er-type moderate deviations for {S}tudentized two-sample
  {$U$}-statistics with applications.
\newblock \textit{Ann. Statist.}, \textbf{44} 1931--1956.
\newblock \urlprefix\url{http://dx.doi.org/10.1214/15-AOS1375}.

\bibitem[{Chatterjee(2014)}]{Chat2014}
\textsc{Chatterjee, S.} (2014).
\newblock A short survey of {S}tein's method.
\newblock To appear in {Proceedings of ICM 2014},
  \urlprefix\url{http://arxiv.org/abs/1404.1392}.

\bibitem[{Chen(1975)}]{Chen1975}
\textsc{Chen, L. H.~Y.} (1975).
\newblock Poisson approximation for dependent trials.
\newblock \textit{Ann. Probab.}, \textbf{3} 534--545.
\newblock \urlprefix\url{http://dx.doi.org/10.1214/aop/1176996359}.

\bibitem[{Chen et~al.(2013{\natexlab{a}})Chen, Fang and
  Shao}]{ChenFangShao2013a}
\textsc{Chen, L. H.~Y.}, \textsc{Fang, X.} and \textsc{Shao, Q.-M.}
  (2013{\natexlab{a}}).
\newblock From {S}tein identities to moderate deviations.
\newblock \textit{Ann. Probab.}, \textbf{41} 262--293.
\newblock \urlprefix\url{http://dx.doi.org/10.1214/12-AOP746}.

\bibitem[{Chen et~al.(2013{\natexlab{b}})Chen, Fang and
  Shao}]{ChenFangShao2013b}
\textsc{Chen, L. H.~Y.}, \textsc{Fang, X.} and \textsc{Shao, Q.-M.}
  (2013{\natexlab{b}}).
\newblock Moderate deviations in {P}oisson approximation: a first attempt.
\newblock \textit{Statist. Sinica}, \textbf{23} 1523--1540.

\bibitem[{Chen et~al.(2011)Chen, Goldstein and Shao}]{ChenGoldShao2011}
\textsc{Chen, L. H.~Y.}, \textsc{Goldstein, L.} and \textsc{Shao, Q.-M.}
  (2011).
\newblock \textit{Normal approximation by {S}tein's method}.
\newblock Probability and its Applications (New York), Springer, Heidelberg.
\newblock \urlprefix\url{http://dx.doi.org/10.1007/978-3-642-15007-4}.

\bibitem[{Chen et~al.(2016)Chen, Shao, Wu and Xu}]{ChenShaoWuXu2016}
\textsc{Chen, X.}, \textsc{Shao, Q.-M.}, \textsc{Wu, W.~B.} and \textsc{Xu, L.}
  (2016).
\newblock Self-normalized {C}ram\'er-type moderate deviations under dependence.
\newblock \textit{Ann. Statist.}, \textbf{44} 1593--1617.
\newblock \urlprefix\url{http://dx.doi.org/10.1214/15-AOS1429}.

\bibitem[{Cram\'er(1938)}]{Cram1938}
\textsc{Cram\'er, H.} (1938).
\newblock Sur un nouveau th\'or\`eme-limite de la th\'eorie des probabilit\'es.
\newblock \textit{Actualit\'es Scientifiques et Industrielles}, \textbf{736}
  5--23.

\bibitem[{Dai et~al.(2014)Dai, Dieker and Gao}]{DaiDiekGao2014}
\textsc{Dai, J.~G.}, \textsc{Dieker, A.} and \textsc{Gao, X.} (2014).
\newblock Validity of heavy-traffic steady-state approximations in many-server
  queues with abandonment.
\newblock \textit{Queueing Systems}, \textbf{78} 1--29.
\newblock \urlprefix\url{http://dx.doi.org/10.1007/s11134-014-9394-x}.

\bibitem[{Dai and He(2013)}]{DaiHe2013}
\textsc{Dai, J.~G.} and \textsc{He, S.} (2013).
\newblock Many-server queues with customer abandonment: Numerical analysis of
  their diffusion model.
\newblock \textit{Stochastic Systems}, \textbf{3} 96--146.
\newblock \urlprefix\url{http://dx.doi.org/10.1214/11-SSY029}.

\bibitem[{Dai et~al.(2010)Dai, He and Tezcan}]{DaiHeTezc2010}
\textsc{Dai, J.~G.}, \textsc{He, S.} and \textsc{Tezcan, T.} (2010).
\newblock Many-server diffusion limits for ${G/Ph/n+GI}$ queues.
\newblock \textit{Annals of Applied Probability}, \textbf{20} 1854--1890.

\bibitem[{Dai and Lin(2008)}]{DaiLin2008}
\textsc{Dai, J.~G.} and \textsc{Lin, W.} (2008).
\newblock Asymptotic optimality of maximum pressure policies in stochastic
  processing networks.
\newblock \textit{Annals of Applied Probability}, \textbf{18} 2239--2299.

\bibitem[{Dai and Tezcan({2011})}]{DaiTezc2011}
\textsc{Dai, J.~G.} and \textsc{Tezcan, T.} ({2011}).
\newblock {State space collapse in many-server diffusion limits of parallel
  server systems}.
\newblock \textit{Mathematics of Operations Research}, \textbf{{36}}
  {271--320}.

\bibitem[{Dieker and Gao(2013)}]{DiekGao2013}
\textsc{Dieker, A.} and \textsc{Gao, X.} (2013).
\newblock Positive recurrence of piecewise {O}rnstein--{U}hlenbeck processes
  and common quadratic {L}yapunov functions.
\newblock \textit{The Annals of Applied Probability}, \textbf{23} 1291--1317.
\newblock \urlprefix\url{http://dx.doi.org/10.1214/12-AAP870}.

\bibitem[{Ehm(1991)}]{Ehm1991}
\textsc{Ehm, W.} (1991).
\newblock Binomial approximation to the {Poisson} binomial distribution.
\newblock \textit{Statistics \& Probability Letters}, \textbf{11} 7 -- 16.
\newblock
  \urlprefix\url{http://www.sciencedirect.com/science/article/pii/016771529190170V}.

\bibitem[{Eryilmaz and Srikant(2012)}]{EryiSrik2012}
\textsc{Eryilmaz, A.} and \textsc{Srikant, R.} (2012).
\newblock Asymptotically tight steady-state queue length bounds implied by
  drift conditions.
\newblock \textit{Queueing Systems}, \textbf{72} 311--359.
\newblock \urlprefix\url{http://dx.doi.org/10.1007/s11134-012-9305-y}.

\bibitem[{Ethier and Kurtz(1986)}]{EthiKurt1986}
\textsc{Ethier, S.~N.} and \textsc{Kurtz, T.~G.} (1986).
\newblock \textit{{M}arkov Processes: Characterization and Convergence}.
\newblock Wiley, New York.

\bibitem[{Foschini and Salz(1978)}]{FoscSalz1978}
\textsc{Foschini, G.~J.} and \textsc{Salz, J.} (1978).
\newblock A basic dynamic routing problem and diffusion.
\newblock \textit{IEEE Transactions on Communications}, \textbf{26} 320--327.

\bibitem[{Gamarnik and Stolyar(2012)}]{GamaStol2012}
\textsc{Gamarnik, D.} and \textsc{Stolyar, A.~L.} (2012).
\newblock Multiclass multiserver queueing system in the {Halfin-Whitt} heavy
  traffic regime: asymptotics of the stationary distribution.
\newblock \textit{Queueing Systems}, \textbf{71} 25--51.
\newblock \urlprefix\url{http://dl.acm.org/citation.cfm?id=2339029}.

\bibitem[{Gamarnik and Zeevi(2006)}]{GamaZeev2006}
\textsc{Gamarnik, D.} and \textsc{Zeevi, A.} (2006).
\newblock Validity of heavy traffic steady-state approximation in generalized
  {J}ackson networks.
\newblock \textit{Ann. Appl. Probab.}, \textbf{16} 56--90.

\bibitem[{Gan and Xia(2015)}]{GanXia2015}
\textsc{Gan, H.} and \textsc{Xia, A.} (2015).
\newblock {S}tein’s method for conditional compound {P}oisson approximation.
\newblock \textit{Statistics \& Probability Letters}, \textbf{100} 19 -- 26.
\newblock
  \urlprefix\url{http://www.sciencedirect.com/science/article/pii/S0167715215000486}.

\bibitem[{Gans et~al.(2003)Gans, Koole and Mandelbaum}]{GansKoolMand2003}
\textsc{Gans, N.}, \textsc{Koole, G.} and \textsc{Mandelbaum, A.} (2003).
\newblock Telephone call centers: {T}utorial, review, and research prospects.
\newblock \textit{Manufacturing \& Service Operations Management}, \textbf{5}
  79--141.
\newblock \eprint{http://msom.journal.informs.org/cgi/reprint/5/2/79.pdf},
  \urlprefix\url{http://msom.journal.informs.org/cgi/content/abstract/5/2/79}.

\bibitem[{Gibbs and Su(2002)}]{GibbSu2002}
\textsc{Gibbs, A.~L.} and \textsc{Su, F.~E.} (2002).
\newblock On choosing and bounding probability metrics.
\newblock \textit{International Statistical Review / Revue Internationale de
  Statistique}, \textbf{70} pp. 419--435.
\newblock \urlprefix\url{http://www.jstor.org/stable/1403865}.

\bibitem[{Glynn and Zeevi(2008)}]{GlynZeev2008}
\textsc{Glynn, P.~W.} and \textsc{Zeevi, A.} (2008).
\newblock Bounding stationary expectations of {M}arkov processes.
\newblock In \textit{Markov processes and related topics: a {F}estschrift for
  {T}homas {G}. {K}urtz}, vol.~4 of \textit{Inst. Math. Stat. Collect.} Inst.
  Math. Statist., Beachwood, OH, 195--214.
\newblock \urlprefix\url{http://dx.doi.org/10.1214/074921708000000381}.

\bibitem[{G{\"o}tze(1991)}]{Gotz1991}
\textsc{G{\"o}tze, F.} (1991).
\newblock On the rate of convergence in the multivariate {CLT}.
\newblock \textit{Ann. Probab.}, \textbf{19} 724--739.
\newblock \urlprefix\url{http://dx.doi.org/10.1214/aop/1176990448}.

\bibitem[{Gurvich(2014{\natexlab{a}})}]{Gurv2014}
\textsc{Gurvich, I.} (2014{\natexlab{a}}).
\newblock Diffusion models and steady-state approximations for exponentially
  ergodic {M}arkovian queues.
\newblock \textit{The Annals of Applied Probability}, \textbf{24} 2527--2559.
\newblock \urlprefix\url{http://dx.doi.org/10.1214/13-AAP984}.

\bibitem[{Gurvich(2014{\natexlab{b}})}]{Gurv2014a}
\textsc{Gurvich, I.} (2014{\natexlab{b}}).
\newblock Validity of heavy-traffic steady-state approximations in multiclass
  queueing networks: the case of queue-ratio disciplines.
\newblock \textit{Mathematics of Operations Research}, \textbf{39} 121--162.
\newblock \urlprefix\url{http://dx.doi.org/10.1287/moor.2013.0593}.

\bibitem[{Gurvich et~al.(2014)Gurvich, Huang and Mandelbaum}]{GurvHuanMand2014}
\textsc{Gurvich, I.}, \textsc{Huang, J.} and \textsc{Mandelbaum, A.} (2014).
\newblock Excursion-based universal approximations for the {Erlang-A} queue in
  steady-state.
\newblock \textit{Mathematics of Operations Research}, \textbf{39} 325--373.
\newblock \urlprefix\url{http://dx.doi.org/10.1287/moor.2013.0606}.

\bibitem[{Halfin and Whitt(1981)}]{HalfWhit1981}
\textsc{Halfin, S.} and \textsc{Whitt, W.} (1981).
\newblock Heavy-traffic limits for queues with many exponential servers.
\newblock \textit{Oper. Res.}, \textbf{29} 567--588.

\bibitem[{Harrison(1978)}]{Harr1978}
\textsc{Harrison, J.~M.} (1978).
\newblock The diffusion approximation for tandem queues in heavy traffic.
\newblock \textit{Advances in Applied Probability}, \textbf{10} 886--905.

\bibitem[{Harrison(1998)}]{Harr1998}
\textsc{Harrison, J.~M.} (1998).
\newblock Heavy traffic analysis of a system with parallel servers: asymptotic
  analysis of discrete-review policies.
\newblock \textit{Annals of Applied Probability}, \textbf{8} 822--848.
\newblock \urlprefix\url{http://projecteuclid.org/euclid.aoap/1028903452}.

\bibitem[{Harrison and L\'opez(1999)}]{HarrLope1999}
\textsc{Harrison, J.~M.} and \textsc{L\'opez, M.~J.} (1999).
\newblock Heavy traffic resource pooling in parallel-server systems.
\newblock \textit{Queueing Systems}, \textbf{33} 339--368.
\newblock \urlprefix\url{http://dx.doi.org/10.1023/A:1019188531950}.

\bibitem[{Harrison and Nguyen(1993)}]{HarrNguy1993}
\textsc{Harrison, J.~M.} and \textsc{Nguyen, V.} (1993).
\newblock {Brownian} models of multiclass queueing networks: Current status and
  open problems.
\newblock \textit{Queueing Systems: Theory and Applications}, \textbf{13}
  5--40.

\bibitem[{Harrison and Williams(1987)}]{HarrWill1987}
\textsc{Harrison, J.~M.} and \textsc{Williams, R.~J.} (1987).
\newblock {Brownian} models of open queueing networks with homogeneous customer
  populations.
\newblock \textit{Stochastics}, \textbf{22} 77--115.

\bibitem[{Henderson(1997)}]{Hend1997}
\textsc{Henderson, S.~G.} (1997).
\newblock \textit{Variance reduction via an approximating Markov process}.
\newblock Ph.D. thesis, Department of Operations Research, Stanford University.
\newblock \url{http://people.orie.cornell.edu/shane/pubs/thesis.pdf}.

\bibitem[{Huang and Gurvich(2016)}]{GurvHuan2016}
\textsc{Huang, J.} and \textsc{Gurvich, I.} (2016).
\newblock Beyond heavy-traffic regimes: universal bounds and controls for the
  single-server queue.
\newblock Submitted for publication,
  \urlprefix\url{http://papers.ssrn.com/sol3/papers.cfm?abstract_id=2784752}.

\bibitem[{Iglehart and Whitt(1970{\natexlab{a}})}]{IgleWhit1970}
\textsc{Iglehart, D.~L.} and \textsc{Whitt, W.} (1970{\natexlab{a}}).
\newblock Multiple channel queues in heavy traffic {I}.
\newblock \textit{Advances in Applied Probability}, \textbf{2} 150--177.

\bibitem[{Iglehart and Whitt(1970{\natexlab{b}})}]{IgleWhit1970a}
\textsc{Iglehart, D.~L.} and \textsc{Whitt, W.} (1970{\natexlab{b}}).
\newblock Multiple channel queues in heavy traffic {II}: sequences, networks,
  and batches.
\newblock \textit{Advances in Applied Probability}, \textbf{2} 355--369.

\bibitem[{Janssen et~al.(2008{\natexlab{a}})Janssen, van Leeuwaarden and
  Zwart}]{JansLeeuZwar2008}
\textsc{Janssen, A. J. E.~M.}, \textsc{van Leeuwaarden, J. S.~H.} and
  \textsc{Zwart, B.} (2008{\natexlab{a}}).
\newblock Corrected asymptotics for a multi-server queue in the
  {H}alfin-{W}hitt regime.
\newblock \textit{Queueing Syst.}, \textbf{58} 261--301.
\newblock \urlprefix\url{http://dx.doi.org/10.1007/s11134-008-9070-0}.

\bibitem[{Janssen et~al.(2008{\natexlab{b}})Janssen, van Leeuwaarden and
  Zwart}]{JansLeeuZwar2008a}
\textsc{Janssen, A. J. E.~M.}, \textsc{van Leeuwaarden, J. S.~H.} and
  \textsc{Zwart, B.} (2008{\natexlab{b}}).
\newblock Gaussian expansions and bounds for the {P}oisson distribution applied
  to the {E}rlang {B} formula.
\newblock \textit{Adv. in Appl. Probab.}, \textbf{40} 122--143.
\newblock \urlprefix\url{http://dx.doi.org/10.1239/aap/1208358889}.

\bibitem[{Janssen et~al.(2011)Janssen, van Leeuwaarden and
  Zwart}]{JansLeeuZwar2011}
\textsc{Janssen, A. J. E.~M.}, \textsc{van Leeuwaarden, J. S.~H.} and
  \textsc{Zwart, B.} (2011).
\newblock Refining square-root safety staffing by expanding {E}rlang {C}.
\newblock \textit{Operations Research}, \textbf{59} 1512--1522.
\newblock \eprint{http://dx.doi.org/10.1287/opre.1110.0991},
  \urlprefix\url{http://dx.doi.org/10.1287/opre.1110.0991}.

\bibitem[{Kang et~al.(2009)Kang, Kelly, Lee and Williams}]{KangKellLeeWill2009}
\textsc{Kang, W.}, \textsc{Kelly, F.}, \textsc{Lee, N.} and \textsc{Williams,
  R.} (2009).
\newblock State space collapse and diffusion approximation for a network
  operating under a fair bandwidth sharing policy.
\newblock \textit{The Annals of Applied Probability}, \textbf{19} 1719--1780.

\bibitem[{Katsuda(2010)}]{Kats2010}
\textsc{Katsuda, T.} (2010).
\newblock State-space collapse in stationarity and its application to a
  multiclass single-server queue in heavy traffic.
\newblock \textit{Queueing Systems: Theory and Applications}, \textbf{65}
  237--273.

\bibitem[{Knoblauch(2008)}]{Knob2008}
\textsc{Knoblauch, A.} (2008).
\newblock Closed-form expressions for the moments of the binomial probability
  distribution.
\newblock \textit{SIAM Journal on Applied Mathematics}, \textbf{69} 197--8.
\newblock
  \urlprefix\url{http://search.proquest.com/docview/915993823?accountid=10267}.

\bibitem[{Kusuoka and Tudor(2012)}]{KusuTudo2012}
\textsc{Kusuoka, S.} and \textsc{Tudor, C.~A.} (2012).
\newblock Stein's method for invariant measures of diffusions via {Malliavin}
  calculus.
\newblock \textit{Stochastic Processes and their Applications}, \textbf{122}
  1627 -- 1651.
\newblock
  \urlprefix\url{http://www.sciencedirect.com/science/article/pii/S0304414912000270}.

\bibitem[{Lindvall(1992)}]{Lind1992}
\textsc{Lindvall, T.} (1992).
\newblock \textit{Lectures on the coupling method}.
\newblock Wiley series in probability and mathematical statistics, Wiley, New
  York.
\newblock A Wiley-Interscience publication.

\bibitem[{Loh(1992)}]{Loh1992}
\textsc{Loh, W.-L.} (1992).
\newblock Stein's method and multinomial approximation.
\newblock \textit{Ann. Appl. Probab.}, \textbf{2} 536--554.
\newblock \urlprefix\url{http://dx.doi.org/10.1214/aoap/1177005648}.

\bibitem[{Mackey and Gorham(2016)}]{GorhMack2016}
\textsc{Mackey, L.} and \textsc{Gorham, J.} (2016).
\newblock Multivariate {S}tein factors for a class of strongly log-concave
  distributions.
\newblock \textit{Electron. Commun. Probab.}, \textbf{21} 14 pp.
\newblock \urlprefix\url{http://dx.doi.org/10.1214/16-ECP15}.

\bibitem[{Mandelbaum et~al.(1998)Mandelbaum, Massey and
  Reiman}]{MandMassReim1998}
\textsc{Mandelbaum, A.}, \textsc{Massey, W.~A.} and \textsc{Reiman, M.~I.}
  (1998).
\newblock Strong approximations for {Markovian} service networks.
\newblock \textit{Queueing Systems}, \textbf{30} 149--201.

\bibitem[{Meyn and Tweedie(1993)}]{MeynTwee1993b}
\textsc{Meyn, S.~P.} and \textsc{Tweedie, R.~L.} (1993).
\newblock Stability of {M}arkovian processes {III}: {F}oster-{L}yapunov
  criteria for continuous time processes.
\newblock \textit{Adv.\ Appl.\ Probab.}, \textbf{25} 518--548.

\bibitem[{Pardoux and Veretennikov(2001)}]{PardVere2001}
\textsc{Pardoux, E.} and \textsc{Veretennikov, Y.} (2001).
\newblock On the {P}oisson equation and diffusion approximation. {I}.
\newblock \textit{Ann. Probab.}, \textbf{29} 1061--1085.
\newblock \urlprefix\url{http://dx.doi.org/10.1214/aop/1015345596}.

\bibitem[{Peterson(1991)}]{Pete1991}
\textsc{Peterson, W.~P.} (1991).
\newblock A heavy traffic limit theorem for networks of queues with multiple
  customer types.
\newblock \textit{Mathematics of Operations Research}, \textbf{16} 90--118.

\bibitem[{Petrov(1975)}]{Petr1975}
\textsc{Petrov, V.~V.} (1975).
\newblock \textit{Sums of independent random variables}.
\newblock Springer-Verlag, New York-Heidelberg.
\newblock Translated from the Russian by A. A. Brown, Ergebnisse der Mathematik
  und ihrer Grenzgebiete, Band 82.

\bibitem[{Reed(2009)}]{Reed2009}
\textsc{Reed, J.} (2009).
\newblock The ${G/GI/N}$ queue in the {Halfin-Whitt} regime.
\newblock \textit{Annals of Applied Probability}, \textbf{19} 2211--2269.

\bibitem[{Reiman(1984{\natexlab{a}})}]{Reim1984}
\textsc{Reiman, M.~I.} (1984{\natexlab{a}}).
\newblock Open queueing networks in heavy traffic.
\newblock \textit{Mathematics of Operations Research}, \textbf{9} 441--458.
\newblock \urlprefix\url{http://dx.doi.org/10.1287/moor.9.3.441}.

\bibitem[{Reiman(1984{\natexlab{b}})}]{Reim1984b}
\textsc{Reiman, M.~I.} (1984{\natexlab{b}}).
\newblock Some diffusion approximations with state space collapse.
\newblock In \textit{Modeling and Performance Evaluation Methodology}
  (F.~Baccelli and G.~Fayolle, eds.). Springer, Berlin, 209--240.
\newblock \urlprefix\url{http://dx.doi.org/10.1007/BFb0005175}.

\bibitem[{Ross(2011)}]{Ross2011}
\textsc{Ross, N.} (2011).
\newblock Fundamentals of {S}tein's method.
\newblock \textit{Probab. Surv.}, \textbf{8} 210--293.
\newblock \urlprefix\url{http://dx.doi.org/10.1214/11-PS182}.

\bibitem[{Shao and Zhou(2016)}]{ShaoZhou2016}
\textsc{Shao, Q.-M.} and \textsc{Zhou, W.-X.} (2016).
\newblock Cram\'er type moderate deviation theorems for self-normalized
  processes.
\newblock \textit{Bernoulli}, \textbf{22} 2029--2079.
\newblock \urlprefix\url{http://dx.doi.org/10.3150/15-BEJ719}.

\bibitem[{Shi et~al.(2016)Shi, Chou, Dai, Ding and Sim}]{Shietal2015}
\textsc{Shi, P.}, \textsc{Chou, M.~C.}, \textsc{Dai, J.~G.}, \textsc{Ding, D.}
  and \textsc{Sim, J.} (2016).
\newblock Models and insights for hospital inpatient operations: Time-dependent
  ed boarding time.
\newblock \textit{Management Science}, \textbf{62} 1--28.

\bibitem[{Stein(1972)}]{Stei1972}
\textsc{Stein, C.} (1972).
\newblock A bound for the error in the normal approximation to the distribution
  of a sum of dependent random variables.
\newblock In \textit{Proceedings of the Sixth Berkeley Symposium on
  Mathematical Statistics and Probability, Volume 2: Probability Theory}.
  University of California Press, Berkeley, Calif., 583--602.
\newblock \urlprefix\url{http://projecteuclid.org/euclid.bsmsp/1200514239}.

\bibitem[{Stolyar(2004)}]{Stol2004}
\textsc{Stolyar, A.~L.} (2004).
\newblock Maxweight scheduling in a generalized switch: state space collapse
  and workload minimization in heavy traffic.
\newblock \textit{Ann. Appl. Probab.}, \textbf{14} 1--53.
\newblock \urlprefix\url{http://dx.doi.org/10.1214/aoap/1075828046}.

\bibitem[{Stolyar(2015)}]{Stol2015}
\textsc{Stolyar, A.~L.} (2015).
\newblock Tightness of stationary distributions of a flexible-server system in
  the {H}alfin-{W}hitt asymptotic regime.
\newblock \textit{Stoch. Syst.}, \textbf{5} 239--267.
\newblock \urlprefix\url{http://dx.doi.org/10.1214/14-SSY139}.

\bibitem[{Stroock and Varadhan(1979)}]{StroVara1979}
\textsc{Stroock, D.~W.} and \textsc{Varadhan, S. R.~S.} (1979).
\newblock \textit{Multidimensional Diffusion Processes}.
\newblock Springer, New York.

\bibitem[{Tezcan(2008)}]{Tezc2008}
\textsc{Tezcan, T.} (2008).
\newblock Optimal control of distributed parallel server systems under the
  {Halfin and Whitt} regime.
\newblock \textit{Mathematics of Operations Research}, \textbf{33} 51--90.
\newblock
  \urlprefix\url{http://search.proquest.com/docview/212618995?accountid=10267}.

\bibitem[{Ward and Glynn(2003)}]{GlynWard2003}
\textsc{Ward, A.} and \textsc{Glynn, P.} (2003).
\newblock A diffusion approximation for a markovian queue with reneging.
\newblock \textit{Queueing Systems}, \textbf{43} 103--128.
\newblock \urlprefix\url{http://dx.doi.org/10.1023/A%3A1021804515162}.

\bibitem[{Ward(2012)}]{Ward2012}
\textsc{Ward, A.~R.} (2012).
\newblock Asymptotic analysis of queueing systems with reneging: {A} survey of
  results for {FIFO}, single class models.
\newblock \textit{Surveys in Operations Research and Management Science},
  \textbf{17} 1 -- 14.
\newblock
  \urlprefix\url{http://www.sciencedirect.com/science/article/pii/S1876735411000237}.

\bibitem[{Whitt(1971)}]{Whit1971}
\textsc{Whitt, W.} (1971).
\newblock Weak convergence theorems for priority queues: preemptive-resume
  discipline.
\newblock \textit{J.\ Appl.\ Probab.}, \textbf{8} 74--94.

\bibitem[{Whitt(2002)}]{Whit2002}
\textsc{Whitt, W.} (2002).
\newblock \textit{Stochastic-process limits}.
\newblock Springer, New York.

\bibitem[{Whitt(2003)}]{Whit2003}
\textsc{Whitt, W.} (2003).
\newblock How multiserver queues scale with growing congestion-dependent
  demand.
\newblock \textit{Operations Research}, \textbf{51} 531--542.

\bibitem[{Williams(1998)}]{Will1998a}
\textsc{Williams, R.~J.} (1998).
\newblock Diffusion approximations for open multiclass queueing networks:
  sufficient conditions involving state space collapse.
\newblock \textit{Queueing Systems}, \textbf{30} 27--88.

\bibitem[{Ye and Yao(2012)}]{YaoYe2012}
\textsc{Ye, H.-Q.} and \textsc{Yao, D.~D.} (2012).
\newblock A stochastic network under proportional fair resource
  control---diffusion limit with multiple bottlenecks.
\newblock \textit{Operations Research}, \textbf{60} 716--738.
\newblock \urlprefix\url{http://dx.doi.org/10.1287/opre.1120.1047}.

\bibitem[{Ying(2016{\natexlab{a}})}]{ying2016}
\textsc{Ying, L.} (2016{\natexlab{a}}).
\newblock On the approximation error of mean-field models.
\newblock In \textit{Proceedings of the 2016 ACM SIGMETRICS International
  Conference on Measurement and Modeling of Computer Science}. ACM, Antibes
  Juan-les-Pins, France, 285--297.
\newblock \urlprefix\url{http://dx.doi.org/10.1145/2964791.2901463}.

\bibitem[{Ying(2016{\natexlab{b}})}]{ying2016b}
\textsc{Ying, L.} (2016{\natexlab{b}}).
\newblock On the rate of convergence of the power-of-two-choices to its
  mean-field limit.
\newblock \urlprefix\url{http://arxiv.org/abs/1605.06581}.

\bibitem[{Zhang et~al.(2012)Zhang, van Leeuwaarden and
  Zwart}]{LeeuZhanZwar2012}
\textsc{Zhang, B.}, \textsc{van Leeuwaarden, J.} and \textsc{Zwart, B.} (2012).
\newblock Staffing call centers with impatient customers: refinements to
  many-server asymptotics.
\newblock \textit{Operations Research}, \textbf{60} 461--474.
\newblock \urlprefix\url{http://dx.doi.org/10.1287/opre.1110.1016}.

\bibitem[{Zhang and Zwart(2008)}]{ZhanZwar2008}
\textsc{Zhang, J.} and \textsc{Zwart, B.} (2008).
\newblock Steady state approximations of limited processor sharing queues in
  heavy traffic.
\newblock \textit{Queueing Systems: Theory and Applications}, \textbf{60}
  227--246.
\newblock \urlprefix\url{http://dx.doi.org/10.1007/s11134-008-9095-4}.

\end{thebibliography}

\end{document}